\documentclass[leqno,11pt,oneside,a4paper]{amsart}
\pdfoutput=1
\usepackage[textwidth=14cm,textheight=23cm]{geometry} 
\usepackage[english]{babel}
\usepackage[utf8]{inputenc}
\usepackage{mathtools}
\usepackage{mathabx,amsmath,amssymb,amsthm,stmaryrd,appendix,fontawesome,accents}
\usepackage{microtype}
\usepackage{eqparbox}
\usepackage{graphicx,xcolor}
\usepackage{float}
\usepackage{etoolbox}
\patchcmd{\thmhead}{(#3)}{#3}{}{} 

\usepackage[all,2cell]{xy}

\usepackage{datetime}

\usepackage{enumitem}
\setlist[enumerate]{label=(\textit{\roman*}\hspace{.08em}),font=\rm,itemsep=.25em}

\newcommand{\itemii}{\mbox{(\textit{ii}\hspace{.08em})}}

\newcommand{\itemiv}{\mbox{(\textit{iv})}}

\newcommand{\hair}{\ifmmode\mskip1.5mu\else\kern0.08em\fi}

\usepackage[colorlinks,pagebackref]{hyperref}
\hypersetup{
  pdftitle   = {Partially wrapped Fukaya categories of orbifold surfaces},
  pdfauthor  = {Severin Barmeier, Sibylle Schroll, Zhengfang Wang},
  pdfcreator = {},
  linkcolor  = blue,
  citecolor  = blue
}
\usepackage{tikz}
\usetikzlibrary{backgrounds,matrix,arrows.meta,decorations.markings,decorations.pathmorphing,shapes,positioning,calc}
\tikzset{>=stealth}
\renewcommand{\to}{\mathrel{\tikz[baseline]\draw[ ->,line width=.4pt] (0ex,0.65ex) -- (3ex,0.65ex);}}
\renewcommand{\mapsto}{\mathrel{\tikz[baseline]\draw[|->,line width=.4pt] (0pt,0.65ex) -- (3ex,0.65ex);}}

\newcommand{\doublerightarrow}{\mathrel{\tikz[baseline]\path[->,line width=.4pt] (0ex,1ex) edge (3ex,1ex) (0ex,0.3ex) edge (3ex,0.3ex);}}
\newcommand{\triplerightarrow}{\mathrel{\tikz[baseline]\path[->,line width=.4pt] (0ex,1.35ex) edge (3ex,1.35ex) (0ex,0.65ex) edge (3ex,0.65ex) (0ex,-0.05ex) edge (3ex,-0.05ex);}}

\renewcommand{\hookrightarrow}{\mathrel{\tikz[baseline]\draw[{Hooks[right]}->,line width=.4pt] (0ex,0.65ex) -- (3ex,0.65ex);}}
\renewcommand{\hookleftarrow}{\mathrel{\tikz[baseline]\draw[{Hooks[left]}->,line width=.4pt] (3ex,0.65ex) -- (0ex,0.65ex);}}
\renewcommand{\leftrightarrow}{\mathrel{\tikz[baseline]\draw[<->,line width=.4pt] (0ex,0.65ex) -- (3.5ex,0.65ex);}}
\renewcommand{\rightsquigarrow}{\mathrel{\tikz[baseline]\draw[|->,line width=.4pt,line join=round,decorate,decoration={zigzag,segment length=4,amplitude=.9,pre=lineto,pre length=.5pt,post=lineto,post length=3pt}] (0ex,0.65ex) -- (3.5ex,0.65ex);}}
\renewcommand{\leftsquigarrow}{\mathrel{\tikz[baseline]\draw[|->,line width=.4pt,line join=round,decorate,decoration={zigzag,segment length=4,amplitude=.9,pre=lineto,pre length=.5pt,post=lineto,post length=3pt}] (3.5ex,0.65ex) -- (0ex,0.65ex);}}

\newcommand{\toarg}[1]{\mathrel{\tikz[baseline]\path[->,line width=.4pt] (0ex,0.65ex) edge node[above=-.4ex, overlay, font=\scriptsize] {$#1$} (3.5ex,.65ex);}}

\newcommand{\toarglong}[1]{\mathrel{\tikz[baseline]\path[->,line width=.4pt] (0ex,0.65ex) edge node[above=-.4ex, overlay, font=\scriptsize] {$#1$} (8ex,.65ex);}}

\newcommand{\tosim}{\mathrel{\tikz[baseline] \path[->,line width=.4pt] (0ex,0.65ex) edge node[above=-.9ex, overlay, font=\normalsize, pos=.45] {$\sim$} (3.5ex,.65ex);}}

\definecolor{stopcolour}{HTML}{2D67B1}
\definecolor{arccolour}{HTML}{00AB57}
\definecolor{ribboncolour}{HTML}{F7971C}

\renewcommand{\theta}{\vartheta}
\renewcommand{\phi}{\varphi}
\renewcommand{\emptyset}{\varnothing}

\numberwithin{equation}{section}
\newtheorem{theorem}[equation]{Theorem}
\newtheorem*{theorem*}{Theorem}
\newtheorem{proposition}[equation]{Proposition}
\newtheorem{lemma}[equation]{Lemma}
\newtheorem*{lemma*}{Lemma}
\newtheorem{corollary}[equation]{Corollary}
\newtheorem{conjecture}[equation]{Conjecture}
\newtheorem*{corollary*}{Corollary}
\theoremstyle{definition}
\newtheorem{definition}[equation]{Definition}
\newtheorem{example}[equation]{Example}

\newtheorem{notation}[equation]{Notation}
\theoremstyle{remark}
\newtheorem{remark}[equation]{Remark}

\renewcommand{\b}{\mathrm{b}}

\DeclareMathOperator{\val}{val}

\renewcommand{\H}{\mathrm{H}}

\DeclareMathOperator{\HH}{HH}
\DeclareMathOperator{\hocolim}{hocolim}
\DeclareMathOperator{\Hom}{Hom}

\DeclareMathOperator{\Mod}{Mod}
\newcommand{\id}{\mathrm{id}}

\newcommand{\op}{\mathrm{op}}
\DeclareMathOperator{\per}{per}

\newcommand{\s}{\mathrm{s}}
\DeclareMathOperator{\Sing}{Sing}

\DeclareMathOperator{\tw}{tw}

\newcommand{\mucirc}{\accentset{\circ}{\mu}}
\newcommand{\mutimes}{\accentset{\times}{\mu}}
\newcommand{\mubar}{\bar\mu}

\makeatletter
\newcommand{\superimpose}[2]{{%
  \ooalign{%
    \hfil$\m@th#1\@firstoftwo#2$\hfil\cr
    \hfil$\m@th#1\@secondoftwo#2$\hfil\cr
  }%
}}
\makeatother

\DeclareRobustCommand{\vcirc}{\accentset{\circ}{v}}
\DeclareRobustCommand{\vtimes}{\accentset{\times}{v}}
\DeclareRobustCommand{\vbullet}{\accentset{\bullet}{v}}
\DeclareRobustCommand{\vodot}{\accentset{\odot}{v}}

\begin{document}


\title[Partially wrapped Fukaya categories of orbifolds]{Partially wrapped Fukaya categories of orbifold surfaces}

\author{Severin Barmeier}
\address{Mathematical Institute, University of Cologne, Weyertal 86-90, 50931 Köln, Germany}
\email{s.barmeier@gmail.com}
\author{Sibylle Schroll}
\address{Mathematical Institute, University of Cologne, Weyertal 86-90, 50931 Köln, Germany \textit{and} Norwegian University of Science and Technology (NTNU), Department of Mathematical Sciences, 7491 Trondheim, Norway}
\email{schroll@math.uni-koeln.de}
\author{Zhengfang Wang}
\address{Nanjing University, School of Mathematics, Nanjing 210093, Jiangsu, China \\ \textup{and} Universität Stuttgart, Institut für Algebra und Zahlentheorie, Pfaffen\-wald\-ring 57, 70569 Stuttgart, Germany}
\email{zhengfangw@nju.edu.cn}


\begin{abstract}
We give a complete description of partially wrapped Fukaya categories of graded orbifold surfaces with stops. We show that a construction via global sections of a natural cosheaf of A$_\infty$ categories on a Lagrangian core of the surface is equivalent to a global construction via the (equivariant) orbit category of a smooth cover. We therefore establish the local-to-global properties of partially wrapped Fukaya categories of orbifold surfaces closely paralleling a proposal by Kontsevich for Fukaya categories of smooth Weinstein manifolds.

From the viewpoint of Weinstein sectorial descent in the sense of Ganatra, Pardon and Shende, our results show that orbifold surfaces also have Weinstein sectors of type $\mathrm D$ besides the type $\mathrm A$ or type $\widetilde{\mathrm A}$ sectors on smooth surfaces.

We describe the global sections of the cosheaf explicitly for any generator given by an admissible dissection of the orbifold surface and we give a full classification of the formal generators which arise in this way. This shows in particular that the partially wrapped Fukaya category of an orbifold surface can always be described as the perfect derived category of a graded associative algebra. We conjecture that associative algebras obtained from dissections of orbifold surfaces form a new class of associative algebras closed under derived equivalence.
\end{abstract}

\maketitle

\setcounter{tocdepth}{1} 
\tableofcontents

\section*{Introduction}

In this paper we give a complete description of partially wrapped Fukaya categories of orbifold surfaces. Our results establish the equivalence of several natural descriptions of these Fukaya categories and we give a concrete description via formal generators. The underlying surfaces we consider are {\it graded orbifold surfaces with stops} which consist of an orbifold surface $S$ with nonempty boundary and a nonempty proper closed subset $\Sigma \subset \partial S$ together with a grading structure given by a (global) line field $\eta$. We show that the partially wrapped Fukaya category $\mathcal W (\mathbf S)$ of $\mathbf S = (S, \Sigma, \eta)$ admits the following three equivalent descriptions:
\begin{align*}
\mathcal W (\mathbf S) &\simeq \H^0 (\mathbf \Gamma (\mathcal T_{\mathrm G})^\natural) && \text{global sections of a cosheaf $\mathcal T_{\mathrm G}$ of A$_\infty$ categories} \\[-.25em]
&&& \text{on a ribbon graph $\mathrm G$ of $\mathbf S$} \\
&\simeq (\mathcal W (\widetilde{\mathbf S}) / \mathbb Z_2)^\natural && \text{orbit category of a double cover $\widetilde{\mathbf S}$ of $\mathbf S$} \\
&\simeq \H^0 (\tw (\mathbf A_\Delta)^\natural) && \text{twisted complexes of an arc system $\Delta$ for $\mathbf S$}
\end{align*}
The above are equivalences of triangulated categories induced by Morita equivalences at the pretriangulated level. Here $\H^0$ denotes the homotopy category (triangulated envelope) of a pretriangulated A$_\infty$ category (or, equivalently, of a $\Bbbk$-linear $\infty$-category) and $^\natural$ idempotent completion. In particular, our results show that any of the categories on the right-hand side may be taken as a definition of $\mathcal W (\mathbf S)$. Our point of departure is the definition of the Fukaya category as the category of global sections of a cosheaf of categories on a (usually singular) Lagrangian core. A (co)sheaf-theoretical approach to Fukaya categories of Weinstein manifolds was proposed by Kontsevich \cite{kontsevich2} and this proposal has led to much progress in (co)sheaf-theoretical techniques in Fukaya categories, see for example \cite{sibillatreumannzaslow,lee,haidenkatzarkovkontsevich,dyckerhoffkapranov,pascaleffsibilla} for the case of surfaces.

The theory of partially wrapped Fukaya categories was introduced by Auroux \cite{auroux1,auroux2} (see also \cite{sylvan,ganatrapardonshende1}). It refines the (fully) wrapped Fukaya categories of noncompact symplectic manifolds or Liouville domains of Abouzaid and Seidel \cite{abouzaidseidel}. The wrapping at infinity (partial or full) is needed for a workable intersection theory of (possibly noncompact) exact Lagrangians. The classical theories of Landau--Ginzburg models $f \colon X \to \mathbb C$ \cite{orlov} and Fukaya--Seidel categories of symplectic Lefschetz fibrations \cite{seidel2}, which have played a central role in the work motivated by Kontsevich's Homological Mirror Symmetry Conjecture \cite{kontsevich1}, can also be understood from this viewpoint. Stop removal functors relate partially wrapped Fukaya categories with different sets of stops, the empty stop corresponding to the fully wrapped case \cite{ganatrapardonshende2}.

The foundational work of Ganatra, Pardon and Shende \cite{ganatrapardonshende1,ganatrapardonshende2,ganatrapardonshende3} shows among many things that the categories forming the cosheaf in Kontsevich's proposal can themselves be viewed as partially wrapped Fukaya categories. Given an open cover $\{ X_i \}_{i \in I}$ of a Weinstein manifold $X$ by Weinstein sectors $X_i$, the partially wrapped Floer theory (pseudo-holomorphic disks and partial wrapping) entering the partially wrapped Fukaya categories of the individual sectors do not probe the larger space $X$. Yet, \cite{ganatrapardonshende1,ganatrapardonshende2} show that the wrapped Fukaya category of $X$ can be described as the homotopy colimit
\[
\hocolim_{J \subset I} \mathcal W \bigl( \textstyle \bigcap_{i \in J} U_i \bigr) \to \mathcal W (X)
\]
thus proving the local-to-global property of $\mathcal W (X)$ envisioned by Kontsevich \cite{kontsevich2} and shedding light on its Floer-theoretic origins.

In this paper we prove an analogous local-to-global description for surfaces with orbifold singularities and show its equivalence with the global-to-local construction via orbit categories. There is a growing body of evidence that Floer theoretic methods and Fukaya categories should admit an extension to symplectic manifolds with singularities and the goal of this paper is to lay solid foundations in the case of graded orbifold surfaces with stops. Symplectic manifolds with orbifold singularities and their Fukaya categories have made an appearance in various settings, for example in the Lagrangian intersection theory in the pillowcase (a $2$-sphere with four orbifold points of order $2$) and its relation to instanton knot homology by Hedden, Herald and Kirk \cite{heddenheraldkirk1,heddenheraldkirk2}, in the context of orbifold relative Fukaya categories for projective hypersurfaces appearing in work of Sheridan \cite{sheridan1,sheridan2}, or in the Lagrangian Floer theory on orbifolds in work of Chen, Ono and Wang \cite{chenonowang}. A central consideration is the Floer theory of invariant Lagrangians in symplectic manifolds with a finite group of symplectomorphisms as for example studied by Seidel and Smith \cite{seidelsmith} and Wu \cite{wu}.

A concrete description of the partially wrapped Fukaya category of graded smooth surfaces was given by Haiden, Katzarkov and Kontsevich \cite{haidenkatzarkovkontsevich}. In the setup of \cite{ganatrapardonshende1,ganatrapardonshende2} a ribbon graph for a graded smooth surface with stops $\mathbf S$ determines an open cover of Weinstein sectors symplectomorphic to a closed disk $\overline{\mathbb D}$ with $n$ stops in the boundary circle, where $n$ is the valency of the vertex in the ribbon graph. Our results show that a ribbon graph for a graded orbifold surface with stops determines an open cover of an orbifold surface by Weinstein sectors of type $\mathrm A_{n-1}$ (corresponding to a smooth disk $\overline{\mathbb D}$ with $n$ stops), type $\widetilde{\mathrm A}_{n-1}$ (corresponding to a smooth annulus with one full boundary stop $\mathrm S^1$ and $n$ stops in the other boundary component) or of type $\mathrm D_{n+1}$ (corresponding to an orbifold disk $\overline{\mathbb D} / \mathbb Z_2$ with $n$ stops).

Another motivation for our work is the study of the deformation theory of partially wrapped Fukaya categories of smooth surfaces. In a companion paper \cite{barmeierschrollwang} we prove that partially wrapped Fukaya categories of graded orbifold surfaces (described in this paper) are the geometric manifestation of algebraic deformations of partially wrapped Fukaya categories of graded {\it smooth} surfaces. This is a first step towards clarifying the role of stop data in the relation between A$_\infty$ deformations of Fukaya categories and partial compactifications of the underlying symplectic manifolds outlined in Seidel's ICM 2002 address \cite{seidel1}.

In \cite{haidenkatzarkovkontsevich} it was shown that partially wrapped Fukaya categories of smooth surfaces admit a formal generator whose endomorphism algebra is a graded gentle algebra. Gentle algebras are associative algebras given as path algebras of certain quivers with quadratic monomial relations. They were introduced by Assem and Skowroński \cite{assemskowronski} and (especially in the ungraded case) have a rich body of representation-theoretic literature. Ribbon graphs associated to (ungraded) gentle algebras appeared in \cite{schroll} as Brauer graphs associated to trivial extensions of gentle algebras. The corresponding smooth surfaces give a geometric model of the derived category of the gentle algebra \cite{opperplamondonschroll,amiotplamondonschroll} as well as of the (Abelian) module category \cite{baurcoelho,chang}. Gentle algebras arising from surfaces triangulations appear also in cluster theory \cite{assembruestlecharbonneaujodoinplamondon,labardini} and in (fully) wrapped Fukaya categories \cite{bocklandt1,bocklandt2,vandekreeke}. The results of \cite{haidenkatzarkovkontsevich} provided new impetus for studying also the graded case and graded gentle algebras and their relationship to partially wrapped Fukaya categories were further studied by Lekili and Polishchuk \cite{lekilipolishchuk2} also in the context of mirror symmetry for surfaces \cite{lekilipolishchuk1}. (See also \cite{abouzaidauroux,pascaleffsibilla} for mirror symmetry in the fully wrapped case.) The results of \cite{haidenkatzarkovkontsevich} also provide a large class of triangulated categories whose stability manifolds can be described explicit as moduli of quadratic differentials, see also  \cite{haiden,ikedaqiu,christhaidenqiu}.

That an extension of these results to orbifold surfaces should exist may be anticipated from the representation-theoretic literature on skew-gentle algebras which were introduced by Geiß and de la Peña \cite{geisspena}. If one regards a quiver of type $\mathrm A_n$ as the simplest kind of gentle algebra, then a quiver of type $\mathrm D_n$ is the simplest kind of skew-gentle algebra which is not gentle. Some of the above-mentioned results for gentle algebras and smooth surfaces have seen generalizations to skew-gentle algebras where the associated surface contains orbifold points of order $2$. For example, \cite{amiotbruestle,labardinifragososchrollvaldivieso} give a surface model for (ungraded) skew-gentle algebras in terms of orbifold surfaces, \cite{hezhouzhu} give a geometric model for the module category and \cite{qiuzhangzhou} a surface model for the derived category also in the graded case.

In our work, skew-gentle algebras appear as cohomology algebras of particular formal generators of $\mathcal W (\mathbf S)$ for a graded orbifold surface $\mathbf S$. In the case of graded smooth surfaces with stops, the algebras associated to formal arc systems on the surface are always gentle algebras as shown in \cite{haidenkatzarkovkontsevich}. Our classification of formal generators obtained from arc systems shows that this is no longer the case in the case of orbifold surfaces, that is, $\mathcal W (\mathbf S)$ admits formal generators whose resulting algebras need no longer be skew-gentle and we conjecture that the algebras obtained in this way form a class of associative algebras closed under derived equivalence.

\section{Main results}

We now give a summary of the main results of this paper. A {\it graded orbifold surface with stops} $\mathbf S = (S, \Sigma, \eta)$ consists of an orbifold surface $S$ with nonempty boundary, a proper nonempty closed subset $\Sigma \subset \partial S$ of stops and a grading structure $\eta$ given by a line field on $S$. In Lemma \ref{lemma:linefield} we show that the existence of a line field on $S$ is equivalent to all orbifold points having order $2$.

\subsection{Fukaya categories of orbifold disks as local model}

In Section \ref{section:disk} we give an explicit description of the partially wrapped Fukaya category of a disk with one orbifold point via the A$_\infty$ orbit category of its double cover. We describe the endomorphism algebra of a particular type of generator of this category and show that it carries a unique higher product (Proposition \ref{proposition:orbifolddisk}) which in combination with Theorem~\ref{theorem:moritaequivalenceadmissible} shows that it is derived equivalent to a linearly oriented quiver of type $\mathrm D$.

\begin{theorem}
[(Corollary~\ref{corollary:typeD} and Theorem~\ref{theorem:moritaequivalenceadmissible})]
Let $\mathbf S$ be a graded orbifold disk with $n$ stops and one orbifold point of order $2$.
Then we have a triangulated equivalence
\[
\mathcal W (\mathbf S) \simeq \per (A)
\]
where $A \simeq \Bbbk Q$ is the (ungraded) path algebra of a linearly ordered quiver $Q$ of type $\mathrm D_{n+1}$ and $\per (A)$ is the perfect derived category of $A$.
\end{theorem}

In Section \ref{section:cosheaves} we use this local model as an input for a cosheaf construction of the partially wrapped Fukaya category of any graded orbifold surface in the spirit of Kontsevich's conjectural description for Weinstein manifolds \cite{kontsevich2} which was proved in \cite{ganatrapardonshende2}.

\subsection{Fukaya categories via cosheaves of A$_\infty$ categories}

We give two constructions which generalize the duality between ribbon graphs of a graded smooth surface and surface dissections into polygons to the case of orbifold surfaces. One is a ribbon graph with a set of distinguished vertices corresponding to the orbifold points which may be viewed as a Lagrangian core of the orbifold surface. The other is what we call a {\it ribbon complex} which is an analogue of a cell complex of dimension $2$ with a ribbon graph structure on its $1$-skeleton. The ribbon complex also contains orbifold $2$-cells which correspond to the orbifold points of the surface. Such a complex has recently been considered in \cite{bahrinotbohmsarkarsong} to compute the integral cohomology of orbifolds.

We define the {\it partially wrapped Fukaya category} of a graded orbifold surface $\mathbf S$ as the triangulated category
\[
\mathcal W (\mathbf S) := \H^0 (\mathbf \Gamma (\mathcal T_{\mathbb G}))
\]
associated to the global sections $\mathbf \Gamma (\mathcal T_{\mathbb G})$ of a cosheaf $\mathcal T_{\mathbb G}$ of pretriangulated A$_\infty$ categories on a ribbon complex $\mathbb G$ for $\mathbf S$ (Definition \ref{definition:fukayacategory}). These global sections are the {\it homotopy colimit} of a diagram of pretriangulated A$_\infty$ categories viewed as a diagram in the homotopy category of DG categories or (equivalently) in the $\infty$-category of $\Bbbk$-linear $\infty$-categories. More precisely, we have a cosheaf $\mathcal E_\Delta$ of A$_\infty$ categories on $\mathbb G = \mathbb G (\Delta)$ associated to a dissection $\Delta$ of $\mathbf S$ whose associated categories of twisted complexes form the cosheaf $\mathcal T_{\mathbb G} = \tw (\mathcal E_\Delta)^\natural$ where $^\natural$ denotes the idempotent completion. The cosheaf $\mathcal E_\Delta$ uses the local model of the orbifold disk from Section \ref{section:disk} as one of the inputs and up to Morita equivalence its category $\mathbf \Gamma (\mathcal E_\Delta)$ of global sections, and hence the partially wrapped Fukaya category of $\mathbf S$, is independent of the particular dissection $\Delta$ (Theorem \ref{theorem:moritaequivalenceglobalsection}).

We then show that there is an equivalent construction via a cosheaf on a (1-dimensional) Lagrangian core corresponding to a ribbon {\it graph} of $\mathbf S$.

\begin{theorem}[(Theorem \ref{theorem:ribbongraphinterpretation})]
\label{main:ribbongraph}
Let $\mathbf S = (S, \Sigma, \eta)$ be a graded orbifold surface with stops and let $\mathbb G$ be a ribbon complex for $\mathbf S$. For any ribbon {\it graph} $\mathrm G$ of $\mathbf S$ there exists a cosheaf $\mathcal T_{\mathrm G}$ of pretriangulated A$_\infty$ categories on $\mathrm G$ such that $\mathbf \Gamma (\mathcal T_{\mathbb G}) \simeq \mathbf \Gamma (\mathcal T_{\mathrm G})$. In particular, 
\[
\mathcal W (\mathbf S) = \H^0 (\mathbf \Gamma (\mathcal T_{\mathbb G})) \simeq \H^0 (\mathbf \Gamma (\mathcal T_{\mathrm G})).
\]
\end{theorem}

This can be viewed as an analogue of Kontsevich's proposal for orbifold surfaces, generalizing the case of smooth surfaces described explicitly in \cite{haidenkatzarkovkontsevich}.

The Weinstein sectors of $\mathbf S$ corresponding to the open cover of the ribbon graph $\mathrm G$ contain sectors of type $\mathrm D$ (corresponding to orbifold disks) in addition to the sectors of type $\mathrm A$ and type $\widetilde{\mathrm A}$ for smooth surfaces (see \S\ref{subsection:core}).

\subsection{Fukaya categories via dissections of orbifold surfaces}

We compute the homotopy colimit used in the definition of $\mathcal W (\mathbf S)$ explicitly by introducing the notion of an {\it admissible dissection} (Definition \ref{definition:admissible}). This allows us to give a concrete description of $\mathcal W (\mathbf S)$ in terms of a generator and we describe all of the higher structures on its endomorphism algebra explicitly.

\begin{theorem}[(Proposition \ref{proposition:ainfinity} and Theorem \ref{theorem:moritapartiallywrapped})]
\label{theoremintroduction:admissible}
Let $\mathbf S = (S, \Sigma, \eta)$ be a graded orbifold surface with stops and let $\Delta$ be an admissible dissection of $\mathbf S$. Viewing the arcs in the dissection $\Delta$ as objects and the boundary and orbifold paths as morphisms, there is an A$_\infty$ category $\mathbf A_\Delta$ whose higher multiplications $\mu_n$ for $n \geq 1$ are of the following three types:
\begin{itemize}
\item $\mubar_2$, the (associative) concatenation of boundary or orbifold paths
\item $\mucirc_n$ for $n \geq 2$ arising from smooth disk sequences in the dissection $\Delta$
\item $\mutimes_n$ for $n \geq 1$ arising from orbifold disk sequences in $\Delta$.
\end{itemize}
Then $\tw (\mathbf A_\Delta)^\natural$ is equivalent to the homotopy colimit $\mathbf \Gamma (\mathcal T_{\mathbb G}) \simeq \mathbf \Gamma (\mathcal T_{\mathrm G})$ in Theorem \ref{main:ribbongraph} inducing a triangulated equivalence
\[
\mathcal W (\mathbf S) = \H^0 (\mathbf \Gamma (\mathcal T_{\mathbb G})) \simeq \H^0 (\tw (\mathbf A_\Delta)^\natural).
\]
\end{theorem}

The multiplications $\mubar_2$ and $\mucirc_{\geq 2}$ already appeared in \cite{haidenkatzarkovkontsevich}. The products $\mutimes_{\geq 1}$ are a new type of higher product that appear for certain polygons (pseudo-holomorphic disks) whose arcs (Lagrangians) intersect at orbifold points.

Using the notion of \emph{tagged arc system} on an orbifold surface, an A$_\infty$ category was independently constructed in \cite{chokim}. This category appears to be Morita equivalent to the category $\mathbf A_\Delta$ in Theorem~\ref{theoremintroduction:admissible} (cf.\ Remarks \ref{remark:withck} and \ref{remark:tagging}).

\subsection{Fukaya categories via A$_\infty$ orbit categories}

Our definition of partially wrapped Fukaya categories of orbifold surfaces builds on the local-to-global principle of Fukaya categories for noncompact manifolds. On the other hand, a graded orbifold surface $\mathbf S$ with nonempty boundary arises as a global quotient of a (smooth) double cover $\widetilde{\mathbf S}$ equipped with an almost free $\mathbb Z_2$-action. The partially wrapped Fukaya category $\mathcal W (\widetilde{\mathbf S})$ inherits a $\mathbb Z_2$-action. Therefore it is natural to relate $\mathcal W (\mathbf S)$ to $\mathcal W (\widetilde{\mathbf S})$ via the notion of (A$_\infty$) orbit categories. The following result proves a global-to-local principle for partially wrapped Fukaya categories of orbifold surfaces.

\begin{theorem}[(Theorem \ref{theorem:doublecoverorbitcategory})]
\label{main:doublecoverorbitcategory}
Let $\mathbf S = (S, \Sigma, \eta)$ be a graded orbifold surface with stops and let $\widetilde{\mathbf S}$ be a double cover of $\mathbf S$. There exists a weakly admissible dissection $\Delta$ of $\mathbf S$ and a weakly admissible dissection $\widetilde \Delta$ of $\widetilde{\mathbf S}$ such that
\[
(\tw (\mathbf A_\Delta))^\natural \simeq (\tw (\mathbf A_{\widetilde \Delta}) / \mathbb Z_2)^\natural \simeq (\tw (\mathbf A_{\widetilde \Delta} / \mathbb Z_2))^\natural
\]
are equivalences of pretriangulated A$_\infty$ categories, the middle term being an A$_\infty$ orbit category. These equivalences induce a triangulated equivalence
\[
\mathcal W (\mathbf S) \simeq (\mathcal W (\widetilde{\mathbf S}) / \mathbb Z_2)^\natural.
\]
\end{theorem}

This global point of view is the perspective taken in \cite{amiotplamondon}, in particular, to construct tilting objects in the derived categories of skew-gentle algebras.

\subsection{Classification of formal generators}

We give a full classification of the formal generators of $\mathcal W (\mathbf S)$ arising from dissections of $\mathbf S$. For this purpose we introduce the notion of a {\it DG dissection} and a {\it formal dissection} (Definition \ref{definition:DGformal}) and show the following.

\begin{theorem}[(Theorem \ref{theorem:dgtriangleequivalent} and Theorem \ref{theorem:formaldg})]
Let $\mathbf S = (S, \Sigma, \eta)$ be a graded orbifold surface with stops.
\begin{enumerate}
\item Let $\Delta$ be a DG dissection of $\mathbf S$. Then the A$_\infty$ category $\mathbf A_\Delta$ is a DG category and we have a triangulated equivalence
\[
\mathcal W (\mathbf S) \simeq \per (\mathbf A_\Delta)
\]
where $\per (\mathbf A_\Delta)$ denotes the perfect derived category of $\mathbf A_\Delta$.
\item Let $\Delta$ be a weakly admissible dissection and suppose that $\mathbf S$ is not a disk with two orbifold points and one stop in the boundary. Then the A$_\infty$ category $\mathbf A_\Delta$ is formal if and only if $\Delta$ is a formal dissection of $\mathbf S$. In this case we have triangulated equivalences
\[
\mathcal W (\mathbf S) \simeq \per (\mathbf A_\Delta) \simeq \per (\H^\bullet (\mathbf A_\Delta))
\]
where $\H^\bullet (\mathbf A_\Delta)$ is the (associative graded) cohomology algebra of $\mathbf A_\Delta$.
\end{enumerate}
\end{theorem}

\subsection{Derived equivalences of associative algebras}

Our classification of formal dissections shows that $\mathcal W (\mathbf S)$ always admits a formal generator whose cohomology is a graded skew-gentle algebra.

\begin{theorem}[(Theorem \ref{theorem:skewgentlealgebras})]
Let $\mathbf S = (S, \Sigma, \eta)$ be a graded orbifold surface. Then $\mathbf S$ admits a formal dissection $\Delta$ such that $A = \H^\bullet (\mathbf A_\Delta)$ is a graded skew-gentle algebra and we have a triangulated equivalence
\[
\mathcal W (\mathbf S) \simeq \per (A).
\]
\end{theorem}

We conjecture that the graded associative algebras arising from formal dissections of graded orbifold surfaces form a new class of algebras closed under derived equivalence.

\begin{conjecture}[(Conjecture \ref{conjecture:skewgentle})]
\label{main:skewgentle}
Let $A$ be a graded skew-gentle algebra and let $\mathbf S = (S, \Sigma, \eta)$ be the graded orbifold surface associated to $A$.

For any graded associative algebra $B$ which is (perfect) derived equivalent to $A$, there exists a formal dissection $\Delta$ of $\mathbf S$ such that $B \simeq \H^\bullet (\mathbf A_\Delta)$ as graded associative algebras.

In particular, the formal generators of $\mathcal W (\mathbf S)$ obtained from dissections give a class of algebras closed under derived equivalence and describe the complete class of graded associative algebras derived equivalent to $A$.
\end{conjecture}

Note that ungraded gentle algebras are closed under derived equivalence by a result of Schröer and Zimmermann \cite{schroeerzimmermann}. Moreover, it is shown in \cite{amiotplamondonschroll} (ungraded case) and \cite{jinschrollwang} (graded case) building on \cite{lekilipolishchuk2} that for a gentle algebra $A$, the homotopy class of the line field on the corresponding surface model gives a complete derived invariant of $A$ generalizing the combinatorial invariant by Avella-Alaminos and Geiss \cite{avellaalaminosgeiss}. In \S\ref{subsection:orbifolddiskthree} we give examples of associative algebras arising from formal dissections of an orbifold disk with three orbifold points. One of these algebras is skew-gentle. The other algebras are not and they do not appear to belong to a specific known class of (derived tame) associative algebras.

\medskip

Shortly before the completion of this paper, a category which appears to be equivalent to the A$_\infty$ category $\mathbf A_\Delta$ we define was constructed independently by Cho and Kim \cite{chokim} using a similar but slightly different approach (cf.\ Remark \ref{remark:withck}). Our independent work on applications to representation theory will now appear jointly \cite{barmeierchokimrhoschrollwang}. Even more recently, Amiot and Plamondon posted independent work studying the orbit category $\mathcal W (\widetilde{\mathbf S}) / \mathbb Z_2$ of the double cover $\widetilde{\mathbf S}$ of an orbifold surface \cite{amiotplamondon}. A consequence of our results is that this category is equivalent to the partially wrapped Fukaya category obtained via a cosheaf construction on any Lagrangian core of the orbifold surface as in Section \ref{section:cosheaves} (see Theorem \ref{theorem:doublecoverorbitcategory}).

\subsection{Structure of the paper}

In Section \ref{section:ainfinitytwisted} we recall some general properties of A$_\infty$ categories and twisted complexes. In \S\ref{subsection:orbitcategories} we introduce the notion of {\it A$_\infty$ orbit categories} generalizing the orbit categories in the sense of \cite{cibilsmarcos,keller2}. In Section \ref{section:orbifoldsurfaces} we define graded orbifold surfaces with stops whose partially wrapped Fukaya categories are the main subject of this paper. In Section \ref{section:disk}, we give a complete description of the partially wrapped Fukaya of a disk with a single orbifold point and any number of boundary stops. The results of Section \ref{section:disk} are used as the main input for the case of general orbifold surfaces. In Section \ref{section:arcsystems} we give the formal definitions of arcs, dissections and ribbon complexes for general graded orbifold surfaces with stops.

In Section \ref{section:cosheaves} we define the partially wrapped Fukaya category of a graded orbifold surface with stops in terms of the category of global sections of a cosheaf of pretriangulated A$_\infty$ categories on the ribbon complex associated to any dissection. In \S\ref{subsection:core} we give an equivalent construction on a ribbon graph in spirit of Kontsevich's conjecture \cite{kontsevich2}. In Section \ref{section:ainfinity} we construct an explicit A$_\infty$ category $\mathbf A_\Delta$ associated to an admissible dissection $\Delta$ which gives a concrete description of the partially wrapped Fukaya categories of graded orbifold surfaces. In Section \ref{section:formal}, we study formal generators of the partially wrapped Fukaya category. This defines a new family of algebras which are derived equivalent to graded skew-gentle algebras which we conjecture to be closed under derived equivalence. In \S\ref{subsection:orbifolddiskthree} we give concrete examples of new derived equivalences obtained via Fukaya categories. In Section \ref{section:doublecover} we show that the partially wrapped Fukaya categories of graded orbifold surface are equivalent to orbit categories of the partially Fukaya categories of their smooth double covers.

\begin{notation}
Throughout we work over a field $\Bbbk$ and we write $\otimes$ for $\otimes_\Bbbk$. For relating the Fukaya category of the orbifold surface to that of its double cover we need to assume $\operatorname{char} \Bbbk \neq 2$ (see \S\ref{subsection:partiallyorbifolddisk}) and for simplicity we shall make this assumption throughout.
\end{notation}

\section{A$_\infty$ categories and twisted complexes}
\label{section:ainfinitytwisted}

\subsection{A$_\infty$ categories} 
We first recall some notions on A$_\infty$ categories and twisted complexes. For details we refer to \cite{seidel2,keller0,lefevre}.

\subsubsection*{Sign convention and shorthands} To keep the signs to a minimum we shall work with a shifted version of A$_\infty$ categories. This reduces all signs to those coming from the Koszul sign rule. The price to pay for this is the presence of shifts which we denote by $\s$ and which we choose to keep track of explicitly throughout the paper. 

 Let $\mathbf A$ be the data of a set $\mathrm{Ob}(\mathbf A)$ of objects and for any two objects $X, Y$, a graded $\Bbbk$-vector space $\mathbf A (X, Y) = \bigoplus_{i\in \mathbb Z} \mathbf A (X, Y)^i$. We denote by $\s \mathbf A (X, Y)$ the $1$-shifted graded space, namely $\s \mathbf A (X, Y)^i = \mathbf A (X, Y)^{i+1}$. Denoting by $|a| = i$ the degree of $a \in \mathbf A (X, Y)^i$, the degree of $\s a$ is $|\s a| = |a| - 1$. To streamline the notation, we shall also use the following shorthand for the tensor product of shifted Hom-spaces in $\mathbf A$
\[
\s \mathbf A (X_{0 \dotsc n}) := \s \mathbf A (X_{n-1}, X_n) \otimes \s \mathbf A (X_{n-2}, X_{n-1}) \otimes \dotsb \otimes \s \mathbf A (X_0, X_1)
\]
and for elements therein we use the shorthand
\[
\s a_{j \dotsc i+1} := \s a_{j} \otimes \dotsb \otimes \s a_{i+1} \in \s \mathbf A (X_{i \dotsc j})
\]
where $1 \leq i < j \leq n$ and $a_k \in \mathbf A (X_{k-1}, X_k)$. The signs in the A$_\infty$ relations will involve the degree
\[
|\s a_{j \dotsc i+1}| = (|a_j| - 1) + \dotsb + (|a_{i+1}| - 1) = |a_j| + \dotsb + |a_{i+1}| - (j - i)
\]
where our shorthand (left) also comes in handy.

\begin{definition}
An {\it A$_\infty$ category} $\mathbf A$ consists of a set $\mathrm{Ob} (\mathbf A)$ of objects, a graded $\Bbbk$-vector space $\mathbf A (X, Y)$ for each pair $X, Y \in \mathrm{Ob}(\mathbf A)$, and a collection of maps of degree $1$
\[
\mu_n \colon \s \mathbf A (X_{0 \dotsc n}) \to \s \mathbf A (X_0, X_n)
\]
for each $n \geq 1$, satisfying the A$_\infty$ relations
\begin{equation}
\label{eq:ainfinityrelations}
\sum_{i=0}^{n-1}\sum_{j=1}^{n-i} (-1)^{|\s a_{i \dotsc 1}|} \mu_{n-j+1} (\s a_{n \dotsc i+j+1} \otimes \mu_j (\s a_{i+j \dotsc i+1}) \otimes \s a_{i \dotsc 1}) = 0.
\end{equation}
\end{definition}

\begin{remark}
\label{remark:ainfinity}
We may set
\[
(\mu_{n-j+1} \bullet_i \mu_j) (\s a_{n \dotsc 1}) = (-1)^{|\s a_{i \dotsc 1}|} \mu_{n-j+1} (\s a_{n \dotsc i+j+1} \otimes \mu_j (\s a_{i+j \dotsc i+1}) \otimes \s a_{i \dotsc 1})
\]
where $i$ is the number of tensor factors appearing before $\mu_j$, read from the right. If we now define $$\mu_{n-j+1} \bullet \mu_j = \sum_{i = 0}^{n-j} \mu_{n-j+1} \bullet_i \mu_j$$ then \eqref{eq:ainfinityrelations} can be written simply as $\mu \bullet \mu = 0$. Here $\mu = (\mu_n)_{n \geq 1}$ and $\mu \bullet \mu = (\sum_j \mu_{n-j+1} \bullet \mu_j)_{n \geq 1}$ is to be computed componentwise.
\end{remark}

An A$_\infty$ category $\mathbf A$ is {\it strictly unital} if for every object $X$ there exists $1_X \in \mathbf A (X, X)$ in degree $0$ such that 
\begin{enumerate}
\item $\mu_1 (\s 1_X) =0$
\item $\mu_2 (\s a \otimes \s 1_X) = \s a =(-1)^{|\s a|}\mu_2 (\s 1_Y \otimes \s a)$ for any $a \in \mathbf A (X, Y)$
\item $\mu_n (\dotsb \otimes \s 1_X \otimes \dotsb) = 0$ for $n \geq 3$. 
\end{enumerate}
Throughout this paper all A$_\infty$ categories will be strictly unital, unless otherwise stated.

\begin{remark}
If $\mu_n = 0$ for all $n \geq 3$ then an A$_\infty$ category $\mathbf A$ is a DG category (in the usual unshifted sense) with differential and composition maps given by 
\[
d a_1 = (-1)^{|a_1|} \s^{-1} \mu_1 (\s a_1), \qquad a_2 \cdot a_1 = (-1)^{|a_1|} \s^{-1} \mu_2 (\s a_2 \otimes \s a_1)
\]
for any $a_1 \in \mathbf A (X_0, X_1)$ and $a_2 \in \mathbf A (X_1, X_2)$.

As a quick sanity check for the sign convention, let $\mathbf A$ be a DG category with trivial differential. Then the composition (denoted by $\cdot$) of morphisms induces a map $$\mu_2 \colon \s \mathbf A(X_1, X_2) \otimes \s \mathbf A(X_0, X_1) \to \s \mathbf A(X_0, X_2)$$ 
given by $\mu_2(\s a_2 \otimes \s a_1) = (-1)^{|a_1|} \s (a_2 \cdot a_1)$. The associativity of the composition is precisely the condition $\mu \bullet \mu = 0$ since
\begin{align*}
&(\mu \bullet \mu) (\s a_3 \otimes \s a_2 \otimes \s a_1) \\
= {} & \undergroup{\mu_2 (\s a_3 \otimes \mu_2 (\s a_2 \otimes \s a_1))}_{\mu_2 \bullet_0 \mu_2} +  \undergroup{{}  (-1)^{|\s a_1|} \mu_2 (\mu_2 (\s a_3 \otimes \s a_2) \otimes \s a_1)}_{\mu_2 \bullet_1 \mu_2}\\
={} & (-1)^{|a_1| +|a_1a_2|}  \s (a_3\cdot (a_2 \cdot a_1)) + (-1)^{|\s a_1|+|a_1|+|a_2|} \s ((a_3 \cdot a_2) \cdot a_1).
\end{align*}
Since $|\s a_1| = |a_1| - 1$ we obtain $(\mu \bullet \mu) (\s a_3 \otimes \s a_2 \otimes \s a_1) = 0$.
\end{remark}

\begin{definition}\label{definition:ainfinityfunctorquasi}
An {\it A$_\infty$ functor} $F \colon \mathbf A \to \mathbf A'$ is given by a map $F \colon \mathrm{Ob}(\mathbf A) \to \mathrm{Ob}(\mathbf A')$ and a collection of maps of degree $0$
\[
F_n \colon \s \mathbf A (X_{0 \dotsc n}) \to \s \mathbf A' (F (X_0), F (X_n)), \quad n\geq 1
\]
satisfying the following relation for each $n \geq 1$
\begin{align*}
& \sum_{i=0}^{n-1} \sum_{j=1}^{n-i} (-1)^{|\s a_{i \dotsc 1}|} F_{n-j+1} (\s a_{n \dotsc i+j+1} \otimes \mu_j(\s a_{i+j \dotsc i+1}) \otimes \s a_{i \dotsc 1}) \\
 ={}& \sum_{\substack{1 \leq r \leq n \\ n = i_1 + \dotsb + i_r}}  \mu'_r (F_{i_r}(\s a_{n \dotsc i_1+\dotsb+i_{r-1}+1})  \otimes \dotsb \otimes F_{i_2}(\s a_{i_1+i_2 \dotsc i_1+1})  \otimes F_{i_1}(\s a_{i_1 \dotsc 1}))
\end{align*}
where $\mu'_r$ denotes the higher multiplication in $\mathbf A'$.

An A$_\infty$ functor $F \colon \mathbf A \to \mathbf A'$ is an {\it A$_\infty$ quasi-equivalence} if  $F_1$ induces a quasi-isomorphism of complexes between $\mathbf A(X, Y)$ and $\mathbf A'(F(X), F(Y))$ for any objects $X, Y$ and moreover, $\H^0(F)$ is an equivalence of categories. 
\end{definition}

\subsection{Formal A$_\infty$ categories}
\label{subsection:formal}

Let $\mathbf A$ be an A$_\infty$ category. The {\it cohomology category} $\H^\bullet(\mathbf A)$ is a graded (associative) category (with trivial higher products) whose objects are the same as those of $\mathbf A$ and whose morphism space from $X$ to $Y$ is given by the graded vector space 
$$\H^\bullet(\mathbf A)(X, Y) = \bigoplus_{i \in \mathbb Z} \H^i(\mathbf A(X, Y)).$$
The composition of $\H^\bullet(\mathbf A)$ is induced by $\mu_2$ of $\mathbf A$.

\begin{definition}
\label{definition:formal}
An A$_\infty$ category $\mathbf A$ is called {\it formal} if it is A$_\infty$-quasi-equivalent to the cohomology category $\H^\bullet (\mathbf A)$. 
\end{definition}

\begin{remark}
For any A$_\infty$ category $\mathbf A$, the homotopy transfer theorem (see for example \cite[\S 1.4]{lefevre}) endows $\H^\bullet(\mathbf A)$ with an A$_\infty$ structure (called the {\it A$_\infty$ minimal model} of $\mathbf A$) such that the resulting A$_\infty$ category is A$_\infty$-quasi-equivalent to $\mathbf A$. An A$_\infty$ category $\mathbf A$ is formal if and only if its A$_\infty$ minimal model is A$_\infty$-isomorphic to the cohomology category $\H^\bullet(\mathbf A)$ with trivial higher products.
\end{remark}

The following result can be used to show that an A$_\infty$ category $\mathbf A$ is not formal by finding a full subcategory which is not formal.

\begin{proposition}\label{proposition:formalcriterion}
Let $\mathbf A$ be a formal A$_\infty$ category. Then any full A$_\infty$ subcategory $\mathbf B \subset \mathbf A$ is formal.
\end{proposition}

\begin{proof}
Since $\mathbf A$ is formal, there is an A$_\infty$ quasi-equivalence $F\colon \mathbf A \to \H^\bullet(\mathbf A)$ where the latter carries trivial higher products. In particular, $\H^\bullet(F)\colon \H^\bullet(\mathbf A) \tosim \H^\bullet(\mathbf A)$ is an equivalence of graded categories.  Consider the following composition of A$_\infty$ functors 
\[
\mathbf A  \toarg{F} \H^\bullet(\mathbf A) \toarglong{\H^\bullet(F)^{-1}} \H^\bullet(\mathbf A)
\] 
where $\H^\bullet(F)^{-1}$ is the inverse of the equivalence $\H^\bullet(F)$. Note that this composition acts as the identity on objects. It follows that its restriction to $\mathbf B$ induces an A$_\infty$ quasi-equivalence between $\mathbf B$ and $\H^\bullet(\mathbf B)$, which implies that $\mathbf B$ is formal. 
\end{proof}

\subsection{Twisted complexes}

Let $\mathbf A$ be an A$_\infty$ category. To define the A$_\infty$ category $\tw(\mathbf A)$ of twisted complexes, we first introduce the A$_\infty$ category $\mathbb Z \mathbf A$ whose objects are the pairs $(X, m)$ of an object $X$ in $\mathbf A$ and $m \in \mathbb Z$. We usually write $(X, m)$ as $\s^m X$. Morphism spaces are defined by 
\[
\mathbb Z \mathbf A (\s^m X, \s^l Y) = \s^{l-m} \mathbf A (X, Y). 
\]
The A$_\infty$ product of $\mathbb Z \mathbf A$ is obtained by extending the A$_\infty$ product of $\mathbf A$. Precisely, for any objects $\s^{m_0} X_0, \dotsc, \s^{m_n} X_n$ in $\mathbb Z \mathbf A$ and morphisms $a_k \in \mathbb Z \mathbf A (\s^{m_{k-1}} X_{k-1}, \s^{m_k} X_k)$ for $1 \leq k \leq n$ we define 
\[
\mu_n^{\mathbb Z \mathbf A} (\s a_{n \dotsc 1}) = (-1)^{\sum_{i>j} (m_{i-1}-m_i) |\s a_j|} \mu_n^{\mathbf A} ( \s^{m_{n-1}-m_n+1} a_n \otimes \dotsb \otimes \s^{m_0-m_1+1} a_1).
\] 

\begin{definition}
\label{definition:twisted}
A {\it twisted complex} $X^\bullet$ is given by a sequence
\[
(X^1, X^2, \dotsc, X^r) := (\s^{m_1} X_1, \s^{m_2} X_2, \dotsc, \s^{m_r} X_r)
\]
of objects in $\mathbb Z \mathbf A$, together with a matrix $\delta = (\delta^{ij})_{1 \leq i, j \leq r}$ of degree $1$ morphisms
\[
\delta^{ij} \in \mathbb Z \mathbf A (X^j, X^i) = \s^{m_i-m_j} \mathbf A (X_j, X_i)
\]
such that $\delta^{ij} = 0$ for all $i \geq j$ and 
\[
\sum_{n \geq 1} \mu^{\mathbb Z\mathbf A}_n (\s\delta \otimes \dotsb \otimes \s \delta) = 0.
\]
We may write $\delta^{ij} = \s^{m_i-m_j} \delta_{ij}$ for $\delta_{ij} \in \mathbf A (X_j, X_i)$.
\end{definition}

Twisted complexes form an A$_\infty$ category $\tw (\mathbf A)$ and the morphism space from $X^\bullet$ to $Y^\bullet$ in $\tw (\mathbf A)$ is given by 
\[
\bigoplus_{i,j} \mathbb Z\mathbf A (X^j, Y^i).
\]
The A$_\infty$ product of $\tw (\mathbf A)$ is given by the twisted higher product 
\[
\mu_n^{\tw(\mathbf A)}(\s a_{n\dotsc1})  = \sum_{i_0,\dotsc,i_n\geq0} \mu_{n+i_0+\dotsb+i_n}^{\mathbb Z \mathbf A} ((\s\delta)^{\otimes i_n} \otimes \s a_n  \otimes \dotsb \otimes   (\s\delta)^{\otimes i_1}  \otimes \s a_1 \otimes (\s\delta)^{\otimes i_0}).
\]

\begin{remark}
The homotopy category $\H^0(\tw (\mathbf A))$ is triangulated and its shift functor is induced by the functor $\s \colon \tw (\mathbf A) \to \tw (\mathbf A)$ defined as follows. For a twisted complex $X^\bullet$ with components $X^i = \s^{m_i} X_i$ and $\delta^{ij} = \s^{m_i-m_j} \delta_{ij} \colon X^j \to X^i$, the $1$-shifted object is  the twisted complex $\s X^\bullet$ with components $(\s X^\bullet)^i = \s^{m_i+1} X_i$ and $(\s \delta)^{ij} = -\s^{m_i-m_j} \delta_{ij} \colon \s X^j \to \s X^i$.

There is a natural embedding $\mathbf A \to \tw(\mathbf A)$ of A$_\infty$ categories, which sends an object $X$ in $\mathbf A$ to the twisted complex $X^\bullet$ with a single (unshifted) component $X^1 = X$ and $\delta = 0$. 
Let $f \colon X_0 \to X_1$ be a morphism of degree $0$ in $\mathbf A$ such that $\mu_1^{\mathbf A}(\s f) = 0$. The mapping cone of $f$ is defined to be the twisted complex $(\s X_0, X_1)$ with the differential $-f \colon \s X_0 \to  X_1$ \cite[Eq.~(3.28)]{seidel2}. This shows that twisted complexes with $r + 1$ components can be viewed as iterated mapping cones of $r$ morphisms.
\end{remark}

\subsection{Idempotent completion}
\label{subsection:idempotent}

\begin{definition}
\label{definition:idempotent}
Let $\mathbf C$ be a $\Bbbk$-linear category. Let $e \colon X \to X$ be an idempotent in $\mathbf C (X, X)$ for an object $X$ in $\mathbf C$. The {\it abstract image} of $e$ is the right $\mathbf C$-module, i.e.\ the functor $\mathcal Z_e \colon \mathbf C^\op \to \mathrm{Vect}_\Bbbk$ which sends $Y$ to $\mathbf C(Y, X) e = \{fe \mid f \in \mathbf C(Y, X)\}$. 

The {\it idempotent completion} $\mathbf C^\natural$ of $\mathbf C$ is the full subcategory of $\Mod(\mathbf C)$ consisting of $\mathbf C$-modules that are isomorphic to the abstract images of idempotents. We may show that if $e$ (resp.\ $e'$) is an idempotent of $X$ (resp.\ $X'$) then $$\mathbf C^\natural (\mathcal Z_e,\mathcal Z_{e'}) = e' \mathbf C(X, X')e.$$
\end{definition}

Note that $\mathbf C^\natural$ is additive and idempotent complete and the functor $\iota \colon \mathbf C \to \mathbf C^\natural$ given by $\iota (C) = \mathcal Z_{1_C}$ and $\iota (f) = f$ is fully faithful. Moreover, if $\mathbf C$ is triangulated, then there is a unique triangulated structure on $\mathbf C^\natural$ such that $\iota$ is an exact functor \cite[Theorem 1.5]{balmerschlichting}.

The notion of idempotent completion can be generalized to A$_\infty$ categories as follows.

\begin{definition}[{\cite[\S I.4]{seidel2}}]
\label{definition:idempotent2}
Let $\mathbf B$ be an A$_\infty$ category. An {\it idempotent up to homotopy} for an object $X$ of $\mathbf B$ is defined to be a non-unital A$_\infty$ functor $e \colon \Bbbk \to \mathbf B$ with $e (1) = X$. Equivalently, it is given by a sequence $(e_1, e_2, \dotsc)$ where each $e_i \in \mathbf B (X, X)^{1-i}$ for $i \geq 1$ satisfying the following condition (see \cite[Eq.~(4.2)]{seidel2})
\[
\sum_k \sum_{i_1+\dotsb+i_k=d} \mu_k (\s e_{i_k} \otimes \dotsb \otimes \s e_{i_1}) = \begin{cases}
e_{d-1} & \text{if $d$ is even}\\
0 & \text{otherwise.}
\end{cases}
\]
\end{definition}

Let $e$ be an idempotent up to homotopy. We associate to it an A$_\infty$ module $\mathcal Z_e$ of $\mathbf B$, called the {\it abstract image}, whose underlying space is given by
\[
\mathcal Z_e(Y) = \mathbf B(Y, X)[q], \quad \text{for any $Y \in \mathbf B$}
\]
the formal polynomial in one graded variable $q, |q|=-1$, with coefficient in $\mathbf B(Y, X)$. The {\it idempotent completion} $\mathbf B^\natural$ of an A$_\infty$ category $\mathbf B$ is the full A$_\infty$ subcategory of $\Mod(\mathbf B)$ consisting of all objects that are isomorphic, in $\H^0(\Mod(\mathbf B))$, to the abstract images of idempotents up to homotopy.  Let $e$ (resp.\ $e'$) be an idempotent up to homotopy of $X$ (resp.\ $X'$). Then there is a natural isomorphism
\begin{align*}
\H^0(\mathbf B^\natural ( \mathcal Z_e, \mathcal Z_{e'})) \simeq \H^0(e_1' \mathbf B(X, X') e_1).
\end{align*}

\begin{remark}
Let $\mathbf B$ be a $\Bbbk$-linear DG category. Let $e \colon X \to X$ be a usual idempotent of $X$, i.e.\ $e$ is a cocycle of degree zero such that $e^2= e$.  Clearly, $e$ can be viewed as an idempotent up to homotopy of $X$. The abstract image $\mathcal Z_{e}$ is a DG $\mathbf B$-module $\mathcal Z_{e}$ which sends an object $Y$ of $\mathbf B$ to  the complex $\mathbf B(Y, X) [q]$ with the differential  
\begin{align*}
d(a q^{2n}) =  a e q^{2n-1} \quad \text{and} \quad d(aq^{2n-1}) = a(1-e) q^{2n-2}
\end{align*}
Note that there is a natural quasi-isomorphism of complexes   $$\pi \colon \mathbf B(Y, X) [q] \to \mathbf B(Y, X)e$$ given by $\pi(aq^i) = 0$ for $i > 0$ and $\pi(a) =ae $. It follows that $\mathcal Z_{e}$ is quasi-isomorphic to the DG $\mathbf B$-module $\widetilde{\mathcal Z}_{e}$ sending $Y$ to $\mathbf B(Y, X)e$, which is much smaller. For this reason, we may use the latter DG modules $\widetilde{\mathcal Z}_e$ for defining the idempotent completion $\mathbf B^\natural$ of a DG category $\mathbf B$.
\end{remark}

Our primary interest lies in the idempotent completion of $\tw(\mathbf A)$ for an A$_\infty$ category $\mathbf A$.  The triangulated category $\H^0 (\tw (\mathbf A)^\natural)$  associated to any A$_\infty$ category $\mathbf A$ is called the {\it split-closed derived category} of $\mathbf A$ in \cite[Page 60]{seidel2}. By \cite{seidel2} there is a natural equivalence between triangulated categories 
\begin{equation}
\label{eq:idempotent}
\H^0 (\tw (\mathbf A))^{\natural} \simeq \H^0 (\tw (\mathbf A)^\natural).
\end{equation}
Namely, taking the idempotent completion commutes with passing to the homotopy category. 
 
\begin{remark}\label{remark:perfectderived}
Let $\mathbf A$ be an A$_\infty$ category. Denote by $\mathrm D(\mathbf A)$ the derived (i.e.\ homotopy) category of A$_\infty$ modules of $\mathbf A$. Then $\H^0 (\tw (\mathbf A)^\natural)$ is triangle equivalent to the {\it perfect derived category} $\per(\mathbf A)$, which is the smallest full subcategory of $\mathrm D(\mathbf A)$ containing $\mathbf A$ and closed under direct summands, see \cite[Corollary 4.9]{seidel2}. 
\end{remark}

\subsection{Morita equivalence}
It is known from \cite[\S I.3]{seidel2} that any A$_\infty$ functor $F\colon \mathbf A \to \mathbf A'$ induces an A$_\infty$ functor $\tw(F) \colon \tw(\mathbf A) \to \tw(\mathbf A')$.

\begin{definition}\label{definition:moritaequivalence}
A strictly unital A$_\infty$ functor $F\colon \mathbf A \to \mathbf A'$ is a {\it Morita equivalence} if the functor $\tw(F)^\natural \colon \tw(\mathbf A)^\natural \to \tw(\mathbf A')^\natural$ is an A$_\infty$ quasi-equivalence (cf.\ Definition \ref{definition:ainfinityfunctorquasi}). Here $^\natural$ denotes idempotent completion (cf.\ Definition \ref{definition:idempotent2}).

More generally, two A$_\infty$ categories $\mathbf A$ and $\mathbf A'$ are called {\it Morita equivalent} if there is a zigzag of Morita equivalences connecting $\mathbf A$ and $\mathbf A'$.
\end{definition}

\begin{remark}\label{remark:quasimoritaequivalence}
It follows from \cite[Lemma 3.25]{seidel2} that an A$_\infty$ quasi-equivalence $F\colon \mathbf A \to \mathbf A'$ induces a Morita equivalence between $\mathbf A$ and $\mathbf A'$. In particular, any formal A$_\infty$ category $\mathbf A$ is Morita equivalent to the graded category $\H^\bullet(\mathbf A)$. 
\end{remark}

In the following we recall certain properties of the twisted functor $\tw(F)$ for a {\it strict} A$_\infty$ functor $F$, as we will mainly use strict A$_\infty$ functors later.

Let $F \colon \mathbf A \to \mathbf A'$ be a strict A$_\infty$ functor, i.e.\ its components $F_n$ vanish for $n > 1$. Then the induced A$_\infty$ functor $\tw(F) \colon \tw (\mathbf A) \to \tw (\mathbf A')$ is also strict. Given a twisted complex $X^\bullet$ with components $X^i = \s^{m_i} X_i$ and $\delta^{ij} = \s^{m_i-m_j} \delta_{ij} \colon X^j \to X^i$ its image $\tw (F) (X^\bullet)$ under $\tw (F)$ is the twisted complex obtained from $X^\bullet$ by applying $F$ componentwise, i.e.\
\begin{align*}
(\tw(F) (X^\bullet))^i &= F (X^i) = \s^{m_i} F (X_i) \\
\delta^{ij}_{\tw (F) (X^\bullet)} &= F_1 (\delta^{ij}) = \s^{m_i-m_j} F_1 (\delta_{ij}).
\end{align*}
For a morphism $f \colon X^\bullet \to Y^\bullet$ with components $f^{ij} = \s^{l_i-m_j} f_{ij} \colon X^j \to Y^i$ we have
\[
(\tw (F) (f))^{ij} = F_1 (f^{ij}) = \s^{l_i - m_j} F_1 (f_{ij}).
\]
Note that if $F$ is fully faithful then so is $\tw(F)$.

The following result allows us to check Morita equivalences simply on objects.

\begin{lemma}\label{lemma:twistedgenerators}
Let $F \colon \mathbf A \subset \mathbf A'$ be an inclusion of A$_\infty$ categories. Then $F$ is a Morita equivalence if and only if each object in $\mathbf A'$ (viewed as a twisted complex) lies in the image of $\H^0(\tw(F)^\natural)$.
\end{lemma}

\begin{proof}
Note that $\tw(F)^\natural$ is strict and fully faithful, see \cite[Lemma 3.23]{seidel2}. It remains to show that $\H^0(\tw(F)^\natural)$ is essentially surjective if and only if each object in $\mathbf A'$ lies in the image of $\H^0(\tw(F)^\natural)$. By \cite[Lemma 3.32]{seidel2} $\H^0(\tw(\mathbf A')^\natural)$ is split-generated by $\mathrm{Obj}(\mathbf A')$, namely  $$\mathrm{thick}(\mathrm{Obj}(\mathbf A')) = \H^0(\tw(\mathbf A')^\natural).$$ Here, $\mathrm{thick} (\mathrm{Obj}(\mathbf A')) $ is the smallest full triangulated subcategory of $ \H^0(\tw(\mathbf A')^\natural)$ containing $\mathrm{Obj}(\mathbf A')$ and closed under direct summands.
 On the other hand, the condition that $\mathrm{Obj}(\mathbf A')$ lies in the image of $\H^0(\tw(F)^\natural)$ is equivalent to 
\begin{flalign*}
&& \mathrm{thick}(\mathrm{Obj}(\mathbf A'))= \mathrm{im}(\H^0(\tw(F)^\natural)). && \qedhere
\end{flalign*}
\end{proof}

\subsection{A$_\infty$ orbit categories}
\label{subsection:orbitcategories}

We record here several properties of orbit categories for A$_\infty$ categories and their twisted complexes since at the time of writing we were unable to find a reference in the literature. The definition of A$_\infty$ orbit categories appeared in \cite[\S 5.1.3]{opperzvonareva} and as we were finalizing the paper, some of the properties of their categories of twisted complexes that we describe below were also obtained independently by Amiot and Plamondon \cite{amiotplamondon}. A standard reference for the classical (associative) case is \cite{cibilsmarcos}, but in the terminology we shall follow \cite{keller2} and call the skew categories of \cite{cibilsmarcos} (and their A$_\infty$ analogues) {\it orbit categories}.

Let $\mathbf A$ be an A$_\infty$ category and let $G$ be a finite group. An {\it action} of $G$ on $\mathbf A$ is a group homomorphism from $G$ to the group of strict A$_\infty$-automorphisms of $\mathbf A$. Each $g \in G$ thus gives a strict A$_\infty$-automorphism
\[
F_g \colon \mathbf A \to \mathbf A
\]
satisfying $F_g \circ F_h = F_{gh}$ for all $g, h \in G$ and $F_{1_G} = \id_{\mathbf A}$. We may write the action of $g \in G$ on an object $X$ of $\mathbf A$ as $g X = F_g (X)$ and on a morphism $f \in \mathbf A (X, Y)$ as $g f = F_g (f) \in \mathbf A (g X, g Y)$.

\begin{definition}[{\cite[Definition 5.6]{opperzvonareva}}]
The {\it orbit category} $\mathbf A / G$ is the A$_\infty$ category with the same objects as $\mathbf A$ and morphisms between two objects $X$ and $Y$ given by
\[
(\mathbf A / G) (X, Y) = \bigoplus_{g \in G} \mathbf A (X, g Y).
\]
Compositions
\[
\mu_n^{\mathbf A / G} \colon (\mathbf A / G) (X_{n-1}, X_n) \otimes \dotsb \otimes (\mathbf A / G) (X_0, X_1) \to (\mathbf A / G) (X_0, X_n)
\]
are defined for all $n \geq 1$ and induced componentwise by the higher multiplications of $\mathbf A$ by
\begin{align*}
\mathbf A (X_{n-1}, g_n X_n) \otimes \dotsb \otimes \mathbf A (X_0, g_1 X_1) &\to \mathbf A (X_0, g_1 \dotsb g_n X_n) \\
\phi_n \otimes \dotsb \otimes \phi_2 \otimes \phi_1 &\mapsto \mu_n^{\mathbf A} (F_{g_1 \dotsb g_{n-1}} (\phi_n) \otimes \dotsb \otimes F_{g_1} (\phi_2) \otimes \phi_1).
\end{align*}
\end{definition}

The following three results may be obtained directly from the definition. 

\begin{lemma}[{\cite[Lemma 5.7]{opperzvonareva}}]
The compositions $\mu_n^{\mathbf A / G}$ make $\mathbf A / G$ into an A$_\infty$ category.
\end{lemma}

\begin{lemma}
If $\mathbf A$ is an A$_\infty$ with a $G$-action, then the A$_\infty$ category $\tw \mathbf A$ of twisted complexes also inherits a natural $G$-action.
\end{lemma}

\begin{proposition}
\label{proposition:triangulatedgroupaction}
If $\mathbf A$ is a pretriangulated A$_\infty$ category with a $G$-action, then $\H^0 (\mathbf A)$ is a triangulated category with $G$-action and
\[
\H^0 (\mathbf A / G) \simeq \H^0 (\mathbf A) / G.
\]
\end{proposition}

\subsection{Twisted complexes of A$_\infty$ orbit categories}

An action of $G$ on an A$_\infty$ category $\mathbf A$ naturally induces an action of $G$ on the A$_\infty$ category $\tw (\mathbf A)$ of twisted complexes. Concretely, given a twisted complex $(\bigoplus_i \s^{m_i} X_i, \delta)$ in $\tw (\mathbf A)$ for objects $X_i$ in $\mathbf A$, we set
\[
g (\textstyle\bigoplus_i \s^{m_i} X_i, \delta) = (\bigoplus_i \s^{m_i} g X_i, g \delta)
\]
where $g \delta = (g \delta_{ij})_{i,j}$.

The following result is similar to \cite[Example~2.6]{chen17}.

\begin{proposition}\label{proposition:groupactiontriangulated}
Let $G$ be a finite group acting on an A$_\infty$ category $\mathbf A$. Assume that $\mathrm{char}(\Bbbk)$ does not divide $|G|$. Then the natural functor
\[
(\tw (\mathbf A) / G)^\natural \to \tw (\mathbf A / G)^\natural
\]
is an A$_\infty$-quasi-equivalence.
\end{proposition}

\begin{proof}
The objects in $\tw (\mathbf A) / G$ are twisted complexes $X$ in $\tw(\mathbf A)$ where $\delta_X = (\delta_X^{ij})_{1\leq i, j\leq r}$ is a twisted differential on $(\s^{m_1} X_1, \dotsc, \s^{m_r} X_r)$ with $\delta_X^{ij} \in \s^{m_i-m_j} \mathbf A (X_j, X_i)$. 
Morphism spaces between two twisted complexes $X$ and $X'$ are given by 
\[
(\tw (\mathbf A) / G) (X, X') = \bigoplus_{g\in G} \bigoplus_{i,j} \s^{m'_i-m_j} \mathbf A(X_j,g\cdot X'_i).
\]

On the other hand, the objects in $\tw (\mathbf A / G)$ are given by sequences
\[
(\s^{m_1} Y_1, \dotsc, \s^{m_r} Y_r)
\]
of objects in $\mathbb Z \mathbf A$, together with a twisted differential $\delta_Y = (\delta_Y^{ij})_{1\leq i, j\leq r}$ where $$\delta_Y^{ij} =\sum_{g\in G} \delta^{ijg}_Y \in s^{m_i-m_j} \mathbf A(Y_j, g\cdot Y_i).$$

There is a natural functor $F \colon \tw (\mathbf A) / G \to \tw (\mathbf A / G)$ which sends an object $X=(\s^{m_1} X_1, \dotsc, \s^{m_r} X_r; \delta_X)$ to $Y=(\s^{m_1} X_1, \dotsc, \s^{m_r} X_r; \delta_Y)$ with 
$$
\delta_Y^{ijg} = 
\begin{cases}
 \delta_Y^{ij} & \text{if $g=1$ is the unit of $G$}\\
 0 & \text{otherwise.}
\end{cases}
$$
Note that $F$ is fully faithful. Then the idempotent completion $F^\natural\colon (\tw (\mathbf A) / G)^\natural \to \tw (\mathbf A / G)^\natural$ is also fully faithful. 

It remains to show that $\H^0 (F^\natural)$ is essentially surjective. Let $(Y, \delta_Y)$ be an object in $\tw (\mathbf A / G)$, i.e.\ $Y \in \mathbb Z \mathbf A$ and $\delta_Y = \sum_{g\in G} \delta_{Y, g} \in \mathbb Z \mathbf A(Y, g\cdot Y)$. 
Consider the twisted complex  $\widetilde Y = \bigoplus_{g \in G} g\cdot Y$ with the twisted differential  $\delta_{\widetilde Y} = \sum_{g, h \in G} g (\delta_{Y, g^{-1}h}) \in  \bigoplus_{g, h\in G} \mathbb Z\mathbf A(g\cdot Y, h\cdot Y)$. Similar to the proof of \cite[Proposition 2.4]{chen17} we may show that $(\widetilde Y, \delta_{\widetilde Y})$ is indeed an object in $\tw (\mathbf A) $. Note that there are two morphisms in $\tw (\mathbf A / G)^\natural$
\[
\varphi\colon (Y, \delta_Y) \to (\widetilde Y, \delta_{\widetilde Y}) \quad\text{and} \quad \psi \colon (\widetilde Y, \delta_{\widetilde Y}) \to (Y, \delta_Y), 
\]
where $\varphi$ is induced by 
$$\frac{1}{|G|} \sum_{g\in G} \id_Y \colon Y \to \bigoplus_{g\in G} g^{-1} g\cdot Y  \hookrightarrow \bigoplus_{h\in G} \bigoplus_{g\in G} hg\cdot Y  $$
and $\psi$ is induced by 
$$
\sum_{g\in G} \id_{g\cdot Y} \colon \bigoplus_{g\in G} g\cdot Y \to \bigoplus_{g\in G} g \cdot Y.
$$
Note that $\psi \circ \varphi = \id$, it follows that $(Y, \delta_Y)$ is a direct summand of $(\widetilde Y, \delta_{\widetilde Y})$ and thus $(Y, \delta_Y)$ lies in the image of $F^\natural$.
\end{proof}

\section{Orbifold surfaces}
\label{section:orbifoldsurfaces}

Recall that a smooth orbifold is ``a topological space which is locally modelled on the quotient of a smooth manifold by the action of a finite group'' \cite{satake,thurston} and this notion can be reformulated in terms of proper étale groupoids \cite{moerdijk} or Deligne--Mumford stacks in the category of smooth manifolds \cite{lerman}. We will consider closed smooth orbifold {\it surfaces} (i.e.\ smooth orbifolds of dimension $2$) with smooth boundary. As we will not look at the internal structure of the orbifolds nor at maps between them, it essentially suffices to consider compact topological surfaces $S$ with boundary $\partial S$ together with a finite set of ``singular'' orbifold points $\Sing (S) \subset S \smallsetminus \partial S$ with prescribed orders and invoke the existence of a smooth orbifold structure on $S$ with the prescribed set of orbifold points from any of the aforementioned viewpoints. The order of an orbifold point $x \in \Sing (S)$ is the order of the stabilizer of $x$ in a local chart.

\subsection{Line fields}

In the setting of smooth surfaces, the $\mathbb Z$-grading of the partially wrapped Fukaya category of a given surface is given by a line field which is a section of the projectivized tangent bundle \cite[\S 2.1]{haidenkatzarkovkontsevich}.

Similarly, a line field on a smooth orbifold surface $S$ is given by a smooth section
\[
\eta \in \Gamma (S, \mathbb P (\mathrm T S))
\]
of the projectivized tangent (orbi)bundle of $S$. The following lemma shows that the existence of a line field restricts the orbifold points to having order $2$.

\begin{lemma}
\label{lemma:linefield}
Let $S$ be a smooth orbifold surface with nonempty boundary. There exists a section $\eta \in \Gamma (S, \mathbb P (\mathrm T S))$ if and only if all orbifold points $x \in \Sing (S)$ have order $2$.
\end{lemma}

\begin{proof}
The neighbourhood of an orbifold point $x \in \Sing (S)$ with (finite) stabilizer group $G$ is given by, say, the quotient $\mathbb D / G$ of the unit disk $\mathbb D \subset \mathbb R^2$. In this neighbourhood, a section of $\mathbb P (\mathrm T S)$ is given by a $G$-invariant section of $\mathbb P (\mathrm T \mathbb D)$. But the only finite subgroup acting faithfully on $\mathrm T_x \mathbb D$ whose induced action on $\mathbb P (\mathrm T_x \mathbb D) \simeq \mathbb R \mathbb P^1$ is trivial is $G \simeq \mathbb Z_2$. This proves the ``only if'' part.

The ``if'' part of the statement follows for example from considering the smooth compact surface $\widehat S = S \smallsetminus \bigcup_{x \in \Sing (S)} \mathbb D_x$ obtained from $S$ by removing a small (orbifold) disk $\mathbb D_x$ around each neighbourhood (cf.\ Remark \ref{remark:compactification}). Since $\widehat S$ has nonempty boundary, its tangent bundle can be trivialized, thus there exists a nonzero section of the tangent bundle inducing an everywhere-defined section of the projectivized tangent bundle $\mathbb P (\mathrm T \widehat S)$. By the action of elements of the (symplectic) mapping class group one may construct a line field with specified winding numbers around the boundary components of $\widehat S$, see e.g.\ \cite[\S 1]{lekilipolishchuk2}. In particular, there exists a line field $\widehat \eta$ on $\widehat S$ whose winding numbers around the (new) boundary components $\partial \mathbb D_x$ are $1$. Since the line field on an orbifold disk has winding number $1$ (see Fig.~\ref{fig:winding}), the line field $\widehat \eta$ extends to the interior of $\mathbb D_x$ for each $x \in \Sing (S)$, in particular, it extends to all of $S$.
\end{proof}

Letting $\mathbb D \subset \mathbb R^2$ denote the unit open disk, the orbifold surfaces in this paper are therefore locally modelled on $\mathbb D$ (around a smooth interior point), $\mathbb D / \mathbb Z_2$ (around an orbifold point) and $\mathbb D \cap \overline{\mathbb H}$ (around a point in the boundary), where $\overline{\mathbb H} \subset \mathbb R^2$ is the closed upper half-plane. We shall always think of $\mathbb Z_2$ as a multiplicative group with elements $\pm 1$.

\begin{figure}
\begin{tikzpicture}[x=1em,y=1em]
\begin{scope}
\node[font=\scriptsize,left] at (-3.5,3.5) {$\mathbb D$};
\draw[clip] (0,0) circle(4.5em);
\foreach \r in {-4.25,-4,...,4.25} {%
\draw[line width=.2pt] (-4.5,\r) -- (4.5,\r);
}
\end{scope}
\begin{scope}[shift={(13em,0)}]
\draw[<->,line width=.5pt] (0.2,-2) arc[start angle=-90,end angle=90,radius=2em];
\draw[fill=black] circle(.07em);
\node[font=\scriptsize,right] at (2.2,0) {identify};
\draw[clip] (0,4.5) -- (0,-4.5) arc[start angle=270,end angle=90,radius=4.5em] -- cycle;
\draw[line width=.5pt] (0,4.5) -- (0,-4.5) arc[start angle=270,end angle=90,radius=4.5em] -- cycle;
\foreach \r in {-4.25,-4,...,4.25} {%
\draw[line width=.2pt] (-4.5,\r) -- (4.5,\r);
}
\end{scope}
\begin{scope}[shift={(26em,0)}]
\node[font=\scriptsize,left] at (-3.5,3.5) {$\mathbb D / \mathbb Z_2$};
\draw[line width=.5pt] (0,0) circle(4.5em);
\node[font=\scriptsize] at (0,0) {$\times$};
\foreach \r in {0.25,0.5,...,4.25} {%
\draw[line width=.2pt,line cap=round] (0,0) ++(270:\r em) arc[start angle=-90,end angle=90,radius=\r em];
}
\clip (0,4.5) -- (0,-4.5) arc[start angle=270,end angle=90,radius=4.5em] -- cycle;
\foreach \r in {-4.25,-4,...,4.25} {%
\draw[line width=.2pt] (-4.5,\r) -- (4.5,\r);
}
\end{scope}
\end{tikzpicture}
\caption{The induced line field of winding number $1$ on a disk with one orbifold point.}
\label{fig:winding}
\end{figure}
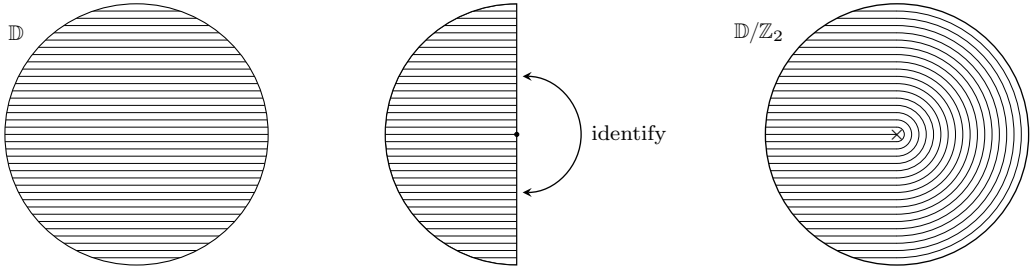

\subsection{Symplectic structure}

Orbifold surfaces with cyclic stabilizers carry a natural symplectic structure which is unique up to a global nonvanishing function. In a smooth local chart this structure is induced by the standard symplectic structure on the unit disk $\mathbb D \subset \mathbb R^2$ given by $\mathrm d x \wedge \mathrm d y$. In particular, this standard symplectic structure is invariant with respect to the $\mathbb Z_2$-action given by
\[
-1 \cdot \begin{pmatrix} x \\ y \end{pmatrix} = \begin{pmatrix} \cos \pi & -\sin \pi \\ \sin \pi & \cos \pi \end{pmatrix} \begin{pmatrix} x \\ y \end{pmatrix} = \begin{pmatrix} -x \\ -y \end{pmatrix}
\]
i.e.\ by rotation through $\pi$, and it induces a symplectic structure on the orbifold charts $\mathbb D / \mathbb Z_2$.

\subsection{Graded orbifold surfaces with stops}

In order to consider partially wrapped Fukaya categories of orbifold surfaces, we consider extra geometric data on orbifold surfaces, generalizing the graded surfaces with stops considered in \cite{haidenkatzarkovkontsevich}.

\begin{definition}
\label{definition:orbifoldsurface}
A {\it graded orbifold surface with stops} $\mathbf S = (S, \Sigma, \eta)$ consists of
\begin{itemize}
\item a smooth compact orbifold surface $S$ with nonempty smooth boundary $\partial S$
\item a nonempty closed subset $\Sigma \subsetneq \partial S$ whose connected components are either boundary components or closed intervals properly contained in boundary components
\item a grading structure given by a smooth line field $\eta$ on $S$.
\end{itemize}
We call a connected component of $\Sigma$ a {\it boundary stop} if it is homeomorphic to a closed interval and a {\it full boundary stop} if it is homeomorphic to $\mathrm S^1$ and we shall require that $\Sigma$ contains at least one boundary stop.
\end{definition}

\begin{remark}
The existence of the line field $\eta$ implies (Lemma \ref{lemma:linefield}) that $S$ only has orbifold points of order $2$. Note that $\eta$ restricts to a line field on the smooth locus $S \smallsetminus \Sing (S)$ with winding number $1$ around every $x \in \Sing (S)$ as illustrated in Fig.~\ref{fig:winding}.
\end{remark}

In illustrations we will mark stops by {\color{stopcolour} $\bullet$} and orbifold points by $\times$ as in Fig.~\ref{fig:orbifoldsurface}. Some notions do not depend on the grading structure $\eta$ in which case we may drop the adjective ``graded'' and refer to $(S, \Sigma)$ as an {\it orbifold surface with stops}.

The terminology of a (boundary) stop is used in the sense of \cite{sylvan,ganatrapardonshende1,ganatrapardonshende2}, where stops are extra data in the contact boundary of a Liouville domain. Although our surfaces are no longer smooth manifolds, the smooth locus $S \smallsetminus \Sing (S)$ is a smooth symplectic manifold and the wrapping occurs away from the orbifold points. This suggests the existence of a well-behaved partially wrapped Floer theory for orbifold surfaces as soon as one makes sense of the role of (interior) orbifold points. We will relate the partially wrapped Floer theory of orbifold surfaces to the usual partially wrapped Floer theory of their smooth double covers. See Section \ref{section:disk} for a detailed description of the case of a disk with a single orbifold point and Section \ref{section:doublecover} for the general case.

The term ``full boundary stop'' is introduced primarily for notational convenience. Equivalently one could simply exclude Lagrangian submanifolds having endpoints on the corresponding boundary components in which case no stop data would be necessary on these boundary components. For example, if $S$ is a graded smooth surface with boundary and $\Sigma = \partial S$ is the full boundary, then $\mathcal W (S, \Sigma, \eta)$ is the category of compact Lagrangians in $S$, as for example studied in \cite{lekiliperutz,lekilipolishchuk0} for $S$ a genus $1$ surface with $n \geq 1$ boundary components. Since these categories do not admit a description via gentle algebras, we assume $\Sigma \neq \partial S$ in Definition \ref{definition:orbifoldsurface}. The assumption that $\Sigma$ contain at least one boundary stop is mostly out of convenience, as it allows us to define special dissections $\Delta_0$ with particular simple ribbon complexes (see \S\ref{subsection:special} and Fig.~\ref{fig:ribbon}). Note that if $\Sigma' \subset \Sigma$ is a subset consisting of connected components of $\Sigma$, then there are {\it stop removal} functors which relate the two partially wrapped Fukaya categories with the two different sets $\Sigma'$ and $\Sigma$ of stop data. See \cite[\S 3.5]{haidenkatzarkovkontsevich} for the case of smooth surfaces and \cite{sylvan,ganatrapardonshende2} for the higher-dimensional case.

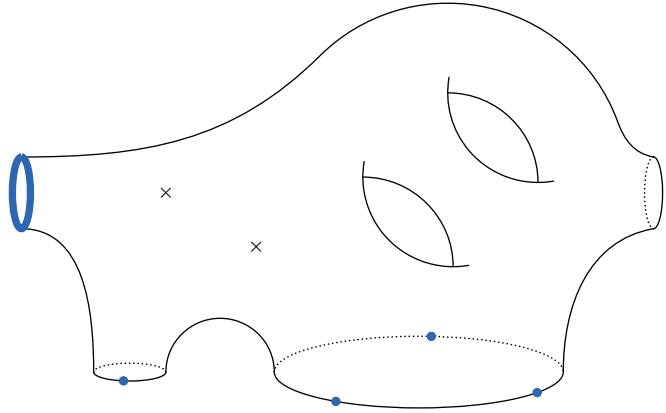
\begin{figure}
\begin{tikzpicture}[x=1.3em,y=1.3em,decoration={markings,mark=at position 0.99 with {\arrow[black]{Stealth[length=4.8pt]}}}, scale=.95]
\node[font=\scriptsize] at (-7,5) {$\times$};
\node[font=\scriptsize] at (-4.5,3.5) {$\times$};
\draw[line width=.5pt] (0,0) ++ (180:4 and 1) arc[start angle=180,end angle=360,x radius=4,y radius=1] to[out=90,in=190,looseness=1] ++(2.5,4) arc[start angle=-90,end angle=90,x radius=.25,y radius=1] to[out=170,in=290,looseness=1] ++(-1,1) arc[start angle=20,end angle=135,radius=5] to[out=225,in=0,looseness=1] (-11,6) ++ (0,-2) to[out=0,in=90] (-9,0) arc[start angle=180,end angle=360,x radius=1,y radius=.25] arc[start angle=180,end angle=0,radius=1.5];
\begin{scope}
\draw[line width=.5pt,line cap=round] (5.5,7) ++(200:5) ++(0:2.5) arc[start angle=0,end angle=90,radius=2.5];
\draw[line width=.5pt,line cap=round] (5.5,7) ++(200:5) ++(0:2.5) ++(90:2.5) ++(280:2.5) arc[start angle=280,end angle=170,radius=2.5];
\end{scope}
\begin{scope}[shift={(-2.35,-2.35)}]
\draw[line width=.5pt,line cap=round] (5.5,7) ++(200:5) ++(0:2.5) arc[start angle=0,end angle=90,radius=2.5];
\draw[line width=.5pt,line cap=round] (5.5,7) ++(200:5) ++(0:2.5) ++(90:2.5) ++(280:2.5) arc[start angle=280,end angle=170,radius=2.5];
\end{scope}
\draw[line width=.6pt,dash pattern=on 0pt off 1.5pt,line cap=round] (0,0) ++ (180:4 and 1) arc[start angle=180,end angle=0,x radius=4,y radius=1] (6.5,4) arc[start angle=270,end angle=90,x radius=0.25,y radius=1] (-9,0) arc[start angle=180,end angle=0,x radius=1,y radius=.25];
\draw[line width=.25em,color=stopcolour] (-11,6) arc[start angle=90,end angle=450,x radius=.25,y radius=1];
\draw[fill=stopcolour, color=stopcolour] (85:4 and 1) circle(.15em);
\draw[fill=stopcolour, color=stopcolour] (235:4 and 1) circle(.15em);
\draw[fill=stopcolour, color=stopcolour] (325:4 and 1) circle(.15em);
\draw[fill=stopcolour, color=stopcolour] (-8,0) ++(260:1 and .25) circle(.15em);
\end{tikzpicture}
\caption{An orbifold surface with stops $(S, \Sigma)$ where $S$ is a genus $2$ surface with two orbifold points and four boundary components and $\Sigma \subset \partial S$ consists of four boundary stops and one full boundary stop.}
\label{fig:orbifoldsurface}
\end{figure}

\begin{figure}
\begin{tikzpicture}[x=1em,y=1em, scale=.9]
\draw[line width=.5pt] (0,0) circle(6em);
\draw[line width=.5pt] (0,0) circle(2.5em);
\foreach \r in {2.75,3,...,5.75} {%
\draw[line width=.2pt] (0,0) ++(118:\r em) arc[start angle=118,end angle=422,radius=\r em];
}
\draw[line width=.2pt] (91.5:2.5em) -- ++(92:2.5em);
\path[line width=.2pt,out=92,in=-5] (91.5:2.5em) ++(92:2.5em) edge (99:6em);
\draw[line width=.2pt] (93:2.5em) -- ++(95:2.25em);
\path[line width=.2pt,out=95,in=5] (93:2.5em) ++(95:2.25em) edge (109:6em);
\draw[line width=.2pt] (94.5:2.5em) -- ++(98:2em);
\path[line width=.2pt,out=98,in=28] (94.5:2.5em) ++(98:2em) edge (118:5.75em);
\draw[line width=.2pt] (96:2.5em) -- ++(101:1.75em);
\path[line width=.2pt,out=101,in=28] (96:2.5em) ++(101:1.75em) edge (118:5.5em);
\draw[line width=.2pt] (97.5:2.5em) -- ++(104:1.5em);
\path[line width=.2pt,out=104,in=28] (97.5:2.5em) ++(104:1.5em) edge (118:5.25em);
\draw[line width=.2pt] (99:2.5em) -- ++(107:1.25em);
\path[line width=.2pt,out=107,in=28] (99:2.5em) ++(107:1.25em) edge (118:5em);
\draw[line width=.2pt] (100.5:2.5em) -- ++(110:1em);
\path[line width=.2pt,out=110,in=28] (100.5:2.5em) ++(110:1em) edge (118:4.75em);
\draw[line width=.2pt] (102:2.5em) -- ++(113:.75em);
\path[line width=.2pt,out=113,in=28] (102:2.5em) ++(113:.75em) edge (118:4.5em);
\draw[line width=.2pt] (103.5:2.5em) -- ++(116:.5em);
\path[line width=.2pt,out=116,in=28] (103.5:2.5em) ++(116:.5em) edge (118:4.25em);
\draw[line width=.2pt] (105:2.5em) -- ++(119:.25em);
\path[line width=.2pt,out=119,in=28] (105:2.5em) ++(119:.25em) edge (118:4em);
\path[line width=.2pt,out=122,in=28] (106.5:2.5em) edge (118:3.75em);
\path[line width=.2pt,out=125,in=28] (108:2.5em) edge (118:3.5em);
\path[line width=.2pt,out=128,in=28] (109.5:2.5em) edge (118:3.25em);
\path[line width=.2pt,out=128,in=28] (111:2.5em) edge (118:3em);
\path[line width=.2pt,out=128,in=28] (112.5:2.5em) edge (118:2.75em);
\draw[line width=.2pt] (88.5:2.5em) -- ++(88:2.5em);
\path[line width=.2pt,out=88,in=185] (88.5:2.5em) ++(88:2.5em) edge (81:6em);
\draw[line width=.2pt] (87:2.5em) -- ++(85:2.25em);
\path[line width=.2pt,out=85,in=175] (87:2.5em) ++(85:2.25em) edge (71:6em);
\draw[line width=.2pt] (85.5:2.5em) -- ++(82:2em);
\path[line width=.2pt,out=82,in=152] (85.5:2.5em) ++(82:2em) edge (62:5.75em);
\draw[line width=.2pt] (84:2.5em) -- ++(79:1.75em);
\path[line width=.2pt,out=79,in=152] (84:2.5em) ++(79:1.75em) edge (62:5.5em);
\draw[line width=.2pt] (82.5:2.5em) -- ++(76:1.5em);
\path[line width=.2pt,out=76,in=152] (82.5:2.5em) ++(76:1.5em) edge (62:5.25em);
\draw[line width=.2pt] (81:2.5em) -- ++(73:1.25em);
\path[line width=.2pt,out=73,in=152] (81:2.5em) ++(73:1.25em) edge (62:5em);
\draw[line width=.2pt] (79.5:2.5em) -- ++(70:1em);
\path[line width=.2pt,out=70,in=152] (79.5:2.5em) ++(70:1em) edge (62:4.75em);
\draw[line width=.2pt] (78:2.5em) -- ++(67:.75em);
\path[line width=.2pt,out=67,in=152] (78:2.5em) ++(67:.75em) edge (62:4.5em);
\draw[line width=.2pt] (76.5:2.5em) -- ++(64:.5em);
\path[line width=.2pt,out=64,in=152] (76.5:2.5em) ++(64:.5em) edge (62:4.25em);
\draw[line width=.2pt] (75:2.5em) -- ++(61:.25em);
\path[line width=.2pt,out=61,in=152] (75:2.5em) ++(61:.25em) edge (62:4em);
\path[line width=.2pt,out=58,in=152] (73.5:2.5em) edge (62:3.75em);
\path[line width=.2pt,out=55,in=152] (72:2.5em) edge (62:3.5em);
\path[line width=.2pt,out=52,in=152] (70.5:2.5em) edge (62:3.25em);
\path[line width=.2pt,out=52,in=152] (69:2.5em) edge (62:3em);
\path[line width=.2pt,out=52,in=152] (67.5:2.5em) edge (62:2.75em);
\draw[line width=.2pt] (90:6em) -- (90:2.5em);
\begin{scope}[shift={(-17em,0)}]
\draw[line width=.5pt] (0,0) circle(6em);
\draw[line width=.5pt] (0,0) circle(2.5em);
\foreach \r in {2.75,3,...,5.75} {%
\draw[line width=.2pt] (0,0) ++(0:\r em) arc[start angle=0,end angle=360,radius=\r em];
}
\end{scope}
\end{tikzpicture}
\caption{Two different grading structures on an annulus given by two non-isotopic line fields.}
\label{fig:grading}
\end{figure}
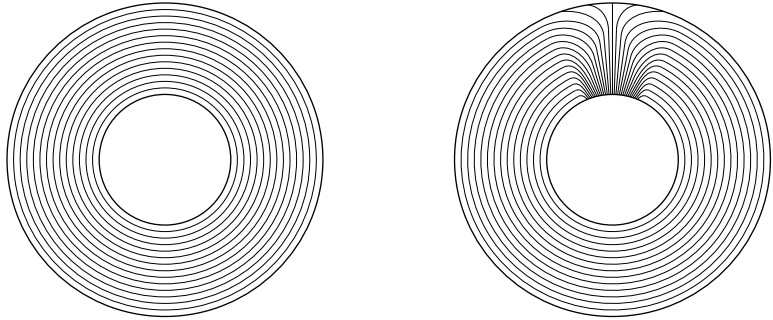

\section{The Fukaya category of an orbifold disk with one orbifold point}
\label{section:disk}

When describing the geometry of quotient spaces, one faces the choice of either working with objects on the quotient or with equivariant objects on the cover. (For example, a quasi-coherent sheaf on a quotient stack $[X / G]$ may be given by a $G$-equivariant sheaf on $X$.) In this section we shall start with the latter viewpoint and study $\mathbb Z_2$-equivariant higher structures in the partially wrapped Fukaya category of a smooth disk with stops.

It turns out that the resulting higher structures on the quotient orbifold disk can be described more efficiently when working directly on the orbifold disk. This remains true for general orbifold surfaces. In later sections we will reverse the direction, by starting with a concrete definition of the Fukaya category of a graded orbifold surface with stops. We then prove that our construction satisfies a universal property and is compatible with the $\mathbb Z_2$-equivariant theory of the double cover using more abstract arguments.

The main result of this section is Proposition \ref{proposition:orbifolddisk} which computes the unique higher structure on a dissection of the orbifold disk. In Corollary \ref{corollary:typeD} we note that the orbifold disk can be described by the perfect derived category of a type $\mathrm D$ quiver. From the viewpoint of a Weinstein sectorial cover (in the sense of \cite{ganatrapardonshende1,ganatrapardonshende2}) of an orbifold surface obtained from a ribbon graph a small neighbourhood of an orbifold point can thus be seen to correspond to a ``Weinstein sector of type $\mathrm D$'' (see also \S\ref{subsection:core}).

Some of the notions for this computation will only be formally introduced in later sections in which case we refer to the pertinent definitions. The reader familiar with partially wrapped Fukaya categories of smooth surfaces as in \cite{haidenkatzarkovkontsevich,lekilipolishchuk1,lekilipolishchuk2} should be able to follow this section on a first reading.

\subsection{An orbifold disk with stops and its double cover}
\label{subsection:orbifolddisk}

Consider the graded orbifold surface $\mathbf S = (S, \Sigma, \eta)$ consisting of a closed disk with a single orbifold point and $l + 1$ boundary stops, i.e.\ $S = \overline{\mathbb D} / \mathbb Z_2$ with $\Sing (S) = \{ x \}$ and $\Sigma = \sigma_0 \sqcup \dotsb \sqcup \sigma_l \subset \partial S \simeq \mathrm S^1$. 
The double cover of $\mathbf S$ is the graded surface with stops $\widetilde{\mathbf S} = (\widetilde S, \widetilde \Sigma, \widetilde \eta)$, where
\begin{itemize}
\item $\widetilde S = \overline{\mathbb D}$ is a smooth closed disk with $\mathbb Z_2$-action given by $-1 \cdot (x, y) = (-x, -y)$, i.e.\ by a rotation through $\pi$, whose only fixpoint is the origin
\item $\widetilde \Sigma = \sigma^\pm_0 \sqcup \dotsb \sqcup \sigma^\pm_l$ is a $\mathbb Z_2$-invariant set of $2l + 2$ boundary stops $\sigma_i^+, \sigma_i^-$ where $-1 \cdot \sigma^\pm_i = \sigma^\mp_i$
\item $\widetilde \eta$ is a $\mathbb Z_2$-invariant line field on $\widetilde S$ which (up to isotopy) may be taken to be the horizontal line field as in Fig.~\ref{fig:winding}.
\end{itemize}

\subsubsection{Arcs}

An arc system $\Gamma$ on the orbifold surface $\mathbf S$ (see \S\ref{subsection:arcsystem}) is given by a collection of arcs in $S$ which may include arcs connecting boundary points in $\partial S \smallsetminus \Sigma$ to the orbifold point.

Consider the arc system
\[
\Gamma = \{ \alpha, \beta, \gamma_1, \dotsc, \gamma_l \}
\]
where $\gamma_1, \dotsc, \gamma_l$ are arcs bounding $\sigma_1, \dotsc, \sigma_l$ and $\alpha, \beta$ are two arcs connecting boundary points in $\partial S \smallsetminus \Sigma$ to the orbifold point $x$ as illustrated in Fig.~\ref{fig:doublecover}.

In the double cover $\widetilde{\mathbf S}$, each $\gamma_i$ lifts to a pair of arcs $\gamma_i^+, \gamma_i^-$ bounding the stops $\sigma_i^+, \sigma_i^-$, respectively. The arcs $\alpha, \beta$ lift to $\mathbb Z_2$-invariant arcs $\widetilde\alpha, \widetilde\beta$ which intersect in the fixed point of the $\mathbb Z_2$-action (see the right-hand side of Fig.~\ref{fig:doublecover}).

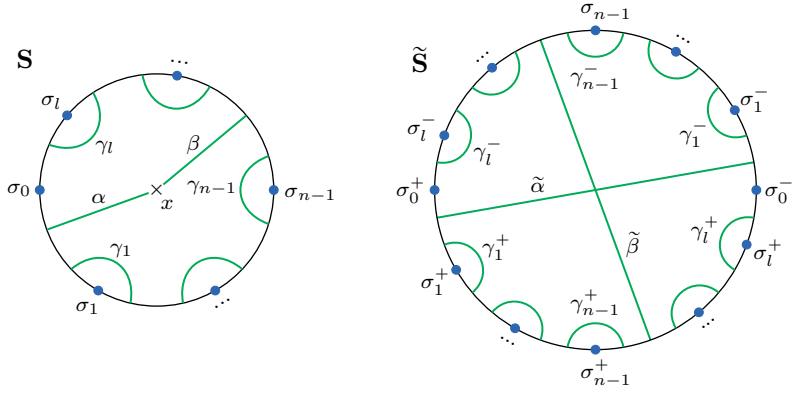
\begin{figure}
\begin{tikzpicture}[x=1em,y=1em,decoration={markings,mark=at position 0.55 with {\arrow[black]{Stealth[length=4.2pt]}}}]
\draw[line width=.5pt] circle(4em);
\node[font=\small] at (-4.5em,4.5em) {$\mathbf S$};
\node[font=\scriptsize,shape=circle,scale=.6,fill=white] (X) at (0,0) {};
\node[font=\scriptsize] at (0,0) {$\times$};
\draw[line width=.75pt,color=arccolour] (200:4em) to (X);
\draw[line width=.75pt,color=arccolour] (40:4em) to (X);
\node[font=\scriptsize] at (180:4.7em) {$\sigma_0$};
\node[font=\scriptsize] at (240:4.7em) {$\sigma_1$};
\node[font=\scriptsize] at (303:4.5em) {$.$};
\node[font=\scriptsize] at (297:4.5em) {$.$};
\node[font=\scriptsize] at (300:4.5em) {$.$};
\node[font=\scriptsize] at (-2:4.6em) {$\sigma\mathrlap{_{n-1}}$};
\node[font=\scriptsize] at (77:4.5em) {$.$};
\node[font=\scriptsize] at (80:4.5em) {$.$};
\node[font=\scriptsize] at (83:4.5em) {$.$};
\node[font=\scriptsize] at (140:4.7em) {$\sigma_l$};
\node[font=\scriptsize] at (306:0.6em) {$x$};
\node[font=\scriptsize] at (188:2em) {$\alpha$};
\node[font=\scriptsize] at (51:2em) {$\beta$};
\node[font=\scriptsize] at (240:2.4em) {$\gamma_1$};
\node[font=\scriptsize] at (2:1.9em) {$\gamma_{n-1}$};
\node[font=\scriptsize] at (140:2.3em) {$\gamma_l$};
\foreach \a in {240,300,0,80,140} {
\draw[fill=stopcolour, color=stopcolour] (\a:4em) circle(.15em);
\path[line width=.75pt,out=\a-160,in=\a+160,looseness=1.5,color=arccolour] (\a+17:4em) edge (\a-17:4em);
}
\draw[fill=stopcolour, color=stopcolour] (180:4em) circle(.15em);
\begin{scope}[xshift=15em]
\node[font=\small] at (-6em,4.5em) {$\widetilde{\mathbf S}$};
\draw[line width=.75pt, color=arccolour] (190:5.5em) to (10:5.5em);
\draw[line width=.75pt, color=arccolour] (290:5.5em) to (110:5.5em);
\draw[line width=.5pt] circle(5.5em);
\foreach \a in {210,240,270,310,340,30,60,90,130,160} {
\draw[fill=stopcolour, color=stopcolour] (\a:5.5em) circle(.15em);
\path[line width=.75pt,out=\a-170,in=\a+170,looseness=1.5,color=arccolour] (\a+10:5.5em) edge (\a-10:5.5em);
}
\draw[fill=stopcolour, color=stopcolour] (180:5.5em) circle(.15em);
\draw[fill=stopcolour, color=stopcolour] (0:5.5em) circle(.15em);
\node[font=\scriptsize] at (180:6.3em) {$\sigma_0^+$};
\node[font=\scriptsize] at (210:6.3em) {$\sigma_1^+$};
\node[font=\scriptsize] at (238:6em) {$.$};
\node[font=\scriptsize] at (240:6em) {$.$};
\node[font=\scriptsize] at (242:6em) {$.$};
\node[font=\scriptsize] at (270:6.3em) {$\sigma_{n\mathrlap{-1}}^+$};
\node[font=\scriptsize] at (308:6em) {$.$};
\node[font=\scriptsize] at (310:6em) {$.$};
\node[font=\scriptsize] at (312:6em) {$.$};
\node[font=\scriptsize] at (340:6.3em) {$\sigma_l^+$};
\node[font=\scriptsize] at (0:6.3em) {$\sigma_0^-$};
\node[font=\scriptsize] at (30:6.3em) {$\sigma_1^-$};
\node[font=\scriptsize] at (58:6em) {$.$};
\node[font=\scriptsize] at (60:6em) {$.$};
\node[font=\scriptsize] at (62:6em) {$.$};
\node[font=\scriptsize] at (90:6.3em) {$\sigma_{n\mathrlap{-1}}^-$};
\node[font=\scriptsize] at (128:6em) {$.$};
\node[font=\scriptsize] at (130:6em) {$.$};
\node[font=\scriptsize] at (132:6em) {$.$};
\node[font=\scriptsize] at (160:6.3em) {$\sigma_l^-$};
\node[font=\scriptsize] at (210:3.9em) {$\gamma_1^+$};
\node[font=\scriptsize] at (270:3.9em) {$\gamma_{n-1}^+$};
\node[font=\scriptsize] at (340:4em) {$\gamma_l^+$};
\node[font=\scriptsize] at (30:3.9em) {$\gamma_1^-$};
\node[font=\scriptsize] at (90:3.9em) {$\gamma_{n-1}^-$};
\node[font=\scriptsize] at (160:3.9em) {$\gamma_l^-$};
\node[font=\scriptsize, above=-.2ex] at (190:2em) {$\widetilde\alpha$};
\node[font=\scriptsize, right] at (110:-2em) {$\widetilde\beta$};
\end{scope}
\end{tikzpicture}
\caption{Arcs on a disk with a single orbifold point and $l+1$ stops and the corresponding $\mathbb Z_2$-invariant (pairs of) arcs in the double cover.}
\label{fig:doublecover}
\end{figure}

\subsubsection{Twisted complexes and morphisms in the double cover}

On graded smooth surfaces, the arc systems considered in \cite{haidenkatzarkovkontsevich} are by definition comprised of {\it pairwise disjoint} arcs. Since $\widetilde\alpha$ and $\widetilde\beta$ intersect in the double cover, we consider the (full formal) arc system
\[
\widetilde \Gamma = \{ \widetilde\alpha, \gamma_1^\pm, \dotsc, \gamma_l^\pm \}
\]
without $\widetilde \beta$.

\begin{notation}
\label{notation:smoothdisk}
We denote by $\mathbf A_{\widetilde \Gamma}$ the graded gentle algebra, viewed as a $\Bbbk$-linear graded category, associated the arc system $\widetilde \Gamma$ as in \cite{haidenkatzarkovkontsevich, lekilipolishchuk2}. Then
\[
\H^0 (\tw (\mathbf A_{\widetilde \Gamma})) \simeq \per (\mathbf A_{\widetilde \Gamma}) \simeq \mathcal W (\widetilde{\mathbf S})
\]
is the partially wrapped Fukaya category of the smooth disk $\widetilde{\mathbf S}$, where $\tw (\mathbf A_{\widetilde \Gamma})$ is the DG category of twisted complexes over $\mathbf A_{\widetilde \Gamma}$ and $\H^0 (\tw (\mathbf A_{\widetilde \Gamma})) \simeq \per (\mathbf A_{\widetilde \Gamma})$ is the associated (triangulated) perfect derived category of $\mathbf A_{\widetilde \Gamma}$, see Remark \ref{remark:perfectderived}.
\end{notation}

There are morphisms $p_1^\pm \in \mathbf A_{\widetilde \Gamma} (\widetilde \alpha, \gamma_1^\pm)$ and
\[
p_i^\pm \in \mathbf A_{\widetilde \Gamma} (\gamma_{i-1}^\pm, \gamma_i^\pm) \quad\text{for $1 < i \leq l$}
\]
given by boundary paths in $\widetilde{\mathbf S}$.

The lift $\widetilde\beta$ may be realized as the twisted complex
\begin{equation}
\label{eq:beta}
\widetilde\beta = \widetilde\alpha \oplus \bigoplus_{\substack{1 \leq i \leq n-1 \\ \ast \in \{ +, - \}}} \s^{|\s p^*_{i \dotsc 1}|} \gamma_i^*
\end{equation}
which should be viewed as an iterated mapping cone of the $p_i^\pm$'s. (Each mapping cone corresponds to the nonzero direct summand of the connected sum of the Lagrangians after a suitable Hamiltonian isotopy.) The twisted differential $\delta_{\widetilde \beta}$ is given by the maps appearing in the diagram
\begin{equation}
\label{eq:betamaps}
\begin{tikzpicture}[baseline=-2.6pt,description/.style={fill=white,inner sep=1.75pt}]
\matrix (m) [matrix of math nodes, row sep={1.2em,between origins}, text height=1.5ex, column sep={5em,between origins}, text depth=0.25ex, ampersand replacement=\&]
{
\&[.4em] \gamma_1^+ \& \gamma_{2}^+ \& \cdots \& \gamma_{n-1}^+ \\
\widetilde\alpha \& \& \& \& \\
\&       \gamma_1^- \& \gamma_{2}^- \& \cdots \& \gamma_{n-1}^- \\
};
\path[->,line width=.4pt]
(m-2-1) edge node[above=-.2ex,font=\scriptsize] {$p_1^+$} (m-1-2)
(m-2-1) edge node[below=-.4ex,font=\scriptsize] {$p_1^-$} (m-3-2)
(m-1-2) edge node[above=-.2ex,font=\scriptsize] {$p_{2}^+$} (m-1-3)
(m-3-2) edge node[below=-.5ex,font=\scriptsize] {$p_{2}^-$} (m-3-3)
(m-1-3) edge node[above=-.2ex,font=\scriptsize] {$p_{3}^+$} (m-1-4)
(m-3-3) edge node[below=-.5ex,font=\scriptsize] {$p_{3}^-$} (m-3-4)
(m-1-4) edge node[above=-.2ex,font=\scriptsize] {$p_{n-1}^+$} (m-1-5)
(m-3-4) edge node[below=-.5ex,font=\scriptsize] {$p_{n-1}^-$} (m-3-5)
;
\end{tikzpicture}
\end{equation}
where for simplicity we have ignored the extra shifts that arise by ordering the direct summands of \eqref{eq:beta} and making all entries of $\delta_{\widetilde \beta} $ of degree $1$. For example, the component of $\delta_{\widetilde \beta} $ from $\widetilde \alpha$ to $\s^{|\s p_1^+|} \gamma_1^+$ is given by $\s^{|\s p_1^+|} p_1^+ \in \s^{|\s p_1^+|} \mathbf A_{\widetilde \Gamma}(\widetilde \alpha, \gamma^+_1)$ which is indeed of degree
\[
|\s^{|\s p_1^+|} p_1^+| = |p_1^+| - |\s p_1^+| = |p_1^+| - (|p_1^+| - 1) = 1.
\]

Note that since $-1 \cdot \gamma_i^\pm = \gamma_i^\mp$ and $-1 \cdot p_i^\pm = p_i^\mp$ and $-1 \cdot \widetilde\alpha = \widetilde\alpha$ the twisted complex $\widetilde\beta$ is clearly $\mathbb Z_2$-invariant, as is also apparent from its graphical representation in Fig.~\ref{fig:doublecover}.

We now compute some of the morphism spaces involving $\widetilde\beta$ in the DG category $\tw (\mathbf A_{\widetilde\Gamma})$. Recall that the differential in the morphism spaces of $\tw (\mathbf A_{\widetilde \Gamma})$ is given by 
\begin{align}\label{align:differentialtwisted}
\mu_1^{\tw (\mathbf A_{\widetilde\Gamma})}(\s a) = \mu_2^{\mathbb Z \mathbf A_{\widetilde\Gamma}}( \s \delta_Y \otimes \s a) + \mu_2^{\mathbb Z \mathbf A_{\widetilde\Gamma}}(\s a \otimes \s \delta_X)
\end{align}
for any $a \in \tw (\mathbf A_{\widetilde \Gamma})(X, Y)$, where $\delta_X$ and $\delta_Y$ are the differentials of $X$ and $Y$ respectively and $\mu_2^{\mathbb Z \mathbf A_{\widetilde\Gamma}}$ is induced by the composition of $\mathbf A_{\widetilde\Gamma}$ (cf.\ Definition \ref{definition:twisted}).

\begin{lemma}\label{lemma:morphismspacetw}
\begin{enumerate}
\item The complex $\tw (\mathbf A_{\widetilde\Gamma}) (\widetilde\alpha, \widetilde\beta) = \Bbbk \{ \id_{\widetilde\alpha} \} \oplus  \Bbbk \{ \s^{|\s p_1^+|}p_1^+,  \s^{|\s p_1^-|}p_1^- \}$ is $3$-dimensional and the nonzero part of its differential is given by 
\[
\id_{\widetilde\alpha}\mapsto \s^{|\s p_1^+|}p_1^+ + \s^{|\s p_1^-|}p_1^-.
\]
So the cohomology $\H^0 (\tw (\mathbf A_{\widetilde\Gamma}) (\widetilde\alpha, \widetilde\beta))$ is $1$-dimensional and lies in degree $1$.  

\item The complex $\tw (\mathbf A_{\widetilde\Gamma}) (\widetilde\beta, \widetilde\alpha) = \Bbbk \{ \id_{\widetilde\alpha} \}$ lies in degree $0$. 

\item The complex $\tw (\mathbf A_{\widetilde\Gamma}) (\widetilde\beta, \widetilde\beta) = \Bbbk \{ \id_{\widetilde\alpha} \} \oplus \bigoplus_{\substack{1 \leq i \leq n-1 \\ \ast \in \{ +, - \}}} \Bbbk \{ \id_{\gamma_i^\ast}, \s^{|\s p_i^\ast|} p_i^{\ast} \}$ is $(2n-1)$-dimensional and the nonzero part of the differential is given by 
\begin{align*}
\id_{\widetilde\alpha} \mapsto \s^{|\s p_1^+|}p_1^+ + \s^{|\s p_1^-|}p_1^-  \quad  \quad  \id_{\gamma_i^{\ast}}  \mapsto \s^{|\s p_{i+1}^\ast|}p_{i+1}^\ast - \s^{|\s p_i^\ast|}p_i^\ast
\end{align*}
where $1\leq i \leq n-1$ and $\ast \in \{ +,- \}$ and we set $p_{n}^{\ast}=0$. Its cohomology is $1$-dimensional and lies in degree $0$. 
\item The complex $\tw (\mathbf A_{\widetilde\Gamma}) (\gamma_{n-1}^\ast, \widetilde\beta) = \Bbbk \{ \s^{|\s p_{n-1 \dotsc 1}^\ast|} \id_{\gamma_{n-1}^\ast} \}$ is $1$-dimensional and lies in degree $|\s p^\ast_{n-1 \dotsc 1}|$.
\end{enumerate}
\end{lemma}

\begin{proof}
This is a straightforward computation that can easily be verified from the definitions. The explicit formulas for the differentials follow from \eqref{align:differentialtwisted} and the twisted differential $\delta_{\widetilde \beta} = \sum_{\substack{1 \leq i \leq n-1 \\ \ast \in \{ +, - \}}} \s^{|\s p_i^\ast|} p_i^\ast$ of $\widetilde \beta$.
\end{proof}

\begin{remark}\label{remark:Z2action}
Since $\widetilde \alpha$ and $\widetilde \beta$ are $\mathbb Z_2$-invariant, it induces a $\mathbb Z_2$-action on $$\tw (\mathbf A_{\widetilde \Gamma}) (\widetilde\alpha, \widetilde\beta) = \Bbbk \{ \id_{\widetilde \alpha} \} \oplus  \Bbbk \{ \s^{|\s p_1^+|}p_1^+,  \s^{|\s p_1^-|}p_1^- \}$$ given by $-1\cdot \id_{\widetilde \alpha} = \id_{\widetilde \alpha}$ and $
-1\cdot \s^{|\s p_1^\pm|}p_1^\pm = \s^{|\s p_1^\mp|}p_1^\mp$. Note that the latter action is nontrivial. 
Similarly, there is also a $\mathbb Z_2$-action on 
\[
\tw (\mathbf A_{\widetilde\Gamma}) (\widetilde\beta, \widetilde\alpha) = \Bbbk \{ \id_{\widetilde \alpha} \}
\]
given by $-1\cdot \id_{\widetilde \alpha} = \id_{\widetilde \alpha}$. That is, the action is trivial. 
\end{remark}

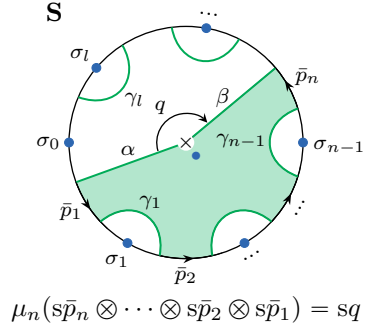
\begin{figure}
\begin{tikzpicture}[x=1em,y=1em,decoration={markings,mark=at position 0.55 with {\arrow[black]{Stealth[length=4.2pt]}}}]
\draw[fill=arccolour!30!white,line width=0pt] (200:4em) arc[start angle=200, end angle=240-17, radius=4em] to[out=240+160, in=240-160, looseness=1.5] (240+17:4em) arc[start angle=240+17, end angle=300-17, radius=4em] to[out=300+160, in=300-160, looseness=1.5] (300+17:4em) arc[start angle=300+17, end angle=360-17, radius=4em] to[out=360+160, in=360-160, looseness=1.5] (360+17:4em) arc[start angle=360+17, end angle=400, radius=4em] to (0,0);
\draw[->, line width=.5pt] (197:1em) arc[start angle=197, end angle=43, radius=1em];
\node[font=\scriptsize] at (126:1.5em) {$q$};
\draw[line width=.5pt] circle(4em);
\node[font=\small] at (-4.5em,4.5em) {$\mathbf S$};
\node[font=\scriptsize,shape=circle,scale=.6,fill=white] (X) at (0,0) {};
\node[font=\scriptsize] at (0,0) {$\times$};
\draw[line width=.75pt,color=arccolour] (200:4em) to (X);
\draw[line width=.75pt,color=arccolour] (40:4em) to (X);
\node[font=\scriptsize] at (180:4.7em) {$\sigma_0$};
\node[font=\scriptsize] at (240:4.7em) {$\sigma_1$};
\node[font=\scriptsize] at (303:4.5em) {$.$};
\node[font=\scriptsize] at (297:4.5em) {$.$};
\node[font=\scriptsize] at (300:4.5em) {$.$};
\node[font=\scriptsize] at (-2:4.6em) {$\sigma\mathrlap{_{n-1}}$};
\node[font=\scriptsize] at (77:4.5em) {$.$};
\node[font=\scriptsize] at (80:4.5em) {$.$};
\node[font=\scriptsize] at (83:4.5em) {$.$};
\node[font=\scriptsize] at (140:4.7em) {$\sigma_l$};
\node[font=\scriptsize, color=stopcolour] at (306:0.6em) {$\bullet$};
\node[font=\scriptsize] at (188:2em) {$\alpha$};
\node[font=\scriptsize] at (51:2em) {$\beta$};
\node[font=\scriptsize] at (240:2.4em) {$\gamma_1$};
\node[font=\scriptsize] at (2:1.9em) {$\gamma_{n-1}$};
\node[font=\scriptsize] at (140:2.3em) {$\gamma_l$};
\draw[line width=0pt,postaction={decorate}] (200:4em) arc[start angle=200, end angle=240-17+5, radius=4em];
\node[font=\scriptsize] at (211.5:4.6em) {$\bar p_1$};
\draw[line width=0pt,postaction={decorate}] (240+17:4em) arc[start angle=240+17, end angle=300-17+5, radius=4em];
\node[font=\scriptsize] at (270:4.6em) {$\bar p_2$};
\draw[line width=0pt,postaction={decorate}] (300+17:4em) arc[start angle=300+17, end angle=360-17+5, radius=4em];
\node[font=\scriptsize] at (330:4.5em) {$.$};
\node[font=\scriptsize] at (327:4.5em) {$.$};
\node[font=\scriptsize] at (333:4.5em) {$.$};
\draw[line width=0pt,postaction={decorate}] (17:4em) arc[start angle=17, end angle=40+5, radius=4em];
\node[font=\scriptsize] at (28.5:4.75em) {$\bar p_n$};
\foreach \a in {240,300,0,80,140} {
\draw[fill=stopcolour, color=stopcolour] (\a:4em) circle(.15em);
\path[line width=.75pt,out=\a-160,in=\a+160,looseness=1.5,color=arccolour] (\a+17:4em) edge (\a-17:4em);
}
\draw[fill=stopcolour, color=stopcolour] (180:4em) circle(.15em);
\node[font=\small] at (0,-5.75em) {$\mu_n (\s \bar p_n \otimes \dotsb \otimes \s \bar p_2 \otimes \s \bar p_1) = \s q$};
\end{tikzpicture}
\caption{The unique higher product on an orbifold disk where the ``orbifold stop'' near the orbifold point indicates that there is no morphism from $\beta$ to $\alpha$.}
\label{fig:doublecoverdisk}
\end{figure}

\subsection{The partially wrapped Fukaya category of an orbifold disk}
\label{subsection:partiallyorbifolddisk}

Let $\mathbf S$ be the graded orbifold disk as in \S\ref{subsection:orbifolddisk} and let $\mathbf A_{\widetilde \Gamma}$ be the graded gentle algebra in Notation \ref{notation:smoothdisk}. The $\mathbb Z_2$-action on $\mathbf A_{\widetilde\Gamma}$ induces a $\mathbb Z_2$-action on the DG category $\tw (\mathbf A_{\widetilde\Gamma})$. We now define
\[
\mathbf W (\mathbf S) := (\tw (\mathbf A_{\widetilde\Gamma})/\mathbb Z_2)^\natural
\]

\begin{definition}
We call the triangulated category
\[
\mathcal W (\mathbf S) := \H^0 (\mathbf W (\mathbf S))
\]
the {\it partially wrapped Fukaya category} of the orbifold disk $\mathbf S$.
\end{definition}

Here we use bold $\mathbf W$ to denote the pretriangulated DG category and calligraphic $\mathcal W$ for the associated triangulated category.

We now study generators of $\mathcal W (\mathbf S)$ for the orbifold disk $\mathbf S$. These concrete results will be used to define the partially wrapped Fukaya category of any graded orbifold surface with stops via a cosheaf of pretriangulated A$_\infty$ categories in the spirit of Kontsevich's conjectural proposal for describing Fukaya categories of Weinstein manifolds, proved in \cite{haidenkatzarkovkontsevich} for graded smooth surfaces with stops and in \cite{ganatrapardonshende2} in all dimensions. In Theorem \ref{theorem:doublecoverorbitcategory} we will relate the cosheaf construction to a global quotient construction using A$_\infty$ orbit categories.

We shall write the morphisms spaces $\mathbf W (\mathbf S) (X, Y)$ as column vectors, where the top entry corresponds to a morphism $X \to Y$ and the bottom entry to a morphism $X \to -1 \cdot Y$. The composition of morphisms (cf.\ \S\ref{subsection:orbitcategories}) is given by
\begin{align}
\label{align:compositionorbit}
\begin{pmatrix} a_1 \\ a_2 \end{pmatrix} \circ \begin{pmatrix} b_1 \\ b_2 \end{pmatrix} = \begin{pmatrix} a_1\circ b_1 + (-1\cdot a_2) \circ b_2 \\ a_2 \circ b_1 +  (-1\cdot a_1) \circ b_2 \end{pmatrix},
\end{align}
where $\circ$ on the right hand side of the equality is the composition in $\tw (\mathbf A_{\widetilde\Gamma})$.

Since $-1 \cdot \widetilde \alpha = \widetilde \alpha$ it follows that 
\[
\mathbf W (\mathbf S)(\widetilde \alpha, \widetilde \alpha) := \bigoplus_{g \in \mathbb Z_2} \tw (\mathbf A_{\widetilde\Gamma}) (\widetilde \alpha, g \cdot \widetilde \alpha) = \tw (\mathbf A_{\widetilde\Gamma})(\widetilde \alpha, \widetilde \alpha)^{\oplus 2}
\]
is $2$-dimensional with trivial differential. We can therefore find two primitive orthogonal idempotents in $\mathbf W (\mathbf S) (\widetilde \alpha, \widetilde \alpha)$ 
\[
\id_{\widetilde \alpha}^+ =  \frac{1}{2} \begin{pmatrix} \id_{\widetilde \alpha} \\ \id_{\widetilde \alpha} \end{pmatrix} \quad \text{and}  \quad  \id_{\widetilde \alpha}^- = \frac{1}{2} \begin{pmatrix} \id_{\widetilde \alpha}\\ -\id_{\widetilde \alpha} \end{pmatrix}.
\]
Similarly, since $-1 \cdot \widetilde \beta = \widetilde \beta$ we have $\mathbf W (\mathbf S)(\widetilde \beta, \widetilde \beta) = \tw (\mathbf A_{\widetilde\Gamma})(\widetilde \beta, \widetilde \beta)^{\oplus 2}$ with two primitive orthogonal idempotents
\[
\id_{\widetilde \beta}^+ = \frac{1}{2} \begin{pmatrix} \id_{\widetilde \beta} \\ \id_{\widetilde \beta}\end{pmatrix} \quad \text{and}  \quad  \id_{\widetilde \beta}^- = \frac{1}{2} \begin{pmatrix} \id_{\widetilde \beta} \\ -\id_{\widetilde \beta}\end{pmatrix}.
\]
Here $\id_{\widetilde \beta}$ is the identity on the twisted complex $\widetilde \beta$ \eqref{eq:beta} which is given by $\id_{\widetilde \alpha} +\sum_{\substack{1 \leq i \leq n-1 \\ \ast \in \{ +, - \}}} \id_{\gamma_i^{*}}$ in the complex $\tw (\mathbf A_{\widetilde\Gamma}) (\widetilde\beta, \widetilde \beta)$.

The existence of primitive orthogonal idempotents implies that both $\widetilde \alpha$ and $\widetilde \beta$ decompose into two direct summands and we shall write
$$
\widetilde \alpha = \widetilde \alpha^+\oplus \widetilde \alpha^- \quad \text{and} \quad \widetilde \beta = \widetilde \beta^+ \oplus \widetilde \beta^-
$$
where $\mathbf W (\mathbf S)  (\widetilde \alpha^\pm, \widetilde \alpha^\pm)  = \Bbbk \, \id_{\widetilde \alpha}^\pm$ and $\mathbf W (\mathbf S)(\widetilde\beta^\pm, \widetilde \beta^\pm)  = \id_{\widetilde \beta}^\pm \mathbf W (\mathbf S)(\widetilde \beta, \widetilde \beta) \id_{\widetilde \beta}^\pm$. 

Let us consider the following direct sum of objects in the DG category $\mathbf W(\mathbf S)$ 
$$
\Gamma = \widetilde \alpha^+ \oplus \bigoplus_{i=1}^l \gamma_i^+\oplus \widetilde \beta^-
$$
which by Lemma \ref{lemma:splitgenerator} turns out to be a split-generator of $\mathbf W(\mathbf S)$, i.e.\ $\mathrm{thick} (\Gamma) \simeq \H^0(\mathbf W (\mathbf S))= \mathcal W(\mathbf S)$. In particular, the perfect derived category of the DG endomorphism algebra of $\Gamma$ is triangle equivalent to $\mathcal W (\mathbf S)$. 

We will now describe the DG endomorphism algebra $\mathbf W(\mathbf S)(\Gamma, \Gamma)$ of $\Gamma$ concretely. To do this, we need the following result which describes some of the morphism spaces in $\mathbf W (\mathbf S)$.

\begin{lemma}
\label{lemma:morphisminvariant}
\begin{enumerate}
\item $\mathbf W (\mathbf S) (\widetilde \beta^\pm, \widetilde \alpha^\pm) =\Bbbk \begin{pmatrix} \id_{\widetilde \alpha} \\ \pm \id_{\widetilde \alpha} \end{pmatrix}  $ and $\mathbf W (\mathbf S) (\widetilde \beta^\pm, \widetilde \alpha^\mp)=0$
\item   Each of the two complexes $$\mathbf W (\mathbf S) (\widetilde \alpha^\pm, \widetilde \beta^\mp)  = \Bbbk  \begin{pmatrix} \s^{|\s p_1^+|}p_1^+-\s^{|\s p_1^-|}p_1^-\\ \mp\s^{|\s p_1^+|}p_1^+\pm\s^{|\s p_1^-|}p_1^-   \end{pmatrix}$$ is $1$-dimensional and the complexes $\mathbf W (\mathbf S) (\widetilde \alpha^\pm, \widetilde \beta^\pm)$ are acyclic.
\item $\mathbf W(\mathbf S) (\widetilde \alpha^+, \gamma_1^+) =  \Bbbk \begin {pmatrix} p_1^+\\ p_1^- \end{pmatrix} \quad \text{and} \quad \mathbf W(\mathbf S) (\widetilde \beta^-,  \gamma_n^+) =  \Bbbk \begin{pmatrix} \s^{-|\s p_{n-1 \dotsc 1}^+|}  p_n^+\\ -\s^{-|\s p_{n-1 \dotsc 1}^+|}  p_n^-\end{pmatrix}$.
\item \label{itemiv}For each $1\leq i \leq n-2$ the complex
$$
\mathbf W (\mathbf S)( \gamma_{i}^+, \widetilde \beta^-) =  \Bbbk \left\{ \begin{pmatrix}\s^{|\s p_{i \dotsc 1}^+|} \id_{\gamma_{i}^+} \\  -\s^{|\s p_{i \dotsc 1}^+|} \id_{\gamma_{i}^+} \end{pmatrix}, \begin{pmatrix}\s^{|\s p_{i+1 \dotsc 1}^+|} p_{i+1}^+ \\  -\s^{|\s p_{i+1 \dotsc 1}^+|} p_{i+1}^+ \end{pmatrix} \right\}
$$
whose differential is given by  
\[
\begin{pmatrix}\s^{|\s p_{i \dotsc 1}^+|} \id_{\gamma_{i}^+} \\  -\s^{|\s p_{i \dotsc 1}^+|} \id_{\gamma_{i}^+} \end{pmatrix} \mapsto  \begin{pmatrix}\s^{|\s p_{i+1 \dotsc 1}^+|} p_{i+1}^+ \\  -\s^{|\s p_{i+1 \dotsc 1}^+|} p_{i+1}^+ \end{pmatrix}
\]
is acyclic
and the complex 
$$\mathbf W (\mathbf S)( \gamma_{n-1}^+, \widetilde \beta^-) =  \Bbbk \begin{pmatrix}\s^{|\s p_{n-1 \dotsc 1}^+|} \id_{\gamma_{n-1}^+} \\  -\s^{|\s p_{n-1 \dotsc 1}^+|} \id_{\gamma_{n-1}^+} \end{pmatrix}$$ is $1$-dimensional. 
\end{enumerate}
\end{lemma}
\begin{proof}
Recall that the complex 
\[
\mathbf W (\mathbf S) (\widetilde \beta, \widetilde \alpha) = \tw (\mathbf A_{\widetilde\Gamma})(\widetilde \beta, \widetilde \alpha)^{\oplus 2} = \Bbbk \{ \id_{\widetilde \alpha}\}^{\oplus 2}
\]
is $2$-dimensional with trivial differential. By definition, we have 
\[
\mathbf W (\mathbf S) (\widetilde \beta^+, \widetilde \alpha^+) = \id_{\widetilde \alpha}^+ \mathbf W (\mathbf S)  (\widetilde \beta, \widetilde \alpha) \id_{\widetilde \beta}^+.
\]
Using the composition formula \eqref{align:compositionorbit} we obtain 
\begin{align*}
\id_{\widetilde \alpha}^\pm \circ \begin{pmatrix}\id_{\widetilde \alpha} \\0\end{pmatrix} \circ  \id_{\widetilde \beta}^\pm &= \frac{1}{4} \begin{pmatrix}\id_{\widetilde \alpha}\\  \pm \id_{\widetilde \alpha}\end{pmatrix} \circ \begin{pmatrix} \id_{\widetilde \alpha} \\ 0 \end{pmatrix}  \circ \begin{pmatrix} \id_{\widetilde \beta}\\ \pm \id_{\widetilde \beta}\end{pmatrix}\\
& =  \frac{1}{4} \begin{pmatrix}\id_{\widetilde \alpha}\\  \pm \id_{\widetilde \alpha}\end{pmatrix} \circ \begin{pmatrix} \id_{\widetilde \beta}\\ 0\end{pmatrix} +  \frac{1}{4} \begin{pmatrix}\id_{\widetilde \alpha}\\  \pm \id_{\widetilde \alpha}\end{pmatrix} \circ \begin{pmatrix} 0 \\ \pm \id_{\widetilde \beta}\end{pmatrix} \\
&=  \frac{1}{4} \begin{pmatrix}\id_{\widetilde \alpha}\\  \pm \id_{\widetilde \alpha}\end{pmatrix} + \frac{1}{4} \begin{pmatrix}\id_{\widetilde \alpha}\\  \pm \id_{\widetilde \alpha}\end{pmatrix}\\ 
& = \frac{1}{2} \begin{pmatrix}\id_{\widetilde \alpha}\\  \pm \id_{\widetilde \alpha}\end{pmatrix}
\end{align*}
where the third equality uses the fact that $\id_{\widetilde \beta} =\id_{\widetilde \alpha} +\sum_{\substack{1 \leq i \leq n-1 \\ \ast \in \{ +, - \}}} \id_{\gamma_i^{*}}$. We also have 
\[
\id_{\widetilde \alpha}^\pm \circ \begin{pmatrix} 0 \\ \id_{\widetilde \alpha} \end{pmatrix}\circ  \id_{\widetilde \beta}^\pm  = \pm \frac{1}{2} \begin{pmatrix}\id_{\widetilde \alpha}\\  \pm \id_{\widetilde \alpha}\end{pmatrix}
\]
since $\bigl( \begin{smallmatrix} \id_{\widetilde \alpha} \\ \pm \id_{\widetilde \alpha} \end{smallmatrix} \bigr)
 \circ \bigl( \begin{smallmatrix} 0 \\ \id_{\widetilde \alpha} \end{smallmatrix} \bigr) = \pm \bigl( \begin{smallmatrix} \id_{\widetilde \alpha} \\ \pm \id_{\widetilde \alpha} \end{smallmatrix} \bigr)
. $  
By a similar computation we may obtain $\id_{\widetilde \alpha}^\pm \circ \bigl( \begin{smallmatrix}\lambda_1 \id_{\widetilde \alpha} \\ \lambda_2 \id_{\widetilde \alpha} \end{smallmatrix} \bigr) \circ \id_{\widetilde \gamma}^\mp = 0$ for any $\lambda_1, \lambda_2 \in \Bbbk$.

Let us prove the second assertion. Recall that $\mathbf W (\mathbf S)(\widetilde \alpha, \widetilde \beta) = \tw (\mathbf A_{\widetilde\Gamma}) (\widetilde\alpha, \widetilde\beta)^{\oplus 2}$ and $\tw (\mathbf A_{\widetilde\Gamma}) (\widetilde\alpha, \widetilde\beta) = \Bbbk \{ \id_{\widetilde\alpha} \} \oplus  \Bbbk \{ \s^{|\s p_1^+|}p_1^+,  \s^{|\s p_1^-|}p_1^- \}.$
From
\[
\mathbf W (\mathbf S) ( \widetilde \alpha^\pm, \widetilde \beta^\mp) = \id_{\widetilde \beta}^\mp \mathbf W (\mathbf S)  ( \widetilde \alpha, \widetilde \beta) \id_{\widetilde \alpha}^\pm,
\]
 it follows that 
\begin{align*}
\id_{\widetilde \beta}^\mp \circ \begin{pmatrix}\s^{|\s p_1^+|}p_1^+\\ 0 \end{pmatrix}  \circ \id_{\widetilde \alpha}^\pm &= \frac{1}{4}  \begin{pmatrix} \id_{\widetilde \beta}\\ \mp \id_{\widetilde \beta} \end{pmatrix}\circ  \begin{pmatrix} \s^{|\s p_1^+|}p_1^+\\ 0 \end{pmatrix}  \circ  \begin{pmatrix} \id_{\widetilde \alpha}\\  \pm \id_{\widetilde \alpha} \end{pmatrix}\\
&=   \frac{1}{4}  \begin{pmatrix} \s^{|\s p_1^+|}p_1^+\\  \mp\s^{|\s p_1^+|}p_1^+ \end{pmatrix} \circ  \begin{pmatrix} \id_{\widetilde \alpha}\\ 0 \end{pmatrix}  + \frac{1}{4}  \begin{pmatrix} \s^{|\s p_1^+|}p_1^+\\ \mp \s^{|\s p_1^+|}p_1^+ \end{pmatrix}\circ   \begin{pmatrix} 0\\ \pm \id_{\widetilde \alpha} \end{pmatrix}\\
&= \frac{1}{4}  \begin{pmatrix} \s^{|\s p_1^+|}p_1^+\\ \mp\s^{|\s p_1^+|}p_1^+ \end{pmatrix}     +   \frac{1}{4}  \begin{pmatrix} -\s^{|\s p_1^-|}p_1^-\\  \pm\s^{|\s p_1^-|}p_1^-  \end{pmatrix}    \\
&= \frac{1}{4}  \begin{pmatrix} \s^{|\s p_1^+|}p_1^+-\s^{|\s p_1^-|}p_1^-\\ \mp\s^{|\s p_1^+|}p_1^+\pm \s^{|\s p_1^-|}p_1^-  \end{pmatrix}
\end{align*}
where the second equality uses the fact that $\id_{\widetilde \beta} =\id_{\widetilde \alpha} +\sum_{\substack{1 \leq i \leq n-1 \\ \ast \in \{ +, - \}}} \id_{\gamma_i^{*}}$
 and  the third equality follows since by the nontrivial action $-1\cdot \s^{|\s p_1^\pm|}p_1^\pm = \s^{|\s p_1^\mp|}p_1^\mp$ in Remark \ref{remark:Z2action} and the composition formula \eqref{align:compositionorbit} we have 
\[
 \begin{pmatrix} \s^{|\s p_1^\pm|}p_1^\pm\\ 0 \end {pmatrix} \circ  \begin{pmatrix} 0\\ \id_{\widetilde \alpha} \end{pmatrix} =  \begin{pmatrix} 0\\ \s^{|\s p_1^\mp|}p_1^\mp \end{pmatrix} \quad \text{and} \quad  \begin{pmatrix} 0\\  \s^{|\s p_1^\pm|}p_1^\pm \end{pmatrix} \circ  \begin{pmatrix} 0\\  \id_{\widetilde \alpha} \end{pmatrix} =  \begin{pmatrix} \s^{|\s p_1^\mp|}p_1^\mp\\  0 \end{pmatrix}.
\]
We may show that each of the two complexes $$ \mathbf W (\mathbf S) (\widetilde \alpha^\pm, \widetilde \beta^\mp)  = \Bbbk  \begin{pmatrix} \s^{|\s p_1^+|}p_1^+-\s^{|\s p_1^-|}p_1^-\\ \mp\s^{|\s p_1^+|}p_1^+\pm\s^{|\s p_1^-|}p_1^-   \end{pmatrix}$$  is $1$-dimensional and concentrated in degree $1$.

By a similar computation we obtain that each of the two complexes 
\begin{align*}
\id_{\widetilde \beta}^\pm \mathbf W (\mathbf S)  (\widetilde \alpha, \widetilde \beta) \id_{\widetilde \alpha}^\pm = \Bbbk \left\{ \begin{pmatrix} \id_{\widetilde \alpha}\\  \pm \id_{\widetilde \alpha} \end{pmatrix},  \begin{pmatrix} \s^{|\s p_1^+|}p_1^++ \s^{|\s p_1^-|}p_1^-\\ \pm\s^{|\s p_1^+|}p_1^+\pm \s^{|\s p_1^-|}p_1^- \end{pmatrix} \right\}
\end{align*}
is $2$-dimensional and the nonzero part of the differentials is
\[
\begin{pmatrix} \id_{\widetilde \alpha}\\  \pm \id_{\widetilde \alpha} \end{pmatrix}  \mapsto \begin{pmatrix} \s^{|\s p_1^+|}p_1^++ \s^{|\s p_1^-|}p_1^-\\ \pm\s^{|\s p_1^+|}p_1^+\pm \s^{|\s p_1^-|}p_1^- \end{pmatrix} .
\]
In particular, both of the two complexes $\id_{\widetilde \beta}^\pm \mathbf W (\mathbf S)  ( \widetilde \alpha, \widetilde \beta) \id_{\widetilde \alpha}^\pm$ are acyclic. 

The third and fourth assertions may be obtained in a similar way using Lemma \ref{lemma:morphismspacetw}.
\end{proof}

\begin{remark}
It follows from Lemma \ref{lemma:morphisminvariant} that 
\begin{align*}
\dim \mathcal W(\mathbf S) (\widetilde \beta^\pm, \widetilde \alpha^\pm) &= \dim \mathcal W(\mathbf S)(\widetilde \alpha^\pm,\widetilde \beta^\mp) = 1\\
\dim \mathcal W(\mathbf S) (\widetilde \beta^\pm, \widetilde \alpha^\mp) &= \dim \mathcal W(\mathbf S)(\widetilde \alpha^\pm,\widetilde \beta^\pm) = 0.
\end{align*}
\end{remark}

\begin{lemma}\label{lemma:splitgenerator}
The object $\Gamma = \widetilde \alpha^+ \oplus \bigoplus_{i=1}^l \gamma_i^+\oplus \widetilde \beta^-$ is a split-generator of $\mathbf W(\mathbf S)$.
\end{lemma}
\begin{proof}
It suffices to show that $\widetilde \alpha^-, \widetilde \beta^+ \in \mathrm{thick}(\Gamma) \subset \mathcal W(\mathbf S)$. Consider the twisted complexes
$$\gamma^\pm = \bigoplus_{1 \leq i \leq n-1}  \s^{|\s p^\pm_{i \dotsc 1}|} \gamma_i^\pm \quad \text{with} \quad \delta= \sum_{i=2}^{n-1}  \s^{|\s p^\pm_i|} p^\pm_i
$$
in $\tw (\mathbf A_{\widetilde\Gamma})$. Clearly, $\gamma^+  \in \mathrm{thick}(\Gamma)$.  Since $-1\cdot \gamma^+ = \gamma^-$ it follows that $ \gamma^+ \simeq \gamma^-$ in $\mathcal W(\mathbf S)$ and thus $\gamma^- \in \mathrm{thick} (\Gamma)$. The distinguished triangle in $\H^0(\tw (\mathbf A_{\widetilde\Gamma}))$ 
\[
\gamma^+ \oplus \gamma^- \to  \widetilde \beta \to \widetilde \alpha
\]
splits into two distinguished triangles in $\mathcal W(\mathbf S)$

\begin{align*}
 \gamma^+ \to \widetilde \beta^+ \toarglong{\begin{pmatrix} \id_{\widetilde \alpha} \\  \id_{\widetilde \alpha} \end{pmatrix}} \widetilde \alpha^+  \quad  \text{and} \quad \gamma^- \to \widetilde \beta^- \toarglong{\begin{pmatrix} \id_{\widetilde \alpha} \\  -\id_{\widetilde \alpha} \end{pmatrix}}   \widetilde \alpha^-. 
\end{align*}
The first triangle yields $\widetilde \beta^+ \in \mathrm{thick}(\Gamma)$ and the second one yields $\widetilde \alpha^- \in \mathrm{thick}(\Gamma)$. 
\end{proof}

Thanks to Lemma \ref{lemma:morphisminvariant} we may now compute the DG endomorphism algebra $\mathbf W(\mathbf S)(\Gamma, \Gamma)$ of the split-generator $\Gamma $ of $\mathbf W(\mathbf S)$. It is given by the following DG quiver
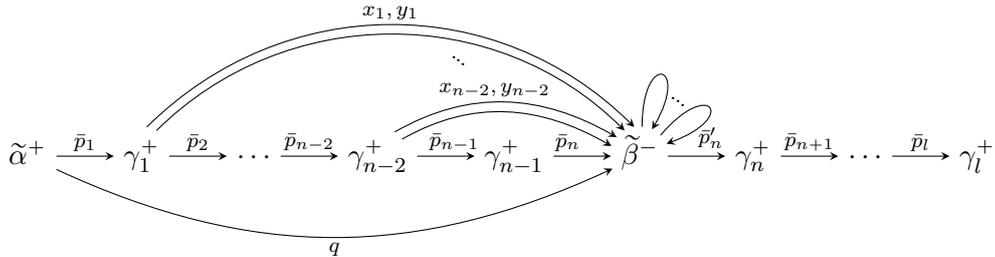
\[
\begin{tikzpicture}[baseline=-2.6pt,description/.style={fill=white,inner sep=1pt,outer sep=0}]
\matrix (m) [matrix of math nodes, row sep=2.5em, text height=1.5ex, column sep=2em, text depth=0.25ex, ampersand replacement=\&, inner sep=3.5pt]
{
\widetilde \alpha^+ \& \gamma_1^+ \& \cdots \& \gamma_{n-2}^+ \& \gamma_{n-1}^+ \& \widetilde \beta^- \& \gamma_n^+ \& \cdots \& \gamma_l^+ \\
};
\node[circle, minimum size=2em] (A) at (m-1-2) {};
\node[circle, minimum size=2em] (B) at (m-1-4) {};
\node[circle, minimum size=2em] (C) at (m-1-6) {};
\path[->,line width=.4pt, font=\scriptsize] (m-1-1) edge node[above=-.4ex] {$\bar p_1$} (m-1-2);
\path[->,line width=.4pt, font=\scriptsize] (m-1-2) edge node[above=-.4ex] {$\bar p_2$} (m-1-3);
\path[->,line width=.4pt, font=\scriptsize] (m-1-3) edge node[above=-.4ex] {$\bar p_{n-2}$} (m-1-4);
\path[->,line width=.4pt, font=\scriptsize, out=45, in=135, looseness=1.02] (A.80) edge node[above=-.4ex] {$x_1, y_1$} (C.100);
\path[->,line width=.4pt, font=\scriptsize, out=45, in=135] (A.60) edge (C.120);
\path[->,line width=.4pt, font=\scriptsize, out=30, in=150] (B.55) edge node[above=-.4ex, pos=.45] {$x_{n-2}, y_{n-2}$} (C.140);
\path[->,line width=.4pt, font=\scriptsize, out=30, in=150] (B.35) edge (C.160);
\path[->,line width=.4pt, font=\scriptsize] (m-1-4) edge node[above=-.4ex, pos=.65] {$\bar p_{n-1} $} (m-1-5);
\path[->,line width=.4pt, font=\scriptsize] (m-1-5) edge node[above=-.4ex, pos=.25] {$\bar p_{n}$} (m-1-6);
\path[->,line width=.4pt, font=\scriptsize, out=90, in=60, looseness=19] (C.85) edge (C.65);
\path[->,line width=.4pt, font=\scriptsize, out=50, in=20, looseness=19] (C.45) edge (C.25);
\node[font=\scriptsize] at ($(C)+(55:2.3em)$) {.};
\node[font=\scriptsize] at ($(C)+(50:2.3em)$) {.};
\node[font=\scriptsize] at ($(C)+(60:2.3em)$) {.};
\node[font=\scriptsize] at ($(B)+(50:4.3em)$) {.};
\node[font=\scriptsize] at ($(B)+(50:4.3em)+(140:.2em)$) {.};
\node[font=\scriptsize] at ($(B)+(50:4.3em)+(140:-.2em)$) {.};
\path[->,line width=.4pt, font=\scriptsize] (m-1-6) edge node[above=-.2ex, pos=.75] {$\bar p_{n}'$} (m-1-7);
\path[->,line width=.4pt, font=\scriptsize] (m-1-7) edge node[above=-.4ex] {$\bar p_{n+1}$} (m-1-8);
\path[->,line width=.4pt, font=\scriptsize] (m-1-8) edge node[above=-.4ex] {$\bar p_l$} (m-1-9);
\path[->,line width=.4pt, font=\scriptsize, out=-25, in=-155] (m-1-1) edge node[below=-.4ex] {$q$} (C);
\end{tikzpicture}
\]
where $\bar p_i = \begin{pmatrix}  p_i^+\\ 0 \end{pmatrix}$ for $2\leq i \leq l$ and $i \neq n$,
\[
\bar p_1=\begin{pmatrix} p_1^+ \\ p_1^- \end{pmatrix}, \quad  \bar p_n= \begin{pmatrix} \s^{|\s p_{n-1 \dotsc 1}^+|} \id_{\gamma_{n-1}^+}\\  -\s^{|\s p_{n-1 \dotsc 1}^+|} \id_{\gamma_{n-1}^+} \end{pmatrix}, \quad \bar p_n'= \begin{pmatrix} \s^{-|\s p_{n-1 \dotsc 1}^+|}  p_n^+\\ -\s^{-|\s p_{n-1 \dotsc 1}^+|}  p_n^-\end{pmatrix},
\]
 and $q =  \begin{pmatrix} \s^{|\s p_1^+|}p_1^+-\s^{|\s p_1^-|}p_1^-\\ -\s^{|\s p_1^+|}p_1^++\s^{|\s p_1^-|}p_1^- \end{pmatrix}.$
 The differential is given by $
d(x_i) = y_i$ for each $1\leq i \leq n-2$, where $x_i, y_i$ are the basis elements appeared in Lemma \ref{lemma:morphisminvariant} \ref{itemiv}.

Let us describe the $(2n-1)$ loops at the vertex $\widetilde \beta^-$ of the above quiver, which correspond to the following basis elements
$$
\mathbf W (\mathbf S) (\widetilde\beta^-, \widetilde \beta^-) = \Bbbk \begin{pmatrix} \id_{\widetilde \alpha} \\ - \id_{\widetilde \alpha}\end{pmatrix} \oplus  \bigoplus_{\substack{1 \leq i \leq n-1 \\ \ast \in \{ +, - \}}} \Bbbk \left\{ \begin{pmatrix} \id_{\gamma_i^\ast} \\ -\id_{\gamma_i^\ast}\end{pmatrix}, \begin{pmatrix}\s^{|\s p_{i+1}^\ast|}p_{i+1}^\ast - \s^{|\s p_i^\ast|}p_i^\ast\\ - \s^{|\s p_{i+1}^\ast|}p_{i+1}^\ast + \s^{|\s p_i^\ast|}p_i^\ast\end{pmatrix} \right\}
$$
The differential is given by 
\begin{align*}
\begin{pmatrix}\id_{\widetilde \alpha}\\ - \id_{\widetilde \alpha}\end{pmatrix} \mapsto \begin{pmatrix} \s^{|\s p_1^+|}p_1^++\s^{|\s p_1^-|}p_1^- \\ - \s^{|\s p_1^+|}p_1^+- \s^{|\s p_1^+|}p_1^+\end{pmatrix} , \quad  \begin{pmatrix}\id_{\gamma_i^{\ast}}\\ -  \id_{\gamma_i^{\ast}}\end{pmatrix}   \mapsto \begin{pmatrix}\s^{|\s p_{i+1}^\ast|}p_{i+1}^\ast - \s^{|\s p_i^\ast|}p_i^\ast\\ - \s^{|\s p_{i+1}^\ast|}p_{i+1}^\ast + \s^{|\s p_i^\ast|}p_i^\ast\end{pmatrix}
\end{align*}
where $1\leq i \leq n-1$ and $\ast \in \{ +,- \}$ and we set $p_{n}^{\ast}=0$. In particular, we obtain that $\H^\bullet (\mathbf W(\mathbf S)(\widetilde \beta^-, \widetilde \beta^- )) = \mathcal W(\mathbf S)(\widetilde \beta^-, \widetilde \beta^-)$ is $1$-dimensional and represented by the unit of $\widetilde \beta$
\[
\id_{\widetilde \beta} =\id_{\widetilde \alpha} +\sum_{\substack{1 \leq i \leq n-1 \\ \ast \in \{ +, - \}}} \id_{\gamma_i^{*}}.
\]

As a result, the cohomology $\H^\bullet (\mathbf W(\mathbf S)(\Gamma, \Gamma)) = \bigoplus_{n\in \mathbb Z} \mathcal W(\mathbf S)(\Gamma, \s^n\Gamma)$ is given by the following graded quiver 
\[
\begin{tikzpicture}[baseline=-2.6pt,description/.style={fill=white,inner sep=1pt,outer sep=0}]
\matrix (m) [matrix of math nodes, row sep=2.5em, text height=1.5ex, column sep=2.2em, text depth=0.25ex, ampersand replacement=\&, inner sep=3.5pt]
{
\widetilde \alpha^+ \& \gamma_1^+ \& \cdots \& \gamma_{n-1}^+ \& \widetilde \beta^- \& \gamma_n^+ \& \cdots \& \gamma_l^+ \\
};
\node[circle, minimum size=2em] (C) at (m-1-5) {};
\path[->,line width=.4pt, font=\scriptsize] (m-1-1) edge node[above=-.4ex] {$\bar p_1$} (m-1-2);
\path[->,line width=.4pt, font=\scriptsize] (m-1-2) edge node[above=-.4ex] {$\bar p_2$} (m-1-3);
\path[->,line width=.4pt, font=\scriptsize] (m-1-3) edge node[above=-.4ex] {$\bar p_{n-1}$} (m-1-4);
\path[->,line width=.4pt, font=\scriptsize] (m-1-4) edge node[above=-.4ex] {$\bar p_n$} (m-1-5);
\path[->,line width=.4pt, font=\scriptsize] (m-1-5) edge node[above=-.2ex] {$\bar p_n'$} (m-1-6);
\path[->,line width=.4pt, font=\scriptsize] (m-1-6) edge node[above=-.4ex] {$\bar p_{n+1}$} (m-1-7);
\path[->,line width=.4pt, font=\scriptsize] (m-1-7) edge node[above=-.4ex] {$\bar p_l$} (m-1-8);
\path[->,line width=.4pt, font=\scriptsize, out=-25, in=-155] (m-1-1) edge node[below=-.4ex] {$q$} (C);
\end{tikzpicture}
\]
Let us describe the A$_\infty$ minimal model of the endomorphism algebra of $\Gamma$.

\begin{proposition}
\label{proposition:orbifolddisk}
After passing to the homotopy category $\H^0 (\mathbf W (\mathbf S)) = \mathcal W (\mathbf S)$, the object $\Gamma$ is a generator of the partially wrapped Fukaya category $\mathcal W (\mathbf S)$ of the orbifold disk $\mathbf S$.

Moreover, the only nonzero higher A$_\infty$ product in the A$_\infty$ minimal model of the endomorphism algebra $\mathcal W(\mathbf S)(\Gamma, \Gamma)$ of $\Gamma$ is given by 
\begin{align*}
\begin{tikzpicture}[baseline=-2.6pt,description/.style={fill=white,inner sep=2pt}]
\matrix (m) [matrix of math nodes, row sep=.5em, text height=1.5ex, column sep=0em, text depth=0.25ex, ampersand replacement=\&, column 2/.style={anchor=base west}]
{
\s \mathcal W (\mathbf S) (\gamma_{n-1}^+, \widetilde\beta^-) \otimes \s \mathcal W (\mathbf S) (\gamma_{1 \dotsc n-1}^+) \otimes \s \mathcal W (\mathbf S) (\widetilde\alpha^+, \gamma_1^+) \&[2em] \s \mathcal W (\mathbf S) (\widetilde\alpha^+, \widetilde\beta^-) \\
\s \bar p_n \otimes \s \bar p_{n-1 \dotsc 2} \otimes \s \bar p_1  \& \s q. \\
};
\path[->,line width=.4pt,font=\scriptsize]
(m-1-1) edge (m-1-2)
;
\path[|->,line width=.4pt]
(m-2-1) edge (m-2-2)
;
\end{tikzpicture}
\end{align*}
\end{proposition}

\begin{proof}
The fact that $\Gamma$ generates $\mathcal W (\mathbf S)$ follows directly from Lemma \ref{lemma:splitgenerator}.

We now construct a homotopy deformation retract
\[
\begin{tikzpicture}[baseline=-2.6pt,description/.style={fill=white,inner sep=1pt,outer sep=0}]
\matrix (m) [matrix of math nodes, row sep=0em, text height=1.5ex, column sep=2.5em, text depth=0.25ex, ampersand replacement=\&, inner sep=3.5pt]
{
\bigoplus_{n\in \mathbb Z} \mathcal W(\mathbf S)(\Gamma, \s^n\Gamma) = \H^\bullet (\mathbf W(\mathbf S)(\Gamma, \Gamma)) \& \mathbf W(\mathbf S)(\Gamma, \Gamma)\\
};
\path[->,line width=.4pt,font=\scriptsize, transform canvas={yshift=.4ex}]
(m-1-1) edge node[above=-.4ex] {$\iota$} (m-1-2)
;
\path[->,line width=.4pt,font=\scriptsize, transform canvas={yshift=-.4ex}]
(m-1-2) edge node[below=-.4ex] {$\pi$} (m-1-1)
;
\path[->,line width=.4pt,font=\scriptsize, looseness=12, in=30, out=330]
($(m-1-2.east)+(0,-.2em)$) edge node[right=-.4ex] {$h$} ($(m-1-2.east)+(0,+.2em)$)
;
\end{tikzpicture}
\]
i.e.\ $\iota \circ \pi - \id = dh + hd$ and $\pi \circ \iota = \id$, which restricts to
\[
\begin{tikzpicture}[baseline=-2.6pt,description/.style={fill=white,inner sep=1pt,outer sep=0}]
\matrix (m) [matrix of math nodes, row sep=0em, text height=1.5ex, column sep=2.5em, text depth=0.25ex, ampersand replacement=\&, inner sep=3.5pt]
{
0 \& \mathbf W (\mathbf S)( \gamma_{i}^+, \widetilde \beta^-) \\
};
\path[->,line width=.4pt,font=\scriptsize, transform canvas={yshift=.4ex}]
(m-1-1) edge node[above=-.4ex] {$\iota$} (m-1-2)
;
\path[->,line width=.4pt,font=\scriptsize, transform canvas={yshift=-.4ex}]
(m-1-2) edge node[below=-.4ex] {$\pi$} (m-1-1)
;
\path[->,line width=.4pt,font=\scriptsize, looseness=12, in=30, out=330]
($(m-1-2.east)+(0,-.2em)$) edge node[right=-.4ex] {$h$} ($(m-1-2.east)+(0,+.2em)$)
;
\end{tikzpicture}
\]
for each $1\leq i \leq n-2$. Here, the nonzero part of $h$ is given by 
\begin{align*}
h \left( \begin{pmatrix}\s^{|\s p_{i+1 \dotsc 1}^+|} p_{i+1}^+ \\  -\s^{|\s p_{i+1 \dotsc 1}^+|} p_{i+1}^+ \end{pmatrix} \right) &= -\begin{pmatrix}\s^{|\s p_{i \dotsc 1}^+|} \id_{\gamma_{i}^+} \\  -\s^{|\s p_{i \dotsc 1}^+|} \id_{\gamma_{i}^+} \end{pmatrix}. 
\end{align*}
Then using the homotopy transfer theorem we obtain the desired higher product. 
\end{proof}

The following corollary shows that an orbifold disk can be described as the perfect derived category of a type $\mathrm D$ quiver.

\begin{corollary}
\label{corollary:typeD}
Let $\mathbf S$ be the graded orbifold disk with $l + 1$ stops as in \S\ref{subsection:orbifolddisk}. Then we have a triangulated equivalence
\[
\mathcal W (\mathbf S) \simeq \per (A)
\]
where $A \simeq \Bbbk Q / I$ is the (ungraded) path algebra of a linearly ordered quiver $Q$ of type $\mathrm D$ and $I$ is the ideal generated by all paths of length $2$. Here $\per (A)$ is the perfect derived category of $A$.
\end{corollary}

\begin{proof}
By Proposition \ref{proposition:orbifolddisk} the object $\Gamma$ is a generator of $\mathcal W (\mathbf S)$ for any $n$ so that we may consider the case $n = 1$. The unique extra (``higher'') structure in Proposition \ref{proposition:orbifolddisk} is given by the differential $\mu_1 (\s p_1) = \s q$. The endomorphism algebra $\mathbf W (\mathbf S) (\Gamma, \Gamma)$ of $\Gamma$ is thus the DG algebra given by the quiver
\begin{equation*}
\begin{tikzpicture}[baseline=-2.6pt,description/.style={fill=white,inner sep=1pt,outer sep=0}]
\matrix (m) [matrix of math nodes, row sep=2.5em, text height=1.5ex, column sep=3em, text depth=0.25ex, ampersand replacement=\&, inner sep=3.5pt]
{
\widetilde \alpha^+ \& \widetilde \beta^- \& \gamma_1^+ \& \gamma_2^+ \& \dotsb \& \gamma_l^+ \\
};
\path[->,line width=.4pt, font=\scriptsize, transform canvas={yshift=.5ex}] (m-1-1) edge node[above] {$\bar p_1$} (m-1-2);
\path[->,line width=.4pt, font=\scriptsize, transform canvas={yshift=-.5ex}] (m-1-1) edge node[below] {$q$} (m-1-2);
\path[->,line width=.4pt, font=\scriptsize] (m-1-2) edge node[above] {$\bar p_1'$} (m-1-3);
\path[->,line width=.4pt, font=\scriptsize] (m-1-3) edge node[above] {$\bar p_2$} (m-1-4);
\path[->,line width=.4pt, font=\scriptsize] (m-1-4) edge node[above] {$\bar p_3$} (m-1-5);
\path[->,line width=.4pt, font=\scriptsize] (m-1-5) edge node[above] {$\bar p_l$} (m-1-6);
\end{tikzpicture}
\end{equation*}
with relations $\bar p_2 \bar p_1' = 0$ and $\bar p_{i+1} \bar p_i = 0$ for all $1 < i < l$. Its cohomology algebra $A = \H^\bullet (\mathbf W (\mathbf S) (\Gamma, \Gamma))$ is given by the quiver
\begin{equation}
\label{eq:typed}
\begin{tikzpicture}[baseline=-2.6pt,description/.style={fill=white,inner sep=1pt,outer sep=0}]
\matrix (m) [matrix of math nodes, row sep=2.5em, text height=1.5ex, column sep=3em, text depth=0.25ex, ampersand replacement=\&, inner sep=3.5pt]
{
\widetilde \alpha^+ \& \widetilde \beta^- \& \gamma_1^+ \& \gamma_2^+ \& \dotsb \& \gamma_l^+ \\
};
\path[->,line width=.4pt, font=\scriptsize] (m-1-1) edge[bend right=30] node[below] {$q'$} (m-1-3);
\path[->,line width=.4pt, font=\scriptsize] (m-1-2) edge node[above] {$\bar p_1'$} (m-1-3);
\path[->,line width=.4pt, font=\scriptsize] (m-1-3) edge node[above] {$\bar p_2$} (m-1-4);
\path[->,line width=.4pt, font=\scriptsize] (m-1-4) edge node[above] {$\bar p_3$} (m-1-5);
\path[->,line width=.4pt, font=\scriptsize] (m-1-5) edge node[above] {$\bar p_l$} (m-1-6);
\end{tikzpicture}
\end{equation}
with all quadratic monomial relations. Here the arrow $q'$ represents the nonzero cocycle $\bar p_1' q$. This quiver is of type $\mathrm D_{l+2}$. It follows from Theorem \ref{theorem:formaldg} below that $\mathbf W (\mathbf S) (\Gamma, \Gamma)$ is formal, i.e.\ it is A$_\infty$-quasi-isomorphic to its cohomology algebra which proves that
\[
\per (A) \simeq \H^0 (\tw (\mathbf W (\mathbf S) (\Gamma, \Gamma))^\natural) \simeq \mathcal W (\mathbf S).
\]
That the degrees of the arrows in \eqref{eq:typed} can be chosen to lie in degree $0$ follows from the observation that up to homotopy, the arcs can be chosen to be transverse to the line field and the line field can be chosen to be tranverse to the boundary everywhere except in the part of the boundary containing the stop $\sigma$ which does not contribute to the grading of the arrows (see \S\ref{subsection:grading}).
\end{proof}

\begin{remark}
The orbifold disk is a local model for the neighbourhood of an orbifold point in a general orbifold surface. Corollary \ref{corollary:typeD} shows that the corresponding partially wrapped Fukaya category can be described by a quiver of type $\mathrm D$. This can be used to construct a Weinstein sectorial cover of a general orbifold surface containing Weinstein sectors of type $\mathrm D$ in addition to the sectors of type $\mathrm A$ or type $\widetilde{\mathrm A}$ for smooth surfaces (see \S\ref{subsection:core}).

Although it is natural to consider the case of full quadratic monomial relations in the algebra $A = \Bbbk Q / I$, by changing the dissection of the orbifold disk one can reverse arrows in $Q$ and remove relations of $I$ without changing the derived equivalence class of $A$. This fact will be a straightforward consequence of our main results of this paper, in particular the description of $\mathcal W (\mathbf S)$ for any $\mathbf S$ via explicit A$_\infty$ categories in Section \ref{section:ainfinity} and the characterization of formal generators in Section \ref{section:formal}.
\end{remark}

Proposition \ref{proposition:orbifolddisk} gives higher structures in terms of curves in the orbit category of a double cover. If we had worked directly on the orbifold disk, then we could have simply defined
\begin{equation}
\label{eq:mutimesunique1}
\mu_n (\s \bar p_{n \dotsc 1}) = \s q
\end{equation}
as in Fig.~\ref{fig:doublecoverdisk}. (In the following sections, we will denote this special type of higher product by $\mutimes_n$ to indicate its relation to the orbifold point since we illustrate orbifold points by $\times$.) In the next few sections, we will take this type of higher structure as a local model and generalize it to arbitrary graded orbifold surfaces with stops.


Note that we may replace the arcs $\gamma_n, \dotsc, \gamma_l$ by arcs $\gamma'_n, \dotsc, \gamma'_l$, where $\gamma'_n$ is the cone of the morphism from $\beta$ to $\gamma_n$ and similarly $\gamma'_{n+1}$ is the cone of the morphism from $\gamma'_n$ to $\gamma_{n+1}$ and so on. One may check that $\alpha \toarg{q} \beta$ factorizes into $q = q_{l-n+2} \dotsb q_2 q_1$ and the unique higher product involving $\alpha, \beta, \gamma_1, \dotsc, \gamma_{n-1}, \gamma_n', \dotsc, \gamma_l'$ is still given by
\begin{equation}
\label{eq:mutimesunique2}
\mu_n (\s \bar p_{n \dotsc 1}) = \s q = \s q_{l-n+2} \dotsb q_2 q_1.
\end{equation}
See Fig.~\ref{fig:orbifolddisk} for an illustration (where in the figure $k = l-n+2$).

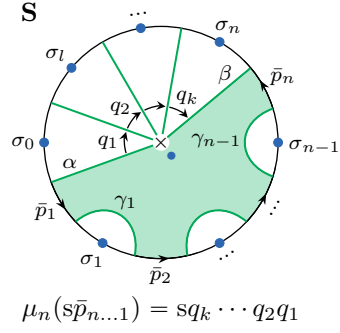
\begin{figure}
\begin{tikzpicture}[x=1em,y=1em,decoration={markings,mark=at position 0.55 with {\arrow[black]{Stealth[length=4.2pt]}}}]
\draw[fill=arccolour!30!white,line width=0pt] (200:4em) arc[start angle=200, end angle=240-17, radius=4em] to[out=240+160, in=240-160, looseness=1.5] (240+17:4em) arc[start angle=240+17, end angle=300-17, radius=4em] to[out=300+160, in=300-160, looseness=1.5] (300+17:4em) arc[start angle=300+17, end angle=360-17, radius=4em] to[out=360+160, in=360-160, looseness=1.5] (360+17:4em) arc[start angle=360+17, end angle=400, radius=4em] to (0,0);
\draw[->, line width=.5pt] (197:1.25em) arc[start angle=197, end angle=163, radius=1.25em];
\draw[->, line width=.5pt] (157:1.25em) arc[start angle=157, end angle=123, radius=1.25em];
\draw[->, line width=.5pt] (117:1.25em) arc[start angle=117, end angle=83, radius=1.25em];
\draw[->, line width=.5pt] (77:1.25em) arc[start angle=77, end angle=43, radius=1.25em];
\node[font=\scriptsize] at (180:1.8em) {$q_1$};
\node[font=\scriptsize] at (140:1.8em) {$q_2$};
\node[font=\scriptsize] at (60:1.8em) {$q_k$};
\draw[line width=.5pt] circle(4em);
\node[font=\small] at (-4.5em,4.5em) {$\mathbf S$};
\node[font=\scriptsize,shape=circle,scale=.6,fill=white] (X) at (0,0) {};
\node[font=\scriptsize] at (0,0) {$\times$};
\draw[line width=.75pt,color=arccolour] (200:4em) to (X);
\draw[line width=.75pt,color=arccolour] (40:4em) to (X);
\draw[line width=.75pt,color=arccolour] (160:4em) to (X);
\draw[line width=.75pt,color=arccolour] (120:4em) to (X);
\draw[line width=.75pt,color=arccolour] (80:4em) to (X);
\node[font=\scriptsize] at (180:4.7em) {$\sigma_0$};
\node[font=\scriptsize] at (240:4.7em) {$\sigma_1$};
\node[font=\scriptsize] at (303:4.5em) {$.$};
\node[font=\scriptsize] at (297:4.5em) {$.$};
\node[font=\scriptsize] at (300:4.5em) {$.$};
\node[font=\scriptsize] at (-2:4.6em) {$\sigma\mathrlap{_{n-1}}$};
\node[font=\scriptsize] at (60:4.6em) {$\sigma_n$};
\node[font=\scriptsize] at (103:4.5em) {$.$};
\node[font=\scriptsize] at (100:4.5em) {$.$};
\node[font=\scriptsize] at (97:4.5em) {$.$};
\node[font=\scriptsize] at (140:4.7em) {$\sigma_l$};
\node[font=\scriptsize, color=stopcolour] at (306:0.6em) {$\bullet$};
\node[font=\scriptsize] at (193:3.2em) {$\alpha$};
\node[font=\scriptsize] at (47:3.2em) {$\beta$};
\node[font=\scriptsize] at (240:2.4em) {$\gamma_1$};
\node[font=\scriptsize] at (2:1.9em) {$\gamma_{n-1}$};
\draw[line width=0pt,postaction={decorate}] (200:4em) arc[start angle=200, end angle=240-17+5, radius=4em];
\node[font=\scriptsize] at (211.5:4.6em) {$\bar p_1$};
\draw[line width=0pt,postaction={decorate}] (240+17:4em) arc[start angle=240+17, end angle=300-17+5, radius=4em];
\node[font=\scriptsize] at (270:4.6em) {$\bar p_2$};
\draw[line width=0pt,postaction={decorate}] (300+17:4em) arc[start angle=300+17, end angle=360-17+5, radius=4em];
\node[font=\scriptsize] at (330:4.5em) {$.$};
\node[font=\scriptsize] at (327:4.5em) {$.$};
\node[font=\scriptsize] at (333:4.5em) {$.$};
\draw[line width=0pt,postaction={decorate}] (17:4em) arc[start angle=17, end angle=40+5, radius=4em];
\node[font=\scriptsize] at (28.5:4.75em) {$\bar p_n$};
\foreach \a in {240,300,0} {
\draw[fill=stopcolour, color=stopcolour] (\a:4em) circle(.15em);
\path[line width=.75pt,out=\a-160,in=\a+160,looseness=1.5,color=arccolour] (\a+17:4em) edge (\a-17:4em);
}
\draw[fill=stopcolour, color=stopcolour] (180:4em) circle(.15em);
\draw[fill=stopcolour, color=stopcolour] (140:4em) circle(.15em);
\draw[fill=stopcolour, color=stopcolour] (100:4em) circle(.15em);
\draw[fill=stopcolour, color=stopcolour] (60:4em) circle(.15em);
\node[font=\small] at (0,-5.75em) {$\mu_n (\s \bar p_{n \dotsc 1}) = \s q_k \dotsb q_2 q_1$};
\end{tikzpicture}
\caption{The unique higher product on an orbifold disk where the orbifold path from $\alpha$ to $\beta$ factorizes as a composition of orbifold paths between arcs connecting to the orbifold point.}
\label{fig:orbifolddisk}
\end{figure}

\section{Arc systems, dissections and paths}
\label{section:arcsystems}

We now give the formal definitions of arcs, dissections, ribbon complexes and paths for an arbitrary graded orbifold surface $\mathbf S$ which we use in Section \ref{section:cosheaves} and Section \ref{section:ainfinity} to define A$_\infty$ categories whose homotopy categories of twisted complexes give the partially wrapped Fukaya category of $\mathbf S$.

\subsection{Arcs and arc systems}
\label{subsection:arcsystem}

\begin{definition}
An {\it arc} in an orbifold surface with stops $(S, \Sigma)$ is given by (the image of) a smooth embedding $\gamma \colon [0, 1] \to S \smallsetminus \Sigma$ such that the following conditions are satisfied:
\begin{itemize}
\item The interior of $\gamma$, i.e.\ the image of the open interval $(0,1)$, contains neither boundary nor orbifold points.
\item The endpoints\footnote{We do not consider any orientation on arcs and thus refer to both $\gamma (0)$ and $\gamma (1)$ as {\it endpoints} of $\gamma$ and denote the endpoints by $\partial \gamma$.} of $\gamma$ belong to $\Sing (S) \cup \partial S \smallsetminus \Sigma$ and the endpoints are only allowed to coincide if they are orbifold points. Moreover, we assume that if $\gamma$ intersects $\partial S$, the intersection is transverse.
\item $\gamma$ is not isotopic to an embedded interval in $\partial S \smallsetminus \Sigma$ nor homotopic to the trivial path at an orbifold point.
\end{itemize}
\end{definition}

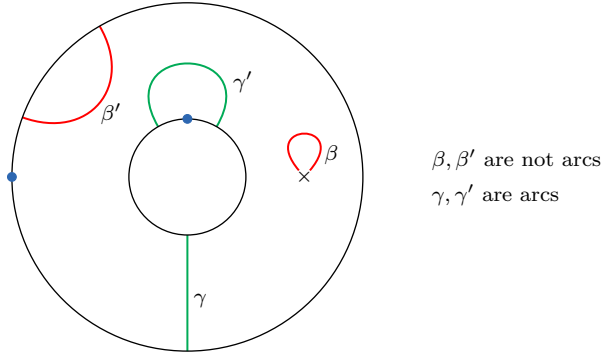
\begin{figure}
\begin{tikzpicture}[x=1em,y=1em]
\node[font=\scriptsize,right] at (8,.6) {$\beta, \beta'$ are not arcs};
\node[font=\scriptsize,right] at (8,-.6) {$\gamma, \gamma'$ are arcs};
\path[line width=.75pt,out=120,in=60,looseness=4.3,color=arccolour] (120:2em) edge (60:2em);
\path[line width=.75pt,color=arccolour] (270:2em) edge (270:6em);
\path[line width=.75pt,out=300,in=-20,looseness=1.5,color=red] (120:6em) edge (160:6em); 
\begin{scope}[xshift=4em]
\node[font=\scriptsize] at (40:1.2em) {$\beta$};
\end{scope}
\node[font=\scriptsize] at (60:3.8em) {$\gamma'$};
\node[font=\scriptsize,right=-.3ex] at (270:4.25em) {$\gamma$};
\node[font=\scriptsize] at (140:3.4em) {$\beta'$};
\draw[line width=.5pt] (0,0) circle(2em);
\draw[line width=.5pt] (0,0) circle(6em);
\draw[fill=stopcolour, color=stopcolour] (90:2em) circle(.15em);
\draw[fill=stopcolour, color=stopcolour] (180:6em) circle(.15em);
\node[font=\scriptsize,shape=circle,scale=.6,fill=white] (X) at (4,0) {};
\path[line width=.75pt,out=130,in=50,looseness=15,color=red] (X) edge (X);
\node[font=\scriptsize] at (4,0) {$\times$};
\end{tikzpicture}
\caption{Examples and non-examples of arcs on an annulus with one orbifold point and two boundary stops.}
\label{fig:full}
\end{figure}

Here and elsewhere we allow isotopies or homotopies of arcs to move the endpoints of the arc within $\partial S \smallsetminus \Sigma$ and to move the interior of the arc within $S \smallsetminus \Sing (S)$. If an arc has an orbifold point as endpoint, this endpoint has to be fixed by the isotopy or homotopy.

The following notion of an arc system on an orbifold surface extends the notion of an arc system on a smooth surface as in \cite{haidenkatzarkovkontsevich}. As we saw in Section \ref{section:disk}, it is necessary to consider arcs connecting to orbifold points. In the literature on cluster algebras related to orbifold surfaces (see for example \cite{feliksonshapirotumarkin,geuenichlabardini}) permitted arcs connecting to orbifold points come in pairs, one of which is usually marked with a ``tagging'' and flips or mutations are defined for tagged triangulations. In our setup, some smoothly isotopic arcs connecting to orbifold points give rise to different objects of the partially wrapped Fukaya category (e.g.\ the arcs $\widetilde \alpha^\pm$ in Section \ref{section:disk}). See also Remark \ref{remark:tagging}. Unlike in the case of cluster theory, we will not restrict ourselves to triangulations and we will not distinguish the curves by a tagging.

\begin{definition}
\label{definition:arcsystem}
An {\it arc system} $\Gamma$ in an orbifold surface with stops $(S, \Sigma)$ is given by a finite collection $\{ \gamma_1, \dotsc, \gamma_n \}$ of arcs such that any two arcs in $\Gamma$ are only allowed to intersect at orbifold points. (We allow $\Gamma$ to contain isotopic arcs.)
\end{definition}

\subsection{Dissections and ribbon complexes}
\label{subsection:ribbongraph}

If $\Gamma$ contains enough arcs to separate the stops and dissect $S$ into ``simple pieces'', then $(S, \Sigma)$ can be encoded into a ribbon graph with extra structure. These ``simple pieces'' are made precise in the following notion of dissection.

\begin{definition}
\label{definition:dissection}
Let $\Gamma$ be an arc system on an orbifold surface with stops $(S, \Sigma)$ as in Definition \ref{definition:arcsystem}. We decompose $S \smallsetminus \Gamma$ into its connected components as follows
\[
S \smallsetminus \Gamma = \bigsqcup_v P_v.
\]
Then $\Gamma$ is called a {\it dissection} of $(S, \Sigma)$ if the following hold:
\begin{enumerate}
\item Each $P_v$ is either a smooth disk or a smooth annulus, i.e.\ $P_v \cap \Sing (S) = \varnothing$.
\item If $P_v$ is a smooth disk then $P_v$ contains at most one single boundary stop, i.e.\ $P_v \cap \Sigma$ is either empty or consists of a single boundary stop $\sigma$.
\item If $P_v$ is a smooth annulus then $P_v$ contains a single full boundary stop, i.e.\ $P_v \cap \Sigma = \sigma \simeq \mathrm S^1$.
\item If $P_{v_{i_1}}, \dotsc, P_{v_{i_k}}$ are the connected components around an orbifold point, then at least one of them does not contain any (full) boundary stop. This condition allows us to introduce the orbifold stop for each orbifold point, see Definition \ref{definition:orbifoldstop}. 
\end{enumerate}
\end{definition}

Since the $P_v$ are connected components of $S \smallsetminus \Gamma$, they are open in the subspace topology of $S$, but may contain parts of the boundary of $S$. Note that the above definition implies that if $\Gamma$ is a dissection, then for each orbifold point $x \in \Sing (S)$ there is at least one arc in $\Gamma$ connecting to it and, moreover, there is at least one orbifold point which is connected to the boundary of $S$. (This is due to the fact that we only allow $P_v$ to be a smooth annulus, if it contains a full boundary stop.)

\subsubsection{Ribbon complex structure}

Let $\Gamma$ be a dissection. Then $\Gamma$ naturally determines a cell complex $\mathbb G (\Gamma)$ as follows:
\begin{itemize}
\item {\it Vertices}.\; The set $\mathbb G_0 (\Gamma)$ of vertices ($0$-cells) corresponds to the connected components $P_v$ in the decomposition $S \smallsetminus \Gamma = \bigsqcup_v P_v$.
\item {\it Edges}.\; The set $\mathbb G_1 (\Gamma) = \{ e_\gamma \}_{\gamma \in \Gamma}$ of edges ($1$-cells) corresponds to the set of arcs in $\Gamma$ as each $\gamma$ belongs to the boundary of exactly two connected components of $S \smallsetminus \Gamma$. The edge $e_\gamma$ then joins $v$ and $v'$ if $\gamma$ belongs to the boundaries of $P_v$ and $P_{v'}$. (Note that $P_v$ and $P_{v'}$ may coincide, in this case, $e_\gamma$ is a loop.)
\item {\it Cyclic order}.\; The cyclic order on the half-edges incident to $v \in \mathbb G_0 (\Gamma)$ is given by the cyclic order in which the corresponding arcs appear in the boundary of $P_v$. This equips the $1$-skeleton of $\mathbb G (\Gamma)$ with the structure of a {\it ribbon graph}.
\end{itemize}

We call $P_v$ an {\it $n$-gon} if its boundary contains exactly $n$ arcs or equivalently if the valency $\val (v)$ of $v$ equals $n$. If the same arc appears twice in the boundary of an $n$-gon, it is counted twice so that $\val (v)$ counts the number of {\it half-edges} incident to $v$.

\begin{remark}
\label{remark:embedribbongraph}
The ribbon graph $\mathbb G (\Gamma)$ can be embedded into the surface $S$ so that each vertex $v \in \mathbb G_0 (\Gamma)$ lies in the interior of the polygon $P_v$ and each edge $e_{\gamma}\in \mathbb G_1 (\Gamma)$ intersects exactly one arc, namely $\gamma\in \Gamma$.
\end{remark}

If $P_v$ does not contain a (full) boundary stop, i.e.\ if $P_v \cap \Sigma = \varnothing$, then $P_v$ is a smooth disk and we shall denote the vertex $v$ by $\vcirc$. The (full) boundary stops may be encoded into the ribbon graph as follows.

\subsubsection*{Boundary stops as linear orders}

If $P_v$ contains a boundary stop $\sigma$ and $\gamma$ and $\gamma'$ appear (counterclockwise) next to each other in the boundary of $P_v$, then we write $e_\gamma \lessdot_v e_{\gamma'}$ if $\gamma$ and $\gamma'$ are joined by a boundary segment containing $\sigma$ and we write $e_\gamma <_v e_{\gamma'}$ otherwise. In particular $<_v$ defines a {\it linear} order on the arcs in the boundary of $v$ which $\lessdot_v$ completes into the natural cyclic order induced by the orientation of the surface. We call $\lessdot_v$ the {\it missing relation} at $v$ since it is the (order) relation that is missing in the cyclic order at $v$. We indicate that $v$ has a missing relation by denoting $v$ by $\vbullet$.

\subsubsection*{Full boundary stops}

If $P_v$ contains a full boundary stop, we shall denote the vertex by $\vodot$ because we may view the corresponding $n$-gon as a ``punctured'' disk. We may view $\odot$ as a colouring of the vertex $v$.

For an illustration of the polygons corresponding to the vertices $\vcirc$, $\vbullet$ and $\vodot$ see the first three diagrams in Fig.~\ref{fig:vertices}.

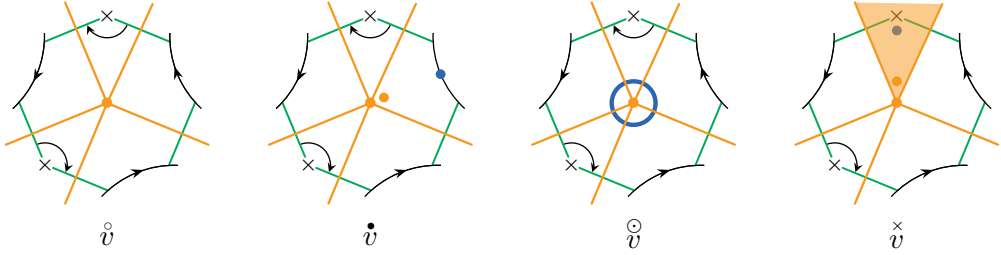
\begin{figure}[ht]
\begin{tikzpicture}[x=1em,y=1em,decoration={markings,mark=at position 0.55 with {\arrow[black]{Stealth[length=4.8pt]}}}]
\clip (-3.6,-5.1) rectangle (30.6,3.6);
\begin{scope}[xshift=9em]
\node[font=\footnotesize,shape=circle,scale=.6] (X) at (90:3em) {};
\node[font=\footnotesize] at (90:3em) {$\times$};
\node[font=\footnotesize,shape=circle,scale=.6] (Y) at (225:3em) {};
\node[font=\footnotesize] at (225:3em) {$\times$};
\draw[line width=.75pt, color=arccolour, line cap=round] (X) to (135:3em) (180:3em) to (Y) to (270:3em) (315:3em) to (0:3em) (45:3em) to (X);
\foreach \a in {0,135} {
\draw[-, line width=.5pt] (157.5+\a:6em) ++(2+\a:3.4em) arc[start angle=2+\a, end angle=-47+\a, radius=3.4em];
\draw[line width=.0pt,postaction={decorate}] ($(157.5+\a:6em)+(357+\a:3.4em)$) arc[start angle=357+\a, end angle=313+\a, radius=3.4em];
}
\draw[-, line width=.5pt] (157.5+225:6em) ++(2+225:3.4em) arc[start angle=2+225, end angle=-47+225, radius=3.4em];
\draw[fill=stopcolour, color=stopcolour] (157.5+225:2.6em) circle(.15em);
\draw[->, line width=.5pt] ($(90:3em)+(-25:.75em)$) arc[start angle=-25, end angle=-155, radius=.75em];
\draw[->, line width=.5pt] ($(90+135:3em)+(-25+135:.75em)$) arc[start angle=-25+135, end angle=-155+135, radius=.75em];
\node at (270:4.5em) {$\vbullet$};
\node[color=ribboncolour] at (0,0) {$\bullet$};
\draw[-, line width=.85, color=ribboncolour, line cap=round] (0,0) to (65.5:3.75em) (0,0) to (112.5:3.75em) (0,0) to (202.5:3.75em) (0,0) to (247.5:3.75em) (0,0) to (-22.5:3.75em);
\draw[fill=ribboncolour, color=ribboncolour] (22.5:.5em) circle(.15em);
\end{scope}
\begin{scope}[xshift=0em]
\node[font=\footnotesize,shape=circle,scale=.6] (X) at (90:3em) {};
\node[font=\footnotesize] at (90:3em) {$\times$};
\node[font=\footnotesize,shape=circle,scale=.6] (Y) at (225:3em) {};
\node[font=\footnotesize] at (225:3em) {$\times$};
\draw[line width=.75pt, color=arccolour, line cap=round] (X) to (135:3em) (180:3em) to (Y) to (270:3em) (315:3em) to (0:3em) (45:3em) to (X);
\foreach \a in {0,135,225} {
\draw[-, line width=.5pt] (157.5+\a:6em) ++(2+\a:3.4em) arc[start angle=2+\a, end angle=-47+\a, radius=3.4em];
\draw[line width=.0pt,postaction={decorate}] ($(157.5+\a:6em)+(357+\a:3.4em)$) arc[start angle=357+\a, end angle=313+\a, radius=3.4em];
}
\draw[->, line width=.5pt] ($(90:3em)+(-25:.75em)$) arc[start angle=-25, end angle=-155, radius=.75em];
\draw[->, line width=.5pt] ($(90+135:3em)+(-25+135:.75em)$) arc[start angle=-25+135, end angle=-155+135, radius=.75em];
\node at (270:4.5em) {$\vcirc$};
\node[color=ribboncolour] at (0,0) {$\bullet$};
\draw[-, line width=.85, color=ribboncolour, line cap=round] (0,0) to (65.5:3.75em) (0,0) to (112.5:3.75em) (0,0) to (202.5:3.75em) (0,0) to (247.5:3.75em) (0,0) to (-22.5:3.75em);
\end{scope}
\begin{scope}[xshift=18em]
\node[font=\footnotesize,shape=circle,scale=.6] (X) at (90:3em) {};
\node[font=\footnotesize] at (90:3em) {$\times$};
\node[font=\footnotesize,shape=circle,scale=.6] (Y) at (225:3em) {};
\node[font=\footnotesize] at (225:3em) {$\times$};
\draw[line width=.75pt, color=arccolour, line cap=round] (X) to (135:3em) (180:3em) to (Y) to (270:3em) (315:3em) to (0:3em) (45:3em) to (X);
\foreach \a in {0,135,225} {
\draw[-, line width=.5pt] (157.5+\a:6em) ++(2+\a:3.4em) arc[start angle=2+\a, end angle=-47+\a, radius=3.4em];
\draw[line width=.0pt,postaction={decorate}] ($(157.5+\a:6em)+(357+\a:3.4em)$) arc[start angle=357+\a, end angle=313+\a, radius=3.4em];
}
\draw[line width=.15em,color=stopcolour] (0,0) circle(.75em);
\draw[->, line width=.5pt] ($(90:3em)+(-25:.75em)$) arc[start angle=-25, end angle=-155, radius=.75em];
\draw[->, line width=.5pt] ($(90+135:3em)+(-25+135:.75em)$) arc[start angle=-25+135, end angle=-155+135, radius=.75em];
\node at (270:4.5em) {$\vodot$};
\node[color=ribboncolour] at (0,0) {$\bullet$};
\draw[-, line width=.85, color=ribboncolour, line cap=round] (0,0) to (65.5:3.75em) (0,0) to (112.5:3.75em) (0,0) to (202.5:3.75em) (0,0) to (247.5:3.75em) (0,0) to (-22.5:3.75em);
\end{scope}
\begin{scope}[xshift=27em]
\node[font=\footnotesize,shape=circle,scale=.6] (X) at (90:3em) {};
\node[font=\footnotesize] at (90:3em) {$\times$};
\node[font=\footnotesize,shape=circle,scale=.6] (Y) at (225:3em) {};
\node[font=\footnotesize] at (225:3em) {$\times$};
\draw[line width=.75pt, color=arccolour, line cap=round] (X) to (135:3em) (180:3em) to (Y) to (270:3em) (315:3em) to (0:3em) (45:3em) to (X);
\foreach \a in {0,135,225} {
\draw[-, line width=.5pt] (157.5+\a:6em) ++(2+\a:3.4em) arc[start angle=2+\a, end angle=-47+\a, radius=3.4em];
\draw[line width=.0pt,postaction={decorate}] ($(157.5+\a:6em)+(357+\a:3.4em)$) arc[start angle=357+\a, end angle=313+\a, radius=3.4em];
}
\draw[->, line width=.5pt] ($(90+135:3em)+(-25+135:.75em)$) arc[start angle=-25+135, end angle=-155+135, radius=.75em];
\draw[fill=stopcolour, color=stopcolour] (90:2.5em) circle(.15em);
\node at (270:4.5em) {$\vtimes$};
\node[color=ribboncolour] at (0,0) {$\bullet$};
\draw[-, line width=.85, color=ribboncolour, line cap=round] (0,0) to (202.5:3.75em) (0,0) to (247.5:3.75em) (0,0) to (-22.5:3.75em);
\draw[-, line width=.85, color=ribboncolour, line cap=round, fill=ribboncolour, fill opacity=.5] (65.5:3.75em) to (0,0) to (112.5:3.75em);
\draw[fill=ribboncolour, color=ribboncolour] (90:.75em) circle(.15em);
\end{scope}
\end{tikzpicture}
\caption{Polygons $P_v$ for different types of vertices in the ribbon complex $\mathbb G (\Gamma)$ associated to a dissection $\Gamma$.}
\label{fig:vertices}
\end{figure}

\subsubsection*{Orbifold structure via orbifold $2$-cells}

Although the above construction defines a ribbon graph (with coloured vertices and linear orders at the vertices of type $\vbullet$), it does not yet contain the data of the orbifold points of $S$. For a number of reasons (that will become apparent in Section \ref{section:cosheaves} and in \cite{barmeierschrollwang}) it is convenient to record the orbifold data by adding an orbifold $2$-cell for each orbifold point as follows.

Since we assume that $S \smallsetminus \Gamma$ contains no orbifold points, each $x \in \Sing (S)$ has at least one arc in $\Gamma$ connecting to it. Let $\gamma_{i_1}, \dotsc, \gamma_{i_k} \in \Gamma$ be arcs connecting to an orbifold point $x \in \Sing (S)$. Then the corresponding edges $e_{\gamma_{i_1}}, \dotsc, e_{\gamma_{i_k}} \in \mathbb G_1 (\Gamma)$ belong to a cycle of $\mathbb G (\Gamma)$. Let $v_{j_1}, \dotsc, v_{j_k}$ denote the corresponding vertices where $e_{\gamma_{i_l}}, e_{\gamma_{i_{l+1}}}$ belong to the boundary of $P_{v_{j_l}}$ for $\gamma_{i_{k+1}} := \gamma_{i_1}$. We have a cycle
\begin{equation}
\label{eq:cycle}
e_{\gamma_{i_1}} <_{v_{j_1}} e_{\gamma_{i_2}} <_{v_{j_2}} \dotsb <_{v_{j_{k-1}}} e_{\gamma_{i_k}} <_{v_{j_k}} e_{\gamma_{i_1}}.
\end{equation}

Whereas such a cycle usually corresponds to a {\it face} of the ribbon graph, i.e.\ to a boundary component of the corresponding surface, we now attach an orbifold $2$-cell $d_x \simeq \overline{\mathbb D} / \mathbb Z_2$ to this cycle with an attaching map of degree $1$ as illustrated in Fig.~\ref{fig:ribboncomplex}. We shall denote the set of $2$-cells by $\mathbb G_2 (\Gamma) = \{ d_x \}_{x \in \Sing (S)}$. We shall call the resulting complex $\mathbb G (\Gamma)$ a {\it ribbon complex}, as it is a type of (orbifold) cell complex of dimension $2$ with a ribbon graph structure on its $1$-skeleton. Such cell complexes with orbifold cells were used in \cite{bahrinotbohmsarkarsong} (under the name {\it {\itshape\bfseries q}-CW complexes}) to compute the integral cohomology of several classes of orbifolds. See also \cite[\S 3]{chen} for a related construction.

The ribbon complex $\mathbb G (\Delta)$ plays a central role in this paper. Although the ribbon complex has (orbifold) cells of dimension $2$ we show in \S\ref{subsection:core} that the partially wrapped Fukaya category can also be described via a ($1$-dimensional) ribbon {\it graph} for the orbifold surface.

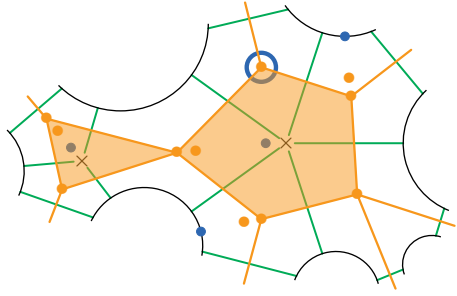
\begin{figure}
\begin{tikzpicture}[x=1em,y=1em,decoration={markings,mark=at position 0.55 with {\arrow[black]{Stealth[length=4.8pt]}}}]
\node[font=\footnotesize,shape=circle,scale=.6] (X) at (0,0) {};
\node[font=\footnotesize] at (0,0) {$\times$};
\node[font=\footnotesize,shape=circle,scale=.6] (Y) at (185:7em) {};
\node[font=\footnotesize] at (185:7em) {$\times$};
\draw[line width=.75pt, color=arccolour, line cap=round] (X) to (0:4em) (X) to (72:4em) (X) to (144:4em) (X) to (216:4em) (X) to (288:4em) ($(72:6.5em)+(-69:2.5em)$) to ($(0:6.5em)+(137:2.5em)$) ($(0:6.5em)+(244:2.5em)$) to ($(-40:6.5em)+(82:1em)$) ($(-40:6.5em)+(190:1em)$) to ($(-72:5.5em)+(20:1.5em)$) ($(-72:5.5em)+(154:1.5em)$) to ($(-144:6em)+(-6:2em)$) ($(-144:6em)+(125:2em)$) to (Y) ($(-144:6em)+(151:2em)$) to ($(185:11em)+(-21:2em)$) ($(185:11em)+(5:2em)$) to (Y) ($(185:11em)+(41:2em)$) to ($(144:7em)+(-132:3em)$) ($(144:7em)+(-105:3em)$) to (Y) ($(144:7em)+(-3:3em)$) to ($(72:6.5em)+(-141:2.5em)$);
\draw[-, line width=.5pt] (0:6.5em) ++(133:2.5em) arc[start angle=133, end angle=248, radius=2.5em];
\draw[-, line width=.5pt] (72:6.5em) ++(-65:2.5em) arc[start angle=-65, end angle=-145, radius=2.5em];
\draw[-, line width=.5pt] (-72:5.5em) ++(15:1.5em) arc[start angle=15, end angle=160, radius=1.5em];
\draw[-, line width=.5pt] (144:7em) ++(0:3em) arc[start angle=0, end angle=-135, radius=3em];
\draw[-, line width=.5pt] (185:11em) ++(45:2em) arc[start angle=45, end angle=-25, radius=2em];
\draw[-, line width=.5pt] (-144:6em) ++(-10:2em) arc[start angle=-10, end angle=155, radius=2em];
\draw[-, line width=.5pt] (-40:6.5em) ++(200:1em) arc[start angle=200, end angle=72, radius=1em];
\draw[line width=.15em,color=stopcolour] (108:2.75em) circle(.5em);
\draw[fill=stopcolour,color=stopcolour] ($(72:6.5em)+(-90:2.5em)$) circle(.15em);
\draw[fill=stopcolour,color=stopcolour] ($(-144:6em)+(14:2em)$) circle(.15em);
\draw[fill=stopcolour,color=stopcolour] (180:.7em) circle(.15em);
\draw[fill=stopcolour,color=stopcolour] ($(185:7em)+(130:.6em)$) circle(.15em);
\node[color=ribboncolour] at ($(185:7em)+(130:1.9em)$) {$\bullet$};
\node[inner sep=0] (A) at ($(185:7em)+(130:1.9em)$) {};
\node[color=ribboncolour] at ($(185:7em)+(130:1.3em)$) {$\bullet$};
\node[color=ribboncolour] at ($(185:7em)+(235:1.2em)$) {$\bullet$};
\node[inner sep=0] (B) at ($(185:7em)+(235:1.2em)$) {};
\node[color=ribboncolour] at (185:3.75em) {$\bullet$};
\node[inner sep=0] (C) at (185:3.75em) {};
\node[color=ribboncolour] at (185:3.1em) {$\bullet$};
\node[color=ribboncolour] at (108:2.75em) {$\bullet$};
\node[inner sep=0] (D) at (108:2.75em) {};
\node[color=ribboncolour] at (36:2.75em) {$\bullet$};
\node[inner sep=0] (E) at (36:2.75em) {};
\node[color=ribboncolour] at ($(36:2.75em)+(97:.6em)$) {$\bullet$};
\node[color=ribboncolour] at (-36:3em) {$\bullet$};
\node[inner sep=0] (F) at (-36:3em) {};
\node[color=ribboncolour] at (-108:2.75em) {$\bullet$};
\node[inner sep=0] (G) at (-108:2.75em) {};
\node[color=ribboncolour] at ($(-108:2.75em)+(192:.6em)$) {$\bullet$};
\draw[line width=0, fill=ribboncolour, draw=ribboncolour, opacity=.5] ($(185:7em)+(130:1.9em)$) -- ($(185:7em)+(235:1.2em)$) -- (185:3.75em) -- cycle;
\draw[line width=0, fill=ribboncolour, draw=ribboncolour, opacity=.5] (185:3.75em) -- (108:2.75em) -- (36:2.75em) -- (-36:3em) -- (-108:2.75em) -- cycle;
\draw[-, line width=.85, color=ribboncolour, line cap=round] (A) to (B) to (C) to (A) (C) to (D) to (E) to (F) to (G) to (C) (A) to ++(130:1em) (B) to ++(249:1.2em) (D) to ++(104:2.3em) (E) to ++(38:2.6em) (F) to ++(-18:3.5em) (F) to ++(-68:3.5em) (G) to ++(-105:2.3em);
\end{tikzpicture}
\caption{Two orbifold $2$-cells attached to a ribbon graph encoding two orbifold points.}
\label{fig:ribboncomplex}
\end{figure}

\begin{remark}
This correspondence between $S$ and $\mathbb G (\Gamma)$ extends the usual correspondence between ribbon graphs and dissected surfaces
\begin{equation}
\label{eq:duality}
\begin{tikzpicture}[x=6.5em,y=3em,baseline=-3.5em]
\node[align=center, anchor=north] at (-1.3,0) {$S$\strut \\[-.25em] {\small\it orbifold surface}\strut};
\node[align=center, anchor=north] at (-1.3,-1) {$\mathbb G (\Gamma)$\strut \\[-.25em] {\small\it ribbon complex}\strut};
\node[align=center, anchor=north] at (0,0) {$P_v \subset S \smallsetminus \Gamma$\strut \\[-.25em] {\footnotesize dimension $2$} \strut};
\node[align=center, anchor=north] at (0,-1) {$v \in \mathbb G_0 (\Gamma)$\strut \\[-.25em] {\footnotesize $0$-cell} \strut};
\node[align=center, anchor=north] at (1,0) {$\gamma \in \Gamma$\strut \\[-.25em] {\footnotesize dimension $1$} \strut};
\node[align=center, anchor=north] at (1,-1) {$e_\gamma \in \mathbb G_1 (\Gamma)$\strut \\[-.25em] {\footnotesize $1$-cell} \strut};
\node[align=center, anchor=north] at (2,0) {$x \in \Sing (S)$\strut \\[-.25em] {\footnotesize dimension $0$} \strut};
\node[align=center, anchor=north] at (2,-1) {$d_x \in \mathbb G_2 (\Gamma)$\strut \\[-.25em] {\footnotesize $2$-cell} \strut};
\end{tikzpicture}
\end{equation}
which one may view as a type of duality (see also \S~\ref{subsubsection:gradedribbon}). The $1$-cells $e_\gamma$ in the boundary of a $2$-cell $d_x$ are then dual to the arcs $\gamma$ connecting to $x$. Note that if we embed $\mathbb G (\Gamma)$ into $S$ as in Remark \ref{remark:embedribbongraph}, then the orbifold point $x$ lies in the interior of $d_x$ (see Fig.~\ref{fig:ribboncomplex}).
\end{remark}

One may now rephrase properties of $\Gamma$ in terms of properties of the associated ribbon complex $\mathbb G (\Gamma)$. For example if a $2$-cell $d_x$ has a loop as its boundary (a single edge connecting to a single vertex), then there is only one arc in $\Gamma$ connecting to $x$.

\subsubsection*{Orbifold stops}

Before proceeding, we shall consider one extra piece of structure on the ribbon complex $\mathbb G (\Gamma)$ of a dissection $\Gamma$.

\begin{definition}\label{definition:orbifoldstop}
An {\it orbifold stop} at an orbifold point $x \in \Sing (S)$ is a linear order at one of the vertices of type $\vcirc$ in the boundary of the $2$-cell $d_x$ such that the missing relation belongs to the cycle \eqref{eq:cycle}. (Note that by {\itemiv} of Definition \ref{definition:dissection} there is at least one such vertex.)
\end{definition}

We denote the corresponding vertex by $\vtimes$ (instead of $\vcirc$), see the last diagram of Fig.~\ref{fig:vertices} for an illustration. There are at least two reasons for considering this extra data.

For one, the choice of an orbifold stop is essentially equivalent to the choice of suitable direct summands in the orbit category of the $\mathbb Z_2$-invariant arcs in the double cover of $S$. Such a choice forces there to be one maximal path around an orbifold point as in Section \ref{section:disk} (see in particular Fig.~\ref{fig:orbifolddisk}). The orbifold stop indicates the missing morphism in one of the polygons around the orbifold point which in the figures we illustrate by {\color{stopcolour} $\bullet$} near each orbifold point $\times$.

A second reason can be given by viewing the partially wrapped Fukaya category of an orbifold surface as a deformation of the partially wrapped Fukaya category of a smooth surface, each orbifold point corresponding to a boundary component of winding number $1$ with one stop as in Remarks \ref{remark:compactification} and \ref{remark:deformationinterpretation}. (The details of this second perspective will be given in \cite{barmeierschrollwang}.)

For now, we may view the choice of orbifold stops simply as an extra piece of combinatorial data analogous to the linear order already chosen at the vertices of type $\vbullet$. Since there is a unique vertex with orbifold stop in the boundary of the $2$-cell $d_x$ for each orbifold point $x \in \Sing (S)$, we usually index the vertex by $x$ and write $\vtimes_x$.

\begin{remark}
\label{remark:valency}
Note that $\val(\vtimes_x) \geq 2$ for any $x \in \Sing (S)$ since we assume that there is at least one arc in $\Gamma$ connecting to $x$ and hence at least two half-edges connecting to $\vtimes_x$.
\end{remark}

\subsubsection*{Weakly admissible dissections}

We now introduce the notion of a {\it weakly admissible dissection} which also takes the orbifold stops into account. We use this notion in Section \ref{section:cosheaves} to define the partially wrapped Fukaya category $\mathcal W (\mathbf S)$ as global sections of a cosheaf of categories of the associated ribbon complex. The slightly stronger notion of {\it admissible dissection} will allow us to give all higher structures in $\mathcal W (\mathbf S)$ of orbifold surfaces explicitly (see Section \ref{section:ainfinity}).

\begin{definition}
\label{definition:admissible}
A {\it weakly admissible dissection} $\Delta$ consists of a dissection $\Gamma$ together with a choice of orbifold stop for each orbifold point $x \in \Sing (S)$. It is called {\it admissible} if for each orbifold $2$-cell $d_x$ there are at least two vertices in the boundary of $d_x$ (i.e.\ the boundary of $d_x$ is not a loop) and moreover one of the following is true:
\begin{itemize}
\item at least one vertex in the boundary of $d_x$ is of type $\vbullet$
\item at least one vertex in the boundary of $d_x$ is of type $\vodot$
\item at least three vertices in the boundary of $d_x$ are of type $\vtimes$.
\end{itemize}
\end{definition}

\begin{remark}
\label{remark:same}
If $S$ is a smooth surface, then $\Sing (S) = \varnothing$ and any dissection is vacuously weakly admissible and admissible.
\end{remark}

\begin{figure}[ht]
\begin{tikzpicture}[x=1em,y=1em]
\begin{scope}
\node[font=\scriptsize,shape=circle,scale=.6] (L) at (150:3em) {};
\node[font=\scriptsize,shape=circle,scale=.6] (R) at (30:3em) {};
\draw[line width=.5pt] (0,0) circle(7em);
\draw[line width=.15em,color=stopcolour] (270:1.75em) circle(.75em);
\foreach \a in {30,90,150,210,270,330} {
\draw[fill=stopcolour,color=stopcolour] (\a:7em) circle(.15em);
}
\path[line width=.75pt,color=arccolour] (L) edge (R);
\path[line width=.75pt,out=275,in=265,looseness=3,color=arccolour] (L) edge (R);
\path[line width=.75pt,out=-10,in=160,looseness=1,color=arccolour] (170:7em) edge (L);
\path[line width=.75pt,out=10,in=200,looseness=1,color=arccolour] (190:7em) edge (L);
\path[line width=.75pt,out=60,in=240,looseness=1,color=arccolour] (240:7em) edge (L);
\path[line width=.75pt,out=285,in=255,looseness=1.5,color=arccolour] (105:7em) edge (75:7em);
\path[line width=.75pt,out=225,in=195,looseness=1.5,color=arccolour] (45:7em) edge (15:7em);
\path[line width=.75pt,out=165,in=135,looseness=1.5,color=arccolour] (-15:7em) edge (-45:7em);
\path[line width=.75pt,out=180,in=-40,looseness=1,color=arccolour] (0:7em) edge (R);
\path[line width=.75pt,out=240,in=105,looseness=1,color=arccolour] (60:7em) edge (R);
\node[font=\scriptsize] at (150:3em) {$\times$};
\node[font=\scriptsize] at ($(150:3em)+(90:.7em)$) {$y$};
\node[font=\scriptsize] at (30:3em) {$\times$};
\node[font=\scriptsize] at ($(30:3em)+(140:.8em)$) {$x$};
\draw[fill=stopcolour,color=stopcolour] ($(150:3em)+(182:.85em)$) circle(.15em);
\draw[fill=stopcolour,color=stopcolour] ($(30:3em)+(40:.5em)$) circle(.15em);
\node[circle, fill=ribboncolour, minimum size=.3em, inner sep=0] (A1) at (270:1.75) {};
\node[circle, fill=ribboncolour, minimum size=.3em, inner sep=0] (A2) at (175:5.5) {};
\node[color=ribboncolour] at ($(175:5.5em)+(5:.5em)$) {$\bullet$};
\node[circle, fill=ribboncolour, minimum size=.3em, inner sep=0] (A3) at (33:4.8) {};
\node[color=ribboncolour] at ($(33:4.8)+(220:.5em)$) {$\bullet$};
\node[circle, fill=ribboncolour, minimum size=.3em, inner sep=0] (A4) at (90:6) {};
\node[circle, fill=ribboncolour, minimum size=.3em, inner sep=0] (A5) at (-30:6) {};
\node[circle, fill=ribboncolour, minimum size=.3em, inner sep=0] (A6) at (-60:5) {};
\node[circle, fill=ribboncolour, minimum size=.3em, inner sep=0] (A7) at (120:5) {};
\node[circle, fill=ribboncolour, minimum size=.3em, inner sep=0] (A8) at (210:5.5) {};
\node[circle, fill=ribboncolour, minimum size=.3em, inner sep=0] (A9) at (29:6) {};
\draw[-, line width=.85, color=ribboncolour, line cap=round] (A3) to (A6) to (A1) to (A7);
\draw[-, line width=.85, color=ribboncolour, line cap=round]  (A9) to (A3) to (A7);
\draw[-, line width=.85, color=ribboncolour, line cap=round]  (A4) to (A7) to (A2) to (A8) to (A6) to (A5);
\draw[line width=0, fill=ribboncolour, draw=ribboncolour, opacity=.5] (120:5) -- (175:5.5) -- (210:5.5) -- (-60:5)-- (270:1.75)-- (120:5)--cycle;
\draw[line width=0, fill=ribboncolour, draw=ribboncolour, opacity=.5] (120:5) -- (33:4.8)  -- (-60:5)-- (270:1.75)-- (120:5)--cycle;
\end{scope}
\begin{scope}[xshift=18em]
\node[font=\scriptsize] at ($(150:3em)+(90:.7em)$) {$y$};
\node[font=\scriptsize] at (30:3em) {$\times$};
\node[font=\scriptsize] at ($(30:3em)+(140:.8em)$) {$x$};
\node[font=\scriptsize] at (150:3em) {$\times$};
\node[font=\scriptsize,shape=circle,scale=.6] (L) at (150:3em) {};
\node[font=\scriptsize,shape=circle,scale=.6] (R) at (30:3em) {};
\node[font=\scriptsize] (R1) at ($(30:3em)+(180:7em)+(6:7em)$) {};
\node[font=\scriptsize] (R2) at ($(30:3em)+(180:7em)+(-6:7em)$) {};
\draw[line width=.5pt] (0,0) circle(7em);
\draw[line width=.15em,color=stopcolour] (270:2.75em) circle(.75em);
\foreach \a in {30,90,150,270,330} {
\draw[fill=stopcolour,color=stopcolour] (\a:7em) circle(.15em);
}
\path[line width=.75pt,color=arccolour] (240:7em) edge[out=60,in=265,looseness=1] (237:4.5em) (237:4.5em) edge[out=85,in=95,looseness=2] (303:4.5em) (303:4.5em) edge[out=275,in=120,looseness=1] (300:7em);
\path[line width=.75pt,color=arccolour] (L) edge (R);
\path[line width=.75pt,out=-10,in=160,looseness=1,color=arccolour] (170:7em) edge (L);
\path[line width=.75pt,out=10,in=200,looseness=1,color=arccolour] (210:7em) edge (L);
\path[line width=.75pt,out=180,in=-40,looseness=1,color=arccolour] (0:7em) edge (R);
\path[line width=.75pt,out=240,in=105,looseness=1,color=arccolour] (60:7em) edge (R);
\draw[fill=stopcolour,color=stopcolour] ($(150:3em)+(182:.9em)$) circle(.15em);
\draw[fill=stopcolour,color=stopcolour] ($(30:3em)+(-120:.6em)$) circle(.15em);
\path[line width=.75pt,out=345,in=315,looseness=1.5,color=arccolour] (160:7em) edge (140:7em);
\path[line width=.75pt,out=105,in=75,looseness=1.5,color=arccolour] (285:7em) edge (255:7em);
\path[line width=.75pt,out=285,in=255,looseness=1.5,color=arccolour] (105:7em) edge (75:7em);
\path[line width=.75pt,out=165,in=135,looseness=1.5,color=arccolour] (-15:7em) edge (-45:7em);
\node[circle, fill=ribboncolour, minimum size=.3em, inner sep=0] (A1) at (270:2.75) {};
\node[circle, fill=ribboncolour, minimum size=.3em, inner sep=0] (A2) at (185:6) {};
\node[color=ribboncolour] at ($(185:6)+(25:.7em)$) {$\bullet$};
\node[circle, fill=ribboncolour, minimum size=.3em, inner sep=0] (A3) at (20:5.5) {};
\node[color=ribboncolour] at ($(-10:1.5em)+(70:.7em)$) {$\bullet$};
\node[circle, fill=ribboncolour, minimum size=.3em, inner sep=0] (A4) at (90:6) {};
\node[circle, fill=ribboncolour, minimum size=.3em, inner sep=0] (A5) at (-30:6) {};
\node[circle, fill=ribboncolour, minimum size=.3em, inner sep=0] (A6) at (-90:6) {};
\node[circle, fill=ribboncolour, minimum size=.3em, inner sep=0] (A7) at (120:5) {};
\node[circle, fill=ribboncolour, minimum size=.3em, inner sep=0] (A9) at (-10:1.5) {};
\node[circle, fill=ribboncolour, minimum size=.3em, inner sep=0] (A0) at (150:6.5) {};
\draw[-, line width=.85, color=ribboncolour, line cap=round]  (A6) to (A1);
\draw[-, line width=.85, color=ribboncolour, line cap=round]  (A7) to (A2) to (A9); 
\draw[-, line width=.85, color=ribboncolour, line cap=round]  (A3) to (A7) to (A0);
\draw[-, line width=.85, color=ribboncolour, line cap=round]  (A4) to (A7) to (A9) to (A1)  to (A6)   (A3) to (A9) to (A5);
\draw[line width=0, fill=ribboncolour, draw=ribboncolour, opacity=.5] (185:6) -- (-10:1.5) -- (120:5) --cycle;
\draw[line width=0, fill=ribboncolour, draw=ribboncolour, opacity=.5] (120:5) -- (20:5.5) -- (-10:1.5) --cycle;
\end{scope}
\end{tikzpicture}
\caption{The left dissection is admissible (Definition \ref{definition:admissible}) and formal (Definition \ref{definition:DGformal}). The right dissection is not admissible since the boundary of the $2$-cell $d_y$ only contains two vertices of type $\vtimes$ and no vertices of type $\vbullet$ or $\vodot$. The right dissection also does not satisfy the condition for being formal since $\val (\vtimes_x) = 5$.}
\label{fig:dissection}
\end{figure}
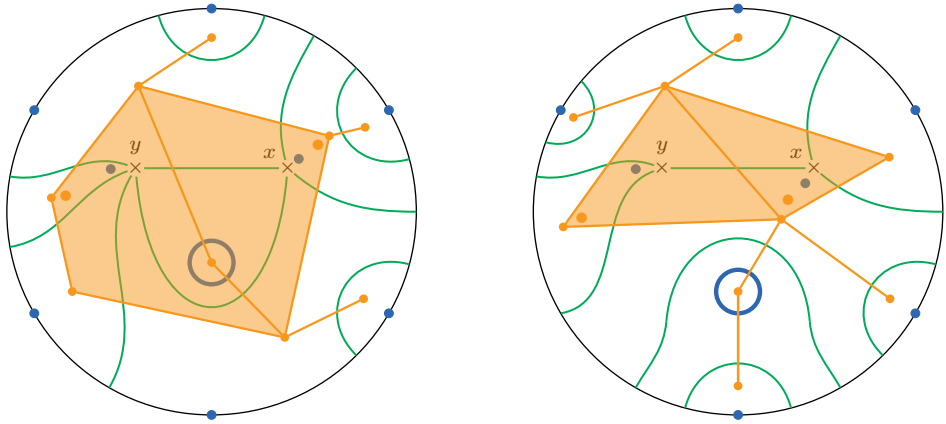

Although the definition of an admissible dissection might appear somewhat mysterious at this point, it reduces the number of possible higher multiplications that are created by introducing new higher multiplications around orbifold points (as in Section \ref{section:disk}) to a very manageable quantity (see Section \ref{section:ainfinity}). Note that the definition of admissibility is still ``local'' on the ribbon complex, as it only involves the $n$-gons surrounding a particular orbifold point. In particular, we do not need to assume any global shape of the arcs connecting to orbifold points.

\begin{definition}
If $\Delta$ is a weakly admissible dissection with underlying dissection $\Gamma$ we denote by $\mathbb G (\Delta)$ the ribbon complex obtained from $\mathbb G (\Gamma)$ by encoding also the missing relations corresponding to orbifold stops as linear orders at the vertices of type $\vtimes$.
\end{definition}

Note that the choice of orbifold stop ensures that there is at least one vertex of type $\vtimes$ in the boundary of each $2$-cell. Admissibility thus requires at least two more such vertices or at least one vertex with a (full) boundary stop. It is straightforward to see that any orbifold surface with stops does admit an admissible dissection.

\begin{remark}[(Partial compactifications and deformations)]
\label{remark:compactification}
We may view $(S, \Sigma)$ as the ``partial compactification'' of a smooth surface $(\widehat S, \widehat \Sigma)$. Concretely, we may put $\widehat S = S \smallsetminus \bigsqcup_{x \in \Sing (S)} \mathbb D_x$ where the $\mathbb D_x$'s are pairwise disjoint open disks containing the orbifold points $x \in \Sing (S)$. The boundary circles $\partial \mathbb D_x \subset S$ turn into new boundary components of $\widehat S$ and we set $\widehat \Sigma = \Sigma \cup \{ \sigma_x \}$ where each $\sigma_x \subset \partial \mathbb D_x$ is a single boundary stop representing the orbifold stop at $x$. Any weakly admissible dissection $\Delta$ of $(S, \Sigma)$ gives a dissection $\widehat \Delta$ of $(\widehat S, \widehat \Sigma)$ by restricting the arcs in $\Delta$ to $\widehat S$ (choosing the $\mathbb D_x$'s such that the arcs in $\Delta$ cross $\partial \mathbb D_x$ at most once transversely). Then $\mathbb G (\widehat \Delta)$ is simply the $1$-skeleton of $\mathbb G (\Delta)$, i.e.\ $\mathbb G (\widehat \Delta)$ is obtained by deleting the $2$-cells from $\mathbb G (\Delta)$.

Although the surface $\widehat S$ is already compact, we may view $S$ as a ``partial compactification'' of $\widehat S$ by gluing the orbifold disks $\mathbb D_x \simeq \mathbb D / \mathbb Z_2$ back in. If we view $(S, \Sigma)$ as the compact ``inner domain'' of a (noncompact) Weinstein manifold, this process indeed corresponds to an orbifold partial compactification of the Weinstein manifold in the usual sense.

The correspondence between $(S, \Sigma)$ with weakly admissible dissection $\Delta$ and $(\widehat S, \widehat \Sigma)$ with (weakly) admissible dissection $\widehat \Delta$ is illustrated for the special type of dissection of \S\ref{subsection:special} in Fig.~\ref{fig:compactification}.

Moreover, any line field $\eta$ on $S$ restricts to a line field on $\widehat S \subset S$ with winding number $1$ around each boundary component $\partial \mathbb D_x$. If we write $\widehat \eta = \eta \vert_{\widehat S}$ and $\mathbf S = (S, \Sigma, \eta)$ and $\widehat{\mathbf S} = (\widehat S, \widehat \Sigma, \widehat \eta)$, then the partially wrapped Fukaya category $\mathcal W (\mathbf S)$ is an {\it algebraic deformation} of $\mathcal W (\widehat{\mathbf S})$. This fact, and the relation to the Hochschild cohomology $\HH^2 (\mathcal W (\widehat{\mathbf S}), \mathcal W (\widehat{\mathbf S}))$, will be proven in \cite{barmeierschrollwang}. The fact that algebraic deformations of partially wrapped Fukaya categories of smooth surfaces is related to partial (orbifold) compactifications confirms a general expectation formulated in P.~Seidel's ICM 2002 address \cite{seidel1}. See also Remark \ref{remark:deformationinterpretation} for the relationship between the higher products of the A$_\infty$ category $\mathbf A_\Delta$ defined in Section \ref{section:ainfinity} and A$_\infty$ deformations.
\end{remark}

\begin{figure}
\begin{tikzpicture}[x=1em,y=1em,decoration={markings,mark=at position 0.55 with {\arrow[black]{Stealth[length=4.8pt]}}}]
\clip (-15.4,-9.4) rectangle (14.1,.3);
\begin{scope}[xshift=-.3em,yshift=-5em]
\path[->, line width=.4pt] (-3.75,.3) edge node[font=\scriptsize, above] {contraction} (-8.25,.3);
\path[->, line width=.4pt] (-8.25,-.3) edge node[font=\scriptsize, below=-.2ex] {expansion} (-3.75,-.3);
\path[->, line width=.4pt] (3.75,.3) edge node[font=\scriptsize, above] {contraction} (8.25,.3);
\path[->, line width=.4pt] (8.25,-.3) edge node[font=\scriptsize, below=-.2ex] {expansion} (3.75,-.3);
\end{scope}
\begin{scope}[xshift=-12em,rotate=80]
\node[font=\footnotesize,shape=circle,scale=.6] (X) at (0,0) {};
\node[font=\footnotesize] at (0,0) {$\times$};
\node[font=\footnotesize,shape=circle,scale=.6] (Y) at (185:7em) {};
\node[font=\footnotesize] at (185:7em) {$\times$};
\draw[line width=.75pt, color=arccolour, line cap=round] (X) to (144:4em) (X) to (216:4em) ($(-144:6em)+(125:2em)$) to (Y) ($(-144:6em)+(151:2em)$) to ($(185:11em)+(-21:2em)$) ($(185:11em)+(5:2em)$) to (Y) ($(185:11em)+(41:2em)$) to ($(144:7em)+(-132:3em)$) ($(144:7em)+(-105:3em)$) to (Y);
\draw[-, line width=.5pt] (144:7em) ++(-32:3em) arc[start angle=-32, end angle=-135, radius=3em];
\draw[-, line width=.5pt] (185:11em) ++(45:2em) arc[start angle=45, end angle=-25, radius=2em];
\draw[-, line width=.5pt] (-144:6em) ++(31:2em) arc[start angle=31, end angle=155, radius=2em];
\draw[fill=stopcolour,color=stopcolour] ($(185:7em)+(130:.6em)$) circle(.15em);
\node[color=ribboncolour] at ($(185:7em)+(130:1.9em)$) {$\bullet$};
\node[inner sep=0] (A) at ($(185:7em)+(130:1.9em)$) {};
\node[color=ribboncolour] at ($(185:7em)+(130:1.3em)$) {$\bullet$};
\node[color=ribboncolour] at ($(185:7em)+(235:1.2em)$) {$\bullet$};
\node[inner sep=0] (B) at ($(185:7em)+(235:1.2em)$) {};
\node[color=ribboncolour] at (185:3.75em) {$\bullet$};
\node[inner sep=0] (C) at (185:3.75em) {};
\draw[line width=0, fill=ribboncolour, draw=ribboncolour, opacity=.5] ($(185:7em)+(130:1.9em)$) -- ($(185:7em)+(235:1.2em)$) -- (185:3.75em) -- cycle;
\draw[-, line width=.85, color=ribboncolour, line cap=round] (A) to (B) to (C) to (A) (A) to ++(130:1em) (B) to ++(249:1.2em) (C) to ++(53:3em) (C) to ++(-52:2.5em);
\end{scope}
\begin{scope}[rotate=80]
\node[font=\footnotesize,shape=circle,scale=.6] (X) at (0,0) {};
\node[font=\footnotesize] at (0,0) {$\times$};
\node[font=\footnotesize,shape=circle,scale=.6] (Y) at (185:7em) {};
\node[font=\footnotesize] at (185:7em) {$\times$};
\draw[line width=.75pt, color=arccolour, line cap=round] (X) to (144:4em) (X) to (216:4em) ($(-144:6em)+(125:2em)$) to (Y) ($(-144:6em)+(151:2em)$) to ($(185:11em)+(-21:2em)$) ($(185:11em)+(5:2em)$) to (Y) ($(185:11em)+(41:2em)$) to ($(144:7em)+(-132:3em)$) ($(144:7em)+(-105:3em)$) to (Y) ($(144:7em)+(280:3em)$) to ($(-144:6em)+(90:2em)$);
\draw[-, line width=.5pt] (144:7em) ++(-32:3em) arc[start angle=-32, end angle=-135, radius=3em];
\draw[-, line width=.5pt] (185:11em) ++(45:2em) arc[start angle=45, end angle=-25, radius=2em];
\draw[-, line width=.5pt] (-144:6em) ++(31:2em) arc[start angle=31, end angle=155, radius=2em];
\draw[fill=stopcolour,color=stopcolour] ($(185:7em)+(135:.5em)$) circle(.15em);
\node[color=ribboncolour] at ($(185:7em)+(130:1.9em)$) {$\bullet$};
\node[inner sep=0] (A) at ($(185:7em)+(130:1.9em)$) {};
\node[color=ribboncolour] at ($(185:7em)+(133:1.1em)$) {$\bullet$};
\node[color=ribboncolour] at ($(185:7em)+(235:1.2em)$) {$\bullet$};
\node[inner sep=0] (B) at ($(185:7em)+(235:1.2em)$) {};
\node[color=ribboncolour] at (185:3em) {$\bullet$};
\node[color=ribboncolour] at (185:5.8em) {$\bullet$};
\node[inner sep=0] (C) at (185:5.8em) {};
\node[inner sep=0] (C') at (185:3em) {};
\node[inner sep=0] (D) at (130:2.2em) {};
\node[inner sep=0] (E) at (230:2.5em) {};
\draw[line width=0, fill=ribboncolour, draw=ribboncolour, opacity=.5] ($(185:7em)+(130:1.9em)$) -- ($(185:7em)+(235:1.2em)$) -- (185:5.8em) -- cycle;
\draw[-, line width=.85, color=ribboncolour, line cap=round] (A) to (B) to (C) to (A) (A) to ++(130:1em) (B) to ++(249:1.2em) (C) to (C') (C') to (D) (C') to (E);
\end{scope}
\begin{scope}[xshift=12em, rotate=80]
\node[font=\footnotesize,shape=circle,scale=.6] (X) at (0,0) {};
\node[font=\footnotesize] at (0,0) {$\times$};
\node[font=\footnotesize,shape=circle,scale=.6] (Y) at (185:7em) {};
\node[font=\footnotesize] at (185:7em) {$\times$};
\draw[line width=.75pt, color=arccolour, line cap=round] (X) to (144:4em) (X) to (216:4em) ($(-144:6em)+(125:2em)$) to (Y) ($(-144:6em)+(151:2em)$) to ($(185:11em)+(-21:2em)$) ($(185:11em)+(5:2em)$) to (Y) ($(185:11em)+(41:2em)$) to ($(144:7em)+(-132:3em)$) ($(144:7em)+(280:3em)$) to ($(-144:6em)+(90:2em)$);
\draw[-, line width=.5pt] (144:7em) ++(-32:3em) arc[start angle=-32, end angle=-135, radius=3em];
\draw[-, line width=.5pt] (185:11em) ++(45:2em) arc[start angle=45, end angle=-25, radius=2em];
\draw[-, line width=.5pt] (-144:6em) ++(31:2em) arc[start angle=31, end angle=155, radius=2em];
\draw[fill=stopcolour,color=stopcolour] ($(185:7em)+(75:.5em)$) circle(.15em);
\node[color=ribboncolour] at ($(185:7em)+(75:1.4em)$) {$\bullet$};
\node[inner sep=0] (A) at ($(185:7em)+(75:1.4em)$) {};
\node[color=ribboncolour] at ($(185:7em)+(75:.95em)$) {$\bullet$};
\node[color=ribboncolour] at ($(185:7em)+(235:1.2em)$) {$\bullet$};
\node[inner sep=0] (B) at ($(185:7em)+(235:1.2em)$) {};
\node[color=ribboncolour] at (185:3em) {$\bullet$};
\node[inner sep=0] (C') at (185:3em) {};
\node[inner sep=0] (D) at (130:2.2em) {};
\node[inner sep=0] (E) at (230:2.5em) {};
\draw[line width=0, fill=ribboncolour, draw=ribboncolour, opacity=.5] ($(185:7em)+(75:1.4em)$) to[bend left=60] ($(185:7em)+(235:1.2em)$) to[bend left=60] cycle;
\draw[-, line width=.85, color=ribboncolour, line cap=round] (A) to[bend left=60] (B) to[bend left=60] (A) to ++(148:2em) (B) to ++(249:1.2em) (A) to (C') (C') to (D) (C') to (E);
\end{scope}
\end{tikzpicture}
\caption{Edge contraction and expansion on ribbon complexes.}
\label{fig:edgecontraction}
\end{figure}
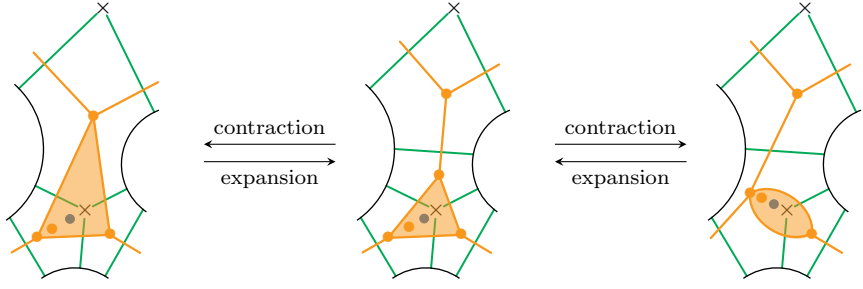

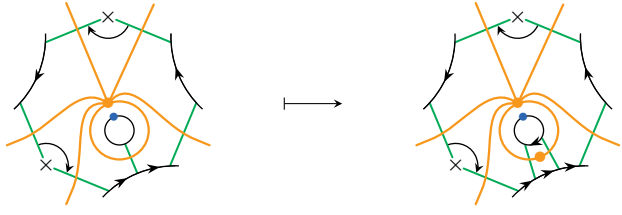
\begin{figure}
\begin{tikzpicture}[x=1em,y=1em,decoration={markings,mark=at position 0.55 with {\arrow[black]{Stealth[length=4.8pt]}}}]
\begin{scope}
\clip (-3.6,-3.6) rectangle (17.6,3.6);
\begin{scope}[xshift=0em]
\node[font=\footnotesize,shape=circle,scale=.6] (X) at (90:3em) {};
\node[font=\footnotesize] at (90:3em) {$\times$};
\node[font=\footnotesize,shape=circle,scale=.6] (Y) at (225:3em) {};
\node[font=\footnotesize] at (225:3em) {$\times$};
\draw[line width=.75pt, color=arccolour] (292.5:1.5em) -- (292.5:2.6em);
\draw[line width=.5pt] (292.5:1em) circle(.5em);
\node[font=\scriptsize, color=stopcolour] at (292.5:.5em) {$\bullet$};
\draw[line width=.75pt, color=arccolour, line cap=round] (X) to (135:3em) (180:3em) to (Y) to (270:3em) (315:3em) to (0:3em) (45:3em) to (X);
\foreach \a in {0,135,225} {
\draw[-, line width=.5pt] (157.5+\a:6em) ++(2+\a:3.4em) arc[start angle=2+\a, end angle=-47+\a, radius=3.4em];
}
\draw[line width=.0pt,postaction={decorate}] ($(157.5:6em)+(357:3.4em)$) arc[start angle=357, end angle=313, radius=3.4em];
\draw[line width=.0pt,postaction={decorate}] ($(157.5+225:6em)+(357+225:3.4em)$) arc[start angle=357+225, end angle=313+225, radius=3.4em];
\draw[line width=.0pt,postaction={decorate}] ($(157.5+135:6em)+(130:3.4em)$) arc[start angle=130, end angle=110, radius=3.4em];
\draw[line width=.0pt,postaction={decorate}] ($(157.5+135:6em)+(110:3.4em)$) arc[start angle=110, end angle=90, radius=3.4em];
\draw[->, line width=.5pt] ($(90:3em)+(-25:.75em)$) arc[start angle=-25, end angle=-155, radius=.75em];
\draw[->, line width=.5pt] ($(90+135:3em)+(-25+135:.75em)$) arc[start angle=-25+135, end angle=-155+135, radius=.75em];
\node[color=ribboncolour] at (0,0) {$\bullet$};
\draw[-, line width=.85, color=ribboncolour, line cap=round] (0,0) to (65.5:3.75em) (0,0) to (112.5:3.75em);
\draw[-, line width=.85, color=ribboncolour, line cap=round] (0,0) to[out=142.5, in=22.5] (202.5:3.2em) to (202.5:3.75em);
\draw[-, line width=.85, color=ribboncolour, line cap=round] (0,0) to[out=172.5, in=90-22.5, looseness=1.1] (247.5:3.2em) to (247.5:3.75em);
\draw[-, line width=.85, color=ribboncolour, line cap=round] (0,0) to[out=52.5, in=180-22.5, looseness=1.1] (-22.5:3.2em) to (-22.5:3.75em);
\draw[-, line width=.85, color=ribboncolour, line cap=round] (292.5:1em) circle(1em);
\end{scope}
\begin{scope}[xshift=14em]
\draw[line width=.0pt,postaction={decorate}, color=white] ($(292.5:1.5em)+(385.5:.5em)$) to ($(292.5:1.5em)+(202.5:.75em)$);
\node[font=\footnotesize,shape=circle,scale=.6] (X) at (90:3em) {};
\node[font=\footnotesize] at (90:3em) {$\times$};
\node[font=\footnotesize,shape=circle,scale=.6] (Y) at (225:3em) {};
\node[font=\footnotesize] at (225:3em) {$\times$};
\draw[line width=.75pt, color=arccolour] (280:1.45em) -- (283:2.67em);
\draw[line width=.75pt, color=arccolour] (305:1.45em) -- (302:2.67em);
\draw[line width=.5pt] (292.5:1em) circle(.5em);
\node[font=\scriptsize, color=stopcolour] at (292.5:.5em) {$\bullet$};
\draw[line width=.75pt, color=arccolour, line cap=round] (X) to (135:3em) (180:3em) to (Y) to (270:3em) (315:3em) to (0:3em) (45:3em) to (X);
\foreach \a in {0,135,225} {
\draw[-, line width=.5pt] (157.5+\a:6em) ++(2+\a:3.4em) arc[start angle=2+\a, end angle=-47+\a, radius=3.4em];
}
\draw[->, line width=.5pt] ($(90:3em)+(-25:.75em)$) arc[start angle=-25, end angle=-155, radius=.75em];
\draw[->, line width=.5pt] ($(90+135:3em)+(-25+135:.75em)$) arc[start angle=-25+135, end angle=-155+135, radius=.75em];
\node[color=ribboncolour] at (0,0) {$\bullet$};
\draw[-, line width=.85, color=ribboncolour, line cap=round] (0,0) to (65.5:3.75em) (0,0) to (112.5:3.75em);
\draw[-, line width=.85, color=ribboncolour, line cap=round] (0,0) to[out=142.5, in=22.5] (202.5:3.2em) to (202.5:3.75em);
\draw[-, line width=.85, color=ribboncolour, line cap=round] (0,0) to[out=172.5, in=90-22.5, looseness=1.1] (247.5:3.2em) to (247.5:3.75em);
\draw[-, line width=.85, color=ribboncolour, line cap=round] (0,0) to[out=52.5, in=180-22.5, looseness=1.1] (-22.5:3.2em) to (-22.5:3.75em);
\draw[-, line width=.85, color=ribboncolour, line cap=round] (292.5:1em) circle(1em);
\node[color=ribboncolour] at (292.5:2em) {$\bullet$};
\draw[line width=.0pt,postaction={decorate}] ($(157.5:6em)+(357:3.4em)$) arc[start angle=357, end angle=313, radius=3.4em];
\draw[line width=.0pt,postaction={decorate}] ($(157.5+225:6em)+(357+225:3.4em)$) arc[start angle=357+225, end angle=313+225, radius=3.4em];
\draw[line width=.0pt,postaction={decorate}] ($(157.5+135:6em)+(130:3.4em)$) arc[start angle=130, end angle=115, radius=3.4em];
\draw[line width=.0pt,postaction={decorate}] ($(157.5+135:6em)+(118:3.4em)$) arc[start angle=118, end angle=100, radius=3.4em];
\draw[line width=.0pt,postaction={decorate}] ($(157.5+135:6em)+(100:3.4em)$) arc[start angle=100, end angle=90, radius=3.4em];
\end{scope}
\begin{scope}[xshift=7em]
\draw[|->, line width=.4pt] (-1,0) to (1,0);
\end{scope}
\end{scope}
\end{tikzpicture}
\caption{Adding isotopic arcs to $\Delta$ is dual to subdividing edges in $\mathbb G (\Delta)$.}
\label{fig:subdividingloops}
\end{figure}

\subsubsection{Operations on ribbon complexes}
\label{subsubsection:operations}

Given a dissection $\Delta$, removing or adding arcs corresponds to {\it edge contraction} or {\it edge expansion} in $\mathbb G (\Delta)$, respectively. More precisely, we may contract any edge $e_\gamma$ connecting any vertex $v$ (of any type) to a vertex $\vcirc$ giving a new vertex $v'$. Conversely, we may split a vertex $v'$ into a pair $v, \vcirc$ of vertices connected by an edge. Note that the polygon $P_{\vcirc}$ contains no (full) boundary stops or missing relations, so $v$ and $v'$ are of the same type (one of $\vbullet, \vcirc, \vodot, \vtimes$). See Fig.~\ref{fig:edgecontraction} for an illustration and 

Note that adding an arc $\gamma'$ which is isotopic to an arc $\gamma$ of $\Delta$ creates a $2$-gon. On the level of the ribbon complex of $\Delta$, adding $\gamma'$ to $\Delta$ corresponds to subdividing the edge $e_\gamma$, creating a new vertex corresponding to the new $2$-gon between $\gamma$ and $\gamma'$. The case of subdividing a loop $e_\gamma$ of $\mathbb G (\Delta)$ is illustrated in Fig.~\ref{fig:subdividingloops}.

\subsection{Special dissections}
\label{subsection:special}

For any orbifold surface with stops $(S, \Sigma)$, there exist special admissible dissections whose associated ribbon complexes are of the particularly simple form illustrated in Fig.~\ref{fig:ribbon}. Such dissections are not unique, but their existence will be proven as part of the proof of Lemma \ref{lemma:4gon}. They can be obtained from any given dissection by choosing one polygon with a boundary stop $\sigma_0$ and performing edge contractions and edge expansions (i.e.\ adding and removing arcs) to bring the ribbon complex into the form in Fig.~\ref{fig:ribbon}. (Note that by Definition \ref{definition:orbifoldsurface} we assume that $\Sigma$ contain at least one boundary stop.) We shall use the notation $\Delta_0$ for any dissection whose ribbon graph $\mathbb G (\Delta_0)$ is of this special form. Local pictures for the different groupings of edges in $\mathbb G (\Delta_0)$ are illustrated in Fig.~\ref{fig:localribbon}.

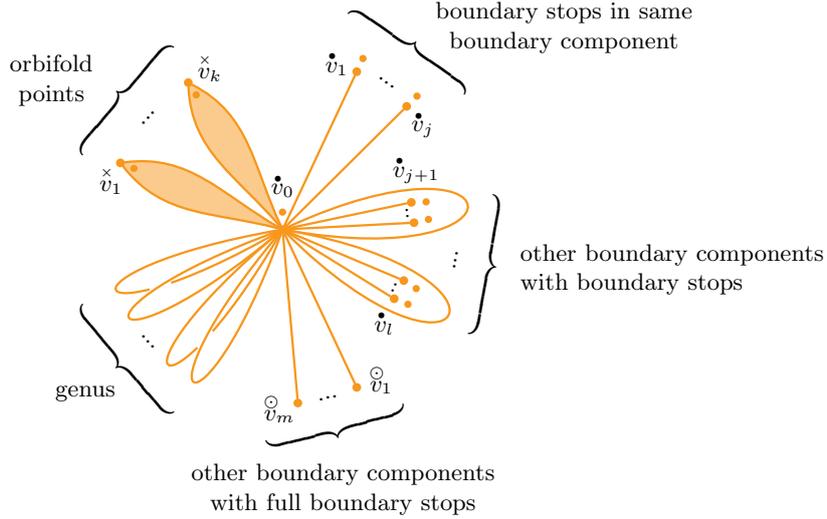
\begin{figure}
\begin{tikzpicture}
\node[circle, fill=ribboncolour, minimum size=.3em, inner sep=0] (Z) at (0,0) {};
\node[circle, fill=ribboncolour, minimum size=.25em, inner sep=0] at (90:.6em) {};
\node[circle, fill=ribboncolour, minimum size=.3em, inner sep=0] (A1) at (65:6em) {};
\node[circle, fill=ribboncolour, minimum size=.25em, inner sep=0] at (65:6.5em) {};
\node[circle, fill=ribboncolour, minimum size=.3em, inner sep=0] (A2) at (45:6em) {};
\node[circle, fill=ribboncolour, minimum size=.25em, inner sep=0] at (45:6.5em) {};
\node[circle, fill=ribboncolour, minimum size=.3em, inner sep=0] (B1) at (12:4.5em) {};
\node[circle, fill=ribboncolour, minimum size=.25em, inner sep=0] at (11:5em) {};
\node[circle, fill=ribboncolour, minimum size=.3em, inner sep=0] (B2) at (3:4.5em) {};
\node[circle, fill=ribboncolour, minimum size=.25em, inner sep=0] at (4:5em) {};
\node[rotate=97.5, font=\scriptsize] at (7.5:4.3em) {..};
\node[circle, fill=ribboncolour, minimum size=.3em, inner sep=0] (B3) at (337:4.5em) {};
\node[circle, fill=ribboncolour, minimum size=.25em, inner sep=0] at (336:5em) {};
\node[circle, fill=ribboncolour, minimum size=.3em, inner sep=0] (B4) at (328:4.5em) {};
\node[circle, fill=ribboncolour, minimum size=.25em, inner sep=0] at (329:5em) {};
\node[rotate=62.5, font=\scriptsize] at (332.5:4.3em) {..};
\node[circle, fill=ribboncolour, minimum size=.3em, inner sep=0] (C1) at (295:6em) {};
\node[circle, fill=ribboncolour, minimum size=.3em, inner sep=0] (C2) at (275:6em) {};
\node[circle, fill=ribboncolour, minimum size=.3em, inner sep=0] (D1) at (157.5:6em) {};
\node[circle, fill=ribboncolour, minimum size=.25em, inner sep=0] at (157.5:5.5em) {};
\node[circle, fill=ribboncolour, minimum size=.3em, inner sep=0] (D2) at (122.5:6em) {};
\node[circle, fill=ribboncolour, minimum size=.25em, inner sep=0] at (122.5:5.5em) {};
\node[font=\footnotesize] at (90:1.5em) {$\vbullet_0$};
\node[font=\footnotesize] at (72:6em) {$\vbullet_1$};
\node[font=\footnotesize] at (37:6em) {$\vbullet_j$};
\node[font=\footnotesize] at (24:5em) {$\vbullet_{j+1}$};
\node[font=\footnotesize] at (317:4.75em) {$\vbullet_l$};
\node[font=\footnotesize] at (164:6.1em) {$\vtimes_1$};
\node[font=\footnotesize] at (114:6.1em) {$\vtimes_k$};
\node[font=\footnotesize] at (303:6.2em) {$\vodot_1$};
\node[font=\footnotesize] at (269:6.2em) {$\vodot_m$};
\draw[-, line width=.85, color=ribboncolour, line cap=round] (Z) to (A1);
\draw[-, line width=.85, color=ribboncolour, line cap=round] (Z) to (A2);
\draw[-, line width=.85, color=ribboncolour, line cap=round] (Z) to[out=25, in=-10, looseness=220] (Z);
\draw[-, line width=.85, color=ribboncolour, line cap=round] (Z) to (B1);
\draw[-, line width=.85, color=ribboncolour, line cap=round] (Z) to (B2);
\draw[-, line width=.85, color=ribboncolour, line cap=round] (Z) to[out=350, in=315, looseness=220] (Z);
\draw[-, line width=.85, color=ribboncolour, line cap=round] (Z) to (B3);
\draw[-, line width=.85, color=ribboncolour, line cap=round] (Z) to (B4);
\draw[-, line width=.85, color=ribboncolour, line cap=round] (Z) to (C1);
\draw[-, line width=.85, color=ribboncolour, line cap=round] (Z) to (C2);
\draw[-, line width=.85, color=ribboncolour, line cap=round] (Z) to[out=240, in=220, looseness=350] node[shape=circle, fill=white, inner sep=3pt, pos=.25] {} (Z);
\draw[-, line width=.85, color=ribboncolour, line cap=round] (Z) to[out=250, in=230, looseness=350] (Z);
\draw[-, line width=.85, color=ribboncolour, line cap=round] (Z) to[out=210, in=190, looseness=350] node[shape=circle, fill=white, inner sep=3pt, pos=.25] {} (Z);
\draw[-, line width=.85, color=ribboncolour, line cap=round] (Z) to[out=220, in=200, looseness=350] (Z);
\draw[-, line width=.85, color=ribboncolour, line cap=round] (Z) to[out=115, in=-27.5] (D2);
\draw[-, line width=.85, color=ribboncolour, line cap=round] (Z) to[out=130, in=-87.5] (D2);
\draw[-, line width=.85, color=ribboncolour, line cap=round] (Z) to[out=150, in=7.5] (D1);
\draw[-, line width=.85, color=ribboncolour, line cap=round] (Z) to[out=165, in=-52.5] (D1);
\draw[line width=0, fill=ribboncolour, draw=ribboncolour, opacity=.5] (0,0) to[out=115, in=-27.5] (122.5:6em) to[out=-87.5, in=130] (0,0) -- cycle;
\draw[line width=0, fill=ribboncolour, draw=ribboncolour, opacity=.5] (0,0) to[out=150, in=7.5] (157.5:6em) to[out=-52.5, in=165] (0,0) -- cycle;
\begin{scope}[rotate=145]
\node[rotate=145] at (270:8em) {$\underbrace{\hspace{5em}}_{}$};
\node[font=\footnotesize, align=center, anchor=west] at (270:8.5em) {boundary stops in same \\ boundary component};
\node[rotate=145, font=\footnotesize] at (270:6.2em) {...};
\end{scope}
\begin{scope}[rotate=15]
\node[rotate=15] at (270:7.5em) {$\underbrace{\hspace{5em}}_{}$};
\node[font=\footnotesize, align=center, anchor=north] at (270:8em) {other boundary components \\ with full boundary stops};
\node[rotate=15, font=\footnotesize] at (270:6em) {...};
\end{scope}
\begin{scope}[rotate=80]
\node[rotate=80] at (270:7.5em) {$\underbrace{\hspace{5em}}_{}$};
\node[font=\footnotesize, align=left, anchor=west] at (270:7.9em) {other boundary components \\ with boundary stops};
\node[rotate=80, font=\footnotesize] at (270:6em) {...};
\end{scope}
\begin{scope}[rotate=-50]
\node[rotate=-50] at (270:7.5em) {$\underbrace{\hspace{5em}}_{}$};
\node[font=\footnotesize, align=left] at (270:8.8em) {genus};
\node[rotate=-50, font=\footnotesize] at (270:6em) {...};
\end{scope}
\begin{scope}[rotate=-130]
\node[rotate=-130] at (270:7.5em) {$\underbrace{\hspace{5em}}_{}$};
\node[font=\footnotesize, align=center, anchor=east] at (270:8em) {orbifold \\ points};
\node[rotate=-130, font=\footnotesize] at (270:6em) {...};
\end{scope}
\end{tikzpicture}
\caption{The ribbon complex $\mathbb G (\Delta_0)$ for the special (formal) admissible dissection $\Delta_0$ for a graded orbifold surface $\mathbf S$ with $k \geq 0$ orbifold points, $l + 1 \geq 1$ boundary stops and $m \geq 0$ full boundary stops.}
\label{fig:ribbon}
\end{figure}

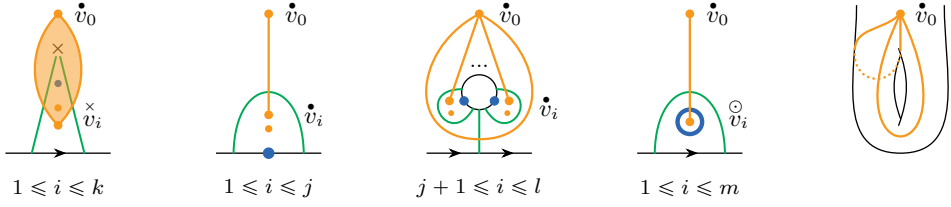
\begin{figure}
\begin{tikzpicture}[x=1em, y=1em, scale=1.2]
\begin{scope}[decoration={markings,mark=at position 0.58 with {\arrow[black]{Stealth[length=4.8pt]}}}]
\draw[line width=.5pt,postaction={decorate}] (-1.5,0) -- (1.5,0);
\node[font=\scriptsize,shape=circle,scale=.6] (X) at (0,3) {};
\node[font=\scriptsize] at (0,3) {$\times$};
\draw[line width=.75pt, color=arccolour, line cap=round] (X) -- (-.75,0);
\draw[line width=.75pt, color=arccolour, line cap=round] (X) -- (.75,0);
\node[circle, fill=ribboncolour, minimum size=.3em, inner sep=0] (T) at (0,4) {};
\node[circle, fill=ribboncolour, minimum size=.3em, inner sep=0] (B) at (0,.8) {};
\node[circle, fill=ribboncolour, minimum size=.25em, inner sep=0] at (0,1.3) {};
\node[circle, fill=stopcolour, minimum size=.25em, inner sep=0] at (0,2) {};
\draw[-, line width=.85, fill=ribboncolour, draw=ribboncolour, line cap=round, fill opacity=.5] (0,.8) to[bend right=45] (0,4) to[bend right=45] (0,.8) -- cycle;
\node[font=\footnotesize] at (1,1.1) {$\vtimes_i$};
\node[font=\footnotesize] at (.8,4) {$\vbullet_0$};
\node[font=\scriptsize] at (0,-1) {$1 \leq i \leq k$};
\end{scope}
\begin{scope}[xshift=6em]
\draw[line width=.5pt] (-1.5,0) -- (1.5,0);
\draw[line width=.75pt, color=arccolour, line cap=round] (-1,0) to[bend left=89, looseness=3] (1,0);
\draw[fill=stopcolour,color=stopcolour] (0,0) circle(.15em);
\node[circle, fill=ribboncolour, minimum size=.3em, inner sep=0] (T) at (0,4) {};
\node[circle, fill=ribboncolour, minimum size=.3em, inner sep=0] (B) at (0,1.1) {};
\node[circle, fill=ribboncolour, minimum size=.25em, inner sep=0] at (0,.7) {};
\draw[-, line width=.85, draw=ribboncolour, line cap=round] (B) to (T);
\node[font=\footnotesize] at (1.3,1.1) {$\vbullet_i$};
\node[font=\footnotesize] at (.8,4) {$\vbullet_0$};
\node[font=\scriptsize] at (0,-1) {$1 \leq i \leq j$};
\end{scope}
\begin{scope}[xshift=12em, decoration={markings,mark=at position 0.66 with {\arrow[black]{Stealth[length=4.8pt]}}}]
\draw[line width=.5pt,postaction={decorate}] (-1.5,0) -- (0,0);
\draw[line width=.5pt,postaction={decorate}] (0,0) -- (1.5,0);
\draw[line width=.5pt] (0,1.75) circle(.5em);
\draw[line width=.75pt, color=arccolour, line cap=round] (0,1.25) to (0,0);
\draw[line width=.75pt, color=arccolour, line cap=round] ($(0,1.75)+(250:.5em)$) to[bend left=120, looseness=5.5] ($(0,1.75)+(170:.5em)$);
\draw[line width=.75pt, color=arccolour, line cap=round] ($(0,1.75)+(10:.5em)$) to[bend left=120, looseness=5.5] ($(0,1.75)+(-70:.5em)$);
\draw[fill=stopcolour,color=stopcolour] ($(0,1.75)+(-30:.5em)$) circle(.12em);
\draw[fill=stopcolour,color=stopcolour] ($(0,1.75)+(210:.5em)$) circle(.12em);
\node[circle, fill=ribboncolour, minimum size=.3em, inner sep=0] (T) at (0,4) {};
\node[circle, fill=ribboncolour, minimum size=.3em, inner sep=0] (L) at (-.85,1.5) {};
\node[circle, fill=ribboncolour, minimum size=.3em, inner sep=0] (R) at (.85,1.5) {};
\node[circle, fill=ribboncolour, minimum size=.2em, inner sep=0] at (-.8,1.15) {};
\node[circle, fill=ribboncolour, minimum size=.2em, inner sep=0] at (.8,1.15) {};
\draw[-, line width=.85, draw=ribboncolour, line cap=round] (L) to (T) (R) to (T);
\draw[-, line width=.85, draw=ribboncolour, line cap=round] (T) to[in=180, out=210, looseness=1.5] (0,.4) to[out=0, in=-30, looseness=1.5] (T);
\node[font=\scriptsize] at (0,2.5) {...};
\node[font=\footnotesize] at (2,1.3) {$\vbullet_i$};
\node[font=\footnotesize] at (.9,4) {$\vbullet_0$};
\node[font=\scriptsize] at (0,-1) {$j+1 \leq i \leq l$};
\end{scope}
\begin{scope}[xshift=18em, decoration={markings,mark=at position 0.58 with {\arrow[black]{Stealth[length=4.8pt]}}}]
\draw[line width=.5pt,postaction={decorate}] (-1.5,0) -- (1.5,0);
\draw[line width=.75pt, color=arccolour, line cap=round] (-1,0) to[bend left=89, looseness=3] (1,0);
\draw[line width=.15em, color=stopcolour] (0,.9) circle(.35em);
\node[circle, fill=ribboncolour, minimum size=.3em, inner sep=0] (T) at (0,4) {};
\node[circle, fill=ribboncolour, minimum size=.3em, inner sep=0] (B) at (0,.9) {};
\draw[-, line width=.85, draw=ribboncolour, line cap=round] (B) to (T);
\node[font=\footnotesize] at (1.4,1.1) {$\vodot_i$};
\node[font=\footnotesize] at (.8,4) {$\vbullet_0$};
\node[font=\scriptsize] at (0,-1) {$1 \leq i \leq m$};
\end{scope}
\begin{scope}[xshift=24em]
\draw[line width=.5pt] (-1.25,4.25) to[out=270, in=180, looseness=1.1] (0,0) to[out=0, in=270, looseness=1.1] (1.25,4.25);
\draw[line width=.5pt, line cap=round] (-.05,3) to[bend left=20] (-.05,.8);
\draw[line width=.5pt, line cap=round] (0,2.8) to[bend right=20] (0,1);
\node[circle, fill=ribboncolour, minimum size=.3em, inner sep=0] (T) at (0,4) {};
\draw[-, line width=.85, draw=ribboncolour, line cap=round] (T) to[out=-65, in=245, looseness=108] (T) to (0,2.95) (T) to[in=85, out=220] (-1.3,3);
\draw[-, line width=.95, dash pattern=on 0pt off 2pt, draw=ribboncolour, line cap=round] (0,2.8) to[bend left=88, looseness=2] (-1.3,3);
\node[font=\footnotesize] at (.75,4) {$\vbullet_0$};
\end{scope}
\end{tikzpicture}
\caption{Local pictures on the surface for the five groups of edges in the ribbon complex $\mathbb G (\Delta_0)$ illustrated in Fig.~\ref{fig:ribbon}. In each picture, the straight boundary segment at the bottom belongs to the boundary component containing the distinguished boundary stop $\sigma_0$.}
\label{fig:localribbon}
\end{figure}

\begin{figure}
\begin{tikzpicture}
\begin{scope}[scale=.8]
\node[circle, fill=ribboncolour, minimum size=.3em, inner sep=0] (Z) at (0,0) {};
\node[circle, fill=ribboncolour, minimum size=.25em, inner sep=0] at (90:.6em) {};
\node[circle, fill=ribboncolour, minimum size=.3em, inner sep=0] (A1) at (65:6em) {};
\node[circle, fill=ribboncolour, minimum size=.25em, inner sep=0] at (65:6.5em) {};
\node[circle, fill=ribboncolour, minimum size=.3em, inner sep=0] (A2) at (45:6em) {};
\node[circle, fill=ribboncolour, minimum size=.25em, inner sep=0] at (45:6.5em) {};
\node[circle, fill=ribboncolour, minimum size=.3em, inner sep=0] (B1) at (12:4.5em) {};
\node[circle, fill=ribboncolour, minimum size=.25em, inner sep=0] at (11:5em) {};
\node[circle, fill=ribboncolour, minimum size=.3em, inner sep=0] (B2) at (3:4.5em) {};
\node[circle, fill=ribboncolour, minimum size=.25em, inner sep=0] at (4:5em) {};
\node[rotate=97.5, font=\scriptsize] at (7.5:4.3em) {..};
\node[circle, fill=ribboncolour, minimum size=.3em, inner sep=0] (B3) at (337:4.5em) {};
\node[circle, fill=ribboncolour, minimum size=.25em, inner sep=0] at (336:5em) {};
\node[circle, fill=ribboncolour, minimum size=.3em, inner sep=0] (B4) at (328:4.5em) {};
\node[circle, fill=ribboncolour, minimum size=.25em, inner sep=0] at (329:5em) {};
\node[rotate=62.5, font=\scriptsize] at (332.5:4.3em) {..};
\node[circle, fill=ribboncolour, minimum size=.3em, inner sep=0] (C1) at (295:6em) {};
\node[circle, fill=ribboncolour, minimum size=.3em, inner sep=0] (C2) at (275:6em) {};
\node[circle, fill=ribboncolour, minimum size=.3em, inner sep=0] (D1) at (157.5:6em) {};
\node[circle, fill=ribboncolour, minimum size=.25em, inner sep=0] at (157.5:5.4em) {};
\node[circle, fill=ribboncolour, minimum size=.3em, inner sep=0] (D2) at (122.5:6em) {};
\node[circle, fill=ribboncolour, minimum size=.25em, inner sep=0] at (122.5:5.4em) {};
\draw[-, line width=.85, color=ribboncolour, line cap=round] (Z) to (A1);
\draw[-, line width=.85, color=ribboncolour, line cap=round] (Z) to (A2);
\draw[-, line width=.85, color=ribboncolour, line cap=round] (Z) to[out=25, in=-10, looseness=173] (Z);
\draw[-, line width=.85, color=ribboncolour, line cap=round] (Z) to (B1);
\draw[-, line width=.85, color=ribboncolour, line cap=round] (Z) to (B2);
\draw[-, line width=.85, color=ribboncolour, line cap=round] (Z) to[out=350, in=315, looseness=173] (Z);
\draw[-, line width=.85, color=ribboncolour, line cap=round] (Z) to (B3);
\draw[-, line width=.85, color=ribboncolour, line cap=round] (Z) to (B4);
\draw[-, line width=.85, color=ribboncolour, line cap=round] (Z) to (C1);
\draw[-, line width=.85, color=ribboncolour, line cap=round] (Z) to (C2);
\draw[-, line width=.85, color=ribboncolour, line cap=round] (Z) to[out=240, in=220, looseness=280] node[shape=circle, fill=white, inner sep=3pt, pos=.25] {} (Z);
\draw[-, line width=.85, color=ribboncolour, line cap=round] (Z) to[out=250, in=230, looseness=280] (Z);
\draw[-, line width=.85, color=ribboncolour, line cap=round] (Z) to[out=210, in=190, looseness=280] node[shape=circle, fill=white, inner sep=3pt, pos=.25] {} (Z);
\draw[-, line width=.85, color=ribboncolour, line cap=round] (Z) to[out=220, in=200, looseness=280] (Z);
\draw[-, line width=.85, color=ribboncolour, line cap=round] (Z) to[out=115, in=-27.5] (D2);
\draw[-, line width=.85, color=ribboncolour, line cap=round] (Z) to[out=130, in=-87.5] (D2);
\draw[-, line width=.85, color=ribboncolour, line cap=round] (Z) to[out=150, in=7.5] (D1);
\draw[-, line width=.85, color=ribboncolour, line cap=round] (Z) to[out=165, in=-52.5] (D1);
\draw[line width=0, fill=ribboncolour, draw=ribboncolour, opacity=.5] (0,0) to[out=115, in=-27.5] (122.5:6em) to[out=-87.5, in=130] (0,0) -- cycle;
\draw[line width=0, fill=ribboncolour, draw=ribboncolour, opacity=.5] (0,0) to[out=150, in=7.5] (157.5:6em) to[out=-52.5, in=165] (0,0) -- cycle;
\begin{scope}[rotate=145]
\node[rotate=145, font=\footnotesize] at (270:6.2em) {...};
\end{scope}
\begin{scope}[rotate=15]
\node[rotate=15, font=\footnotesize] at (270:6em) {...};
\end{scope}
\begin{scope}[rotate=80]
\node[rotate=80, font=\footnotesize] at (270:6em) {...};
\end{scope}
\begin{scope}[rotate=-50]
\node[rotate=-50, font=\footnotesize] at (270:6em) {...};
\end{scope}
\begin{scope}[rotate=-130]
\node[rotate=-130, font=\footnotesize] at (270:6em) {...};
\end{scope}
\draw[<->, line width=.5pt] (-7.75em,1.5em) -- (-6.25em,1.5em);
\begin{scope}[x=1.2em, y=1.2em, xshift=-9.5em, yshift=-1em, decoration={markings,mark=at position 0.58 with {\arrow[black]{Stealth[length=4.8pt]}}}]
\draw[line width=.5pt,postaction={decorate}] (-1.5,0) -- (1.5,0);
\node[font=\scriptsize,shape=circle,scale=.6] (X) at (0,3) {};
\node[font=\scriptsize] at (0,3) {$\times$};
\draw[line width=.75pt, color=arccolour, line cap=round] (X) -- (-.75,0);
\draw[line width=.75pt, color=arccolour, line cap=round] (X) -- (.75,0);
\node[circle, fill=ribboncolour, minimum size=.3em, inner sep=0] (T) at (0,4) {};
\node[circle, fill=ribboncolour, minimum size=.3em, inner sep=0] (B) at (0,.8) {};
\node[circle, fill=ribboncolour, minimum size=.25em, inner sep=0] at (0,1.3) {};
\node[circle, fill=stopcolour, minimum size=.25em, inner sep=0] at (0,2) {};
\draw[-, line width=.85, fill=ribboncolour, draw=ribboncolour, line cap=round, fill opacity=.5] (0,.8) to[bend right=45] (0,4) to[bend right=45] (0,.8) -- cycle;
\end{scope}
\end{scope}
\begin{scope}[xshift=18em, scale=.8]
\node[circle, fill=ribboncolour, minimum size=.3em, inner sep=0] (Z) at (0,0) {};
\node[circle, fill=ribboncolour, minimum size=.25em, inner sep=0] at (90:.6em) {};
\node[circle, fill=ribboncolour, minimum size=.3em, inner sep=0] (A1) at (65:6em) {};
\node[circle, fill=ribboncolour, minimum size=.25em, inner sep=0] at (65:6.5em) {};
\node[circle, fill=ribboncolour, minimum size=.3em, inner sep=0] (A2) at (45:6em) {};
\node[circle, fill=ribboncolour, minimum size=.25em, inner sep=0] at (45:6.5em) {};
\node[circle, fill=ribboncolour, minimum size=.3em, inner sep=0] (B1) at (12:4.5em) {};
\node[circle, fill=ribboncolour, minimum size=.25em, inner sep=0] at (11:5em) {};
\node[circle, fill=ribboncolour, minimum size=.3em, inner sep=0] (B2) at (3:4.5em) {};
\node[circle, fill=ribboncolour, minimum size=.25em, inner sep=0] at (4:5em) {};
\node[rotate=97.5, font=\scriptsize] at (7.5:4.3em) {..};
\node[circle, fill=ribboncolour, minimum size=.3em, inner sep=0] (B3) at (337:4.5em) {};
\node[circle, fill=ribboncolour, minimum size=.25em, inner sep=0] at (336:5em) {};
\node[circle, fill=ribboncolour, minimum size=.3em, inner sep=0] (B4) at (328:4.5em) {};
\node[circle, fill=ribboncolour, minimum size=.25em, inner sep=0] at (329:5em) {};
\node[rotate=62.5, font=\scriptsize] at (332.5:4.3em) {..};
\node[circle, fill=ribboncolour, minimum size=.3em, inner sep=0] (C1) at (295:6em) {};
\node[circle, fill=ribboncolour, minimum size=.3em, inner sep=0] (C2) at (275:6em) {};
\node[circle, fill=ribboncolour, minimum size=.3em, inner sep=0] (D1) at (157.5:6em) {};
\node[circle, fill=ribboncolour, minimum size=.25em, inner sep=0] at (157.5:5.4em) {};
\node[circle, fill=ribboncolour, minimum size=.3em, inner sep=0] (D2) at (122.5:6em) {};
\node[circle, fill=ribboncolour, minimum size=.25em, inner sep=0] at (122.5:5.4em) {};
\draw[-, line width=.85, color=ribboncolour, line cap=round] (Z) to (A1);
\draw[-, line width=.85, color=ribboncolour, line cap=round] (Z) to (A2);
\draw[-, line width=.85, color=ribboncolour, line cap=round] (Z) to[out=25, in=-10, looseness=173] (Z);
\draw[-, line width=.85, color=ribboncolour, line cap=round] (Z) to (B1);
\draw[-, line width=.85, color=ribboncolour, line cap=round] (Z) to (B2);
\draw[-, line width=.85, color=ribboncolour, line cap=round] (Z) to[out=350, in=315, looseness=173] (Z);
\draw[-, line width=.85, color=ribboncolour, line cap=round] (Z) to (B3);
\draw[-, line width=.85, color=ribboncolour, line cap=round] (Z) to (B4);
\draw[-, line width=.85, color=ribboncolour, line cap=round] (Z) to (C1);
\draw[-, line width=.85, color=ribboncolour, line cap=round] (Z) to (C2);
\draw[-, line width=.85, color=ribboncolour, line cap=round] (Z) to[out=240, in=220, looseness=280] node[shape=circle, fill=white, inner sep=3pt, pos=.25] {} (Z);
\draw[-, line width=.85, color=ribboncolour, line cap=round] (Z) to[out=250, in=230, looseness=280] (Z);
\draw[-, line width=.85, color=ribboncolour, line cap=round] (Z) to[out=210, in=190, looseness=280] node[shape=circle, fill=white, inner sep=3pt, pos=.25] {} (Z);
\draw[-, line width=.85, color=ribboncolour, line cap=round] (Z) to[out=220, in=200, looseness=280] (Z);
\draw[-, line width=.85, color=ribboncolour, line cap=round] (Z) to[out=115, in=-27.5] (D2);
\draw[-, line width=.85, color=ribboncolour, line cap=round] (Z) to[out=130, in=-87.5] (D2);
\draw[-, line width=.85, color=ribboncolour, line cap=round] (Z) to[out=150, in=7.5] (D1);
\draw[-, line width=.85, color=ribboncolour, line cap=round] (Z) to[out=165, in=-52.5] (D1);
\begin{scope}[rotate=145]
\node[rotate=145, font=\footnotesize] at (270:6.2em) {...};
\end{scope}
\begin{scope}[rotate=15]
\node[rotate=15, font=\footnotesize] at (270:6em) {...};
\end{scope}
\begin{scope}[rotate=80]
\node[rotate=80, font=\footnotesize] at (270:6em) {...};
\end{scope}
\begin{scope}[rotate=-50]
\node[rotate=-50, font=\footnotesize] at (270:6em) {...};
\end{scope}
\begin{scope}[rotate=-130]
\node[rotate=-130, font=\footnotesize] at (270:6em) {...};
\end{scope}
\draw[<->, line width=.5pt] (-7.75em,1.5em) -- (-6.25em,1.5em);
\begin{scope}[x=1.2em, y=1.2em, xshift=-9.5em, yshift=-1em, decoration={markings,mark=at position 0.58 with {\arrow[black]{Stealth[length=4.8pt]}}}]
\draw[line width=.5pt,postaction={decorate}] (-1.5,0) -- (1.5,0);
\draw[line width=.5pt] (0,2.4) circle(.5em);
\draw[line width=.75pt, color=arccolour, line cap=round] (-.3,2.1) -- (-.75,0);
\draw[line width=.75pt, color=arccolour, line cap=round] (.3,2.1) -- (.75,0);
\node[circle, fill=ribboncolour, minimum size=.3em, inner sep=0] (T) at (0,4) {};
\node[circle, fill=ribboncolour, minimum size=.3em, inner sep=0] (B) at (0,.8) {};
\node[circle, fill=ribboncolour, minimum size=.25em, inner sep=0] at (0,1.3) {};
\node[circle, fill=stopcolour, minimum size=.25em, inner sep=0] at (0,2) {};
\draw[-, line width=.85, draw=ribboncolour, line cap=round] (0,.8) to[bend right=45] (0,4) to[bend right=45] (0,.8) -- cycle;
\end{scope}
\end{scope}
\end{tikzpicture}
\caption{The ribbon graph (right) obtained by removing the orbifold $2$-cells from the ribbon complex (left) associated to a special dissection $\Delta_0$.}
\label{fig:compactification}
\end{figure}
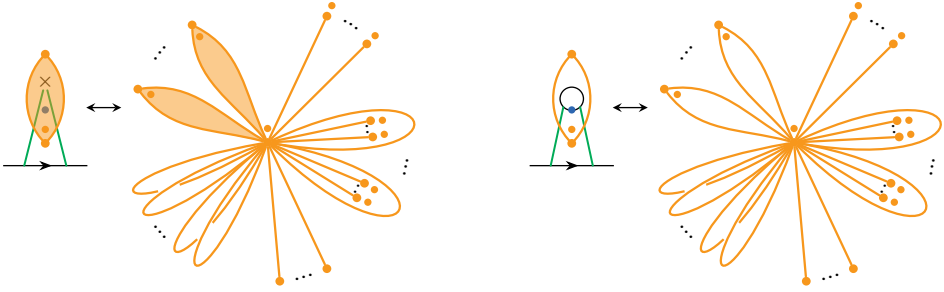

\subsection{Boundary paths and orbifold paths}

In order to describe the morphisms in the Fukaya category of orbifold surfaces, we need two different notions of paths between arcs which we call {\it boundary paths} and {\it orbifold paths}. The former already appeared in \cite{haidenkatzarkovkontsevich} (using a slightly different setup) and the latter are a natural generalization for orbifold surfaces where arcs may end in orbifold points. In Section \ref{section:doublecover} we show that orbifold paths correspond to morphisms of intersecting arcs in the double cover. In this subsection we describe the topological definition. The grading via the line field is discussed in \S\ref{subsection:grading}.

\subsubsection{Boundary paths}

Recall from Definition \ref{definition:orbifoldsurface} that our surfaces are oriented with smooth boundary, so that the boundary is also naturally oriented. A {\it boundary path} is an immersion $p \colon [0, 1] \to \partial S \smallsetminus \Sigma$ such that $p$ follows the orientation of $\partial S$. We consider boundary paths up to reparametrization. Note that if $\partial_j S$ is a boundary component with at least one boundary stop, then there is at most one boundary path with given starting and ending points in $\partial_j S \smallsetminus \Sigma$. If $\partial_j S \cap \Sigma = \varnothing$, then there are countably infinitely many boundary paths with given starting and ending points in $\partial_j S$.

Given two arcs $\gamma, \gamma' \in \Gamma$, each boundary path (up to reparametrization) from a point in $\gamma \cap \partial S$ to a point in $\gamma' \cap \partial S$ will be a morphism in $\mathbf A_\Delta(\gamma, \gamma')$. See \S~\ref{subsection:gradingofpaths} below for the degree of a boundary path.

\subsubsection{Orbifold paths}

Let $\gamma, \gamma' \in \Gamma$ be two arcs at an orbifold point $x \in \Sing (S)$. By an {\it orbifold path} from $\gamma$ to $\gamma'$ at $x$, we mean an anticlockwise angle locally from $\gamma$ to $\gamma'$ based at $x$, which does not pass through the orbifold stop. Note that at each orbifold point $x \in \Sing (S)$ there is a unique maximal orbifold path, i.e.\ any orbifold path at $x$ is its subpath.

We refer to boundary paths and orbifold paths collectively as {\it paths}. Note that under the duality \eqref{eq:duality} each path $p = p_1 \dotsb p_n$ between arcs in a dissection corresponds to a chain
\[
\gamma_0 <_{v_{i_1}} \gamma_1 <_{v_{i_2}} \dotsb <_{v_{i_n}} \gamma_n.
\]
In particular each $p_k$ is dual to a (linear or cyclic) relation $<_{v_{i_k}}$.

\subsection{Grading via line fields}
\label{subsection:grading}

Thus far the notions of arcs, dissections and ribbon complexes were defined for an orbifold surface with stops $(S, \Sigma)$. We now consider a grading structure on $(S, \Sigma)$ given by a line field $\eta \in \Gamma (S, \mathbb P (\mathrm T S))$, so that $\mathbf S = (S, \Sigma, \eta)$ is a {\it graded} orbifold surface with stops (cf.\ Definition \ref{definition:orbifoldsurface}). There are two equivalent definitions of gradings.

Following \cite{haidenkatzarkovkontsevich} the grading may be given as follows. Given a dissection $\Delta$ of $\mathbf S$, each arc $\gamma \colon [0, 1] \to S$ of $\Delta$ may be considered as a {\it graded arc} by fixing a homotopy $\tilde \gamma$ between the restriction of $\eta$ to (the image of) $\gamma$ and the line field $\dot \gamma$ along $\gamma$ determined by the tangent lines of $\gamma$. We consider $\tilde \gamma$ up to homotopy in the space $\Gamma ([0, 1], \gamma^* \mathbb P (\mathrm T S))$ of sections of the projectivized tangent bundle of $S$ restricted to $\gamma$.

The choice of a homotopy class $\tilde \gamma$ for $\gamma$ may be regarded as fixing a reference grading for $\gamma$, akin to choosing to consider a projective module as a perfect complex concentrated in degree $0$.

Let us denote by $\Pi_1 (X, x, y)$ the set of homotopy classes of paths from $x$ to $y$ in a topological space $X$, i.e.\ $\Pi_1 (X, x, y) = \Hom_{\Pi_1 (X)} (x, y)$ is the set of morphisms in the fundamental groupoid $\Pi_1 (X)$ of $X$. We may then write the grading of $\gamma$ more concisely as an element
\[
\tilde \gamma \in \Pi_1 (\Gamma ([0, 1], \gamma^* \mathbb P (\mathrm T S)), \gamma^* \eta, \dot \gamma).
\]

This definition also makes sense when considering any immersed curve $c \colon I \to S$, where $I$ is not necessarily $[0, 1]$ but any $1$-manifold, possibly disconnected and with boundary. That is, a {\it graded curve} is given by a triple $(I, c, \tilde c)$, where $c \colon I \to S$ is an immersion and $\tilde c \in \Pi_1 (\Gamma (S, \mathbb P (\mathrm T S)), c^* \eta, \dot c)$.

Let $(I, c, \tilde c)$ and $(I', c', \tilde c')$ be two graded curves which intersect transversely at a point $s \in S \smallsetminus \Sing (S)$. That is, $c (t_0) = s = c' (t_0')$ for some $t_0 \in I$, $t_0' \in I'$ and $\dot c (t_0) \neq \dot c' (t_0') \in \mathbb P (\mathrm T_s S)$. In this case, their {\it intersection index} is the integer
\[
\mathrm i_s (c, c') = \tilde c (t_0) \cdot \kappa \cdot \tilde c'{}^{-1} (t_0') \in \pi_1 (\mathbb P (\mathrm T_s S)) \simeq \mathbb Z
\]
where $\kappa \colon \mathbb P (\mathrm T_s S) \to \mathbb P (\mathrm T_s S)$ is given by counterclockwise rotation in $\mathrm T_s S$ mapping $\dot c (t_0)$ to $\dot c' (t_0')$ through an angle $< \pi$.

\subsubsection{Grading of paths}\label{subsection:gradingofpaths}

A boundary path may be graded as follows.

\begin{definition}
Let $p \colon [0, 1] \to \partial S \smallsetminus \Sigma$ be a boundary path from $p (0) \in \gamma \cap \partial S$ to $p (1) \in \gamma' \cap \partial S$ for two arcs $\gamma, \gamma'$. Then the {\it degree} of $p$ is
\[
|p| = \mathrm i_{p (0)} (\gamma, p) - \mathrm i_{p (1)} (\gamma', p)
\]
for an arbitrary grading of $p$.
\end{definition}

To grade orbifold paths, let $\mathrm S^1_x$ be an embedded circle enclosing a unique orbifold point $x$ which intersects each arc in $\Gamma$ at most once and does so transversely. (We may take $\mathrm S^1_x$ to be $\partial \mathbb D_x$ of Remark \ref{remark:compactification}.) Let $c_x \colon \mathrm S^1_x \to S$ be the embedding and let $c_x$ be equipped with an arbitrary grading $\tilde c_x$. Note such a grading exists since $\eta$ has winding number $1$ around $x$ whence $\Pi_1 (\Gamma (\mathrm S^1, c_x^* \mathbb P (\mathrm T S)), c_x^* \eta, \dot c_x) \neq \varnothing$.

\begin{definition}
Let $([0, 1], q, \tilde q)$ with $q \colon [0, 1] \to \mathrm S^1_x $ be (a reparametrization of) the graded curve obtained by restricting $(\mathrm S^1_x, c_x, \tilde c_x)$ to the counterclockwise segment of $\mathrm S^1_x$ connecting $\gamma \cap \mathrm S_x^1$ to $\gamma' \cap \mathrm S^1_x$ for two arcs $\gamma, \gamma'$ connecting to $x$. Then we may view $q$ as the orbifold path from $\gamma$ to $\gamma'$ whose {\it degree} is defined as
\[
|q| = \mathrm i_{q (0)} (\gamma, q) - \mathrm i_{q (1)} (\gamma', q).
\]
\end{definition}

\begin{remark}
Following Lekili and Polishchuk \cite[\S 2]{lekilipolishchuk2} an equivalent definition of the degree of a morphism can be given by the {\it winding numbers} of $\eta$, without considering gradings of arcs.

The line field $\eta$ restricts to a line field on each polygon $P_v$. Since different polygons are glued along the arcs in $\Delta$, we may change the line field up to homotopy so that $\eta$ is transverse to the arcs in $\Delta$. Up to homotopy, every line field can be obtained by gluing such line fields on the polygons. 

A line field on a polygon $P_v$, which is transverse to the arcs in $\Delta$, assigns an integer $\theta_i$ to each boundary segment and each segment of $\mathrm S^1_x$ lying in $P_v$, including those with (boundary or orbifold) stops. Here $\theta_i$ is the winding number of the boundary segment in $\partial S$ or the segment of $\mathrm S_x^1$ which is the signed count of how often the tangent line of the boundary or circle segment agrees with the line given by the restriction of $\eta$ to that segment. Letting $n = \val(v)$ these numbers satisfy the following topological constraint
\[
\theta_1 + \dotsb + \theta_n = n - 2
\]
deriving from the Poincaré--Hopf index formula. We have the following constraints:
\begin{itemize}
\item If $P_v$ is an $n$-gon with $v$ of type $\vcirc$ with paths $p_1, \dotsc, p_n$ in the boundary of $P_v$, we have
\[
|p_1| + \dotsb + |p_n| = n - 2.
\]
\item If $P_v$ is an $n$-gon with $v$ of type $\vtimes$ containing an orbifold stop at an orbifold point $x$ with paths $p_1, \dotsc, p_{n-1}$ in the boundary of $P_v$ and $q_1, \dotsc, q_k$ the orbifold paths around $x$, then
\[
|p_1| + \dotsb + |p_{n-1}| = |q_1| + \dotsb + |q_k| + 3 - n.
\]
This follows since $|p_1| + \dotsb + |p_{n-1}| + \theta_n = n - 2$ and $|q_1|+ \dotsb + |q_k| + \theta_n = \mathrm w_\eta(\mathrm S_x^1) = 1$, where $\theta_n$ is the winding number of the segment of $\mathrm S_x^1$ containing the orbifold stop. 
\end{itemize}

The degrees of the (boundary or orbifold) paths are then given by the winding numbers of the paths along $\eta$, i.e.\
\[
|p| = \mathrm w_\eta (p).
\]
\end{remark}

\subsubsection{Grading on ribbon complexes}
\label{subsubsection:gradedribbon}

All relevant information about $\Delta$ can be encoded into extra data for the ribbon complex $\mathbb G (\Delta)$ and the degrees for the paths are no exception. Each (boundary or orbifold) path $p$ between two arcs $\gamma, \gamma'$ appearing consecutively in the boundary of $P_v$ corresponds to an order relation $e_\gamma < e_{\gamma'}$ on the half-edges of $\mathbb G_1 (\Delta)$ corresponding to the arcs. We may encode the degree of $p$ given by the winding number $|p| = \mathrm w_\eta (p)$ by labelling this order relation by the corresponding degree, e.g.\ as $e_\gamma \overset{|p|}{<} e_{\gamma'}$, although we will not use this notation later. The analogous notion for smooth surfaces is called {\it S-graph} in \cite[\S 6]{haidenkatzarkovkontsevich}, but to keep the new terms to a minimum we shall continue to call $\mathbb G (\Delta)$ a {\it ribbon complex} for $\mathbf S$ (rather than {\it S-complex}).

\section{Partially wrapped Fukaya categories of orbifold surfaces}
\label{section:cosheaves}

In this section we define the partially wrapped Fukaya category of a graded orbifold surface with stops in terms of the category of global sections of a cosheaf of A$_\infty$ categories on (an open cover of) the ribbon complex $\mathbb G (\Delta)$ for any admissible dissection $\Delta$.

This is an analogue of a construction advocated in \cite{kontsevich2} where Kontsevich conjectured that the Fukaya category of a Weinstein manifold can be recovered as the category of global sections of a cosheaf of DG categories on a Lagrangian core. This conjecture was recently established by Ganatra, Pardon and Shende \cite{ganatrapardonshende2} who moreover show that the individual categories are partially wrapped Fukaya categories. Each open subset of the Lagrangian core corresponds to a ``Weinstein sector'' of the symplectic manifold and ``sectorial descent'' implies that the global partially wrapped Fukaya category can indeed be computed as the homotopy colimit of the categories associated to the open sets in an open cover.

Although in general, the Lagrangian core might have complicated singularities, in the case of smooth surfaces with at least one boundary component, a Lagrangian core is simply a ribbon graph whose ``singularities'' are the vertices with valency $> 2$. In this case Fukaya categories and mirror symmetry in terms of ribbon graph of surfaces have been studied in several works \cite{sibillatreumannzaslow,lee,haidenkatzarkovkontsevich,dyckerhoffkapranov,pascaleffsibilla}.

We now give an analogous description for orbifold surfaces, where the underlying topological space is the ribbon complex $\mathbb G (\Delta)$ defined in \S\ref{subsection:ribbongraph}. In \S\ref{subsection:core} we show that our construction is in fact a refinement of a similar construction on a ($1$-dimensional) ribbon graph which should be considered as a more apt replacement of the notion of (half-dimensional) Lagrangian core in the case of orbifold surfaces. The detour via the $2$-dimensional ribbon complex will pay off when we give an explicit description of $\mathcal W (\mathbf S)$ via generators in Sections \ref{section:ainfinity}--\ref{section:doublecover}.

\subsection{Cosheaves of A$_\infty$ categories}

We follow the exposition of \cite[\S 2.1]{pascaleffsibilla} (see also \cite[\S 1]{dyckerhoffkapranov}) with the mostly cosmetic difference that we consider not only DG categories but A$_\infty$ categories. (See \cite{pascaleff} for the equivalence between the homotopy theories of DG categories and A$_\infty$ categories.) We may view DG categories either as $1$-categories strictly enriched over cochain complexes of $\Bbbk$-vector spaces or as $\Bbbk$-linear $\infty$-categories. In order to talk about their homotopy theory, we may view them as objects in the homotopy category of DG categories. This homotopy category can be formed by inverting weak equivalences in the Morita model structure of DG categories \cite{tabuada,toen} and it forms an $\infty$-category which is equivalent to the $\infty$-category of DG categories viewed as $\Bbbk$-linear $\infty$-categories \cite{haugseng}. In this sense, the homotopy theory of the enriched $1$-categorical viewpoint agrees with the $\infty$-categorical description and homotopy (co)limits can be computed in either framework. We refer to \cite[\S 2.1]{pascaleffsibilla} for more details and further references.

A {\it precosheaf} $\mathcal E$ of A$_\infty$ categories on a topological space $X$ is a functor assigning to each open set $U \subset X$ an A$_\infty$ category $\mathcal E (U)$ and to each inclusion $U \subset V$ of open sets an inclusion functor $\mathcal E (U) \to \mathcal E (V)$. (This is dual to the restriction functor of a sheaf.)

We say that $\mathcal E$ satisfies {\it Čech descent} if for any open subset $U \subset X$ and any open cover $\mathfrak U = \{ U_i \}_{i \in I \subset \mathbb N}$ of $U$ the inclusion functor
\[
\Bigl( \dotsb \triplerightarrow \mathcal E (\mathfrak U \times_U \mathfrak U) \doublerightarrow \mathcal E (\mathfrak U) \Bigr) \to \mathcal E (U)
\]
realizes $\mathcal E (U)$ as the colimit of the semi-simplicial object formed by evaluating $\mathcal E$ on the Čech nerve of $\mathfrak U$. The Čech nerve $\mathrm N (\mathfrak U)$ is the (semi)simplicial set whose $n$-simplices consist of the intersection of $n + 1$ open sets in $\mathfrak U$. If $\mathcal E$ satisfies Čech descent, we call $\mathcal E$ a {\it cosheaf} of A$_\infty$ categories.

Below we will choose a finite open cover $\mathfrak U$ in such a way that $U_i \cap U_j \cap U_k \cap U_l = \varnothing$ for any quadruple of pairwise distinct open sets. In this case, it suffices to compute the colimit on the truncated Čech nerve
\begin{equation}
\label{eq:truncated}
\prod_{i < j < k} \mathcal E (U_i \cap U_j \cap U_k) \triplerightarrow \prod_{i < j} \mathcal E (U_i \cap U_j) \doublerightarrow \prod_i \mathcal E (U_i).
\end{equation}

\subsection{Open cover of ribbon complexes}

Let $\Delta$ be a weakly admissible dissection of $\mathbf S$ and let $\mathbb G (\Delta)$ be the associated ribbon complex as defined in \S\ref{subsection:ribbongraph}. A workable open cover of $\mathbb G (\Delta)$ can be obtained by further subdividing the cells of $\mathbb G (\Delta)$ as follows.

\subsubsection{Subdividing loops}
\label{subsubsection:loops}

We apply barycentric subdivision to any {\it loop} $e_\gamma$ in the $1$-skeleton of $\mathbb G (\Delta)$, i.e.\ to any edge ($1$-cell) $e_\gamma$ whose boundary vertices coincide. This results in a new vertex along the edge subdividing the interior into two $1$-cells. In terms of the dissection, barycentric subdivision of $e_\gamma$ can simply be understood as adding an extra arc to $\Delta$ which is isotopic to $\gamma$ (see \S\ref{subsubsection:operations} and Fig.~\ref{fig:subdividingloops}). Below we will simply assume that $\Delta$ contains enough arcs so that $\mathbb G (\Delta)$ does not contain any loops in its $1$-skeleton. This will not change the global section of the cosheaf defined in Section \ref{section:cosheaves}. See also Remark \ref{remark:smoothorbifolddisklength12} \ref{remark:item-iii} and Fig.~\ref{fig:differential} (right).

\subsubsection{Subdividing orbifold $2$-cells}
\label{subsubsection:2cells}

We now apply barycentric subdivision to the orbifold $2$-cells $d_x \in \mathbb G_2 (\Delta)$ where the barycenter is the orbifold point in $d_x$. If the boundary of $d_x$ contains $k$ vertices, then barycentric subdivision creates a new vertex $\widetilde v_x$ corresponding to the orbifold point of $d_x$ and $k$ edges connecting from $\widetilde v_x$ to the vertices in the boundary. The original $2$-cell is thus subdivided into $k$ smaller (smooth) $2$-cells. 

We denote the ribbon complex with subdivided $2$-cells by $\widetilde{\mathbb G} (\Delta)$.

\subsubsection{Open cover}
\label{subsubsection:opencover}

We now use the open cells in the subdivided ribbon complex $\widetilde{\mathbb G} (\Delta)$ to give an open cover of the original ribbon complex $\mathbb G (\Delta)$. Consider the following open sets:
\begin{itemize}
\item If $v \in \mathbb G_0 (\Delta)$ is a vertex of type $\vbullet$, $\vodot$ or $\vcirc$ we let $U_v$ be the union of $v$ and all open cells in $\widetilde{\mathbb G} (\Delta)$ whose boundary contains $v$.
\item If $v \in \mathbb G_0 (\Delta)$ is a vertex of type $\vtimes$ for an orbifold point $x \in \Sing (S)$ we let $U_v$ be the union of $v$ and $\widetilde v_x$ and all open cells in $\widetilde{\mathbb G} (\Delta)$ whose boundary contains either $v$ or $\widetilde v_x$. (Here $\widetilde v_x$ is the new vertex subdividing the $2$-cell $d_x \in \mathbb G_2 (\Delta)$ as in \S\ref{subsubsection:2cells}.)
\end{itemize}

\begin{definition}
Let $\Delta$ be a weakly admissible dissection and assume that $\Delta$ contains enough arcs so that the $1$-skeleton of $\mathbb G (\Delta)$ does not contain any loops (see \S\ref{subsubsection:loops}). We denote by
\[
\mathfrak U_\Delta = \{ U_v \}_{v \in \mathbb G_0 (\Delta)}
\]
the open cover of $\mathbb G (\Delta)$.
\end{definition}

Note that for each orbifold point $x \in \Sing (S)$ there is exactly one open set in $\mathfrak U_\Delta $ containing $x$, namely $U_{\vtimes_x}$ for $\vtimes_x \in \mathbb G_0 (\Delta)$.

\begin{proposition}
\label{proposition:cover}
Let $\Delta$ be a weakly admissible dissection and assume that $\Delta$ contains enough arcs so that the $1$-skeleton of $\mathbb G (\Delta)$ does not contain any loops. The open cover $\mathfrak U_\Delta = \{ U_v \}_{v \in \mathbb G_0 (\Delta)}$ has the following properties:
\begin{enumerate}
\item The open sets $U \in \mathfrak U_\Delta$ are connected and contractible, as are the connected components of their intersections.
\item If $v, w \in \mathbb G_0 (\Delta)$ are two distinct vertices, then $U_v \cap U_w$ consists of at most two $2$-cells glued along an open $1$-cell.
\item If $u, v, w \in \mathbb G_0 (\Delta)$ are three distinct vertices, then $U_u \cap U_v \cap U_w$ consists of at most  a single open $2$-cell in $\widetilde{\mathbb G} (\Delta)$.
\item The intersection of four or more pairwise distinct open sets in $\mathfrak U_\Delta$ is empty.
\end{enumerate}
\end{proposition}

\begin{proof}
The open sets $U_v$ contract onto the vertex $v$ which proves the first part of the first assertion. Since $\Delta$ does not contain any loops, any open $1$-cell has exactly two distinct vertices in its boundary, so it can appear in at most two distinct open sets of $\mathfrak U_\Delta$. Any open $2$-cell in $\widetilde{\mathbb G} (\Delta)$ has exactly three distinct vertices in its boundary, so it can appear in at most three open sets of $\mathfrak U_\Delta$. (See Example \ref{example:cosheaf} for a nonempty intersection of three open sets.) Hence the intersection of four or more pairwise distinct open sets is empty which proves the fourth assertion.

The connected components of an intersection of two distinct open sets in $\mathfrak U_\Delta$ consist at most of two $2$-cells glued along a single open $1$-cell, whence the connected components are contractible which proves the second assertion and the second part of the first assertion for intersections of two open sets. The connected components of an intersection of three pairwise distinct open sets consist at most of a single open $2$-cell which is also contractible proving the third assertion and the second part of the first assertion for intersections of three open sets.
\end{proof}

\subsection{Cosheaf on the open cover}

We now construct a cosheaf $\mathcal E_\Delta$ of A$_\infty$ categories on $\mathfrak U_\Delta$.

Let $u, v, w \in \mathbb G_0 (\Delta)$ be three distinct vertices. By Proposition \ref{proposition:cover}, the intersection $U_u \cap U_v \cap U_w$ is either empty or it consists of a single open $2$-cell in $\widetilde{\mathbb G} (\Delta)$. Two of its boundary $1$-cells are those added in the barycentric subdivision of $d_x \in \mathbb G_2 (\Delta)$ and the third $1$-cell corresponds to an arc $\gamma$ of the dissection $\Delta$. If $U_u \cap U_v \cap U_w \neq \varnothing$ we set
\[
\mathcal E_\Delta (U_u \cap U_v \cap U_w) = \underset{\gamma}{\bullet}
\]
to be the DG category with one object corresponding to the arc $\gamma$ with (scalar multiples of) its identity morphism. (Of course, $\mathcal E_\Delta (\varnothing)$ is defined as the empty category with no objects.)

Let $v, w \in \mathbb G_0 (\Delta)$ be two distinct vertices. By Proposition \ref{proposition:cover}, the intersection $U_v \cap U_w$ is either empty or it consists of at most two $2$-cells glued along a single open $1$-cell.

If the open $1$-cell belongs to $\mathbb G (\Delta)$, it corresponds to an arc $\gamma$ of $\Delta$ in which case we put again
\begin{align}\label{align:A1quiver}
\mathcal E_\Delta (U_v \cap U_w) = \underset{\gamma}{\bullet}
\end{align}
to be the category with one object corresponding to $\gamma$. 

If the open $1$-cell does not belong to $\mathbb G (\Delta)$, it glues two $2$-cells in the subdivision of a single $2$-cell $d_x$ of $\mathbb G (\Delta)$ for some orbifold point $x \in \Sing (S)$. The two $2$-cells correspond to two adjacent arcs $\gamma, \gamma'$ connecting to $x$ and the $1$-cell corresponds to the orbifold path between them. In this case we put
\begin{align}\label{align:A2quiver}
\mathcal E_\Delta (U_v \cap U_w) = \underset{\gamma}{\bullet} \toarg{q} \underset{\gamma'}{\bullet}
\end{align}
to be the DG category with two objects corresponding to $\gamma, \gamma'$ and one morphism corresponding to the orbifold path from $\gamma$ to $\gamma'$ at $x$. Note that this morphism has some degree (see \S\ref{subsection:grading}).

Let $v \in \mathbb G_0 (\Delta)$. If $\val (v) = n$ we define $\mathcal E_\Delta (U_v)$ according to its type as follows
\[
\begin{tikzpicture}[x=1em,y=1em]
\node[align=center, anchor=north] at (-5,0) {\it vertex\strut \\ \it type\strut};
\node[align=center, anchor=north] at (6,0) {$\mathcal E_\Delta (U_v)$\strut \\ \small\it as category\strut};
\node[align=center, anchor=north] at (17,0) {\it higher\strut \\ \it structure\strut};
\begin{scope}[yshift=-6em]
\node (v) at (-5,0) {$\vbullet$};
\node[shape=circle,inner sep=1.5pt] (0) at (0,0) {$\bullet$};
\node[shape=circle,inner sep=1.5pt] (1) at (4,0) {$\bullet$};
\node[shape=circle,inner sep=1.5pt] (d) at (8,0) {$...$};
\node[shape=circle,inner sep=1.5pt] (n) at (12,0) {$\bullet$};
\node at (17,0) {none};
\path[->, line width=.4pt] (0) edge node[font=\scriptsize, below=-.3ex] {$p_1$} (1) (1) edge node[font=\scriptsize, below=-.3ex] {$p_2$} (d) (d) edge node[font=\scriptsize, below=-.3ex] {$p_{n-1}$} (n);
\draw[dash pattern=on 0pt off 2pt, line width=.5pt, line cap=round] (1) ++(175:1.1em) arc[start angle=175, end angle=5, radius=1.1em] (7.8,0) ++(175:1.1em) arc[start angle=175, end angle=120, radius=1.1em] (8,0) ++(5:1.1em) arc[start angle=5, end angle=60, radius=1.1em];
\end{scope}
\begin{scope}[yshift=-12em]
\node (v) at (-5,1) {$\vodot$};
\node[shape=circle,inner sep=1.5pt] (0) at (0,0) {$\bullet$};
\node[shape=circle,inner sep=1.5pt] (1) at (4,0) {$\bullet$};
\node[shape=circle,inner sep=1.5pt] (d) at (8,0) {$...$};
\node[shape=circle,inner sep=1.5pt] (n) at (12,0) {$\bullet$};
\node at (17,1) {none};
\path[->, line width=.4pt] (0) edge node[font=\scriptsize, below=-.3ex] {$p_1$} (1) (1) edge node[font=\scriptsize, below=-.3ex] {$p_2$} (d) (d) edge node[font=\scriptsize, below=-.3ex] {$p_{n-1}$} (n) (n) edge[bend right=40] node[font=\scriptsize, above=-.3ex] {$p_n$} (0);
\draw[dash pattern=on 0pt off 2pt, line width=.5pt, line cap=round] (0) ++(4:1.4em) arc[start angle=5, end angle=36, radius=1.4em] (1) ++(175:1.1em) arc[start angle=175, end angle=5, radius=1.1em] (7.8,0) ++(175:1.1em) arc[start angle=175, end angle=120, radius=1.1em] (8,0) ++(5:1.1em) arc[start angle=5, end angle=60, radius=1.1em] (n) ++(176:1.4em) arc[start angle=176, end angle=144, radius=1.4em];
\end{scope}
\begin{scope}[yshift=-17em]
\node (v) at (-5,1) {$\vcirc$};
\node[shape=circle,inner sep=1.5pt] (0) at (0,0) {$\bullet$};
\node[shape=circle,inner sep=1.5pt] (1) at (4,0) {$\bullet$};
\node[shape=circle,inner sep=1.5pt] (d) at (8,0) {$...$};
\node[shape=circle,inner sep=1.5pt] (n) at (12,0) {$\bullet$};
\node at (17,1) {$\mucirc_n$};
\path[->, line width=.4pt] (0) edge node[font=\scriptsize, below=-.3ex] {$p_1$} (1) (1) edge node[font=\scriptsize, below=-.3ex] {$p_2$} (d) (d) edge node[font=\scriptsize, below=-.3ex] {$p_{n-1}$} (n) (n) edge[bend right=40] node[font=\scriptsize, above=-.3ex] {$p_n$} (0);
\draw[dash pattern=on 0pt off 2pt, line width=.5pt, line cap=round] (0) ++(4:1.4em) arc[start angle=5, end angle=36, radius=1.4em] (1) ++(175:1.1em) arc[start angle=175, end angle=5, radius=1.1em] (7.8,0) ++(175:1.1em) arc[start angle=175, end angle=120, radius=1.1em] (8,0) ++(5:1.1em) arc[start angle=5, end angle=60, radius=1.1em] (n) ++(176:1.4em) arc[start angle=176, end angle=144, radius=1.4em];
\end{scope}
\begin{scope}[yshift=-22em]
\node (v) at (-5,1) {$\vtimes$};
\node[shape=circle,inner sep=1.5pt] (0) at (0,0) {$\bullet$};
\node[shape=circle,inner sep=1.5pt] (1) at (4,0) {$\bullet$};
\node[shape=circle,inner sep=1.5pt] (d) at (8,0) {$...$};
\node[shape=circle,inner sep=1.5pt] (n) at (12,0) {$\bullet$};
\node[shape=circle,inner sep=1.5pt] (1') at (2.9,1.9) {$\bullet$};
\node[shape=circle,inner sep=1.5pt] (d') at (6,2.4) {$...$};
\node[shape=circle,inner sep=1.5pt] (n-1') at (9.1,1.9) {$\bullet$};
\node at (17,1) {$\mutimes_{n-1}$};
\path[->, line width=.4pt] (0) edge node[font=\scriptsize, below=-.3ex] {$p_1$} (1) (1) edge node[font=\scriptsize, below=-.3ex] {$p_2$} (d) (d) edge node[font=\scriptsize, below=-.3ex] {$p_{n-1}$} (n) (0) edge[bend left=6] node[font=\scriptsize, above=-.1ex] {$q_1$} (1'.203) (1') edge[bend left=6] node[font=\scriptsize, above=-.3ex] {$q_2$} (d'.180) (d'.0) edge[bend left=6] (n-1') (n-1') edge[bend left=6] node[font=\scriptsize, above=-.1ex] {$q_k$} (n);
\draw[dash pattern=on 0pt off 2pt, line width=.5pt, line cap=round] (1) ++(175:1.1em) arc[start angle=175, end angle=5, radius=1.1em] (7.8,0) ++(175:1.1em) arc[start angle=175, end angle=120, radius=1.1em] (8,0) ++(5:1.1em) arc[start angle=5, end angle=60, radius=1.1em];
\end{scope}
\end{tikzpicture}
\]
where each $\bullet$ denotes an object corresponding to an arc in $\Delta$, each arrow represents a morphism and the dotted lines represent zero relations as usual. See Fig.~\ref{fig:vertices} for illustrations of the polygons corresponding to the different types of vertices.

Here $\mucirc_n$ and $\mutimes_{n-1}$ are given by
\begin{align}
\mucirc_n (\s p_{i+n ... i+1}) &= \s \, \id_{\gamma_i} \qquad \text{for any $0 \leq i < n$} \label{eq:mucirc} \\
\mutimes_{n-1} (\s p_{n-1 ... 1}) &= \s q_k \dotsb q_2 q_1. \label{eq:mutimes}
\end{align}
Here \eqref{eq:mucirc} is the higher multiplication introduced in \cite[\S 3.3]{haidenkatzarkovkontsevich} where $\gamma_{i-1}$ is the object in the domain of the morphism $p_i$ and the codomain of the morphism $p_{i+n}$ and the indices are to be understood modulo $n$. Note that \eqref{eq:mutimes} is the higher product \eqref{eq:mutimesunique2} on an orbifold disk obtained in Proposition \ref{proposition:orbifolddisk}.

Note that the category $\H^0 (\tw (\mathcal E_\Delta (U_v))^\natural)$ can itself be viewed as the partially wrapped Fukaya category of a smooth disk with stops for the types $\vbullet$ and $\vcirc$, the partially wrapped Fukaya category of a smooth annulus with stops and one full boundary stop for type $\vodot$, and of an orbifold disk with a single orbifold point and stops in the boundary as in Section \ref{section:disk}. This is in line with the description in \cite{ganatrapardonshende2} where the local categories of the cosheaf on the Lagrangian are indeed partially wrapped Fukaya categories of the sectors.

Note that $\mathcal E_\Delta (U_v)$ contains objects corresponding to arcs in the boundary of $P_v$ and the inclusion functors
\begin{equation}
\label{eq:inclusionfunctors}
\begin{tikzpicture}[baseline=-2.6pt,description/.style={fill=white,inner sep=1pt,outer sep=0}]
\matrix (m) [matrix of math nodes, row sep=.5em, text height=1.5ex, column sep=2.8em, text depth=0.25ex, ampersand replacement=\&, inner sep=3.5pt]
{
                                          \& \mathcal E_\Delta (U_u \cap U_v) \& \mathcal E_\Delta (U_u) \\
\mathcal E_\Delta (U_u \cap U_v \cap U_w) \& \mathcal E_\Delta (U_u \cap U_w) \& \mathcal E_\Delta (U_v) \\
                                          \& \mathcal E_\Delta (U_v \cap U_w) \& \mathcal E_\Delta (U_w) \\
};
\path[->,line width=.4pt] (m-2-1.4) edge (m-1-2.185) (m-2-1) edge (m-2-2) (m-2-1.-4) edge (m-3-2.175) (m-1-2) edge (m-1-3) (m-1-2.-7) edge (m-2-3.175) (m-2-2.4) edge (m-1-3.190) (m-2-2.-4) edge (m-3-3.170) (m-3-2.7) edge (m-2-3.185) (m-3-2) edge (m-3-3);
\end{tikzpicture}
\end{equation}
send the objects in the intersection to the corresponding objects in $\mathcal E_\Delta (U_v)$.

Finally, for any union $U = \bigcup_{v \in V} U_v$ of open sets in $\mathfrak U_\Delta$, indexed by a subset $V \subset \mathbb G_0 (\Delta)$ of vertices, we set
\[
\mathcal E_\Delta (U) = \hocolim \Biggl( \prod_{u, v, w \in V} \!\! \mathcal E_\Delta (U_u \cap U_v \cap U_w) \triplerightarrow \! \prod_{v, w \in V} \mathcal E_\Delta (U_v \cap U_w) \doublerightarrow \prod_{v \in V} \mathcal E_\Delta (U_v) \Biggr)
\]
where the products are taken over distinct vertices in $V$ and the maps in the (truncated) semi-simplicial diagram on the right are induced by the inclusion functors \eqref{eq:inclusionfunctors}.

\begin{notation}
We denote by
\[
\mathbf \Gamma (\mathcal E_\Delta) = \mathcal E_\Delta (\mathbb G (\Delta)) = \mathcal E_\Delta \bigl( \textstyle\bigcup_{v \in \mathbb G_0 (\Delta)} U_v \bigr)
\]
the A$_\infty$ category of {\it global sections} of $\mathcal E_\Delta$.
\end{notation}

\begin{remark}
The cosheaf $\mathcal E_\Delta$ can be viewed as the restriction to $\mathfrak U_\Delta$ of a cosheaf $\mathcal E$ on $\mathbb G (\Delta)$ where for {\it any} open set $U \subset \mathbb G (\Delta)$, the category $\mathcal E (U)$ is the partially wrapped Fukaya category generated by the arcs in $\Delta$ intersected with the Weinstein sector of $\mathbf S$ corresponding to the open set $U$ (see also \S\ref{subsection:core}).
See also \S\ref{subsection:core} for a description in terms of a ribbon {\it graph} for $\mathbf S$.

In fact, $\mathcal E$ is seen to be a constructible cosheaf with respect to the stratification obtained from the subdivision $\widetilde{\mathbb G} (\Delta)$, although we will not need this fact.
\end{remark}

\begin{theorem}\label{theorem:moritaequivalenceglobalsection}
Let $\Delta, \Delta'$ be two weakly admissible dissections of $\mathbf S$. Then $\mathbf \Gamma (\mathcal E_\Delta)$ and $\mathbf \Gamma (\mathcal E_{\Delta'})$ are Morita equivalent.
\end{theorem}

\begin{proof}
Since the ribbon complexes of two dissections of $\mathbf S$ are related by a finite number of edge contractions and expansions, it suffices to prove the case when $\mathbb G (\Delta')$ is obtained from $\mathbb G (\Delta)$ by contracting a single edge $e_\gamma$ connecting two vertices $u, v$ with one of them, say $u$, of type $\vcirc$ resulting in a single vertex $w \in \mathbb G_0 (\Delta')$ of the same type as $v$. In terms of the dissections we have $\Delta = \Delta' \cup \{ \gamma \}$.

The universal property of the homotopy colimit implies that we may compute the global sections of $\mathcal E_\Delta$ iteratively, see e.g.\ \cite[\S 4.2.3]{lurie} and \cite[Example 2.5]{horevyanovski}. In particular, to compute $\mathbf \Gamma (\mathcal E_\Delta) = \mathcal E_\Delta \bigl( \bigcup_{v \in \mathbb G_0 (\Delta)} U_v \bigr)$ we may as a first step compute $\mathcal E_\Delta (U_u \cup U_v)$ as the homotopy pushout of the diagram
\begin{equation}
\label{eq:union}
\begin{tikzpicture}[baseline=-2.6pt,description/.style={fill=white,inner sep=1pt,outer sep=0}]
\matrix (m) [matrix of math nodes, row sep=0em, text height=1.5ex, column sep=3.5em, text depth=0.25ex, ampersand replacement=\&, inner sep=3.5pt]
{
\& \mathcal E_\Delta (U_u) \\
\mathcal E_\Delta (U_u \cap U_v) \& \\
\& \mathcal E_\Delta (U_v) \\
};
\path[->,line width=.4pt] (m-2-1.5) edge (m-1-2.185) (m-2-1.-5) edge (m-3-2.175);
\end{tikzpicture}
\end{equation}
We claim that $\mathcal E_{\Delta'} (U'_w)$ is quasi-equivalent to $\mathcal E_\Delta (U_u \cup U_v)$, where $U'_w \in \mathfrak U_{\Delta'}$ and $U_u, U_v \in \mathfrak U_\Delta$. In other words, we need to show that $\mathcal E_{\Delta'} (U'_w)$ represents the homotopy colimit of the diagram \eqref{eq:union}.

This can be shown by observing that after a cofibrant replacement, the homotopy colimit can be computed as the naive colimit. Rather than working with the (large) bar--cobar construction, we may use the Bardzell resolution in the monomial case built from overlaps of zero relations. This resolution was constructed in \cite{bardzell} and its homotopical properties were studied in \cite{tamaroff}. The case of non-monomial relations was studied in \cite{chouhysolotar}, see also \cite{barmeierwang} for a recursive formula of the differential. The non-monomial case can be obtained from the monomial case by homological perturbation.

For example, when $v$ is of type $\vtimes$, the cofibrant replacement of $\mathcal E_\Delta (U_v)$ is the DG category described by the DG quiver
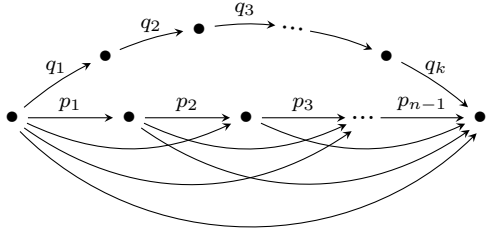
\[
\begin{tikzpicture}[x=1em,y=1em]
\node[shape=circle,inner sep=1.5pt] (0) at (-4,0) {$\bullet$};
\node[shape=circle,inner sep=1.5pt] (1) at (0,0) {$\bullet$};
\node[shape=circle,inner sep=1.5pt] (2) at (4,0) {$\bullet$};
\node[shape=circle,inner sep=1.5pt] (d) at (8,0) {$...$};
\node[shape=circle,inner sep=1.5pt] (n) at (12,0) {$\bullet$};
\node[shape=circle,inner sep=1.5pt] (1') at (-.8,2.1) {$\bullet$};
\node[shape=circle,inner sep=1.5pt] (2') at (2.4,3.03) {$\bullet$};
\node[shape=circle,inner sep=1.5pt, rotate=-8] (d') at (5.6,3.1) {$...$};
\node[shape=circle,inner sep=1.5pt] (n-1') at (8.8,2.1) {$\bullet$};
\path[->, line width=.4pt] (0) edge node[font=\scriptsize, above=-.3ex] {$p_1$} (1) (1) edge node[font=\scriptsize, above=-.3ex] {$p_2$} (2) (2) edge node[font=\scriptsize, above=-.3ex] {$p_3$} (d) (d) edge node[font=\scriptsize, above=-.3ex] {$p_{n-1}$} (n) (0) edge[bend left=6] node[font=\scriptsize, above=-.1ex] {$q_1$} (1'.203) (1'.27) edge[bend left=6] node[font=\scriptsize, above=-.3ex] {$q_2$} (2'.185) (2'.13) edge[bend left=6] node[font=\scriptsize, above=-.2ex] {$q_3$} (d'.180) (d'.-3) edge[bend left=6] (n-1'.152) (n-1'.-20) edge[bend left=6] node[font=\scriptsize, above=-.1ex] {$q_k$} (n) (0.-20) edge[bend right=25] (2.202.5) (1.-20) edge[bend right=25] (d.202.5) (2.-20) edge[bend right=25] (n.202.5) (0.-40) edge[bend right=35] (d.230) (1.-40) edge[bend right=35] (n.230) (0.-60) edge[bend right=45] (n.258);
\end{tikzpicture}
\]
where for each subpath $p_j \dotsb p_{i+1} p_i$ of $p_{n-1} \dotsb p_2 p_1$ there is one parallel arrow of degree $|\s p_{j \dotsc i}| + 1$ which we may denote by $b_{i,j}$. This DG quiver has no relations, but it carries a nonzero differential $d$ which encodes the quadratic monomial relations $p_{i+1} p_i = 0$ and the higher multiplication $\mutimes_{n-1}$ of $\mathcal E_\Delta (U_v)$. Each $b_{i,j}$ can be viewed as a (higher) overlap of the initial monomial relations $p_{i+1} p_i$. (These higher overlaps are also known as higher ambiguities or as Anick chains after \cite{anick} although they can be traced back at least to \cite{greenhappelzacharia}.) The differential is zero on $p_l$'s and $q_l$'s and on the $b_{i,j}$'s it is given by
\begin{align*}
d (b_{i,j}) = \begin{cases}
  \sum_{i\leq k \leq j} (-1)^{|b_{k,j}|} b_{k,j} b_{i,k} & \text{if $(i, j) \neq (1, n-1)$} \\
 \sum_{1\leq k \leq n-1} (-1)^{|b_{k,j}|} b_{k,j} b_{i,k} + q_k \dotsb q_2 q_1 & \text{if $(i, j) = (1, n-1)$},
\end{cases}
\end{align*}
where we set $b_{i, i} = p_{i}$. (This is similar to the proof of \cite[Proposition 3.7]{haidenkatzarkovkontsevich}.)

Note that $\mathcal E_\Delta (U_u \cap U_v)$ consists either of a single object or of two objects and a single morphism which corresponds to an orbifold path. In particular $\mathcal E_\Delta (U_u \cap U_v)$ is described by the graded quiver of type $\mathrm A_1$ or $\mathrm A_2$ (see \eqref{align:A1quiver} and \eqref{align:A2quiver}) and hence already cofibrant. Gluing the cofibrant replacements of $\mathcal E_\Delta (U_u)$ and $\mathcal E_\Delta (U_v)$ along $\mathcal E_\Delta (U_u \cap U_v)$ in the naive sense, i.e.\ by identifying the objects and morphism, does not create any extra relations or overlaps. In particular, the naive colimit is still cofibrant and coincides with the cofibrant replacement of $\mathcal E_{\Delta'} (U'_w)$. This shows that $\mathcal E_{\Delta'} (U'_w)$ indeed represents the homotopy colimit of the diagram \eqref{eq:union}.

The statement now follows from the observation that all other categories in the diagram computing the global sections of $\mathcal E_\Delta$ and $\mathcal E_{\Delta'}$ are equivalent since $\mathbb G (\Delta)$ and $\mathbb G (\Delta')$ are related by a single edge contraction, so the global sections are canonically isomorphic.
\end{proof}

\subsection{The partially wrapped Fukaya category of an orbifold surface}

We may now define the partially wrapped Fukaya category of $\mathbf S$.

\begin{definition}
\label{definition:fukayacategory}
Let $\mathbf S = (S, \Sigma, \eta)$ be a graded orbifold surface with stops and let $\Delta$ be any weakly admissible dissection of $\mathbf S$. We define the {\it partially wrapped Fukaya category} of $\mathbf S$ as the triangulated category
\[
\mathcal W (\mathbf S) := \H^0 (\tw (\mathbf \Gamma (\mathcal E_\Delta))^\natural)
\]
where $\mathbf \Gamma (\mathcal E_\Delta)$ is the category of global sections of the cosheaf $\mathcal E_\Delta$ of A$_\infty$ categories associated to $\Delta$.
\end{definition}

Note that $ \tw (\mathbf \Gamma (\mathcal E_\Delta))^\natural$ and $\mathbf \Gamma (\tw (\mathcal E_\Delta)^\natural)$ are Morita equivalent, since $\mathbf A$ and $\tw(\mathbf A)$ are Morita equivalent for any A$_\infty$ category $\mathbf A$ \cite[Lemma 3.34]{seidel2} and are thus isomorphic in the homotopy category of DG categories.

In Section \ref{section:ainfinity} we give an explicit presentation of $\mathcal W (\mathbf S)$ as $\H^0 (\tw (\mathbf A_\Delta)^\natural)$, where $\mathbf A_\Delta$ is an A$_\infty$ category associated to an {\it admissible} dissection, all of whose higher products we give explicitly. In Section \ref{section:doublecover} we show that we also have an equivalence $\mathcal W (\mathbf S) \simeq (\mathcal W (\widetilde{\mathbf S}) / \mathbb Z_2)^\natural$ where the latter is the A$_\infty$ orbit category of a smooth double cover $\widetilde{\mathbf S}$ of $\mathbf S$.

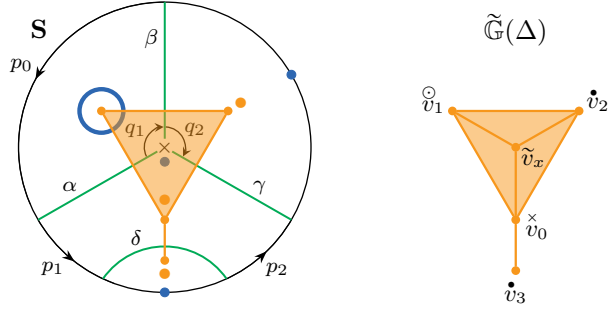
\begin{figure}
\begin{tikzpicture}[x=1em,y=1em,decoration={markings,mark=at position 0.55 with {\arrow[black]{Stealth[length=4.8pt]}}}]
\begin{scope}
\node[font=\small] at (-4.3em,4em) {$\mathbf S$\strut};
\node[font=\scriptsize,shape=circle,scale=.6] (X) at (0,0) {};
\draw[line width=0,postaction={decorate}] (90:5em) arc[start angle=90, end angle=203, radius=5em];
\draw[line width=0,postaction={decorate}] (210:5em) arc[start angle=210, end angle=250, radius=5em];
\draw[line width=0,postaction={decorate}] (295:5em) arc[start angle=295, end angle=330, radius=5em];
\draw[line width=.5pt] (0,0) circle(5em);
\draw[line width=.15em,color=stopcolour] (150:2.5em) circle(.75em);
\draw[fill=stopcolour,color=stopcolour] (30:5em) circle(.15em);
\draw[fill=stopcolour,color=stopcolour] (270:5em) circle(.15em);
\path[line width=.75pt,color=arccolour] (X) edge (330:5em);
\path[line width=.75pt,color=arccolour] (X) edge (210:5em);
\path[line width=.75pt,color=arccolour] (X) edge (90:5em);
\path[line width=.75pt,color=arccolour] (245:5em) edge[bend left=55, looseness=1.1] (295:5em);
\draw[->, line width=.5pt] (207:.7em) arc[start angle=207, end angle=93, radius=.7em];
\draw[->, line width=.5pt] (87:.7em) arc[start angle=87, end angle=-27, radius=.7em];
\node[font=\scriptsize, above] at (210:3.75em) {$\alpha$};
\node[font=\scriptsize, left=-.3ex] at (90:3.75em) {$\beta$};
\node[font=\scriptsize, above=-.2ex] at (-30:3.75em) {$\gamma$};
\node[font=\scriptsize, above=-.2ex] at (255:3.75em) {$\delta$};
\node[font=\scriptsize] at (150:1.15em) {$q_1$};
\node[font=\scriptsize] at (30:1.2em) {$q_2$};
\node[font=\scriptsize] at (150:5.65em) {$p_0$};
\node[font=\scriptsize] at (227.5:5.65em) {$p_1$};
\node[font=\scriptsize] at (312.5:5.7em) {$p_2$};
\node[font=\scriptsize] at (0,0) {$\times$};
\draw[fill=stopcolour,color=stopcolour] (270:.5em) circle(.15em);
\node[circle, fill=ribboncolour, minimum size=.3em, inner sep=0] (0) at (270:2.5em) {};
\node[circle, fill=ribboncolour, minimum size=.3em, inner sep=0] (1) at (150:2.5em) {};
\node[color=ribboncolour] at (30:3em) {$\bullet$};
\node[circle, fill=ribboncolour, minimum size=.3em, inner sep=0] (2) at (30:2.5em) {};
\node[color=ribboncolour] at (270:1.85em) {$\bullet$};
\node[circle, fill=ribboncolour, minimum size=.3em, inner sep=0] (3) at (270:3.9em) {};
\node[color=ribboncolour] at (270:4.4em) {$\bullet$};
\draw[line width=0, fill=ribboncolour, draw=ribboncolour, opacity=.5] (270:2.5em) -- (150:2.5em) -- (30:2.5em) -- cycle;
\draw[-, line width=.85, color=ribboncolour, line cap=round] (3) to (0) to (1) to (2) to (0);
\end{scope}
\begin{scope}[xshift=12em]
\node[circle, fill=ribboncolour, minimum size=.3em, inner sep=0] (X) at (0,0) {};
\node[circle, fill=ribboncolour, minimum size=.3em, inner sep=0] (0) at (270:2.5em) {};
\node[circle, fill=ribboncolour, minimum size=.3em, inner sep=0] (1) at (150:2.5em) {};
\node[circle, fill=ribboncolour, minimum size=.3em, inner sep=0] (2) at (30:2.5em) {};
\node[circle, fill=ribboncolour, minimum size=.3em, inner sep=0] (3) at (270:4.25em) {};
\draw[line width=0, fill=ribboncolour, draw=ribboncolour, opacity=.5] (270:2.5em) -- (150:2.5em) -- (30:2.5em) -- cycle;
\draw[-, line width=.85, color=ribboncolour, line cap=round] (3) to (0) to (1) to (2) to (0) to (X) to (1) (X) to (2);
\node[font=\footnotesize] at (285:2.8em) {$\vtimes_0$};
\node[font=\footnotesize] at (150:3.25em) {$\vodot_1$};
\node[font=\footnotesize] at (30:3.25em) {$\vbullet_2$};
\node[font=\footnotesize] at (270:5em) {$\vbullet_3$};
\node[font=\footnotesize] at (-30:.7em) {$\widetilde v_x$};
\node[font=\small] at (0,4em) {$\widetilde{\mathbb G} (\Delta)$\strut};
\end{scope}
\end{tikzpicture}
\caption{A formal dissection $\Delta$ of an annulus $\mathbf S$ with one orbifold point $x$, its ribbon complex $\mathbb G (\Delta)$ (superimposed on $\mathbf S$) and the subdivided ribbon complex $\widetilde{\mathbb G} (\Delta)$.}
\label{fig:orbifoldannulus}
\end{figure}

\begin{example}
\label{example:cosheaf}
Let $\mathbf S = (S, \Sigma, \eta)$ be an annulus $S$ with $\Sigma \subset \partial S$ consisting of one full boundary stop and two stops in the other boundary component and $\Sing (S) = \{ x \}$ consists of a single orbifold point. Let $\eta$ by any line field. Let $\Delta = \{ \alpha, \beta, \gamma, \delta \}$ be the formal admissible dissection illustrated in Fig.~\ref{fig:orbifoldannulus}.

The Čech nerve of the open cover $\mathfrak U_\Delta$ of the ribbon complex $\mathbb G (\Delta)$ can be visualized as follows
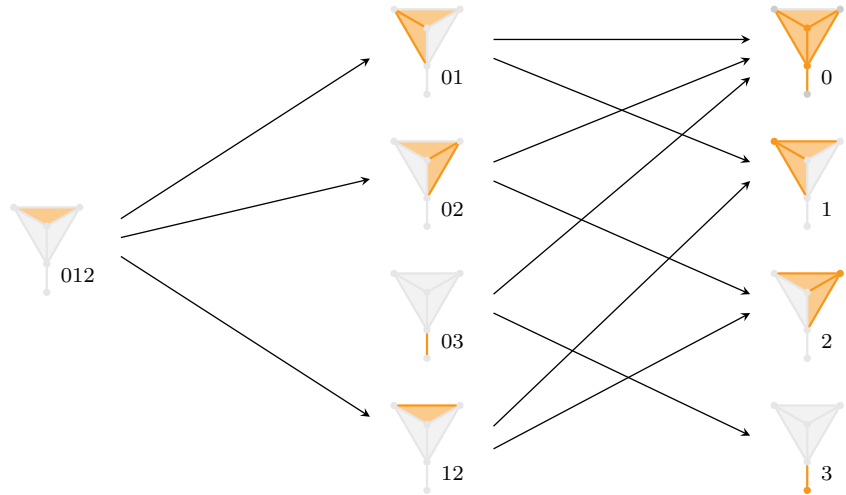
\[
\begin{tikzpicture}[x=1em, y=1em, scale=1.3]
\draw[->, line width=.5pt] (-20+1.95,-5.25-.3+.5) to (-10-1.5,0-.3-.5);
\draw[->, line width=.5pt] (-20+1.95,-5.25-.3) to (-10-1.5,-3.5-.3-.2);
\draw[->, line width=.5pt] (-20+1.95,-5.25-.3-.5) to (-10-1.5,-10.5-.3+.5);
\draw[->, line width=.5pt] (-10+1.75,-.3) to (-1.5,0-.3);
\draw[->, line width=.5pt] (-10+1.75,-.3-.5) to (-1.5,-3.5-.3+.25);
\draw[->, line width=.5pt] (-10+1.75,-3.5-.3+.25) to (-1.5,-.3-.5);
\draw[->, line width=.5pt] (-10+1.75,-3.5-.3-.25) to (-1.5,-7-.3+.25);
\draw[->, line width=.5pt] (-10+1.75,-7-.3+.25) to (-1.5,-.3-1);
\draw[->, line width=.5pt] (-10+1.75,-7-.3-.25) to (-1.5,-10.5-.3);
\draw[->, line width=.5pt] (-10+1.75,-10.5-.3+.25) to (-1.5,-3.5-.3-.25);
\draw[->, line width=.5pt] (-10+1.75,-10.5-.3-.35) to (-1.5,-7-.3-.25);
\begin{scope}
\node[font=\scriptsize, right] at (.1em,-1.3em) {$0$};
\draw[line width=0, fill=ribboncolour, draw=ribboncolour, opacity=.5] (270:1em) -- (150:1em) -- (0,0) -- cycle;
\draw[line width=0, fill=ribboncolour, draw=ribboncolour, opacity=.5] (150:1em) -- (30:1em) -- (0,0) -- cycle;
\draw[line width=0, fill=ribboncolour, draw=ribboncolour, opacity=.5] (270:1em) -- (30:1em) -- (0,0) -- cycle;
\node[circle, fill=ribboncolour, minimum size=.25em, inner sep=0] (X) at (0,0) {};
\node[circle, fill=ribboncolour, minimum size=.25em, inner sep=0] (0) at (270:1em) {};
\node[circle, fill=black!20, minimum size=.25em, inner sep=0] (1) at (150:1em) {};
\node[circle, fill=black!20, minimum size=.25em, inner sep=0] (2) at (30:1em) {};
\node[circle, fill=black!20, minimum size=.25em, inner sep=0] (3) at (270:1.75em) {};
\draw[-, line width=.85, color=ribboncolour, line cap=round] (3) to (0);
\draw[-, line width=.85, color=ribboncolour, line cap=round] (0) to (1);
\draw[-, line width=.85, color=black!10, line cap=round] (1) to (2);
\draw[-, line width=.85, color=ribboncolour, line cap=round] (2) to (0);
\draw[-, line width=.85, color=ribboncolour, line cap=round] (0) to (X);
\draw[-, line width=.85, color=ribboncolour, line cap=round] (X) to (1);
\draw[-, line width=.85, color=ribboncolour, line cap=round] (X) to (2);
\end{scope}
\begin{scope}[yshift=-3.5em]
\node[font=\scriptsize, right] at (.1em,-1.3em) {$1$};
\draw[line width=0, fill=ribboncolour, draw=ribboncolour, opacity=.5] (270:1em) -- (150:1em) -- (0,0) -- cycle;
\draw[line width=0, fill=ribboncolour, draw=ribboncolour, opacity=.5] (150:1em) -- (30:1em) -- (0,0) -- cycle;
\draw[line width=0, fill=black!10, draw=black!10, opacity=.5] (270:1em) -- (30:1em) -- (0,0) -- cycle;
\node[circle, fill=black!10, minimum size=.25em, inner sep=0] (X) at (0,0) {};
\node[circle, fill=black!10, minimum size=.25em, inner sep=0] (0) at (270:1em) {};
\node[circle, fill=ribboncolour, minimum size=.25em, inner sep=0] (1) at (150:1em) {};
\node[circle, fill=black!10, minimum size=.25em, inner sep=0] (2) at (30:1em) {};
\node[circle, fill=black!10, minimum size=.25em, inner sep=0] (3) at (270:1.75em) {};
\draw[-, line width=.85, color=black!10, line cap=round] (3) to (0);
\draw[-, line width=.85, color=ribboncolour, line cap=round] (0) to (1);
\draw[-, line width=.85, color=ribboncolour, line cap=round] (1) to (2);
\draw[-, line width=.85, color=black!10, line cap=round] (2) to (0);
\draw[-, line width=.85, color=black!10, line cap=round] (0) to (X);
\draw[-, line width=.85, color=ribboncolour, line cap=round] (X) to (1);
\draw[-, line width=.85, color=black!10, line cap=round] (X) to (2);
\end{scope}
\begin{scope}[yshift=-7em]
\node[font=\scriptsize, right] at (.1em,-1.3em) {$2$};
\draw[line width=0, fill=black!10, draw=black!10, opacity=.5] (270:1em) -- (150:1em) -- (0,0) -- cycle;
\draw[line width=0, fill=ribboncolour, draw=ribboncolour, opacity=.5] (150:1em) -- (30:1em) -- (0,0) -- cycle;
\draw[line width=0, fill=ribboncolour, draw=ribboncolour, opacity=.5] (270:1em) -- (30:1em) -- (0,0) -- cycle;
\node[circle, fill=black!10, minimum size=.25em, inner sep=0] (X) at (0,0) {};
\node[circle, fill=black!10, minimum size=.25em, inner sep=0] (0) at (270:1em) {};
\node[circle, fill=black!10, minimum size=.25em, inner sep=0] (1) at (150:1em) {};
\node[circle, fill=ribboncolour, minimum size=.25em, inner sep=0] (2) at (30:1em) {};
\node[circle, fill=black!10, minimum size=.25em, inner sep=0] (3) at (270:1.75em) {};
\draw[-, line width=.85, color=black!10, line cap=round] (3) to (0);
\draw[-, line width=.85, color=black!10, line cap=round] (0) to (1);
\draw[-, line width=.85, color=ribboncolour, line cap=round] (1) to (2);
\draw[-, line width=.85, color=ribboncolour, line cap=round] (2) to (0);
\draw[-, line width=.85, color=black!10, line cap=round] (0) to (X);
\draw[-, line width=.85, color=black!10, line cap=round] (X) to (1);
\draw[-, line width=.85, color=ribboncolour, line cap=round] (X) to (2);
\end{scope}
\begin{scope}[yshift=-10.5em]
\node[font=\scriptsize, right] at (.1em,-1.3em) {$3$};
\draw[line width=0, fill=black!10, draw=black!10, opacity=.5] (270:1em) -- (150:1em) -- (0,0) -- cycle;
\draw[line width=0, fill=black!10, draw=black!10, opacity=.5] (150:1em) -- (30:1em) -- (0,0) -- cycle;
\draw[line width=0, fill=black!10, draw=black!10, opacity=.5] (270:1em) -- (30:1em) -- (0,0) -- cycle;
\node[circle, fill=black!10, minimum size=.25em, inner sep=0] (X) at (0,0) {};
\node[circle, fill=black!10, minimum size=.25em, inner sep=0] (0) at (270:1em) {};
\node[circle, fill=black!10, minimum size=.25em, inner sep=0] (1) at (150:1em) {};
\node[circle, fill=black!10, minimum size=.25em, inner sep=0] (2) at (30:1em) {};
\node[circle, fill=ribboncolour, minimum size=.25em, inner sep=0] (3) at (270:1.75em) {};
\draw[-, line width=.85, color=ribboncolour, line cap=round] (3) to (0);
\draw[-, line width=.85, color=black!10, line cap=round] (0) to (1);
\draw[-, line width=.85, color=black!10, line cap=round] (1) to (2);
\draw[-, line width=.85, color=black!10, line cap=round] (2) to (0);
\draw[-, line width=.85, color=black!10, line cap=round] (0) to (X);
\draw[-, line width=.85, color=black!10, line cap=round] (X) to (1);
\draw[-, line width=.85, color=black!10, line cap=round] (X) to (2);
\end{scope}
\begin{scope}[xshift=-10em]
\node[font=\scriptsize, right] at (.1em,-1.3em) {$01$};
\draw[line width=0, fill=ribboncolour, draw=ribboncolour, opacity=.5] (270:1em) -- (150:1em) -- (0,0) -- cycle;
\draw[line width=0, fill=ribboncolour, draw=ribboncolour, opacity=.5] (150:1em) -- (30:1em) -- (0,0) -- cycle;
\draw[line width=0, fill=black!10, draw=black!10, opacity=.5] (270:1em) -- (30:1em) -- (0,0) -- cycle;
\node[circle, fill=black!10, minimum size=.25em, inner sep=0] (X) at (0,0) {};
\node[circle, fill=black!10, minimum size=.25em, inner sep=0] (0) at (270:1em) {};
\node[circle, fill=black!10, minimum size=.25em, inner sep=0] (1) at (150:1em) {};
\node[circle, fill=black!10, minimum size=.25em, inner sep=0] (2) at (30:1em) {};
\node[circle, fill=black!10, minimum size=.25em, inner sep=0] (3) at (270:1.75em) {};
\draw[-, line width=.85, color=black!10, line cap=round] (3) to (0);
\draw[-, line width=.85, color=ribboncolour, line cap=round] (0) to (1);
\draw[-, line width=.85, color=black!10, line cap=round] (1) to (2);
\draw[-, line width=.85, color=black!10, line cap=round] (2) to (0);
\draw[-, line width=.85, color=black!10, line cap=round] (0) to (X);
\draw[-, line width=.85, color=ribboncolour, line cap=round] (X) to (1);
\draw[-, line width=.85, color=black!10, line cap=round] (X) to (2);
\end{scope}
\begin{scope}[xshift=-10em,yshift=-3.5em]
\node[font=\scriptsize, right] at (.1em,-1.3em) {$02$};
\draw[line width=0, fill=black!10, draw=ribboncolour, opacity=.5] (270:1em) -- (150:1em) -- (0,0) -- cycle;
\draw[line width=0, fill=ribboncolour, draw=ribboncolour, opacity=.5] (150:1em) -- (30:1em) -- (0,0) -- cycle;
\draw[line width=0, fill=ribboncolour, draw=ribboncolour, opacity=.5] (270:1em) -- (30:1em) -- (0,0) -- cycle;
\node[circle, fill=black!10, minimum size=.25em, inner sep=0] (X) at (0,0) {};
\node[circle, fill=black!10, minimum size=.25em, inner sep=0] (0) at (270:1em) {};
\node[circle, fill=black!10, minimum size=.25em, inner sep=0] (1) at (150:1em) {};
\node[circle, fill=black!10, minimum size=.25em, inner sep=0] (2) at (30:1em) {};
\node[circle, fill=black!10, minimum size=.25em, inner sep=0] (3) at (270:1.75em) {};
\draw[-, line width=.85, color=black!10, line cap=round] (3) to (0);
\draw[-, line width=.85, color=black!10, line cap=round] (0) to (1);
\draw[-, line width=.85, color=black!10, line cap=round] (1) to (2);
\draw[-, line width=.85, color=ribboncolour, line cap=round] (2) to (0);
\draw[-, line width=.85, color=black!10, line cap=round] (0) to (X);
\draw[-, line width=.85, color=black!10, line cap=round] (X) to (1);
\draw[-, line width=.85, color=ribboncolour, line cap=round] (X) to (2);
\end{scope}
\begin{scope}[xshift=-10em,yshift=-7em]
\node[font=\scriptsize, right] at (.1em,-1.3em) {$03$};
\draw[line width=0, fill=black!10, draw=black!10, opacity=.5] (270:1em) -- (150:1em) -- (0,0) -- cycle;
\draw[line width=0, fill=black!10, draw=black!10, opacity=.5] (150:1em) -- (30:1em) -- (0,0) -- cycle;
\draw[line width=0, fill=black!10, draw=black!10, opacity=.5] (270:1em) -- (30:1em) -- (0,0) -- cycle;
\node[circle, fill=black!10, minimum size=.25em, inner sep=0] (X) at (0,0) {};
\node[circle, fill=black!10, minimum size=.25em, inner sep=0] (0) at (270:1em) {};
\node[circle, fill=black!10, minimum size=.25em, inner sep=0] (1) at (150:1em) {};
\node[circle, fill=black!10, minimum size=.25em, inner sep=0] (2) at (30:1em) {};
\node[circle, fill=black!10, minimum size=.25em, inner sep=0] (3) at (270:1.75em) {};
\draw[-, line width=.85, color=ribboncolour, line cap=round] (3) to (0);
\draw[-, line width=.85, color=black!10, line cap=round] (0) to (1);
\draw[-, line width=.85, color=black!10, line cap=round] (1) to (2);
\draw[-, line width=.85, color=black!10, line cap=round] (2) to (0);
\draw[-, line width=.85, color=black!10, line cap=round] (0) to (X);
\draw[-, line width=.85, color=black!10, line cap=round] (X) to (1);
\draw[-, line width=.85, color=black!10, line cap=round] (X) to (2);
\end{scope}
\begin{scope}[xshift=-10em,yshift=-10.5em]
\node[font=\scriptsize, right] at (.1em,-1.3em) {$12$};
\draw[line width=0, fill=black!10, draw=black!10, opacity=.5] (270:1em) -- (150:1em) -- (0,0) -- cycle;
\draw[line width=0, fill=ribboncolour, draw=ribboncolour, opacity=.5] (150:1em) -- (30:1em) -- (0,0) -- cycle;
\draw[line width=0, fill=black!10, draw=black!10, opacity=.5] (270:1em) -- (30:1em) -- (0,0) -- cycle;
\node[circle, fill=black!10, minimum size=.25em, inner sep=0] (X) at (0,0) {};
\node[circle, fill=black!10, minimum size=.25em, inner sep=0] (0) at (270:1em) {};
\node[circle, fill=black!10, minimum size=.25em, inner sep=0] (1) at (150:1em) {};
\node[circle, fill=black!10, minimum size=.25em, inner sep=0] (2) at (30:1em) {};
\node[circle, fill=black!10, minimum size=.25em, inner sep=0] (3) at (270:1.75em) {};
\draw[-, line width=.85, color=black!10, line cap=round] (3) to (0);
\draw[-, line width=.85, color=black!10, line cap=round] (0) to (1);
\draw[-, line width=.85, color=ribboncolour, line cap=round] (1) to (2);
\draw[-, line width=.85, color=black!10, line cap=round] (2) to (0);
\draw[-, line width=.85, color=black!10, line cap=round] (0) to (X);
\draw[-, line width=.85, color=black!10, line cap=round] (X) to (1);
\draw[-, line width=.85, color=black!10, line cap=round] (X) to (2);
\end{scope}
\begin{scope}[xshift=-20em,yshift=-5.25em]
\node[font=\scriptsize, right] at (.1em,-1.3em) {$012$};
\draw[line width=0, fill=black!10, draw=black!10, opacity=.5] (270:1em) -- (150:1em) -- (0,0) -- cycle;
\draw[line width=0, fill=ribboncolour, draw=ribboncolour, opacity=.5] (150:1em) -- (30:1em) -- (0,0) -- cycle;
\draw[line width=0, fill=black!10, draw=black!10, opacity=.5] (270:1em) -- (30:1em) -- (0,0) -- cycle;
\node[circle, fill=black!10, minimum size=.25em, inner sep=0] (X) at (0,0) {};
\node[circle, fill=black!10, minimum size=.25em, inner sep=0] (0) at (270:1em) {};
\node[circle, fill=black!10, minimum size=.25em, inner sep=0] (1) at (150:1em) {};
\node[circle, fill=black!10, minimum size=.25em, inner sep=0] (2) at (30:1em) {};
\node[circle, fill=black!10, minimum size=.25em, inner sep=0] (3) at (270:1.75em) {};
\draw[-, line width=.85, color=black!10, line cap=round] (3) to (0);
\draw[-, line width=.85, color=black!10, line cap=round] (0) to (1);
\draw[-, line width=.85, color=black!10, line cap=round] (1) to (2);
\draw[-, line width=.85, color=black!10, line cap=round] (2) to (0);
\draw[-, line width=.85, color=black!10, line cap=round] (0) to (X);
\draw[-, line width=.85, color=black!10, line cap=round] (X) to (1);
\draw[-, line width=.85, color=black!10, line cap=round] (X) to (2);
\end{scope}
\end{tikzpicture}
\]
where $0$ corresponds to $U_{\vtimes_0}$, $1$ corresponds to $U_{\vodot_1}$ and $02$ to the intersection $U_{\vtimes_0} \cap U_{\vbullet_2}$ and so forth.

Evaluating the cosheaf $\mathcal E_\Delta$ on this diagram of open sets we obtain the following diagram of A$_\infty$ categories
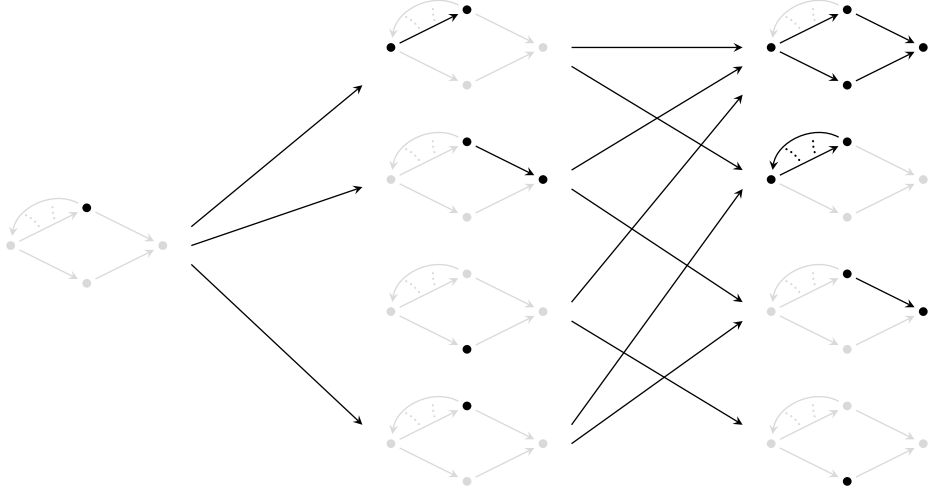
\[
\begin{tikzpicture}[x=1em, y=1em, scale=1.3]
\draw[->, line width=.5pt] (-20+2.75,-5.25+.5) to (-10-2.75,0-1);
\draw[->, line width=.5pt] (-20+2.75,-5.25) to (-10-2.75,-3.5-.2);
\draw[->, line width=.5pt] (-20+2.75,-5.25-.5) to (-10-2.75,-10.5+.5);
\draw[->, line width=.5pt] (-10+2.75,0) to (-2.75,0);
\draw[->, line width=.5pt] (-10+2.75,-.5) to (-2.75,-3.5+.25);
\draw[->, line width=.5pt] (-10+2.75,-3.5+.25) to (-2.75,-.5);
\draw[->, line width=.5pt] (-10+2.75,-3.5-.25) to (-2.75,-7+.25);
\draw[->, line width=.5pt] (-10+2.75,-7+.25) to (-2.75,-1.25);
\draw[->, line width=.5pt] (-10+2.75,-7-.25) to (-2.75,-10.5+.5);
\draw[->, line width=.5pt] (-10+2.75,-10.5+.5) to (-2.75,-3.5-.25);
\draw[->, line width=.5pt] (-10+2.75,-10.5) to (-2.75,-7-.25);
\begin{scope} 
\node[circle, fill=black, minimum size=.3em, inner sep=0, outer sep=2pt] (A) at (-2em,0) {};
\node[circle, fill=black, minimum size=.3em, inner sep=0, outer sep=2pt] (B) at (0,1em) {};
\node[circle, fill=black, minimum size=.3em, inner sep=0, outer sep=2pt] (C) at (2em,0) {};
\node[circle, fill=black, minimum size=.3em, inner sep=0, outer sep=2pt] (D) at (0,-1em) {};
\draw[->, line width=.5pt] (A) to (B);
\draw[->, line width=.5pt, color=black!15] (B) to[bend right=55] (A);
\draw[->, line width=.5pt] (B) to (C);
\draw[->, line width=.5pt] (A) to (D);
\draw[->, line width=.5pt] (D) to (C);
\draw[dash pattern=on 0pt off 2pt, line width=.8pt, line cap=round, color=black!15] (-2em,0) ++(35:.9em) arc[start angle=35, end angle=70, radius=.9em];
\draw[dash pattern=on 0pt off 2pt, line width=.8pt, line cap=round, color=black!15] (0,1em) ++(198:.9em) arc[start angle=198, end angle=170, radius=.9em];
\end{scope}
\begin{scope}[yshift=-3.5em] 
\node[circle, fill=black, minimum size=.3em, inner sep=0, outer sep=2pt] (A) at (-2em,0) {};
\node[circle, fill=black, minimum size=.3em, inner sep=0, outer sep=2pt] (B) at (0,1em) {};
\node[circle, fill=black!15, minimum size=.3em, inner sep=0, outer sep=2pt] (C) at (2em,0) {};
\node[circle, fill=black!15, minimum size=.3em, inner sep=0, outer sep=2pt] (D) at (0,-1em) {};
\draw[->, line width=.5pt] (A) to (B);
\draw[->, line width=.5pt] (B) to[bend right=55] (A);
\draw[->, line width=.5pt, color=black!15] (B) to (C);
\draw[->, line width=.5pt, color=black!15] (A) to (D);
\draw[->, line width=.5pt, color=black!15] (D) to (C);
\draw[dash pattern=on 0pt off 2pt, line width=.8pt, line cap=round] (-2em,0) ++(35:.9em) arc[start angle=35, end angle=70, radius=.9em];
\draw[dash pattern=on 0pt off 2pt, line width=.8pt, line cap=round] (0,1em) ++(198:.9em) arc[start angle=198, end angle=170, radius=.9em];
\end{scope}
\begin{scope}[yshift=-7em] 
\node[circle, fill=black!15, minimum size=.3em, inner sep=0, outer sep=2pt] (A) at (-2em,0) {};
\node[circle, fill=black, minimum size=.3em, inner sep=0, outer sep=2pt] (B) at (0,1em) {};
\node[circle, fill=black, minimum size=.3em, inner sep=0, outer sep=2pt] (C) at (2em,0) {};
\node[circle, fill=black!15, minimum size=.3em, inner sep=0, outer sep=2pt] (D) at (0,-1em) {};
\draw[->, line width=.5pt, color=black!15] (A) to (B);
\draw[->, line width=.5pt, color=black!15] (B) to[bend right=55] (A);
\draw[->, line width=.5pt] (B) to (C);
\draw[->, line width=.5pt, color=black!15] (A) to (D);
\draw[->, line width=.5pt, color=black!15] (D) to (C);
\draw[dash pattern=on 0pt off 2pt, line width=.8pt, line cap=round, color=black!15] (-2em,0) ++(35:.9em) arc[start angle=35, end angle=70, radius=.9em];
\draw[dash pattern=on 0pt off 2pt, line width=.8pt, line cap=round, color=black!15] (0,1em) ++(198:.9em) arc[start angle=198, end angle=170, radius=.9em];
\end{scope}
\begin{scope}[yshift=-10.5em] 
\node[circle, fill=black!15, minimum size=.3em, inner sep=0, outer sep=2pt] (A) at (-2em,0) {};
\node[circle, fill=black!15, minimum size=.3em, inner sep=0, outer sep=2pt] (B) at (0,1em) {};
\node[circle, fill=black!15, minimum size=.3em, inner sep=0, outer sep=2pt] (C) at (2em,0) {};
\node[circle, fill=black, minimum size=.3em, inner sep=0, outer sep=2pt] (D) at (0,-1em) {};
\draw[->, line width=.5pt, color=black!15] (A) to (B);
\draw[->, line width=.5pt, color=black!15] (B) to[bend right=55] (A);
\draw[->, line width=.5pt, color=black!15] (B) to (C);
\draw[->, line width=.5pt, color=black!15] (A) to (D);
\draw[->, line width=.5pt, color=black!15] (D) to (C);
\draw[dash pattern=on 0pt off 2pt, line width=.8pt, line cap=round, color=black!15] (-2em,0) ++(35:.9em) arc[start angle=35, end angle=70, radius=.9em];
\draw[dash pattern=on 0pt off 2pt, line width=.8pt, line cap=round, color=black!15] (0,1em) ++(198:.9em) arc[start angle=198, end angle=170, radius=.9em];
\end{scope}
\begin{scope}[xshift=-10em] 
\node[circle, fill=black, minimum size=.3em, inner sep=0, outer sep=2pt] (A) at (-2em,0) {};
\node[circle, fill=black, minimum size=.3em, inner sep=0, outer sep=2pt] (B) at (0,1em) {};
\node[circle, fill=black!15, minimum size=.3em, inner sep=0, outer sep=2pt] (C) at (2em,0) {};
\node[circle, fill=black!15, minimum size=.3em, inner sep=0, outer sep=2pt] (D) at (0,-1em) {};
\draw[->, line width=.5pt] (A) to (B);
\draw[->, line width=.5pt, color=black!15] (B) to[bend right=55] (A);
\draw[->, line width=.5pt, color=black!15] (B) to (C);
\draw[->, line width=.5pt, color=black!15] (A) to (D);
\draw[->, line width=.5pt, color=black!15] (D) to (C);
\draw[dash pattern=on 0pt off 2pt, line width=.8pt, line cap=round, color=black!15] (-2em,0) ++(35:.9em) arc[start angle=35, end angle=70, radius=.9em];
\draw[dash pattern=on 0pt off 2pt, line width=.8pt, line cap=round, color=black!15] (0,1em) ++(198:.9em) arc[start angle=198, end angle=170, radius=.9em];
\end{scope}
\begin{scope}[xshift=-10em,yshift=-3.5em] 
\node[circle, fill=black!15, minimum size=.3em, inner sep=0, outer sep=2pt] (A) at (-2em,0) {};
\node[circle, fill=black, minimum size=.3em, inner sep=0, outer sep=2pt] (B) at (0,1em) {};
\node[circle, fill=black, minimum size=.3em, inner sep=0, outer sep=2pt] (C) at (2em,0) {};
\node[circle, fill=black!15, minimum size=.3em, inner sep=0, outer sep=2pt] (D) at (0,-1em) {};
\draw[->, line width=.5pt, color=black!15] (A) to (B);
\draw[->, line width=.5pt, color=black!15] (B) to[bend right=55] (A);
\draw[->, line width=.5pt] (B) to (C);
\draw[->, line width=.5pt, color=black!15] (A) to (D);
\draw[->, line width=.5pt, color=black!15] (D) to (C);
\draw[dash pattern=on 0pt off 2pt, line width=.8pt, line cap=round, color=black!15] (-2em,0) ++(35:.9em) arc[start angle=35, end angle=70, radius=.9em];
\draw[dash pattern=on 0pt off 2pt, line width=.8pt, line cap=round, color=black!15] (0,1em) ++(198:.9em) arc[start angle=198, end angle=170, radius=.9em];
\end{scope}
\begin{scope}[xshift=-10em, yshift=-7em] 
\node[circle, fill=black!15, minimum size=.3em, inner sep=0, outer sep=2pt] (A) at (-2em,0) {};
\node[circle, fill=black!15, minimum size=.3em, inner sep=0, outer sep=2pt] (B) at (0,1em) {};
\node[circle, fill=black!15, minimum size=.3em, inner sep=0, outer sep=2pt] (C) at (2em,0) {};
\node[circle, fill=black, minimum size=.3em, inner sep=0, outer sep=2pt] (D) at (0,-1em) {};
\draw[->, line width=.5pt, color=black!15] (A) to (B);
\draw[->, line width=.5pt, color=black!15] (B) to[bend right=55] (A);
\draw[->, line width=.5pt, color=black!15] (B) to (C);
\draw[->, line width=.5pt, color=black!15] (A) to (D);
\draw[->, line width=.5pt, color=black!15] (D) to (C);
\draw[dash pattern=on 0pt off 2pt, line width=.8pt, line cap=round, color=black!15] (-2em,0) ++(35:.9em) arc[start angle=35, end angle=70, radius=.9em];
\draw[dash pattern=on 0pt off 2pt, line width=.8pt, line cap=round, color=black!15] (0,1em) ++(198:.9em) arc[start angle=198, end angle=170, radius=.9em];
\end{scope}
\begin{scope}[xshift=-10em, yshift=-10.5em] 
\node[circle, fill=black!15, minimum size=.3em, inner sep=0, outer sep=2pt] (A) at (-2em,0) {};
\node[circle, fill=black, minimum size=.3em, inner sep=0, outer sep=2pt] (B) at (0,1em) {};
\node[circle, fill=black!15, minimum size=.3em, inner sep=0, outer sep=2pt] (C) at (2em,0) {};
\node[circle, fill=black!15, minimum size=.3em, inner sep=0, outer sep=2pt] (D) at (0,-1em) {};
\draw[->, line width=.5pt, color=black!15] (A) to (B);
\draw[->, line width=.5pt, color=black!15] (B) to[bend right=55] (A);
\draw[->, line width=.5pt, color=black!15] (B) to (C);
\draw[->, line width=.5pt, color=black!15] (A) to (D);
\draw[->, line width=.5pt, color=black!15] (D) to (C);
\draw[dash pattern=on 0pt off 2pt, line width=.8pt, line cap=round, color=black!15] (-2em,0) ++(35:.9em) arc[start angle=35, end angle=70, radius=.9em];
\draw[dash pattern=on 0pt off 2pt, line width=.8pt, line cap=round, color=black!15] (0,1em) ++(198:.9em) arc[start angle=198, end angle=170, radius=.9em];
\end{scope}
\begin{scope}[xshift=-20em, yshift=-5.25em] 
\node[circle, fill=black!15, minimum size=.3em, inner sep=0, outer sep=2pt] (A) at (-2em,0) {};
\node[circle, fill=black, minimum size=.3em, inner sep=0, outer sep=2pt] (B) at (0,1em) {};
\node[circle, fill=black!15, minimum size=.3em, inner sep=0, outer sep=2pt] (C) at (2em,0) {};
\node[circle, fill=black!15, minimum size=.3em, inner sep=0, outer sep=2pt] (D) at (0,-1em) {};
\draw[->, line width=.5pt, color=black!15] (A) to (B);
\draw[->, line width=.5pt, color=black!15] (B) to[bend right=55] (A);
\draw[->, line width=.5pt, color=black!15] (B) to (C);
\draw[->, line width=.5pt, color=black!15] (A) to (D);
\draw[->, line width=.5pt, color=black!15] (D) to (C);
\draw[dash pattern=on 0pt off 2pt, line width=.8pt, line cap=round, color=black!15] (-2em,0) ++(35:.9em) arc[start angle=35, end angle=70, radius=.9em];
\draw[dash pattern=on 0pt off 2pt, line width=.8pt, line cap=round, color=black!15] (0,1em) ++(198:.9em) arc[start angle=198, end angle=170, radius=.9em];
\end{scope}
\end{tikzpicture}
\]
where the two parallel paths of length $2$ in the top right category $\mathcal E_\Delta (U_{\vtimes_0})$ are equal, i.e.\ the square commutes. The homotopy colimit of this diagram is the category $\mathbf \Gamma (\mathcal E_\Delta)$ of global sections of $\mathcal E_\Delta$ which is given by
\[
\begin{tikzpicture}[x=1em, y=1em]
\begin{scope}
\node[circle, fill=black, minimum size=.3em, inner sep=0, outer sep=2pt] (A) at (-4em,0) {};
\node[circle, fill=black, minimum size=.3em, inner sep=0, outer sep=2pt] (B) at (0,2em) {};
\node[circle, fill=black, minimum size=.3em, inner sep=0, outer sep=2pt] (C) at (4em,0) {};
\node[circle, fill=black, minimum size=.3em, inner sep=0, outer sep=2pt] (D) at (0,-2em) {};
\draw[->, line width=.5pt] (A) to (B);
\draw[->, line width=.5pt] (B) to[bend right=55] (A);
\draw[->, line width=.5pt] (B) to (C);
\draw[->, line width=.5pt] (A) to (D);
\draw[->, line width=.5pt] (D) to (C);
\node[font=\scriptsize] at (-4.65,0) {$\alpha$};
\node[font=\scriptsize] at (0,2.7) {$\beta$};
\node[font=\scriptsize] at (4.65,0) {$\gamma$};
\node[font=\scriptsize] at (0,-2.65) {$\delta$};
\node[font=\scriptsize] at (155:1.55em) {$q_1$};
\node[font=\scriptsize] at (30:2.7em) {$q_2$};
\node[font=\scriptsize] at (140:3.85em) {$p_0$};
\node[font=\scriptsize] at (210:2.7em) {$p_1$};
\node[font=\scriptsize] at (-30:2.75em) {$p_2$};
\draw[dash pattern=on 0pt off 2pt, line width=.8pt, line cap=round] (-4em,0) ++(35:1.5em) arc[start angle=35, end angle=68, radius=1.5em];
\draw[dash pattern=on 0pt off 2pt, line width=.8pt, line cap=round] (0,2em) ++(198:1.5em) arc[start angle=198, end angle=170, radius=1.5em];
\end{scope}
\begin{scope}[xshift=13em]
\node at (0,0) {$\mutimes_2 (\s p_2 \otimes \s p_1) = \s q_2 q_1$.};
\end{scope}
\end{tikzpicture}
\]
Note that the only ``higher'' multiplication is $\mutimes_2$ since $\Delta$ is a {\it formal} dissection (see Definition \ref{definition:DGformal}). This product comes from the orbifold disk sequence corresponding to the vertex $\vtimes_0$ (see Fig.~\ref{fig:orbifoldannulus}). In other words, the category of global sections is a DG category with zero differential. (See Section \ref{section:formal} for further results in this direction.)
\end{example}

\subsection{Lagrangian core and Weinstein sectors for orbifold surfaces}
\label{subsection:core}

For a graded {\it smooth} surface with stops, its partially wrapped Fukaya category can be computed via a cosheaf on a ($1$-dimensional) ribbon graph which may be embedded as a (singular) half-dimensional Lagrangian core into the underlying surface \cite{haidenkatzarkovkontsevich}. In our generalization to orbifold surfaces, we associate a ribbon {\it complex} $\mathbb G (\Delta)$ to any weakly admissible dissection $\Delta$ and $\mathbb G (\Delta)$ contains $2$-dimensional orbifold cells $d_x$ for $x \in \Sing (S)$. At first, this construction may not appear to be a direct analogue of Kontsevich's proposal \cite{kontsevich2} since thus far we have replaced the role of the Lagrangian core of the smooth surface by an object which contains $2$-dimensional cells and is therefore not Lagrangian (half-dimensional). In order to resolve this apparent discrepancy, we introduce the following notion of a ribbon {\it graph} for orbifold surfaces (some of whose vertices represent orbifold disks).

\begin{definition}
\label{definition:core}
Let $S$ be an orbifold surface with boundary. A {\it ribbon graph} for $S$ consists of a (coloured) ribbon graph $G$ such that $G$ embeds as a deformation retract into the underlying topological surface in such a way that for each $x \in \Sing (S)$, there is one vertex $v_x^{\times}$ of $G$ at $x$, where $^{\times}$ is viewed as a colouring of the vertex $v_x^{\times}$. (The vertex $v_x^{\times}$ plays a similar role to the vertex $\vtimes_x$ in the ribbon complex which is why we use almost indistinguishable notation, see Theorem \ref{theorem:ribbongraphinterpretation}.)

If $\mathbf S = (S, \Sigma, \eta)$ is a graded orbifold surface with stops, then a ribbon graph for $\mathbf S$ is a coloured ribbon graph with vertices of type $v^{\times}, \vodot, \vbullet, \vcirc$ with linear orders at the vertices of type $\vbullet$ (and gradings as in \S\ref{subsubsection:gradedribbon}).
\end{definition}

Note that the embedding of a ribbon graph $G$ into an orbifold surface $S$ is a Lagrangian core since as a topological surface (ignoring the orbifold structure) $S$ deformation retracts onto (the embedded) $G$.

The following theorem shows that for certain dissections $\Delta$, our ($2$-dimensional) ribbon complex $\mathbb G (\Delta)$ may be viewed as a refinement of a ribbon graph $\mathrm G (\Delta)$ for $\mathbf S$ which can indeed be regarded as a (half-dimensional) Lagrangian core of the orbifold surface $S$. The partially wrapped Fukaya category $\mathcal W (\mathbf S)$ can then be computed via a cosheaf on either object.

\begin{theorem}
\label{theorem:ribbongraphinterpretation}
Let $\mathbf S = (S, \Sigma, \eta)$ be a graded orbifold surface with stops and let $G$ be a ribbon graph for $\mathbf S$.
\begin{enumerate}
\item There exists a weakly admissible dissection $\Delta$ of $\mathbf S$ such that the ($1$-dimensional) ribbon graph $\mathrm G (\Delta)$, obtained from $\mathbb G (\Delta)$ by contracting each orbifold $2$-cell to a vertex, can be identified with $G$.

\item There exists an open cover $\mathfrak U^{\times}_\Delta$ of the ribbon graph $\mathrm G (\Delta)$ and a cosheaf $\mathcal E^{\times}_\Delta$ of A$_\infty$ categories on $\mathfrak U^{\times}_\Delta$ such that $\mathbf \Gamma (\mathcal E^{\times}_\Delta)$ is Morita equivalent to $\mathbf \Gamma (\mathcal E_\Delta)$.
In particular, we have that
\[
\mathcal W (\mathbf S) \simeq \H^0 (\tw (\mathbf \Gamma (\mathcal E^{\times}_\Delta))^\natural) \simeq \H^0 (\mathbf \Gamma (\tw (\mathcal E^{\times}_\Delta)^\natural))
\]
i.e.\ $\mathcal W (\mathbf S)$ can be viewed as the homotopy category of the global sections of the cosheaf $\tw (\mathcal E^{\times}_\Delta)^\natural$ on the open cover $\mathfrak U^{\times}_\Delta$ of the ribbon graph $G \simeq \mathrm G (\Delta)$.
\end{enumerate}
\end{theorem}

\begin{proof}
Let $\Delta'$ be a weakly admissible dissection of $\mathbf S$ such that no arc connects to more than one orbifold point. Let $\Delta$ be the weakly admissible dissection constructed from $\Delta'$ as follows. Let us denote by $\gamma_0, \dotsc, \gamma_k$ the $k + 1$ arcs in $\Delta'$ connecting to an orbifold point $x \in \Sing (S)$ and let $q_1, \dotsc, q_k$ denote the orbifold paths around $x$, where $q_i$ is an orbifold path from $\gamma_{i-1}$ to $\gamma_i$. For each $1 \leq i \leq k + 1$ let $\gamma_i^{\times}$ be an arc isotopic to a smoothening of $\gamma_i \cup \gamma_{i+1}$, where the indices are taken modulo $k + 1$, i.e.\ $\gamma^{\times}_k$ is isotopic to a smoothening of $\gamma_k \cup \gamma_0$ (see Fig.~\ref{fig:ribbongraphorbifold} for an example). We shall choose $\gamma^{\times}_i$ to be close enough to $\gamma_i \cup \gamma_{i+1}$ so as to not intersect any other arcs in $\Delta'$. Let $\Delta$ be the weakly admissible dissection obtained from $\Delta'$ by adding such arcs for all orbifold points. Then the vertices $v_0, \dotsc, v_k$ in the boundary of each $2$-cell in $\mathbb G (\Delta)$ are of type $\vcirc$ except for one special vertex, say $v_0$, of type $\vtimes$ which contains the orbifold stop at $x$. Moreover, $\val (v_i) = 3$ for all $0 \leq i \leq k$.

Now let $\mathrm G (\Delta)$ be the ribbon graph obtained by contracting each $2$-cell $d_x \in \mathbb G_2 (\Delta)$ to a vertex $v_x^{\times}$ which we shall mark by {\color{ribboncolour}$\pmb\times$}. Note that $\val (v_x^{\times}) = \# \{ v \in \partial d_x \}$. Then $\mathrm G (\Delta)$ is a ribbon graph of $\mathbf S$ and any ribbon graph of $\mathbf S$ can be obtained from some weakly admissible dissection $\Delta$ in this way.

Let $\mathfrak U^{\times}_\Delta = \{ U^{\times}_v \}_{v \in \mathbb G_0 (\Delta)}$ be the open cover of $\mathrm G (\Delta)$ whose open sets $U^{\times}_v$ are the union of $v$ and all open $1$-cells whose boundary contains $v$ (cf.\ \S\ref{subsubsection:opencover}).

Then we may define $\mathcal E^{\times}_\Delta$ on $\mathfrak U^{\times}_\Delta$ by setting
\[
\mathcal E^{\times}_\Delta (U^{\times}_v) =
\begin{cases}
\mathcal E_\Delta (\bigcup_{v \in \partial d_x} U_v) & \text{if } v = v_x^{\times} \text{ for some } x \in \Sing (S) \\
\mathcal E_\Delta (U_v) & \text{otherwise}.
\end{cases}
\]
Since the diagram $\mathcal E^{\times}_\Delta (\mathrm N (\mathfrak U^{\times}_\Delta))$, obtained by evaluating $\mathcal E^{\times}_\Delta$ on (the open sets in) the Čech nerve $\mathrm N (\mathfrak U^{\times}_\Delta)$, is obtained by computing a partial homotopy colimit of the (finer) diagram $\mathcal E_\Delta (\mathrm N (\mathfrak U_\Delta))$, their global sections are canonically isomorphic.
\end{proof}

\begin{figure}
\begin{tikzpicture}[x=1em,y=1em,decoration={markings,mark=at position 0.55 with {\arrow[black]{Stealth[length=4.8pt]}}}]
\begin{scope}
\node[font=\scriptsize,shape=circle,scale=.6] (X) at (0,0) {};
\draw[line width=0,postaction={decorate}] (90:5em) arc[start angle=90, end angle=203, radius=5em];
\draw[line width=0,postaction={decorate}] (210:5em) arc[start angle=210, end angle=250, radius=5em];
\draw[line width=0,postaction={decorate}] (295:5em) arc[start angle=295, end angle=330, radius=5em];
\draw[line width=.5pt] (0,0) circle(5em);
\draw[line width=.15em,color=stopcolour] (150:2.5em) circle(.75em);
\draw[fill=stopcolour,color=stopcolour] (30:5em) circle(.15em);
\draw[fill=stopcolour,color=stopcolour] (270:5em) circle(.15em);
\path[line width=.75pt,color=arccolour] (X) edge (330:5em);
\path[line width=.75pt,color=arccolour] (X) edge (210:5em);
\path[line width=.75pt,color=arccolour] (X) edge (90:5em);
\path[line width=.75pt,color=arccolour] (245:5em) edge[bend left=55, looseness=1.1] (295:5em);
\draw[->, line width=.5pt] (207:.7em) arc[start angle=207, end angle=93, radius=.7em];
\draw[->, line width=.5pt] (87:.7em) arc[start angle=87, end angle=-27, radius=.7em];
\node[font=\scriptsize] at (0,0) {$\times$};
\draw[fill=stopcolour,color=stopcolour] (270:.5em) circle(.15em);
\node[circle, fill=ribboncolour, minimum size=.3em, inner sep=0] (0) at (270:2.5em) {};
\node[circle, fill=ribboncolour, minimum size=.3em, inner sep=0] (1) at (150:2.5em) {};
\node[color=ribboncolour] at (30:3em) {$\bullet$};
\node[circle, fill=ribboncolour, minimum size=.3em, inner sep=0] (2) at (30:2.5em) {};
\node[color=ribboncolour] at (270:1.85em) {$\bullet$};
\node[circle, fill=ribboncolour, minimum size=.3em, inner sep=0] (3) at (270:3.9em) {};
\node[color=ribboncolour] at (270:4.4em) {$\bullet$};
\draw[line width=0, fill=ribboncolour, draw=ribboncolour, opacity=.5] (270:2.5em) -- (150:2.5em) -- (30:2.5em) -- cycle;
\draw[-, line width=.85, color=ribboncolour, line cap=round] (3) to (0) to (1) to (2) to (0);
\end{scope}
\begin{scope}[xshift=14em]
\node[font=\scriptsize,shape=circle,scale=.6] (X) at (0,0) {};
\draw[line width=0,postaction={decorate}] (90:5em) arc[start angle=90, end angle=203, radius=5em];
\draw[line width=0,postaction={decorate}] (210:5em) arc[start angle=210, end angle=250, radius=5em];
\draw[line width=0,postaction={decorate}] (295:5em) arc[start angle=295, end angle=330, radius=5em];
\draw[line width=0,postaction={decorate}] (202:5em) arc[start angle=202, end angle=214, radius=5em];
\draw[line width=0,postaction={decorate}] (210:5em) arc[start angle=210, end angle=222, radius=5em];
\draw[line width=0,postaction={decorate}] (322:5em) arc[start angle=322, end angle=334, radius=5em];
\draw[line width=0,postaction={decorate}] (330:5em) arc[start angle=330, end angle=342, radius=5em];
\draw[line width=0,postaction={decorate}] (82:5em) arc[start angle=82, end angle=94, radius=5em];
\draw[line width=0,postaction={decorate}] (90:5em) arc[start angle=90, end angle=102, radius=5em];
\draw[line width=.5pt] (0,0) circle(5em);
\draw[line width=.15em,color=stopcolour] (150:2.5em) circle(.75em);
\draw[fill=stopcolour,color=stopcolour] (30:5em) circle(.15em);
\draw[fill=stopcolour,color=stopcolour] (270:5em) circle(.15em);
\path[line width=.75pt,color=arccolour] (X) edge (330:5em);
\path[line width=.75pt,color=arccolour] (X) edge (210:5em);
\path[line width=.75pt,color=arccolour] (X) edge (90:5em);
\path[line width=.75pt,color=arccolour] (245:5em) edge[bend left=55, looseness=1.1] (295:5em);
\path[line width=.75pt,color=arccolour] (98:5em) edge[bend left=30, looseness=1.5] (202:5em);
\path[line width=.75pt,color=arccolour] (218:5em) edge[bend left=30, looseness=1.5] (322:5em);
\path[line width=.75pt,color=arccolour] (338:5em) edge[bend left=30, looseness=1.5] (442:5em);
\draw[->, line width=.5pt] (207:.7em) arc[start angle=207, end angle=93, radius=.7em];
\draw[->, line width=.5pt] (87:.7em) arc[start angle=87, end angle=-27, radius=.7em];
\node[font=\scriptsize] at (0,0) {$\times$};
\node[color=ribboncolour] at (30:3em) {$\bullet$};
\draw[fill=stopcolour,color=stopcolour] (270:.35em) circle(.12em);
\node[circle, fill=ribboncolour, minimum size=.3em, inner sep=0] (0') at (270:1.15em) {};
\node[circle, fill=ribboncolour, minimum size=.3em, inner sep=0] (1') at (150:1.15em) {};
\node[circle, fill=ribboncolour, minimum size=.3em, inner sep=0] (2') at (30:1.15em) {};
\node[circle, fill=ribboncolour, minimum size=.3em, inner sep=0] (0) at (270:2.5em) {};
\node[circle, fill=ribboncolour, minimum size=.3em, inner sep=0] (1) at (150:2.5em) {};
\node[color=ribboncolour] at (30:3em) {$\bullet$};
\node[circle, fill=ribboncolour, minimum size=.3em, inner sep=0] (2) at (30:2.5em) {};
\node[color=ribboncolour, font=\scriptsize] at (270:.7em) {$\bullet$};
\node[circle, fill=ribboncolour, minimum size=.3em, inner sep=0] (3) at (270:3.9em) {};
\node[color=ribboncolour] at (270:4.4em) {$\bullet$};
\draw[line width=0, fill=ribboncolour, draw=ribboncolour, opacity=.5] (270:1.15em) -- (150:1.15em) -- (30:1.15em) -- cycle;
\draw[-, line width=.85, color=ribboncolour, line cap=round] (3) to (0') to (1') to (2') to (0') to (0) (1') to (1) (2') to (2);
\end{scope}
\begin{scope}[xshift=28em]
\node[font=\scriptsize,shape=circle,scale=.6] (X) at (0,0) {};
\draw[line width=0,postaction={decorate}] (90:5em) arc[start angle=90, end angle=203, radius=5em];
\draw[line width=0,postaction={decorate}] (210:5em) arc[start angle=210, end angle=250, radius=5em];
\draw[line width=0,postaction={decorate}] (295:5em) arc[start angle=295, end angle=330, radius=5em];
\draw[line width=0,postaction={decorate}] (202:5em) arc[start angle=202, end angle=214, radius=5em];
\draw[line width=0,postaction={decorate}] (210:5em) arc[start angle=210, end angle=222, radius=5em];
\draw[line width=0,postaction={decorate}] (322:5em) arc[start angle=322, end angle=334, radius=5em];
\draw[line width=0,postaction={decorate}] (330:5em) arc[start angle=330, end angle=342, radius=5em];
\draw[line width=0,postaction={decorate}] (82:5em) arc[start angle=82, end angle=94, radius=5em];
\draw[line width=0,postaction={decorate}] (90:5em) arc[start angle=90, end angle=102, radius=5em];
\draw[line width=.5pt] (0,0) circle(5em);
\draw[line width=.15em,color=stopcolour] (150:2.5em) circle(.75em);
\draw[fill=stopcolour,color=stopcolour] (30:5em) circle(.15em);
\draw[fill=stopcolour,color=stopcolour] (270:5em) circle(.15em);
\path[line width=.75pt,color=arccolour] (X) edge (330:5em);
\path[line width=.75pt,color=arccolour] (X) edge (210:5em);
\path[line width=.75pt,color=arccolour] (X) edge (90:5em);
\path[line width=.75pt,color=arccolour] (245:5em) edge[bend left=55, looseness=1.1] (295:5em);
\path[line width=.75pt,color=arccolour] (98:5em) edge[bend left=30, looseness=1.5] (202:5em);
\path[line width=.75pt,color=arccolour] (218:5em) edge[bend left=30, looseness=1.5] (322:5em);
\path[line width=.75pt,color=arccolour] (338:5em) edge[bend left=30, looseness=1.5] (442:5em);
\draw[->, line width=.5pt] (207:.7em) arc[start angle=207, end angle=93, radius=.7em];
\draw[->, line width=.5pt] (87:.7em) arc[start angle=87, end angle=-27, radius=.7em];
\node[color=ribboncolour] at (30:3em) {$\bullet$};
\draw[fill=stopcolour,color=stopcolour] (270:.5em) circle(.15em);
\node[shape=circle, fill=ribboncolour, fill opacity=.3, scale=.6, inner sep=4pt] (v) at (0,0) {};
\node[color=ribboncolour] at (0,0) {$\pmb\times$};
\node[circle, fill=ribboncolour, minimum size=.3em, inner sep=0] (0) at (270:2.5em) {};
\node[circle, fill=ribboncolour, minimum size=.3em, inner sep=0] (1) at (150:2.5em) {};
\node[color=ribboncolour] at (0) {$\bullet$};
\node[circle, fill=ribboncolour, minimum size=.3em, inner sep=0] (2) at (30:2.5em) {};
\node[circle, fill=ribboncolour, minimum size=.3em, inner sep=0] (3) at (270:3.9em) {};
\node[color=ribboncolour] at (270:4.4em) {$\bullet$};
\draw[-, line width=.85, color=ribboncolour, line cap=round] (3) to (0) to (v) to (2) (v) to (1);
\end{scope}

\end{tikzpicture}
\caption{The ribbon complex $\mathbb G (\Delta')$ of Fig.~\ref{fig:orbifoldannulus} (left), the ribbon complex $\mathbb G (\Delta)$ (middle) of the dissection $\Delta$ obtained by adding arcs isotopic to the smoothenings of all three adjacent pairs of arcs in $\Delta'$ connecting to the orbifold point $x$, and the associated ribbon {\it graph} $\mathrm G (\Delta)$ (right) obtained by contracting the $2$-cell $d_x$ in $\mathbb G_2 (\Delta)$ to a single vertex $v_x^{\times} \in \mathrm G_0 (\Delta)$ denoted by {\color{ribboncolour}$\pmb\times$}.}
\label{fig:ribbongraphorbifold}
\end{figure}
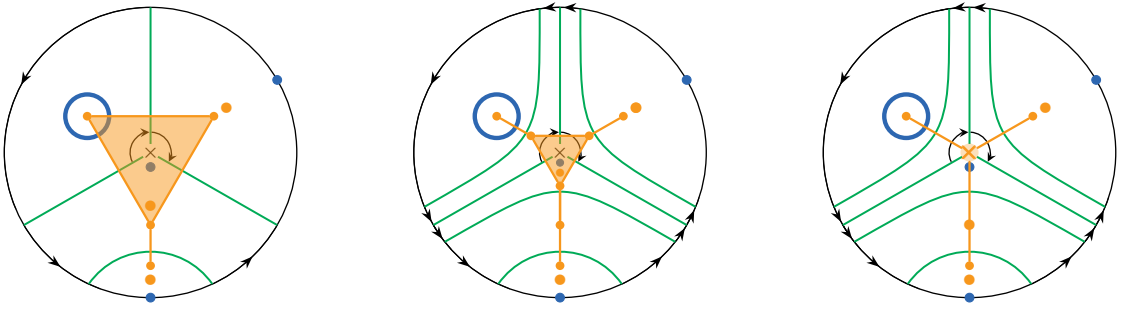

In terms of the notion of Weinstein sectors as in \cite{ganatrapardonshende1,ganatrapardonshende2}, the ribbon {\it graph} $\mathrm G (\Delta)$ of a graded orbifold surface with stops $\mathbf S$ can be covered by open sets $U_v^\times$ where each open set now corresponds to a Weinstein sector of a certain type. Let us assume that $\mathrm G (\Delta)$ contains no vertices of type $\vcirc$ which can be achieved by iterated edge contraction (cf.~\S\ref{subsection:ribbongraph} especially Fig.~\ref{fig:edgecontraction}). Under this assumption, the vertices of $\mathrm G (\Delta)$ are of type $\vbullet$, $\vodot$ and $v^\times$. For each vertex $v$ we have an open set $U^\times_v$ of $\mathrm G (\Delta)$. The A$_\infty$ category $\mathcal E_\Delta^\times (U_v^\times)$ can be viewed as the partially wrapped Fukaya category of the disk $\mathbf S_v$ obtained from the part of $\mathbf S$ given by (an open neighbourhood of) the polygon corresponding to the vertex $v$, together with extra boundary part with stops ensuring that $\mathcal W (\mathbf S_v)$ only records the local data of $\mathbf S$. Let $n = \val (v)$, then we have
\[
\mathcal E_\Delta^\times (U_v^\times) \simeq A
\]
where $A = \Bbbk Q / I$ is the path algebra of a linearly (or cyclically) oriented quiver $Q$ of type $\mathrm A_{n-1}$ (case $\vbullet$), type $\widetilde{\mathrm A}_{n-1}$ (case $\vodot$) or type $\mathrm D_{n+1}$ (case $v^\times$) with $I$ generated by all paths of length $2$ (cf.\ Corollary \ref{corollary:typeD}). Note also that besides the arcs dual to the edges of the ribbon graph $\mathrm G (\Delta)$ we may add two arcs connecting to the orbifold point to generate $\mathcal W (\mathbf S_v)$ to obtain the type $\mathrm D$ quiver. We then obtain triangulated equivalences
\[
\mathcal W (\mathbf S_v) \simeq \H^0 (\mathcal E_\Delta^\times (U_v^\times)) \simeq \per (A).
\]
See Fig.~\ref{fig:weinstein} for an illustration of the type $\mathrm D_{n+1}$ case. Type $\mathrm A_{n-1}$ and type $\widetilde{\mathrm A}_{n-1}$ can be described similarly. A type $\mathrm A_{n-1}$ sector contains one boundary stop and $n-1$ stops are added for $\mathbf S_v$. A type $\widetilde{\mathrm A}_{n-1}$ sector contains one full boundary stop (a circle) and $n$ stops are added for $\mathbf S_v$.

\begin{figure}
\begin{tikzpicture}[x=1.3em,y=1.3em,decoration={markings,mark=at position 0.99 with {\arrow[black]{Stealth[length=4.8pt]}}}, scale=.95]
\draw[line width=0pt, draw opacity=0, fill opacity=.1, fill=black] (-8,0) ++(205:1 and .25) arc[start angle=205, end angle=245, x radius=1, y radius=.25] to[bend left=90, looseness=1.8] ($(-8,0)+(275:1 and .25)$) arc[start angle=275, end angle=318, x radius=1, y radius=.25] to[bend left=89, looseness=1.95] (199:4 and 1) arc[start angle=199, end angle=228.5, x radius=4, y radius=1] to[bend left=90, looseness=1.8] (242:4 and 1) arc[start angle=242, end angle=293, x radius=4, y radius=1] to[out=90, in=-10, looseness=.8] (.65,2.95) to[bend left=30] (-1.45,4.7) to[in=92, out=132, looseness=1.07] ($(-8,0)+(205:1 and .25)$) -- cycle;
\draw[line width=0pt, draw opacity=0, fill opacity=.1, fill=black] (97:4 and 1) to[out=90, in=170, looseness=.73] (.65,2.95) to[bend left=26] (-1.35,4.4) to[out=300, in=90, looseness=1.1] (113:4 and 1);
\node[font=\scriptsize] at (-7,5) {$\times$};
\node[font=\scriptsize] at (-4.5,3.5) {$\times$};
\draw[line width=.5pt] (0,0) ++ (180:4 and 1) arc[start angle=180,end angle=360,x radius=4,y radius=1] to[out=90,in=190,looseness=1] ++(2.5,4) arc[start angle=-90,end angle=90,x radius=.25,y radius=1] to[out=170,in=290,looseness=1] ++(-1,1) arc[start angle=20,end angle=135,radius=5] to[out=225,in=0,looseness=1] (-11,6) ++ (0,-2) to[out=0,in=90] (-9,0) arc[start angle=180,end angle=360,x radius=1,y radius=.25] arc[start angle=180,end angle=0,radius=1.5];
\begin{scope}
\draw[line width=.5pt,line cap=round] (5.5,7) ++(200:5) ++(0:2.5) arc[start angle=0,end angle=90,radius=2.5];
\draw[line width=.5pt,line cap=round] (5.5,7) ++(200:5) ++(0:2.5) ++(90:2.5) ++(280:2.5) arc[start angle=280,end angle=170,radius=2.5];
\end{scope}
\begin{scope}[shift={(-2.35,-2.35)}]
\draw[line width=.5pt,line cap=round] (5.5,7) ++(200:5) ++(0:2.5) arc[start angle=0,end angle=90,radius=2.5];
\draw[line width=.5pt,line cap=round] (5.5,7) ++(200:5) ++(0:2.5) ++(90:2.5) ++(280:2.5) arc[start angle=280,end angle=170,radius=2.5];
\end{scope}
\draw[line width=.6pt,dash pattern=on 0pt off 1.5pt,line cap=round] (0,0) ++ (180:4 and 1) arc[start angle=180,end angle=0,x radius=4,y radius=1] (6.5,4) arc[start angle=270,end angle=90,x radius=0.25,y radius=1] (-9,0) arc[start angle=180,end angle=0,x radius=1,y radius=.25];
\draw[line width=.25em,color=stopcolour] (-11,6) arc[start angle=90,end angle=450,x radius=.25,y radius=1];
\draw[fill=stopcolour, color=stopcolour] (85:4 and 1) circle(.15em);
\draw[fill=stopcolour, color=stopcolour] (235:4 and 1) circle(.15em);
\draw[fill=stopcolour, color=stopcolour] (325:4 and 1) circle(.15em);
\draw[fill=stopcolour, color=stopcolour] (-8,0) ++(260:1 and .25) circle(.15em);
\draw[line width=.75pt, color=arccolour] (-8,0) ++(235:1 and .25) to[bend left=90, looseness=1.8] ($(-8,0)+(285:1 and .25)$);
\draw[line width=.75pt, color=arccolour] (-8,0) ++(305:1 and .25) to[bend left=90, looseness=2] (205:4 and 1);
\draw[line width=.75pt, color=arccolour] (-8,0) ++(220:1 and .25) to[out=90, in=120, looseness=1.2] (-1.1,4);
\draw[line width=.75pt, color=arccolour!70!black] (-1.1,4) to[out=300, in=90] (110:4 and 1);
\draw[line width=.75pt, color=arccolour] (225:4 and 1) to[bend left=90, looseness=1.8] (245:4 and 1);
\draw[line width=.75pt, color=arccolour] (290:4 and 1) to[out=90, in=-10, looseness=.8] (.4,3);
\draw[line width=.75pt, color=arccolour!70!black] (100:4 and 1) to[out=90, in=170, looseness=.7] (.4,3);
\begin{scope}[xshift=17em, yshift=6em, x=1em, y=1em]
\draw[line width=0pt, draw opacity=0, fill opacity=.1, fill=black] (-17:4) to[bend left=57, looseness=1.3] (17:4) arc[start angle=17, end angle=72-17, radius=4] to[bend left=57, looseness=1.3] (72+17:4) arc[start angle=72+17, end angle=144-17, radius=4] to[bend left=57, looseness=1.3] (144+17:4) arc[start angle=144+17, end angle=216-17, radius=4] to[bend left=57, looseness=1.3] (216+17:4) arc[start angle=216+17, end angle=288-17, radius=4] to[bend left=57, looseness=1.3] (288+17:4) arc[start angle=288+17, end angle=360-17, radius=4] -- cycle;
\node[font=\small] at (-4,4) {$\mathbf S_v$};
\node at (0,0) {$\times$};
\foreach \a in {0,72,144,216,288} {
\draw[line width=.5pt, line cap=round, opacity=.2] (\a-17:4) arc[start angle=\a-17, end angle=\a+17, radius=4];
\draw[fill=stopcolour, color=stopcolour] (\a:4) circle(.15em);
\draw[fill=stopcolour, color=white, opacity=.5, line width=0] (\a:4) circle(.18em);
\draw[line width=.75pt, color=arccolour] (\a-20:4) to[bend left=60, looseness=1.3] (\a+20:4);
\draw[line width=.5pt, line cap=round] (\a+17:4) arc[start angle=\a+17, end angle=\a+72-17, radius=4];
}
\end{scope}
\end{tikzpicture}
\caption{A Weinstein sector (shaded) of type $\mathrm D_{n+1}$ on the orbifold surface $\mathbf S$ of Fig.~\ref{fig:orbifoldsurface} containing $n$ arcs and one orbifold point (left). On the right the associated surface $\mathbf S_v$ containing extra parts of the boundary with stops (light).}
\label{fig:weinstein}
\end{figure}
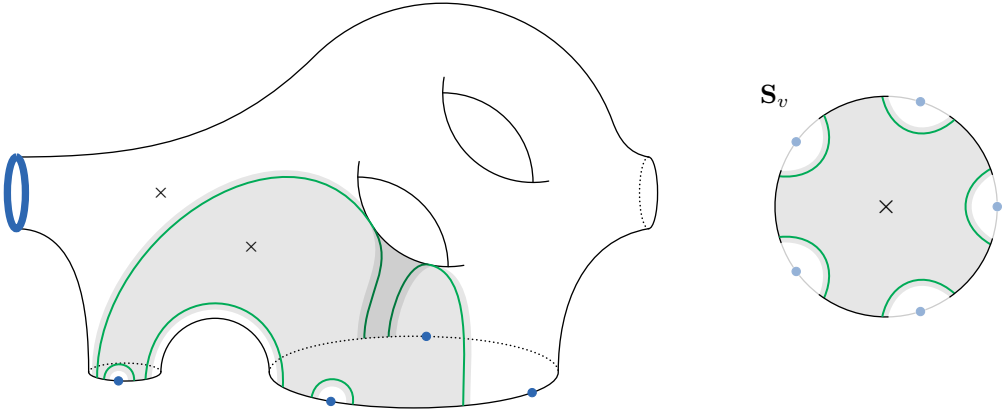

\section{Explicit A$_\infty$ categories from admissible dissections}
\label{section:ainfinity}

Given an admissible dissection $\Delta$ (see Definition \ref{definition:admissible}) we now construct an explicit $\Bbbk$-linear A$_\infty$ category $\mathbf A_\Delta$ whose associated triangulated category $\H^0 (\tw (\mathbf A_\Delta)^\natural)$ is equivalent to $\mathcal W (\mathbf S)$. In other words, we show how to compute the homotopy colimit
\[
\mathcal W (\mathbf S) = \mathbf \Gamma (\mathcal E_\Delta) = \hocolim (\mathrm N (\mathcal E_\Delta (\mathfrak U_\Delta)))
\]
explicitly using any admissible dissection of $\mathbf S$. (See also \S\ref{subsection:notstrongDG} for an extension of this definition to {\it DG dissections} which are only weakly admissible.)

For admissible dissections, the higher products of $\mathbf A_\Delta$ are essentially the products $\mucirc_n$ and $\mutimes_{n-1}$ which we already saw in the definition of $\mathcal E_\Delta (U_v)$ for $U_v \in \mathfrak U_\Delta$. The admissibility of $\Delta$ then implies that the additional higher products one needs to describe $\mathcal W (\mathbf S)$ are obtained from linearly extending $\mucirc$ with respect to the associative composition of morphisms in the first and last entries.

The explicit nature of the higher products on $\mathbf A_\Delta$ allows us in Section \ref{section:formal} to give rather straightforward conditions on $\Delta$ such that $\mathbf A_\Delta$ is a {\it formal} A$_\infty$ category, i.e.\ A$_\infty$-quasi-equivalent to its cohomology (Definition \ref{definition:formal}). This gives rise to explicit derived equivalences between associative algebras arising from what we call {\it formal dissections}.

\subsection{Objects and morphisms}
\label{subsection:objectsandmorphisms}

Let $\mathbf S = (S, \Sigma, \eta)$ be a graded orbifold surface with stops and let $\Delta$ be an admissible dissection. Fix a coefficient field $\Bbbk$ and let $\mathbf A_\Delta$ be the $\Bbbk$-linear A$_\infty$ category whose objects are the arcs in $\Delta$.

Given two arcs $\gamma$ and $\gamma'$ in $\Delta$, a $\Bbbk$-linear basis of morphisms from $\gamma$ to $\gamma'$ is given by (boundary and orbifold) paths from $\gamma$ to $\gamma'$ together with the identity morphism (denoted by $\id_{\gamma}$) in case $\gamma = \gamma'$. 

\subsection{Composition and higher products}
\label{subsection:higher}

We define three types of A$_\infty$ products on $\mathbf A_\Delta$ denoted by
\begin{itemize}
\item $\mubar_2$ for the (associative) concatenation of paths
\item $\mucirc_n$ for $n \geq 2$, where $\circ$ stands for a smooth disk sequence
\item $\mutimes_n$ for $n \geq 1$, where $\times$ stands for an orbifold disk sequence.
\end{itemize}

\subsubsection{Composition}

Let $p, q$ be two paths defining morphisms from $\gamma$ to $\gamma'$ and $\gamma'$ to $\gamma''$, respectively. If $p$ and $q$ can be concatenated, because they are both consecutive boundary paths or consecutive angles around an orbifold point, we define $\mubar_2 (\s q \otimes \s p) = (-1)^{|p|} \s q p$ where $q p$ denotes the concatenation of $p$ and $q$. In particular, we have
\[
(-1)^{|p|} \mubar_2 (\s \, \id_{\gamma'} \otimes \s p) = \s p = \mubar_2 (\s p \otimes \s \, \id_{\gamma}).
\]

\subsubsection{A$_\infty$ products}
\label{subsection:Aproduct}

In addition to the above composition, the remaining A$_\infty$ products are obtained from {\it smooth disk sequences} and {\it orbifold disk sequences}. The former were already introduced in \cite{haidenkatzarkovkontsevich} (under the name {\it disk sequence}) to define $n$-ary higher operations for $n \geq 2$. The latter higher products are defined for $n \geq 1$.

\begin{definition}
\label{definition:disksequence}
Let $\Delta$ be a weakly admissible dissection of $\mathbf S = (S, \Sigma, \eta)$. A {\it smooth disk sequence} of length $n$ is an $n$-gon dual to the vertex of type $\vcirc$ obtained from contracting all edges in a subtree of $\mathbb G (\Delta)$ all of whose vertices are of type $\vcirc$. (In other words, a smooth disk sequence is an $n$-gon in $S$ cut out by a subset of the arcs in $\Delta$ containing neither orbifold points or orbifold stops in its interior nor any (full) boundary stops in its boundary.) We may denote a smooth disk sequence of length $n$ as 
\[
\gamma_0 \overset{p_1}{\frown} \gamma_1 \overset{p_2}{\frown} \dotsb\overset{p_{n-1}}{\frown} \gamma_{n-1} \overset{p_{n}}{\frown} \gamma_0
\]
where $\gamma_0, p_1, \gamma_1, p_2, \dotsc, \gamma_{n-1}, p_{n}$ are the consecutive arcs and paths in the boundary of the $n$-gon. In particular, they are cyclic. 

An {\it orbifold disk sequence} of length $n$ is an $n$-gon in $S$ which is dual to the vertex of type $\vtimes$ obtained from contracting all edges in a subtree of $\mathbb G (\Delta)$ one of whose vertices is of type $\vtimes$ and all of whose other vertices are of type $\vcirc$. We may denote an orbifold disk sequence of length $n$ as 
\[
\gamma_0 \overset{p_1}{\smile} \gamma_1 \overset{p_2}{\smile} \dotsb\overset{p_{n-1}}{\smile} \gamma_{n-1} 
\]
where $\gamma_0, p_1, \gamma_1, p_2, \dotsc, \gamma_{n-2}, p_{n-1}, \gamma_{n-1}$ are the consecutive arcs and paths in the boundary of the $n$-gon such that the orbifold stop lies between $\gamma_0$ and $\gamma_{n-1}$.

An example of a smooth disk sequence (on the left) and an orbifold disk sequence (on the right) is given in Fig.~\ref{fig:smoothorbifolddisksequence}.
\end{definition}

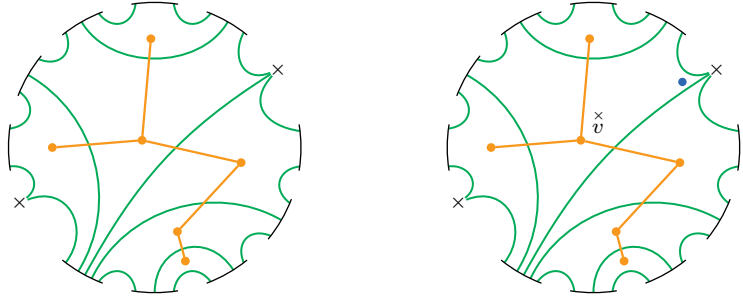
\begin{figure}
\begin{tikzpicture}[x=1em,y=1em]
\begin{scope}[scale=.5]
\node[font=\scriptsize] at (32.5:10) {$\times$};
\node[font=\scriptsize,shape=circle,scale=.6] (X) at (32.5:10) {};
\node[font=\scriptsize,shape=circle,scale=.6] (Y) at (202.5:10) {};
\node[font=\scriptsize] at (202.5:10) {$\times$};
\foreach \s in {82.5,112.5, ...,172.5} {%
\draw[line width=.75pt, color=arccolour, looseness=1.5] (\s:10) to[in=\s+150, out=\s+195] (\s-15:10);
}
\draw[line width=.75pt, color=arccolour, looseness=1.5] (Y) to[in=52.5, out=22.5] (202.5+30:10);
\draw[line width=.75pt, color=arccolour, looseness=1.5] (202.5-15:10) to[in=37.5, out=7.5] (Y);
\draw[line width=.75pt, color=arccolour, looseness=1.5] (X) to[in=180+52.5+10, out=180+32.5-10] (52.5:10);
\draw[line width=.75pt, color=arccolour, looseness=1.5] (6.5:10) to[in=180+32.5+15, out=180+6.5-15] (X);
\draw[line width=.75pt, color=arccolour, looseness=1] (241.5:10) to[out=241.5+180, in=32.5+180] (X);
\foreach \s in {262.5,292.5,...,352.5}{%
\draw[line width=.75pt, color=arccolour, looseness=1.5] (\s:10) to[in=\s+150, out=\s+195] (\s-15:10);
}
\foreach \a/\b in {120/60,150/238.5,244.5/330} {
\draw[line width=.75pt, color=arccolour, looseness=1] (\a:10) to[out=\a+180, in=\b+180] (\b:10);
}
\foreach \a/\b in {270/302.5} {
\draw[line width=.75pt, color=arccolour, looseness=1.5] (\a:10) to[out=\a+180, in=\b+180] (\b:10);
}
\foreach \s in {52.5,82.5,...,172.5} {%
\draw[line width=.5pt, line cap=round] (\s-1.5:10) arc[start angle=\s-1.5, end angle=\s+16.5, radius=10em];
}
\foreach \s in {232.5,262.5,...,352.5} {%
\draw[line width=.5pt, line cap=round] (\s-1.5:10) arc[start angle=\s-1.5, end angle=\s+16.5, radius=10em];
}
\node[circle, fill=ribboncolour, minimum size=.3em, inner sep=0] (A1) at (150:1) {};
\node[circle, fill=ribboncolour, minimum size=.3em, inner sep=0] (A2) at (180:7) {};
\node[circle, fill=ribboncolour, minimum size=.3em, inner sep=0] (A3) at (-10:6) {};
\node[circle, fill=ribboncolour, minimum size=.3em, inner sep=0] (A4) at (92:7.5) {};
\node[circle, fill=ribboncolour, minimum size=.3em, inner sep=0] (A5) at (-75:6) {};
\node[circle, fill=ribboncolour, minimum size=.3em, inner sep=0] (A6) at (-75:8.1) {};
\draw[-, line width=.85, color=ribboncolour, line cap=round]  (A2) to (A1) to (A3);
\draw[-, line width=.85, color=ribboncolour, line cap=round]  (A1) to (A4); 
\draw[-, line width=.85, color=ribboncolour, line cap=round]  (A3) to (A5) to (A6);
\end{scope}
\begin{scope}[scale=.5, xshift=30em]
\node[font=\scriptsize] at (32.5:10) {$\times$};
\draw[fill=stopcolour,color=stopcolour] ($(32.5:10)+(200:2.5)$) circle(.25em);
\node[font=\scriptsize,shape=circle,scale=.6] (X) at (32.5:10) {};
\node[font=\scriptsize,shape=circle,scale=.6] (Y) at (202.5:10) {};
\node[font=\scriptsize] at (202.5:10) {$\times$};
\foreach \s in {82.5,112.5, ...,172.5} {%
\draw[line width=.75pt, color=arccolour, looseness=1.5] (\s:10) to[in=\s+150, out=\s+195] (\s-15:10);
}
\draw[line width=.75pt, color=arccolour, looseness=1.5] (Y) to[in=52.5, out=22.5] (202.5+30:10);
\draw[line width=.75pt, color=arccolour, looseness=1.5] (202.5-15:10) to[in=37.5, out=7.5] (Y);
\draw[line width=.75pt, color=arccolour, looseness=1.5] (X) to[in=180+52.5+10, out=180+32.5-10] (52.5:10);
\draw[line width=.75pt, color=arccolour, looseness=1.5] (6.5:10) to[in=180+32.5+15, out=180+6.5-15] (X);
\draw[line width=.75pt, color=arccolour, looseness=1] (241.5:10) to[out=241.5+180, in=32.5+180] (X);
\foreach \s in {262.5,292.5,...,352.5}{%
\draw[line width=.75pt, color=arccolour, looseness=1.5] (\s:10) to[in=\s+150, out=\s+195] (\s-15:10);
}
\foreach \a/\b in {120/60,150/238.5,244.5/330} {
\draw[line width=.75pt, color=arccolour, looseness=1] (\a:10) to[out=\a+180, in=\b+180] (\b:10);
}
\foreach \a/\b in {270/302.5} {
\draw[line width=.75pt, color=arccolour, looseness=1.5] (\a:10) to[out=\a+180, in=\b+180] (\b:10);
}
\foreach \s in {52.5,82.5,...,172.5} {%
\draw[line width=.5pt, line cap=round] (\s-1.5:10) arc[start angle=\s-1.5, end angle=\s+16.5, radius=10em];
}
\foreach \s in {232.5,262.5,...,352.5} {%
\draw[line width=.5pt, line cap=round] (\s-1.5:10) arc[start angle=\s-1.5, end angle=\s+16.5, radius=10em];
}
\node[circle, fill=ribboncolour, minimum size=.3em, inner sep=0] (A1) at (150:1) {};
\node[circle, fill=ribboncolour, minimum size=.3em, inner sep=0] (A2) at (180:7) {};
\node[circle, fill=ribboncolour, minimum size=.3em, inner sep=0] (A3) at (-10:6) {};
\node[circle, fill=ribboncolour, minimum size=.3em, inner sep=0] (A4) at (92:7.5) {};
\node[circle, fill=ribboncolour, minimum size=.3em, inner sep=0] (A5) at (-75:6) {};
\node[circle, fill=ribboncolour, minimum size=.3em, inner sep=0] (A6) at (-75:8.1) {};
\node[font=\scriptsize] at (80:1.8em) {$\vtimes$};
\draw[-, line width=.85, color=ribboncolour, line cap=round]  (A2) to (A1) to (A3);
\draw[-, line width=.85, color=ribboncolour, line cap=round]  (A1) to (A4); 
\draw[-, line width=.85, color=ribboncolour, line cap=round]  (A3) to (A5) to (A6);
\end{scope}
\end{tikzpicture}
\caption{Subtrees of the ribbon graphs of two dissections which after contracting the edges (dually, removing the internal arcs) give a smooth disk sequence (left) and an orbifold disk sequence (right).}
\label{fig:smoothorbifolddisksequence}
\end{figure}

Let $\gamma_0 \overset{p_1}{\frown} \gamma_1 \overset{p_2}{\frown} \dotsb\overset{p_{n-1}}{\frown} \gamma_{n-1} \overset{p_{n}}{\frown} \gamma_0$ be a smooth disk sequence of length $n$. Then we define
\begin{align}
\label{align:smoothdiskproduct}
\mucirc_n (\s p_i \otimes \dotsb \otimes \s p_1 \otimes \s p_n \otimes \dotsb \otimes \s p_{i+1}) = \s \, \id_{\gamma_i}
\end{align}
for all $0 \leq i < n$. Let $\gamma_0 \overset{p_1}{\smile} \gamma_1 \overset{p_2}{\smile} \dotsb\overset{p_{n-1}}{\smile} \gamma_{n-1} $ be an orbifold disk sequence of length $n$. Let $q$ be the orbifold path from $\gamma_0$ to $\gamma_{n-1}$ at the orbifold point $x$. Denote by $r$ the unique maximal orbifold path at $x$ so that $r = q'' q q'$, where $q', q''$ are orbifold subpaths (possibly trivial) of $p$. Then 
\begin{align}\label{align:orbifolddiskproduct}
\mutimes_{n-1} ( \s p_{n-1} \otimes \dotsb\otimes \s p_1) = (-1)^{|q''|} \s q.
\end{align}

\begin{remark}
\label{remark:smoothorbifolddisklength12}
\begin{enumerate}
\item If $\gamma_0 \overset{p_1}{\frown} \dotsb\overset{p_{n-1}}{\frown} \gamma_{n-1} \overset{p_n}{\frown} \gamma_0
$ is a smooth disk sequence and $p'$ can be concatenated nontrivially with the path $p_{i+1}$, then
\begin{align}\label{align:signmucircle}
\mucirc_n (\s p_i \otimes \dotsb \otimes \s p_1 \otimes \s p_n \otimes \dotsb \otimes \s p_{i+1}p') &= (-1)^{|p'|} \s p'.
\end{align}
Similarly, if $p'' p_i$ is the (nontrivial) concatenation of $p_i$ with a path $p''$ then
\begin{align*}
\mucirc_n (\s p'' p_i \otimes \dotsb \otimes \s p_1 \otimes \s p_n \otimes \dotsb \otimes \s p_{i+1}) &= \s p''.
\end{align*}

\item An orbifold disk sequence $\gamma_0 \overset{p_1}{\smile} \gamma_1 $ of length $2$ produces  a differential
\[
\mutimes_1 (\s p_1) = \s q
\]
as illustrated on the left in Fig.~\ref{fig:differential}.
\item \label{remark:item-iii} Let $\gamma_0 \overset{p_1}{\frown} \gamma_1 \overset{p_2}{\frown}  \gamma_0$ be a smooth disk sequence of length $2$ as illustrated on the right of Fig.~\ref{fig:differential}.  Then  $\gamma_1$ and $\gamma_2$ are isotopic and  
\[
\mucirc_2(\s p_2 \otimes \s p_1) = \s \, \id_{\gamma_0} \quad \text{and} \quad \mucirc_2(\s p_1 \otimes \s p_2) = \s \, \id_{\gamma_1}.
\]
Using this product $\mucirc_2$ we may show that $\gamma_1$ is isomorphic to $\gamma_2$ in $\H^0(\tw(\mathbf A_\Delta))$. That is, we may remove one of the two arcs $\gamma_1$ and $\gamma_2$ to obtain a new admissible dissection $\Delta'$ such that the natural inclusion $\mathbf A_{\Delta'} \subset \mathbf A_{\Delta}$ is a Morita equivalence. Conversely, we may also add more isotopic arcs to any dissection to ensure that the arcs contained in the boundary of an $n$-gon are all different, compare \S\ref{subsubsection:loops}. Isotopic arcs forming a $2$-gon as on the right of Fig.~\ref{fig:differential} are isomorphic in $\H^0(\tw(\mathbf A_\Delta))$. Note, however, that the objects corresponding to two isotopic arcs with an orbifold stop in between them are non-isomorphic and such a configuration gives rise to a differential in $\mathbf A_\Delta$ as described in {\itemii}.
\end{enumerate}
\end{remark}

\begin{figure}[ht]
\begin{tikzpicture}[x=1em,y=1em,decoration={markings,mark=at position 0.55 with {\arrow[black]{Stealth[length=4.8pt]}}}]
\begin{scope}
\draw[line width=.5pt,postaction={decorate}] (1,0) -- (-1,0);
\draw[line width=.5pt,postaction={decorate}] (-1,-4) -- (1,-4);
\draw[line width=.75pt, color=arccolour, line cap=round] (-1,0) -- (-1,-4);
\draw[line width=.75pt, color=arccolour, line cap=round] (1,0) -- (1,-4);
\node[font=\scriptsize,left=-.3ex] at (-1.2,-2) {$\gamma_0$};
\node[font=\scriptsize,right=-.3ex] at (1.2,-2) {$\gamma_1$};
\node[font=\scriptsize] at (0,-4.6) {$p_1$};
\node[font=\scriptsize] at (0,.6) {$p_2$};
\node[font=\small] at (3.5,-2) {$\leftrightarrow $};
\node[font=\small] at (9.5,-1){$\mucirc_2(\s p_2 \otimes \s p_1) = \s \, \id_{\gamma_0}$};
\node[font=\small] at (9.5,-3){$\mucirc_2(\s p_1 \otimes \s p_2) = \s \, \id_{\gamma_1}$};
\end{scope}
\begin{scope}[xshift=-8em]
\draw[line width=.5pt,postaction={decorate}] ($(252:4.5)+(-1,0)$) -- ($(288:4.5)+(1,0)$);
\node[font=\scriptsize] at (0,0) {$\times$};
\draw[->, line width=.5pt] (249:1.5em) arc[start angle=249, end angle=-69, radius=1.5em];
\node[font=\scriptsize,shape=circle,scale=.6] (X) at (0,0) {};
\draw[line width=.75pt, color=arccolour, line cap=round] (X) -- (252:4.5);
\draw[line width=.75pt, color=arccolour, line cap=round] (X) -- (288:4.5);
\draw[line width=.75pt, color=arccolour, line cap=round] (X) -- (190:2em);
\draw[dash pattern=on 0pt off 1.3pt, line width=.8pt, line cap=round, color=arccolour] (190:2.1em) -- (190:2.5em);
\draw[line width=.75pt, color=arccolour, line cap=round] (X) -- (128:2em);
\draw[dash pattern=on 0pt off 1.3pt, line width=.8pt, line cap=round, color=arccolour] (128:2.1em) -- (128:2.5em);
\draw[line width=.75pt, color=arccolour, line cap=round] (X) -- (350:2em);
\draw[dash pattern=on 0pt off 1.3pt, line width=.8pt, line cap=round, color=arccolour] (350:2.1em) -- (350:2.5em);
\node[font=\scriptsize] at (79:.7) {.};
\node[font=\scriptsize] at (59:.7) {.};
\node[font=\scriptsize] at (39:.7) {.};
\node[font=\scriptsize,left=-.3ex] at (252:3) {$\gamma_0$};
\node[font=\scriptsize,right=-.3ex] at (288:3) {$\gamma_1$};
\node[font=\scriptsize,color=stopcolour] at (270:1em) {$\bullet$};
\node[font=\scriptsize] at (270:5) {$p_1$};
\node[font=\scriptsize] at (90:2) {$q$};
\node[font=\small] at (-7,-2) {$\mutimes_1 (\s p_1) = \s q \quad \leftrightarrow$};
\end{scope}
\end{tikzpicture}
\caption{An orbifold disk sequence of length $2$ contributing to the differential (left) and a smooth disk sequence of length $2$, where the paths $p_1$ and $p_2$ can be boundary paths or orbifold paths (right).}
\label{fig:differential}
\end{figure}

\begin{proposition}
\label{proposition:ainfinity}
Let $\Delta$ be an admissible dissection of $\mathbf S = (S, \Sigma, \eta)$. Then equipped with the (higher) products $\mubar_2$, $\mucirc_{\geq 2}$ and $\mutimes_{\geq 1}$ the category $\mathbf A_\Delta$ is a strictly unital A$_\infty$ category. 
\end{proposition}

\begin{figure}
\begin{tikzpicture}[x=1em,y=1em,decoration={markings,mark=at position 0.55 with {\arrow[black]{Stealth[length=4.8pt]}}}]
\begin{scope}
\node[font=\footnotesize,shape=circle,scale=.6] (X) at (90:4em) {};
\node[font=\footnotesize] at (90:4em) {$\times$};
\node[font=\scriptsize,color=stopcolour] at ($(90:4em)+(240:.75em)$) {$\bullet$};
\draw[line width=.75pt, color=arccolour, line cap=round] ($(150:4em)+(30:1em)$) to (X);
\draw[line width=.75pt, color=arccolour, line cap=round] ($(150:4em)+(-90:1em)$) to ($(210:4em)+(90:1em)$);
\draw[line width=.75pt, color=arccolour, line cap=round] ($(210:4em)+(-30:1em)$) to ($(270:4em)+(150:1.5em)$);
\draw[line width=.75pt, color=arccolour, line cap=round] ($(270:4em)+(30:1.5em)$) to ($(-30:4em)+(210:1em)$);
\draw[line width=.75pt, color=arccolour, line cap=round] ($(-30:4em)+(90:1em)$) to ($(30:4em)+(270:1em)$);
\draw[line width=.75pt, color=arccolour, line cap=round] ($(30:4em)+(150:1em)$) to (X);
\draw[line width=.75pt, color=arccolour, line cap=round] ($(270:4em)+(90:1.5em)$) to (X);
\draw[->, line width=.5pt] ($(90:4em)+(207:1em)$) arc[start angle=207, end angle=-27, radius=1em];
\node[font=\scriptsize] at (90:5.5em) {$q$};
\draw[->, line width=.5pt] ($(90:4em)+(-33:1em)$) arc[start angle=-33, end angle=-87, radius=1em];
\node[font=\scriptsize] at ($(90:4em)+(-60:1.6em)$) {$r$};
\draw[->, line width=.5pt] ($(150:4em)+(30:1em)$) arc[start angle=30, end angle=-90, radius=1em];
\node[font=\scriptsize] at (150:3.8em) {$p_1$};
\draw[->, line width=.5pt] ($(210:4em)+(90:1em)$) arc[start angle=90, end angle=-30, radius=1em];
\node[font=\scriptsize] at (207:4em) {$.$};
\node[font=\scriptsize] at (210:4em) {$.$};
\node[font=\scriptsize] at (213:4em) {$.$};
\draw[->, line width=.5pt] ($(270:4em)+(150:1.5em)$) arc[start angle=150, end angle=90, radius=1.5em];
\node[font=\scriptsize] at (258:3.8em) {$p_j$};
\draw[->, line width=.5pt] ($(270:4em)+(90:1.5em)$) arc[start angle=90, end angle=30, radius=1.5em];
\node[font=\scriptsize] at (283:3.8em) {$p_{j+1}$};
\draw[->, line width=.5pt] ($(-30:4em)+(210:1em)$) arc[start angle=210, end angle=90, radius=1em];
\node[font=\scriptsize] at (-27:4em) {$.$};
\node[font=\scriptsize] at (-30:4em) {$.$};
\node[font=\scriptsize] at (-33:4em) {$.$};
\draw[->, line width=.5pt] ($(30:4em)+(270:1em)$) arc[start angle=270, end angle=150, radius=1em];
\node[font=\scriptsize] at (30:3.8em) {$p_n$};
\end{scope}
\begin{scope}[xshift=12em,yshift=.7em]
\node[font=\footnotesize] at (0,0) {$\s p_n \otimes \dotsb \otimes \s p_{j+1} \otimes \s p_j \otimes \dotsb \otimes \s p_1$};
\node[font=\footnotesize] at (0,-.4) {$\undergroup{\phantom{\s p_{j+1} \otimes \s p_j}}$};
\node[font=\footnotesize] at (0,-1.3) {${\mubar_2}$};
\node[font=\footnotesize] at (0,-1.7) {$\undergroup{\phantom{\s p_n \otimes \dotsb \otimes \s p_{j+1} \otimes \s p_j \otimes \dotsb \otimes \s p_1}}$};
\node[font=\footnotesize] at (0,-2.7) {${\mutimes_{n-1}}$};
\node[font=\footnotesize] at (3.7,.4) {$\overgroup{\phantom{\s p_{j} \otimes \dotsb \otimes \s p_1}}$};
\node[font=\footnotesize] at (4.25,1.5) {${\mutimes_j}$};
\node[font=\footnotesize] at (0,2) {$\overgroup{\phantom{\s p_n \otimes \dotsb \otimes \s p_{j+1} \otimes \s p_j \otimes \dotsb \otimes \s p_1}}$};
\node[font=\footnotesize] at (0,3) {${\mucirc_{n-j+1}}$};
\node[font=\small, anchor=west] at (7,1.8) {$\mucirc_{n-j+1} \bullet_0 \mutimes_j$};
\node[font=\small, anchor=west] at (7,-1.6) {$\mutimes_{n-1} \bullet_{j-1} \mubar_2$};
\end{scope}
\end{tikzpicture}
\caption{Composable morphisms $p_1, \dotsc, p_n$ in an orbifold disk and two possible ways of applying nested higher products which cancel and therefore satisfy the A$_\infty$ equations on $\s p_n \otimes \dotsb \otimes \s p_1$. Here, for simplicity, we omit the arcs in the interior of a disk sequence.}
\label{fig:cancel}
\end{figure}
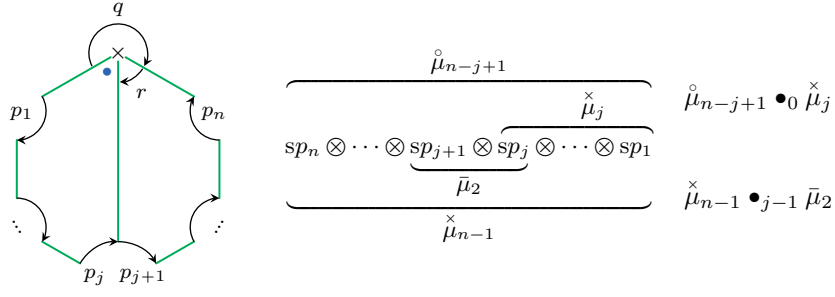

\begin{figure}
\begin{tikzpicture}[x=1em,y=1em,decoration={markings,mark=at position 0.55 with {\arrow[black]{Stealth[length=4.8pt]}}}]
\begin{scope}[xshift=0em]
\node[font=\footnotesize,shape=circle,scale=.6] (X) at (90:3em) {};
\node[font=\footnotesize] at (90:3em) {$\times$};
\node[font=\scriptsize,color=stopcolour] at ($(90:3em)+(-60:.75em)$) {$\bullet$};
\draw[line width=.75pt, color=arccolour, line cap=round] ($(150:3em)+(30:.75em)$) to (X);
\draw[line width=.75pt, color=arccolour, line cap=round] ($(150:3em)+(-90:.75em)$) to ($(210:3em)+(90:.75em)$);
\draw[line width=.75pt, color=arccolour, line cap=round] ($(210:3em)+(-30:.75em)$) to ($(270:3em)+(150:1em)$);
\draw[line width=.75pt, color=arccolour, line cap=round] ($(270:3em)+(30:1em)$) to ($(-30:3em)+(210:.75em)$);
\draw[line width=.75pt, color=arccolour, line cap=round] ($(-30:3em)+(90:.75em)$) to ($(30:3em)+(270:.75em)$);
\draw[line width=.75pt, color=arccolour, line cap=round] ($(30:3em)+(150:.75em)$) to (X);
\draw[line width=.75pt, color=arccolour, line cap=round] ($(270:3em)+(90:1em)$) to (X);
\draw[->, line width=.5pt] ($(90:3em)+(207:.75em)$) arc[start angle=207, end angle=-27, radius=.75em] to[in=76] ++(-0.001,-0.001);
\node[font=\scriptsize] at (90:4.25em) {$q$};
\draw[->, line width=.5pt] ($(90:3em)+(267:.75em)$) arc[start angle=267, end angle=213, radius=.75em] to[in=-45] ++(-0.001,0.001);
\node[font=\scriptsize] at ($(90:3em)+(240:1.25em)$) {$r$};
\draw[->, line width=.5pt] ($(150:3em)+(30:.75em)$) arc[start angle=30, end angle=-90, radius=.75em] to[in=10] ++(-0.01,0);
\node[font=\scriptsize] at (150:2.9em) {$p_1$};
\draw[->, line width=.5pt] ($(210:3em)+(90:.75em)$) arc[start angle=90, end angle=-30, radius=.75em] to[in=70] ++(0,-0.001);
\node[font=\scriptsize] at (207:2.8em) {$.$};
\node[font=\scriptsize] at (210:2.8em) {$.$};
\node[font=\scriptsize] at (213:2.8em) {$.$};
\draw[->, line width=.5pt] ($(270:3em)+(150:1em)$) arc[start angle=150, end angle=90, radius=1em];
\node[font=\scriptsize, left=-.8em] at (257:3em) {$p_{n-j}$};
\draw[->, line width=.5pt] ($(270:3em)+(90:1em)$) arc[start angle=90, end angle=30, radius=1em];
\node[font=\scriptsize, right=-.8em] at (284:3em) {$p_{n-j+1}$};
\draw[->, line width=.5pt] ($(-30:3em)+(210:.75em)$) arc[start angle=210, end angle=90, radius=.75em] to[in=190] ++(0.005,0);
\node[font=\scriptsize] at (-27:2.8em) {$.$};
\node[font=\scriptsize] at (-30:2.8em) {$.$};
\node[font=\scriptsize] at (-33:2.8em) {$.$};
\draw[->, line width=.5pt] ($(30:3em)+(270:.75em)$) arc[start angle=270, end angle=150, radius=.75em] to[in=250] ++(0.001,0.001);
\node[font=\scriptsize] at (27:3em) {$p_n$};
\node[font=\small, align=center, anchor=north] at (0,-3.5em) {$\mucirc_{n-j+1} \bullet_{n-j} \mutimes_j$ \strut \\ $\mutimes_{n-1} \bullet_{n-j-1} \mubar_2$ \strut};
\end{scope}
\begin{scope}[xshift=10em]
\node[font=\footnotesize,shape=circle,scale=.6] (X) at (90:3em) {};
\node[font=\footnotesize] at (90:3em) {$\times$};
\node[font=\scriptsize,color=stopcolour] at (90:2.5em) {$\bullet$};
\draw[line width=.75pt, color=arccolour, line cap=round] (70:3em) to (X) to (110:3em) (130:3em) to (150:3em) (170:3em) to (190:3em) (220:3em) to (240:3em) (260:3em) to (280:3em) (300:3em) to (320:3em) (350:3em) to (10:3em) (30:3em) to (50:3em) (205:2.3em) to (335:2.3em);
\draw[->, line width=.5pt] ($(90:3em)+(190:.75em)$) arc[start angle=190, end angle=-10, radius=.75em] to[in=92] ++(-0.001,-0.001);
\node[font=\scriptsize] at (90:4.25em) {$q$};
\draw[->, line width=.5pt] (110:3em) to[out=280, in=320, looseness=1.5] (130:3em);
\node[font=\scriptsize] at (120:3.2em) {$p_1$};
\draw[->, line width=.5pt] (150:3em) to[out=320, in=360, looseness=1.5] (170:3em);
\node[font=\scriptsize] at (160:3em) {$.$};
\node[font=\scriptsize] at (163:3em) {$.$};
\node[font=\scriptsize] at (157:3em) {$.$};
\draw[->, line width=.5pt] (190:3em) to[out=360, in=115, looseness=1] (205:2.3em);
\node[font=\scriptsize] at (199:3.1em) {$p_i$};
\draw[->, line width=.5pt] (205:2.3em) to[out=295, in=50, looseness=1] (220:3em);
\node[font=\scriptsize] at (208:3.7em) {$p_{i+1}$};
\draw[->, line width=.5pt] (240:3em) to[out=50, in=90, looseness=1.5] (260:3em);
\node[font=\scriptsize] at (244:3.4em) {$p_{i+2}$};
\draw[->, line width=.5pt] (280:3em) to[out=90, in=130, looseness=1.5] (300:3em);
\node[font=\scriptsize] at (290:3em) {$.$};
\node[font=\scriptsize] at (287:3em) {$.$};
\node[font=\scriptsize] at (293:3em) {$.$};
\draw[->, line width=.5pt] (320:3em) to[out=130, in=180, looseness=1.4] (350:3em);
\node[font=\scriptsize] at (339:3.5em) {$p_{i+j}$};
\draw[->, line width=.5pt] (10:3em) to[out=180, in=220, looseness=1.5] (30:3em);
\node[font=\scriptsize] at (20:3em) {$.$};
\node[font=\scriptsize] at (23:3em) {$.$};
\node[font=\scriptsize] at (17:3em) {$.$};
\draw[->, line width=.5pt] (50:3em) to[out=220, in=260, looseness=1.4] (70:3em);
\node[font=\scriptsize] at (57:3.3em) {$p_n$};
\node[font=\small, align=center, anchor=north] at (0,-3.5em) {$\mutimes_{n-j+1} \bullet_i \mucirc_j$ \strut \\ $\mutimes_{n-1} \bullet_{i-1} \mubar_2$ \strut};
\end{scope}
\begin{scope}[xshift=20em]
\node[font=\footnotesize,shape=circle,scale=.6] (X) at (90:3em) {};
\node[font=\footnotesize] at (90:3em) {$\times$};
\node[font=\scriptsize,color=stopcolour] at (90:2.5em) {$\bullet$};
\draw[line width=.75pt, color=arccolour, line cap=round] (70:3em) to (X) to (110:3em) (130:3em) to (150:3em) (170:3em) to (190:3em) (220:3em) to (240:3em) (260:3em) to (280:3em) (300:3em) to (320:3em) (350:3em) to (10:3em) (30:3em) to (50:3em) (205:2.3em) to (335:2.3em);
\draw[->, line width=.5pt] ($(90:3em)+(190:.75em)$) arc[start angle=190, end angle=-10, radius=.75em] to[in=92] ++(-0.001,-0.001);
\node[font=\scriptsize] at (90:4.25em) {$q$};
\draw[->, line width=.5pt] (110:3em) to[out=280, in=320, looseness=1.5] (130:3em);
\node[font=\scriptsize] at (120:3.2em) {$p_1$};
\draw[->, line width=.5pt] (150:3em) to[out=320, in=360, looseness=1.5] (170:3em);
\node[font=\scriptsize] at (160:3em) {$.$};
\node[font=\scriptsize] at (163:3em) {$.$};
\node[font=\scriptsize] at (157:3em) {$.$};
\draw[->, line width=.5pt] (190:3em) to[out=360, in=50, looseness=1.4] (220:3em);
\node[font=\scriptsize] at (202:3.3em) {$p_{i+1}$};
\draw[->, line width=.5pt] (240:3em) to[out=50, in=90, looseness=1.5] (260:3em);
\node[font=\scriptsize] at (244:3.4em) {$p_{i+2}$};
\draw[->, line width=.5pt] (280:3em) to[out=90, in=130, looseness=1.5] (300:3em);
\node[font=\scriptsize] at (290:3em) {$.$};
\node[font=\scriptsize] at (287:3em) {$.$};
\node[font=\scriptsize] at (293:3em) {$.$};
\draw[->, line width=.5pt] (320:3em) to[out=130, in=245, looseness=1] (335:2.3em);
\node[font=\scriptsize] at (330:3.9em) {$p_{i+j}$};
\draw[->, line width=.5pt] (335:2.3em) to[out=65, in=180, looseness=1] (350:3em);
\node[font=\scriptsize] at (343:4em) {$p_{i+j+1}$};
\draw[->, line width=.5pt] (10:3em) to[out=180, in=220, looseness=1.5] (30:3em);
\node[font=\scriptsize] at (20:3em) {$.$};
\node[font=\scriptsize] at (23:3em) {$.$};
\node[font=\scriptsize] at (17:3em) {$.$};
\draw[->, line width=.5pt] (50:3em) to[out=220, in=260, looseness=1.4] (70:3em);
\node[font=\scriptsize] at (57:3.3em) {$p_n$};
\node[font=\small, align=center, anchor=north] at (0,-3.5em) {$\mutimes_{n-j+1} \bullet_i \mucirc_j$ \strut \\ $\mutimes_{n-1} \bullet_{i+j-1} \mubar_2$ \strut};
\end{scope}
\begin{scope}[xshift=5em,yshift=-12em]
\node[font=\footnotesize,shape=circle,scale=.6] (X) at (90:3em) {};
\node[font=\footnotesize] at (90:3em) {$\times$};
\node[font=\scriptsize,color=stopcolour] at ($(90:3em)+(240:.75em)$) {$\bullet$};
\draw[line width=.75pt, color=arccolour, line cap=round] ($(150:3em)+(30:.75em)$) to (X);
\draw[line width=.75pt, color=arccolour, line cap=round] ($(150:3em)+(-90:.75em)$) to ($(210:3em)+(90:.75em)$);
\draw[line width=.75pt, color=arccolour, line cap=round] ($(210:3em)+(-30:.75em)$) to ($(270:3em)+(150:1em)$);
\draw[line width=.75pt, color=arccolour, line cap=round] ($(270:3em)+(30:1em)$) to ($(-30:3em)+(210:.75em)$);
\draw[line width=.75pt, color=arccolour, line cap=round] ($(-30:3em)+(90:.75em)$) to ($(30:3em)+(270:.75em)$);
\draw[line width=.75pt, color=arccolour, line cap=round] ($(30:3em)+(150:.75em)$) to (X);
\draw[line width=.75pt, color=arccolour, line cap=round] ($(270:3em)+(90:1em)$) to (X);
\draw[->, line width=.5pt] ($(90:3em)+(207:.75em)$) arc[start angle=207, end angle=-27, radius=.75em] to[in=76] ++(-0.001,-0.001);
\node[font=\scriptsize] at (90:4.25em) {$q$};
\node[font=\scriptsize] at ($(90:3em)+(-57:1.45em)$) {$p_n$};
\draw[->, line width=.5pt] ($(90:3em)+(-33:.75em)$) arc[start angle=-33, end angle=-87, radius=.75em];
\draw[->, line width=.5pt] ($(150:3em)+(30:.75em)$) arc[start angle=30, end angle=-90, radius=.75em] to[in=10] ++(-0.01,0);
\node[font=\scriptsize] at (150:2.9em) {$p_1$};
\draw[->, line width=.5pt] ($(210:3em)+(90:.75em)$) arc[start angle=90, end angle=-30, radius=.75em] to[in=70] ++(0,-0.001);
\node[font=\scriptsize] at (207:3em) {$.$};
\node[font=\scriptsize] at (210:3em) {$.$};
\node[font=\scriptsize] at (213:3em) {$.$};
\draw[->, line width=.5pt] ($(270:3em)+(150:1em)$) arc[start angle=150, end angle=30, radius=1em];
\node[font=\scriptsize] at (270:3em) {$p_{n-j+1}$};
\draw[->, line width=.5pt] ($(-30:3em)+(210:.75em)$) arc[start angle=210, end angle=90, radius=.75em] to[in=190] ++(0.005,0);
\node[font=\scriptsize] at (-27:3em) {$.$};
\node[font=\scriptsize] at (-30:3em) {$.$};
\node[font=\scriptsize] at (-33:3em) {$.$};
\draw[->, line width=.5pt] ($(30:3em)+(270:.75em)$) arc[start angle=270, end angle=150, radius=.75em] to[in=250] ++(0.001,0.001);
\node[font=\scriptsize] at (24:3.4em) {$p_{n-1}$};
\node[font=\small, align=center, anchor=north] at (0,-3.5em) {$\mutimes_{n-j+1} \bullet_{n-j} \mucirc_j$ \strut \\ $\mubar_2 \bullet_0 \mutimes_{n-1}$ \strut};
\end{scope}
\begin{scope}[xshift=15em,yshift=-12em]
\node[font=\footnotesize,shape=circle,scale=.6] (X) at (90:3em) {};
\node[font=\footnotesize] at (90:3em) {$\times$};
\node[font=\scriptsize,color=stopcolour] at ($(90:3em)+(-60:.75em)$) {$\bullet$};
\draw[line width=.75pt, color=arccolour, line cap=round] ($(150:3em)+(30:.75em)$) to (X);
\draw[line width=.75pt, color=arccolour, line cap=round] ($(150:3em)+(-90:.75em)$) to ($(210:3em)+(90:.75em)$);
\draw[line width=.75pt, color=arccolour, line cap=round] ($(210:3em)+(-30:.75em)$) to ($(270:3em)+(150:1em)$);
\draw[line width=.75pt, color=arccolour, line cap=round] ($(270:3em)+(30:1em)$) to ($(-30:3em)+(210:.75em)$);
\draw[line width=.75pt, color=arccolour, line cap=round] ($(-30:3em)+(90:.75em)$) to ($(30:3em)+(270:.75em)$);
\draw[line width=.75pt, color=arccolour, line cap=round] ($(30:3em)+(150:.75em)$) to (X);
\draw[line width=.75pt, color=arccolour, line cap=round] ($(270:3em)+(90:1em)$) to (X);
\draw[->, line width=.5pt] ($(90:3em)+(207:.75em)$) arc[start angle=207, end angle=-24, radius=.75em] to[in=76] ++(-0.001,-0.001);
\node[font=\scriptsize] at (90:4.25em) {$q$};
\draw[->, line width=.5pt] ($(90:3em)+(267:.75em)$) arc[start angle=267, end angle=213, radius=.75em] to[in=-45] ++(-0.001,0.001);
\node[font=\scriptsize] at ($(90:3em)+(240:1.35em)$) {$p_1$};
\draw[->, line width=.5pt] ($(150:3em)+(30:.75em)$) arc[start angle=30, end angle=-90, radius=.75em] to[in=10] ++(-0.01,0);
\node[font=\scriptsize] at (150:2.9em) {$p_2$};
\draw[->, line width=.5pt] ($(210:3em)+(90:.75em)$) arc[start angle=90, end angle=-30, radius=.75em] to[in=70] ++(0,-0.001);
\node[font=\scriptsize] at (207:3em) {$.$};
\node[font=\scriptsize] at (210:3em) {$.$};
\node[font=\scriptsize] at (213:3em) {$.$};
\draw[->, line width=.5pt] ($(270:3em)+(150:1em)$) arc[start angle=150, end angle=30, radius=1em];
\node[font=\scriptsize] at (270:2.7em) {$p_j$};
\draw[->, line width=.5pt] ($(-30:3em)+(210:.75em)$) arc[start angle=210, end angle=90, radius=.75em] to[in=190] ++(0.005,0);
\node[font=\scriptsize] at (-27:3em) {$.$};
\node[font=\scriptsize] at (-30:3em) {$.$};
\node[font=\scriptsize] at (-33:3em) {$.$};
\draw[->, line width=.5pt] ($(30:3em)+(270:.75em)$) arc[start angle=270, end angle=150, radius=.75em] to[in=250] ++(0.001,0.001);
\node[font=\scriptsize] at (27:3em) {$p_n$};
\node[font=\small, align=center, anchor=north] at (0,-3.5em) {$\mutimes_{n-j+1} \bullet_0 \mucirc_j$ \strut \\ $\mubar_2 \bullet_1 \mutimes_{n-1}$ \strut};
\end{scope}
\end{tikzpicture}
\caption{The remaining disks with nested higher multiplications cancelling pairwise in the A$_\infty$ equations for $n > 1$ when applied to $\s p_n \otimes \dotsb \otimes \s p_1$.}
\label{fig:cancel2}
\end{figure}
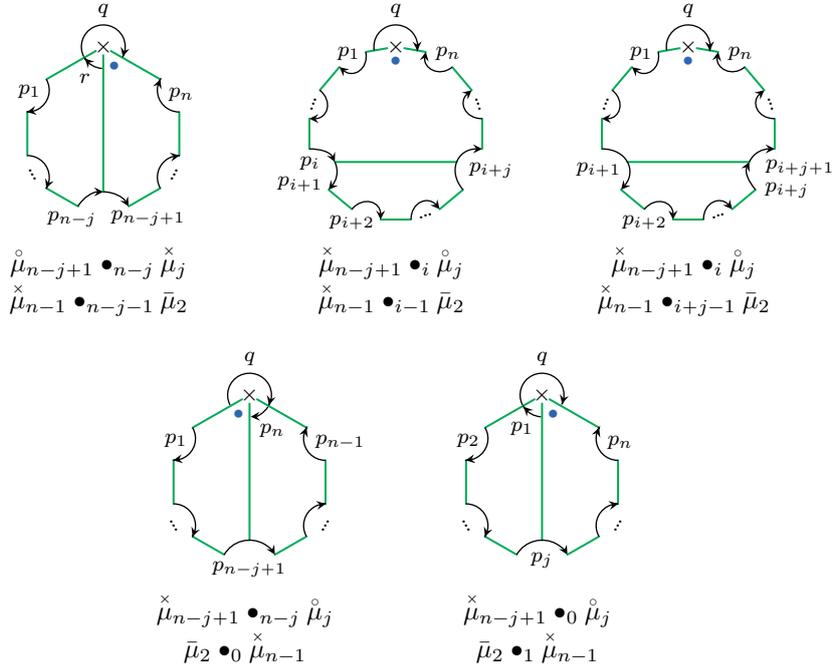

\begin{proof}
Let $p_1, \dotsc, p_n$ be a sequence of composable morphisms corresponding to (either boundary or orbifold) paths and recall that the higher products $\mu_n$ fall into three types --- $\mubar_2$, $\mucirc_{\geq 2}$ and $\mutimes_{\geq 1}$. We verify the A$_\infty$ relations \eqref{eq:ainfinityrelations} for $p_1, \dotsc, p_n$.  

For $n = 1$, let $p_1$ be a path from $\gamma_1$ to $\gamma_2$. We have $\mu_1 = \mutimes_1$ and if $\mutimes_1 (\s p_1) = \s q$ then $q$ is an orbifold path at an orbifold point $x$. Note that $\mutimes_1 (\s q) = 0$ since $q$ cannot lie in a $4$-gon with one missing relation. As a result, we have
\[
\mutimes_1 (\mutimes_1 (\s p_1)) = 0
\]
so that $\mu_1 \bullet \mu_1 = \mutimes_1 \bullet \mutimes_1 = 0$ and hence satisfying the A$_\infty$ relations \eqref{eq:ainfinityrelations} for $n = 1$.

For $n > 1$, the terms appearing in the A$_\infty$ relation \eqref{eq:ainfinityrelations} are all of the form
\begin{equation}
\label{eq:mumu}
\mu_{n-j+1} \bullet_i \mu_j.
\end{equation}
If neither $\mu_j$ nor $\mu_{n-j+1}$ is a higher product $\mutimes_k$ associated to an orbifold disk sequence, then the A$_\infty$ relations are satisfied by \cite[Proposition 3.1]{haidenkatzarkovkontsevich}.

Let us now consider the case where either $\mu_j = \mutimes_j$ or $\mu_{n-j+1} = \mutimes_{n-j+1}$. We first note that
\begin{align*}
\mutimes_{n-j+1} \bullet_i \mutimes_j &= 0 \qquad\text{for all $0 \leq i \leq n-j$} \\
\mucirc_{n-j+1} \bullet_i \mutimes_j &= 0 \qquad\text{for all $0 < i < n-j$.}
\end{align*}
This vanishing follows directly from the definition of admissible dissection since the maximal orbifold path at the orbifold point $x$ of the orbifold disk sequence $p_{i+1}, \dotsc, p_{i+j}$ lies in a polygon with one missing relation or stop.

The term \eqref{eq:mumu} may be nonzero only in one of the cases $\mucirc \bullet \mutimes$, $\mutimes \bullet \mucirc$, $\mutimes \bullet \mubar$ or $\mubar \bullet \mutimes$. The precise terms are
\begin{flalign*}
&&
\begin{gathered}
\mucirc_{n-j+1} \bullet_i \mutimes_j \\
\text{\footnotesize ($i = 0, n-j$)}
\end{gathered}
&&
\begin{gathered}
\mutimes_{n-j+1} \bullet_i \mucirc_j \\
\text{\footnotesize ($0 \leq i \leq n-j$)}
\end{gathered}
&&
\begin{gathered}
\mutimes_{n-1} \bullet_i \mubar_2 \\
\text{\footnotesize ($0 \leq i \leq n-2$)}
\end{gathered}
&&
\begin{gathered}
\mubar_2 \bullet_i \mutimes_{n-1} \\
\text{\footnotesize ($0 \leq i \leq 1$)}
\end{gathered}
&&
\end{flalign*}

We claim that all of these terms cancel in pairs. These pairs arise from cutting a polygon into two smaller polygons along an arc in $\Delta$. (In the Floer-theoretic picture, this is precisely the degeneration of a pseudo-holomorphic disk into two pseudo-holomorphic disks along a Lagrangian.) Up to signs, the term $\mucirc_{n-j+1} \bullet_0 \mutimes_j$ is the same as $\mutimes_{n-1} \bullet_{j-1} \mubar_2$ as illustrated in Fig.~\ref{fig:cancel}. Precisely, we have
\begin{align*}
\mucirc_{n-j+1} (\s p_n \otimes \dotsb \otimes \s p_{j+1} \otimes \undergroup{\mutimes_j (\s p_j \otimes \dotsb \otimes \s p_{1})}_{\s rq}) &= (-1)^{|q|} \s q \\
\mutimes_{n-1} (\s p_n \otimes \dotsb \otimes \s p_{j+2} \otimes \undergroup{\mubar_2 (\s p_{j+1} \otimes \s p_j)}_{(-1)^{|p_j|} \s p_{j+1} p_j} \otimes \dotsb \otimes \s p_{1})) &= (-1)^{|r|+ |p_j|+ |\s p_{j-1 \dotsc 1}|} \s q
\end{align*}
where the sign $(-1)^q$ in the first equality follows from \eqref{align:signmucircle} and the sign $(-1)^{|r|}$ in the second equality is due to \eqref{align:orbifolddiskproduct}. The above two terms cancel each other since $|q|+|r| = |\s p_{j \dotsc 1}| +2$.

The argument for the remaining cases is completely analogous. The relevant disks are illustrated in Fig.~\ref{fig:cancel2}.
\end{proof}

\begin{remark}\label{remark:withck}
An A$_\infty$ category $\mathcal F_\Gamma$ associated to a different kind of dissection $\Gamma$ (called {\it tagged arc system}) was defined independently by Cho and Kim in \cite{chokim} using a similar but slightly different approach. This category $\mathcal F_\Gamma$ appears to be Morita equivalent to the category $\mathbf A_\Delta$ defined above. One of the differences is that the category $\mathcal F_\Gamma$ has trivial differential, since in \cite{chokim}  no orbifold disk sequences of length $2$ exist. Another difference is that arcs connecting to orbifold points come in (isotopic) copies distinguished by a tagging which give rise to nonisomorphic objects in $\mathcal F_\Gamma$. In our setup, these nonisomorphic copies can be expressed as twisted complexes in $\mathbf A_\Delta$ (see Remark \ref{remark:tagging}). 
Furthermore, the A$_\infty$ products in the category $\mathcal F_\Gamma$ have some nontrivial coefficients involving powers of $\frac{1}{2}$ which do not appear in the higher products in $\mathbf A_\Delta$.
\end{remark}

\begin{remark}[(Tagged arcs via twisted complexes)]
\label{remark:tagging}
Let $\Delta$ be an admissible dissection and $\mathbf A_\Delta$ be the associated A$_\infty$ category. For any arc $\gamma$ connecting to an orbifold point in $\Delta$ we may obtain an object $\gamma'$ in $\tw (\mathbf A_\Delta)$ which is not isomorphic to $\gamma$ but which is represented by an isotopic arc. For example, given a minimal orbifold sequence
\[
\gamma \overset{p_1}{\smile} \gamma_1 \overset{p_2}{\smile} \dotsb\overset{p_{n-1}}{\smile} \gamma_{n-1} 
\]
Let $\gamma'$ be the twisted complex 
\[
\bigoplus_{i=1}^{n-1}  \s^{|\s p_{i\dotsc 1}|}\gamma_i 
\]
with differential $\delta = \sum_{i=2}^{n-1} p_i$. Then $\gamma'$ is not isomorphic to $\gamma$, but may be represented by an isotopic arc which is a way of representing the ``tagged'' version of $\gamma$ as a twisted complex. See also the proof of Lemma \ref{lemma:splitgenerator} for an explicit example.
\end{remark}

\subsection{Morita equivalences}

In this subsection we show that the A$_\infty$ categories $\mathbf A_\Delta$ are Morita equivalent (Definition \ref{definition:moritaequivalence}) for all admissible dissections $\Delta$ on a graded orbifold surface with stops.

The following result shows that the contraction of an edge in the ribbon complex connecting to a vertex of type $\vcirc$ preserves the Morita equivalence class of the associated A$_\infty$ category.

\begin{lemma}\label{lemma:edgecontraction}
Let $\mathbf S = (S, \Sigma, \eta)$ be a graded orbifold surface with stops. Let $\Delta$ be an admissible dissection of $\mathbf S$ whose associated ribbon complex is $\mathbb G (\Delta)$. Let $e_\gamma$ be an edge of $\mathbb G (\Delta)$ joining a vertex of type $\vcirc$ and a distinct vertex $v$. Denote by $\Delta'$ the dissection obtained from $\Delta$ by deleting the arc $\gamma$ dual to $e_\gamma$. 

Then $\Delta'$ is also admissible and the natural inclusion $F \colon \mathbf A_{\Delta'} \hookrightarrow \mathbf A_{\Delta}$ is a Morita equivalence.
\end{lemma}

\begin{proof}
Note that the ribbon complex $\mathbb G (\Delta')$ associated to $\Delta'$ is obtained from $\mathbb G (\Delta)$ by contracting the edge $e_\gamma$. Then the first assertion follows from Definition \ref{definition:admissible} since the above operation will not change the number of vertices of type $\vtimes, \vbullet, \vodot$ on the boundary of each $2$-cell. 

Let us prove the second assertion. Note that the inclusion $F$ is a strict A$_\infty$ functor since by definition any smooth (resp.\ orbifold) disk sequence in $\Delta'$ is also a smooth (resp.\ orbifold) disk sequence in $\Delta$. Similar to \cite[Lemma 3.2]{haidenkatzarkovkontsevich} we may show that the object $\gamma$ in $\mathbf A_\Delta$ is isomorphic to a twisted complex of objects (i.e.\ all the other arcs in $P_{\vcirc}$ different from $\gamma$) in $\mathbf A_{\Delta'}$. Then by Lemma \ref{lemma:twistedgenerators} we infer that $F$ is a Morita equivalence. 
\end{proof}

\begin{figure}[ht]
\begin{tikzpicture}[x=1em,y=1em,decoration={markings,mark=at position 0.55 with {\arrow[black]{Stealth[length=4.8pt]}}}]
\begin{scope}[xshift=0em]
\node[font=\footnotesize] at (-80:.8em) {$v$};
\node[color=ribboncolour] at (0,0) {$\bullet$};
\draw[-, line width=.85, color=ribboncolour, line cap=round] (0,0) to (150:2em) (0,0) to (120:2em) (0,0) to (240:2em) (0,0) to (210:2em) (0,0) to (0:3em);
\draw[fill=ribboncolour, color=ribboncolour] (180:.7em) circle(.15em);
\node[color=ribboncolour] at (0:3em) {$\bullet$};
\draw[-, line width=.85, color=ribboncolour, line cap=round] (0:3em) to ($(0:3em)+(-45:2em)$)  (0:3em) to ($(0:3em)+(45:2em)$);
\node[font=\footnotesize] at ($(0:3em)+(-90:.8em)$) {$\vcirc$};
\node[font=\footnotesize] at (1.5em,.5em) {$e_{\gamma}$};
\node at (7em,0) {$\rightsquigarrow$};
\node[font=\footnotesize]  at (1.5em, -2.5em) {$\mathbb G (\Delta)$};
\end{scope}
\begin{scope}[xshift=12em]
\node[font=\footnotesize] at (-80:.8em) {$v$};
\node[color=ribboncolour] at (0,0) {$\bullet$};
\draw[-, line width=.85, color=ribboncolour, line cap=round] (0,0) to (150:2em) (0,0) to (120:2em) (0,0) to (240:2em) (0,0) to (210:2em) (0,0) to (30:2em) (0,0) to (-30:2em);
\draw[fill=ribboncolour, color=ribboncolour] (180:.7em) circle(.15em);
\node[font=\footnotesize]  at (0em, -2.5em) {$\mathbb G (\Delta')$};
\end{scope}
\end{tikzpicture}
\caption{Contracting an edge $e_\gamma$ joining two vertices.}
\label{fig:contractingvcirc}
\end{figure}
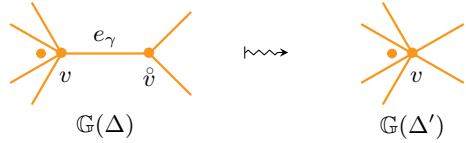

\begin{lemma}
\label{lemma:4gon}
Let $\mathbf S = (S, \Sigma, \eta)$ be a graded orbifold surface with stops and let $\Delta$ be an admissible dissection of $\mathbf S$. Then there exists an admissible dissection $\Delta_0$ of $\mathbf S$, whose ribbon graph is illustrated in Fig.~\ref{fig:ribbon}, such that $\mathbf A_{\Delta}$ is Morita equivalent to $\mathbf A_{\Delta_0}$.
\end{lemma}

\begin{proof}
Let $\mathbb G (\Delta)$ be the ribbon complex associated to $\Delta$. Let us fix a vertex $\vbullet \in \mathbb G_0 (\Delta)$, which exists since we assume that there is at least one boundary stop on $\mathbf S$.

{\it Operation 1.} Assume that $e_{\gamma}$ is an edge joining $\vbullet$ and another vertex $\vbullet_1\neq \vbullet$. If $\val(\vbullet_1) >1$ then we may add an arc $\gamma'$ in $P_{\vbullet_1}$ to obtain a new admissible dissection $\Delta'$. Note that $\mathbb G (\Delta')$ is obtained from $\mathbb G (\Delta)$ by splitting the vertex $\vbullet_1$ into two vertices $\vcirc_1$ and $\vbullet_1'$.  See Fig.~\ref{fig:contractingvbullet}. (It is indeed possible to make $\Delta'$ admissible for instance if $\val(\vcirc) =3$ in $\Delta'$.) Note that $\val(\vbullet_1) > \val(\vbullet_1')$. Denote by $\Delta''$ the admissible dissection obtained from $\Delta'$ by deleting the arc $\gamma$. 

Note that $\mathbb G (\Delta)$ and $\mathbb G (\Delta'')$ are both obtained from $\mathbb G (\Delta')$ by edge contractions. By Lemma \ref{lemma:edgecontraction} we obtain Morita equivalences
\[
\mathbf A_\Delta \hookrightarrow \mathbf A_{\Delta'} \hookleftarrow \mathbf A_{\Delta''}.
\]
Iterating this operation we may decrease $\val(\vbullet_1)$ until it equals $1$. 

\begin{figure}[ht]
\begin{tikzpicture}[x=1em,y=1em,decoration={markings,mark=at position 0.55 with {\arrow[black]{Stealth[length=4.8pt]}}}]
\begin{scope}[xshift=0em]
\node[font=\footnotesize] at (.2em,-.8em) {$\vbullet$};
\node[color=ribboncolour] at (0,0) {$\bullet$};
\draw[-, line width=.85, color=ribboncolour, line cap=round] (0,0) to (150:2em) (0,0) to (120:2em) (0,0) to (240:2em) (0,0) to (210:2em) (0,0) to (0:3em);
\draw[fill=ribboncolour, color=ribboncolour] (180:.7em) circle(.15em);
\node[color=ribboncolour] at (0:3em) {$\bullet$};
\draw[-, line width=.85, color=ribboncolour, line cap=round] (0:3em) to ($(0:3em)+(-20:2em)$)  (0:3em) to ($(0:3em)+(20:2em)$) (0:3em) to ($(0:3em)+(-40:2em)$)  (0:3em) to ($(0:3em)+(40:2em)$) (0:3em) to ($(0:3em)+(60:2em)$) (0:3em) to ($(0:3em)+(-60:2em)$);
\node[font=\footnotesize] at ($(2.8em,-.8em)$) {$\vbullet_1$};
\draw[fill=ribboncolour, color=ribboncolour] (0:3.9em) circle(.15em);
\node[font=\scriptsize] at (1.5em,.5em) {$e_{\gamma}$};
\node at (7em,0) {$\leftsquigarrow$};
\node[font=\footnotesize]  at (1.5em, -2.5em) {$\mathbb G (\Delta)$};
\end{scope}
\begin{scope}[xshift=11em]
\node[font=\footnotesize] at (.2em,-.8em) {$\vbullet$};
\node[color=ribboncolour] at (0,0) {$\bullet$};
\draw[-, line width=.85, color=ribboncolour, line cap=round] (0,0) to (150:2em) (0,0) to (120:2em) (0,0) to (240:2em) (0,0) to (210:2em) (0,0) to (0:2.5em);
\node[color=ribboncolour] at (0:2.5em) {$\bullet$};
\node[font=\scriptsize] at (1.225em,.5em) {$e_{\gamma}$};
\draw[fill=ribboncolour, color=ribboncolour] (180:.7em) circle(.15em);
\node[color=ribboncolour] at (0:5em) {$\bullet$};
\node[font=\footnotesize] at (2.2em,-.8em) {$\vcirc$};
\node[font=\scriptsize] at (4.3em,.5em) {$e_{\gamma'}$};
\draw[-, line width=.85, color=ribboncolour, line cap=round] (0:5em) to (0:2.5em) (0:5em) to ($(0:5em)+(-20:2em)$)  (0:5em) to ($(0:5em)+(20:2em)$) (0:5em) to ($(0:5em)+(-40:2em)$);  \draw[-, line width=.85, color=ribboncolour, line cap=round] (0:2.5em) to ($(0:2.5em)+(40:2em)$) (0:2.5em) to ($(0:2.5em)+(60:2em)$) (0:2.5em) to ($(0:2.5em)+(-60:2em)$);
\node[font=\footnotesize] at (4.8em,-.8em) {$\vbullet_2$};
\node[color=ribboncolour] at (0:5em) {$\bullet$};
\draw[fill=ribboncolour, color=ribboncolour] (0:5.9em) circle(.15em);
\node at (9em,0) {$\rightsquigarrow$};
\node[font=\footnotesize]  at (2.5em, -2.5em) {$\mathbb G (\Delta')$};
\end{scope}
\begin{scope}[xshift=23em]
\node[font=\footnotesize] at (0em,-.8em) {$\vbullet$};
\node[color=ribboncolour] at (0,0) {$\bullet$};
\draw[-, line width=.85, color=ribboncolour, line cap=round] (0,0) to (150:2em) (0,0) to (120:2em) (0,0) to (240:2em) (0,0) to (210:2em) (0,0) to (0:3em);
\draw[fill=ribboncolour, color=ribboncolour] (180:.7em) circle(.15em);
\node[color=ribboncolour] at (0:3em) {$\bullet$};
\node[font=\scriptsize] at (2em,.5em) {$e_{\gamma'}$};
\draw[-, line width=.85, color=ribboncolour, line cap=round] (0:3em) to ($(0:3em)+(-20:2em)$)  (0:3em) to ($(0:3em)+(20:2em)$) (0:3em) to ($(0:3em)+(-40:2em)$);  
\draw[-, line width=.85, color=ribboncolour, line cap=round] (0:0em) to ($(0:0em)+(40:2em)$) (0:0em) to ($(0:0em)+(60:2em)$) (0:0em) to ($(0:0em)+(-60:2em)$);
\node[font=\footnotesize] at (2.8em,-.8em) {$\vbullet_2$};
\draw[fill=ribboncolour, color=ribboncolour] (0:3.9em) circle(.15em);
\node[font=\footnotesize]  at (1.5em, -2.5em) {$\mathbb G (\Delta'')$};
\end{scope}
\end{tikzpicture}
\caption{Contracting the edges $e_{\gamma}$ and $e_{\gamma'}$.}
\label{fig:contractingvbullet}
\end{figure}

{\it Operation 2.} Let $e_\gamma$ be an edge joining $\vbullet$ and a vertex $\vtimes$. Assume that $\val(\vtimes) > 2$. Denote by $\Delta'$ the admissible dissection obtained from $\Delta$ by adding an arc $\gamma'$ in $P_{\vtimes}$. Then the ribbon complex $\mathbb G (\Delta')$ is obtained from $\mathbb G (\Delta)$ by splitting the vertex $\vtimes$ into two vertices $\vcirc$ and $\vtimes'$. Note that $\val(\vtimes') < \val(\vtimes)$. Denote by $\Delta''$ the admissible dissection obtained from $\Delta'$ by deleting the arc $\gamma$, so that $\mathbb G (\Delta'')$ is obtained from $\mathbb G (\Delta')$ by contracting the edge $e_{\gamma}$ as illustrated in Fig.~\ref{fig:contractingvtimes}.

Then by Lemma \ref{lemma:edgecontraction} we obtain Morita equivalences $$\mathbf A_{\Delta}  \hookrightarrow \mathbf A_{\Delta'} \hookleftarrow \mathbf A_{\Delta''}.$$ Iterating this operation we  may decrease $\val(\vtimes)$ until it equals $2$, such that both edges join $\vtimes$ and $\vbullet$.
\begin{figure}[ht]
\begin{tikzpicture}[x=1em,y=1em,decoration={markings,mark=at position 0.55 with {\arrow[black]{Stealth[length=4.8pt]}}}]
\begin{scope}[xshift=0em]
\node[font=\footnotesize] at (.2em,-.8em) {$\vbullet$};
\node[color=ribboncolour] at (0,0) {$\bullet$};
\draw[-, line width=.85, color=ribboncolour, line cap=round] (0,0) to (150:2em) (0,0) to (120:2em) (0,0) to (240:2em) (0,0) to (210:2em) (0,0) to (0:3em);
\draw[fill=ribboncolour, color=ribboncolour] (180:.7em) circle(.15em);
\node[color=ribboncolour] at (0:3em) {$\bullet$};
\draw[-, line width=.85, color=ribboncolour, line cap=round] (0:3em) to ($(0:3em)+(-20:2em)$)  (0:3em) to ($(0:3em)+(20:2em)$) (0:3em) to ($(0:3em)+(-40:2em)$)  (0:3em) to ($(0:3em)+(40:2em)$) (0:3em) to ($(0:3em)+(60:2em)$) (0:3em) to ($(0:3em)+(-60:2em)$);
\node[font=\footnotesize] at ($(2.8em,-.8em)$) {$\vtimes$};
\draw[fill=ribboncolour, color=ribboncolour] (0:3.9em) circle(.15em);
\node[font=\scriptsize] at (1.5em,.5em) {$e_{\gamma}$};
\node at (7em,0) {$\leftsquigarrow$};
\node[font=\footnotesize]  at (1.5em, -2.5em) {$\mathbb G (\Delta)$};
\draw[line width=0, fill=ribboncolour, draw=ribboncolour, opacity=.5] ($(0:3em)+(-20:2em)$)  --(0:3em) -- ($(0:3em)+(20:2em)$) -- cycle;
\end{scope}
\begin{scope}[xshift=11em]
\node[font=\footnotesize] at (.2em,-.8em) {$\vbullet$};
\node[color=ribboncolour] at (0,0) {$\bullet$};
\draw[-, line width=.85, color=ribboncolour, line cap=round] (0,0) to (150:2em) (0,0) to (120:2em) (0,0) to (240:2em) (0,0) to (210:2em) (0,0) to (0:2.5em);
\node[color=ribboncolour] at (0:2.5em) {$\bullet$};
\node[font=\scriptsize] at (1.225em,.5em) {$e_{\gamma}$};
\draw[fill=ribboncolour, color=ribboncolour] (180:.7em) circle(.15em);
\node[color=ribboncolour] at (0:5em) {$\bullet$};
\node[font=\footnotesize] at (2.2em,-.8em) {$\vcirc$};
\node[font=\scriptsize] at (4.3em,.5em) {$e_{\gamma'}$};
\draw[-, line width=.85, color=ribboncolour, line cap=round] (0:5em) to (0:2.5em) (0:5em) to ($(0:5em)+(-20:2em)$)  (0:5em) to ($(0:5em)+(20:2em)$) (0:5em) to ($(0:5em)+(-40:2em)$);  \draw[-, line width=.85, color=ribboncolour, line cap=round] (0:2.5em) to ($(0:2.5em)+(40:2em)$) (0:2.5em) to ($(0:2.5em)+(60:2em)$) (0:2.5em) to ($(0:2.5em)+(-60:2em)$);
\node[font=\footnotesize] at (4.8em,-.8em) {$\vtimes'$};
\node[color=ribboncolour] at (0:5em) {$\bullet$};
\draw[fill=ribboncolour, color=ribboncolour] (0:5.9em) circle(.15em);
\node at (9em,0) {$\rightsquigarrow$};
\node[font=\footnotesize]  at (2.5em, -2.5em) {$\mathbb G (\Delta')$};
\draw[line width=0, fill=ribboncolour, draw=ribboncolour, opacity=.5] ($(0:5em)+(-20:2em)$)  --(0:5em) -- ($(0:5em)+(20:2em)$) -- cycle;
\end{scope}
\begin{scope}[xshift=23em]
\node[font=\footnotesize] at (0em,-.8em) {$\vbullet$};
\node[color=ribboncolour] at (0,0) {$\bullet$};
\draw[-, line width=.85, color=ribboncolour, line cap=round] (0,0) to (150:2em) (0,0) to (120:2em) (0,0) to (240:2em) (0,0) to (210:2em) (0,0) to (0:3em);
\draw[fill=ribboncolour, color=ribboncolour] (180:.7em) circle(.15em);
\node[color=ribboncolour] at (0:3em) {$\bullet$};
\node[font=\scriptsize] at (2em,.5em) {$e_{\gamma'}$};
\draw[-, line width=.85, color=ribboncolour, line cap=round] (0:3em) to ($(0:3em)+(-20:2em)$)  (0:3em) to ($(0:3em)+(20:2em)$) (0:3em) to ($(0:3em)+(-40:2em)$);  
\draw[-, line width=.85, color=ribboncolour, line cap=round] (0:0em) to ($(0:0em)+(40:2em)$) (0:0em) to ($(0:0em)+(60:2em)$) (0:0em) to ($(0:0em)+(-60:2em)$);
\node[font=\footnotesize] at (2.8em,-.8em) {$\vtimes'$};
\draw[fill=ribboncolour, color=ribboncolour] (0:3.9em) circle(.15em);
\node[font=\footnotesize]  at (1.5em, -2.5em) {$\mathbb G (\Delta'')$};
\draw[line width=0, fill=ribboncolour, draw=ribboncolour, opacity=.5] ($(0:3em)+(-20:2em)$)  --(0:3em) -- ($(0:3em)+(20:2em)$) -- cycle;
\end{scope}
\end{tikzpicture}
\caption{Contracting the edges $e_{\gamma}$ and $e_{\gamma'}$.}
\label{fig:contractingvtimes}
\end{figure}

{\it Operation 3.} Let $x \in \Sing (S)$ be an orbifold point such that $\val(\vtimes_x) =2$. Denote by $v$ one of the adjacent vertices of $\vtimes_x$ in $d_x$ as illustrated in Fig.~\ref{fig:operation4}. We may assume that $v$ is neither of type $\vcirc$ (otherwise we may use the first operation to contract the edge $e_\gamma$).  Let $d_x$ be the associated $2$-cell. Then Lemma \ref{lemma:edgecontraction} shows that the zigzag of edge contractions in Fig.~\ref{fig:operation4} induces a zigzag of A$_\infty$ quasi-equivalences between $\mathbf A_\Delta$ and $\mathbf A_{\Delta'}$. See Fig.~\ref{fig:zigzag} for the local behaviour of this operation in terms of dissections. 

\begin{figure}[ht]
\begin{tikzpicture}[x=1em,y=1em,decoration={markings,mark=at position 0.55 with {\arrow[black]{Stealth[length=4.8pt]}}}]
\begin{scope}[xshift=0em]
\node[font=\footnotesize] at (-180:.8em) {$\vtimes_x$};
\node[color=ribboncolour] at (0,0) {$\bullet$};
\draw[-, line width=.85, color=ribboncolour, line cap=round] (0,0) to (120:2em) (0,0) to (0:3em);
\draw[fill=ribboncolour, color=ribboncolour] (70:.5em) circle(.15em);
\node[color=ribboncolour] at (0:3em) {$\bullet$};
\draw[-, line width=.85, color=ribboncolour, line cap=round] (0:3em) to ($(0:3em)+(60:2em)$)(0:3em) to ($(0:3em)+(-70:1.5em)$)  (0:3em) to ($(0:3em)+(-110:1.5em)$);
\node[font=\footnotesize] at (3.8,0) {$v$};
\node[font=\footnotesize] at (1.5em,-.5em) {$e_{\gamma}$};
\node at (5.5em,0) {$\leftsquigarrow$};
\draw[fill=ribboncolour, color=ribboncolour] (3em,-.8em) circle(.15em);
\node[font=\footnotesize]  at (1.5em, -2.5em) {$\mathbb G (\Delta)$};
\draw[line width=0, fill=ribboncolour, draw=ribboncolour, opacity=.5] (120:2em)--(0,0) -- (3em,0) -- ($(3em,0)+(60:2em)$) -- cycle;
\end{scope}
\begin{scope}[xshift=8em]
\node[font=\footnotesize] at (-180:.8em) {$\vtimes_x$};
\node[color=ribboncolour] at (0,0) {$\bullet$};
\draw[-, line width=.85, color=ribboncolour, line cap=round] (0,0) to (120:2em) (0,0) to (0:3em);
\draw[fill=ribboncolour, color=ribboncolour] (70:.5em) circle(.15em);
\node[color=ribboncolour] at (0:3em) {$\bullet$};
\node[color=ribboncolour] at (3em,-1em) {$\bullet$};
\node[font=\footnotesize] at (3.8em, -1em) {$v$};
\draw[-, line width=.85, color=ribboncolour, line cap=round] (0:3em) to ($(0:3em)+(60:2em)$);
\draw[-, line width=.85, color=ribboncolour, line cap=round] (3em, -1em) to ($(3em, -1em)+(-70:1.5em)$) (3em,-1em) to ($(3em, -1em)+(-110:1.5em)$) (3em,-1em) to (0:3em);
\node[font=\footnotesize] at (3.8,0) {$\vcirc$};
\node[font=\footnotesize] at (1.5em,-.5em) {$e_{\gamma}$};
\node at (5.5em,0) {$\rightsquigarrow$};
\draw[fill=ribboncolour, color=ribboncolour] (3em,-1.8em) circle(.15em);
\draw[line width=0, fill=ribboncolour, draw=ribboncolour, opacity=.5] (120:2em)--(0,0) -- (3em,0) -- ($(3em,0)+(60:2em)$) -- cycle;
\end{scope}
\begin{scope}[xshift=16em]
\node[font=\footnotesize] at (-180:.8em) {$\vtimes_x$};
\node[color=ribboncolour] at (0,0) {$\bullet$};
\draw[-, line width=.85, color=ribboncolour, line cap=round] (0,0) to (120:2em);
\draw[fill=ribboncolour, color=ribboncolour] (90:.8em) circle(.15em);
\node[color=ribboncolour] at (0,-1em) {$\bullet$};
\node[font=\footnotesize] at (.8em, -1em) {$v$};
\draw[-, line width=.85, color=ribboncolour, line cap=round] (0:0em) to (60:2em);
\draw[-, line width=.85, color=ribboncolour, line cap=round] (0em, -1em) to ($(0em, -1em)+(-70:1.5em)$) (0em,-1em) to ($(0em, -1em)+(-110:1.5em)$) (0em,-1em) to (0:0em);
\draw[fill=ribboncolour, color=ribboncolour] (0em,-1.8em) circle(.15em);
\node at (2.5em,0) {$\leftsquigarrow$};
\draw[line width=0, fill=ribboncolour, draw=ribboncolour, opacity=.5] (120:2em)--(0,0) -- (60:2em) -- cycle;
\end{scope}
\begin{scope}[xshift=21em]
\node[font=\footnotesize] at (-180:.8em) {$\vcirc$};
\node[color=ribboncolour] at (0,0) {$\bullet$};
\draw[-, line width=.85, color=ribboncolour, line cap=round] (0,0) to (120:2em) (0,0) to (0:3em);
\draw[fill=ribboncolour, color=ribboncolour] ($(3em,0)+(110:.5em)$) circle(.15em);
\node[color=ribboncolour] at (0:3em) {$\bullet$};
\node[color=ribboncolour] at (0em,-1em) {$\bullet$};
\draw[-, line width=.85, color=ribboncolour, line cap=round] (0:3em) to ($(0:3em)+(60:2em)$);
\draw[-, line width=.85, color=ribboncolour, line cap=round] (0em, -1em) to ($(0em, -1em)+(-70:1.5em)$) (0em,-1em) to ($(0em, -1em)+(-110:1.5em)$) (0em,-1em) to (0:0em);
\node[font=\footnotesize] at (3.8,0) {$\vtimes_x$};
\node[font=\footnotesize] at (.8em, -1em) {$v$};
\node[font=\footnotesize] at (1.5em,-.5em) {$e_{\gamma'}$};
\draw[fill=ribboncolour, color=ribboncolour] (0em,-1.8em) circle(.15em);
\node at (5.5em,0) {$\rightsquigarrow$};
\draw[line width=0, fill=ribboncolour, draw=ribboncolour, opacity=.5] (120:2em)--(0,0) -- (3em,0) -- ($(3em,0)+(60:2em)$) -- cycle;
\end{scope}
\begin{scope}[xshift=29em]
\node[font=\footnotesize] at (-180:.8em) {$v$};
\node[color=ribboncolour] at (0,0) {$\bullet$};
\draw[-, line width=.85, color=ribboncolour, line cap=round] (0,0) to (120:2em) (0,0) to (0:3em);
\draw[fill=ribboncolour, color=ribboncolour] ($(3em,0)+(110:.5em)$) circle(.15em);
\node[color=ribboncolour] at (0:3em) {$\bullet$};
\draw[-, line width=.85, color=ribboncolour, line cap=round] (0:3em) to ($(0:3em)+(60:2em)$);
\draw[-, line width=.85, color=ribboncolour, line cap=round] (0em, 0em) to ($(0em, 0em)+(-70:1.5em)$) (0em,0em) to ($(0em, 0em)+(-110:1.5em)$);
\node[font=\footnotesize] at (3.8,0) {$\vtimes_x$};
\node[font=\footnotesize] at (1.5em,-.5em) {$e_{\gamma'}$};
\draw[fill=ribboncolour, color=ribboncolour] (0em,-.8em) circle(.15em);
\node[font=\footnotesize]  at (1.5em, -2.5em) {$\mathbb G (\Delta')$};
\draw[line width=0, fill=ribboncolour, draw=ribboncolour, opacity=.5] (120:2em)--(0,0) -- (3em,0) -- ($(3em,0)+(60:2em)$) -- cycle;
\end{scope}
\end{tikzpicture}
\caption{Flipping $\vtimes_x$ and its adjacent vertex in the $2$-cell $d_x$.}
\label{fig:operation4}
\end{figure}
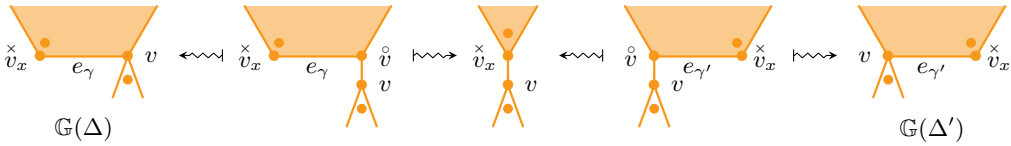

Iterating the above operation we may obtain an admissible dissection $\Delta_0$ whose ribbon complex is illustrated in Fig.~\ref{fig:ribbon}. 
\end{proof}

\newcommand{\flipdiagram}{%
\draw[line width=.5pt] (-4em,0) ++( 60:1.5em) arc[start angle= 60, end angle=-15, radius=1.5em];
\draw[line width=.5pt]  (4em,0) ++(195:1.5em) arc[start angle=195, end angle=120, radius=1.5em];
\node[font=\scriptsize,shape=circle,scale=.6] (X) at (0,4em) {};
\node[font=\scriptsize] at (0,4em) {$\times$};
\node[font=\scriptsize,color=stopcolour] at ($(0,4em)+(270:.65em)$) {$\bullet$};
\draw[line width=.75pt, color=arccolour, line cap=round] (-4em,0) ++( 45:1.5em) -- (X);
\draw[line width=.75pt, color=arccolour, line cap=round] ( 4em,0) ++(135:1.5em) -- (X);
\draw[line width=.75pt, color=arccolour, line cap=round] (X) -- ++(165:2em);
\draw[dash pattern=on 0pt off 1.3pt, line width=.8pt, line cap=round, color=arccolour] (X) ++(165:2.1em) -- ++(165:.4em);
\draw[line width=.75pt, color=arccolour, line cap=round] (X) -- ++(105:2em);
\draw[dash pattern=on 0pt off 1.3pt, line width=.8pt, line cap=round, color=arccolour] (X) ++(105:2.1em) -- ++(105:.4em);
\draw[line width=.75pt, color=arccolour, line cap=round] (X) -- ++(15:2em);
\draw[dash pattern=on 0pt off 1.3pt, line width=.8pt, line cap=round, color=arccolour] (X) ++(15:2.1em) -- ++(15:.4em);
\node[font=\scriptsize] at ($(0,4em)+(60:1em)$) {.};
\node[font=\scriptsize] at ($(0,4em)+(75:1em)$) {.};
\node[font=\scriptsize] at ($(0,4em)+(45:1em)$) {.};
}

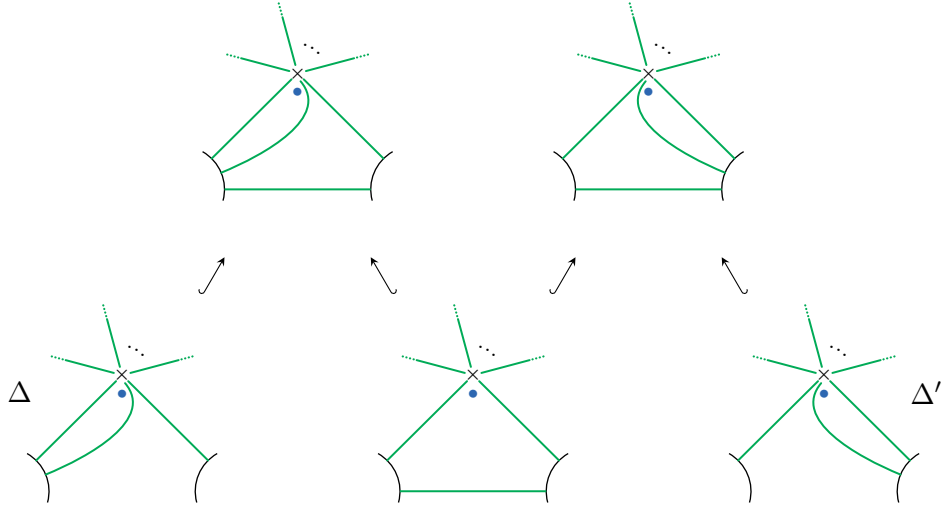
\begin{figure}
\begin{tikzpicture}[x=1em,y=1em,decoration={markings,mark=at position 0.55 with {\arrow[black]{Stealth[length=4.8pt]}}}]
\begin{scope}[xshift=-12em]
\flipdiagram
\draw[line width=.75pt, color=arccolour, line cap=round] (-4em,0) ++(22.5:1.5em) to[out=22.5, in=-50, looseness=1] (X.290);
\draw[{Hooks[right]}->, line width=.4pt] (0,2em) ++(60:5.5em) -- ++(60:1.5em);
\end{scope}
\begin{scope}[xshift=-6em,yshift=10.4em]
\flipdiagram
\node[]  at (-9.5em, -7em) {$\Delta$};
\draw[line width=.75pt, color=arccolour, line cap=round] (-4em,0) ++(22.5:1.5em) to[out=22.5, in=-50, looseness=1] (X.290);
\draw[line width=.75pt, color=arccolour] (-2.5em,0) -- (2.5em,0);
\end{scope}
\begin{scope}
\flipdiagram
\draw[line width=.75pt, color=arccolour] (-2.5em,0) -- (2.5em,0);
\draw[{Hooks[left]}->, line width=.4pt] (0,2em) ++(120:5.5em) -- ++(120:1.5em);
\draw[{Hooks[right]}->, line width=.4pt] (0,2em) ++(60:5.5em) -- ++(60:1.5em);
\end{scope}
\begin{scope}[xshift=6em,yshift=10.4em]
\flipdiagram
\node[]  at (9.5em, -7em) {$\Delta'$};
\draw[line width=.75pt, color=arccolour, line cap=round] ( 4em,0) ++(157.5:1.5em) to[out=157.5, in=-130, looseness=1] (X.250);
\draw[line width=.75pt, color=arccolour] (-2.5em,0) -- (2.5em,0);
\end{scope}
\begin{scope}[xshift=12em]
\flipdiagram
\draw[line width=.75pt, color=arccolour, line cap=round] ( 4em,0) ++(157.5:1.5em) to[out=157.5, in=-130, looseness=1] (X.250);
\draw[{Hooks[left]}->, line width=.4pt] (0,2em) ++(120:5.5em) -- ++(120:1.5em);
\end{scope}
\end{tikzpicture}
\caption{A zigzag of Morita equivalences induced by inclusions of dissections showing that a flip (from lower left to lower right) induces a Morita equivalence of the associated A$_\infty$ categories}
\label{fig:zigzag}
\end{figure}

\begin{theorem}
\label{theorem:moritaequivalenceadmissible}
Let $\mathbf S = (S, \Sigma, \eta)$ be a graded orbifold surface with stops. If $\Delta$ and $\Delta'$ are two admissible dissections of $\mathbf S$, then $\mathbf A_{\Delta}$ and $\mathbf A_{\Delta'}$ are Morita equivalent.
\end{theorem}

\begin{proof}
By Lemma \ref{lemma:4gon} as well as its proof there is an admissible DG dissection $\Delta_0$, which is related to $\Delta$ and $\Delta'$ by edge contractions and expansions, compare to the proof of  \cite[Proposition 3.3]{haidenkatzarkovkontsevich}. As a result, the DG category $\mathbf A_{\Delta_0}$ is Morita equivalent to both $\mathbf A_{\Delta}$ and $\mathbf A_{\Delta'}$. 
\end{proof}

\subsection{Universal property}

The following result now shows that the A$_\infty$ category $\mathbf A_\Delta$ (see Proposition \ref{proposition:ainfinity}) gives an explicit model of the partially wrapped Fukaya category $\mathcal W (\mathbf S)$.

\begin{theorem}\label{theorem:moritapartiallywrapped}
Let $\mathbf S = (S, \Sigma, \eta)$ be a graded orbifold surface with stops. If $\Delta$ is an admissible dissection of $\mathbf S$, then $\tw (\mathbf A_\Delta)^\natural$ is equivalent to the homotopy colimit $\mathbf \Gamma (\tw (\mathcal E_\Delta)^\natural) \simeq \tw (\mathbf \Gamma (\mathcal E_\Delta))^\natural$. In particular, $\mathbf A_\Delta$ is Morita equivalent to $\mathbf \Gamma (\mathcal E_\Delta)$ and
\[
\mathcal W (\mathbf S) = \H^0 (\tw (\mathbf \Gamma (\mathcal E_\Delta))^\natural) \simeq \H^0 (\tw (\mathbf A_\Delta)^\natural).
\]
\end{theorem}

\begin{proof}
By Theorem \ref{theorem:moritaequivalenceadmissible} we obtain a Morita equivalence between $\mathbf A_\Delta$ and $\mathbf A_{\Delta_0}$ for a special dissection $\Delta_0$. That is, $\H^0(\tw(\mathbf A_{\Delta})^\natural) \simeq \H^0(\tw(\mathbf A_{\Delta_0})^\natural).$
From Proposition \ref{proposition:globalsection} below it follows that 
\[
\H^0 (\tw (\mathbf \Gamma (\mathcal E_{\Delta_0}))^\natural) \simeq \H^0(\tw(\mathbf A_{\Delta_0})^\natural).
\]
Then combining the above two equivalences with Theorem \ref{theorem:moritaequivalenceglobalsection} we obtain the desired equivalence. 
\end{proof}

Since the Morita equivalence class of $\mathbf \Gamma (\mathcal E_\Delta)$ resp.\ $\mathbf A_\Delta$ does not depend on the particular weakly admissible resp.\ admissible dissection (Theorems \ref{theorem:moritaequivalenceglobalsection} resp.\ \ref{theorem:moritaequivalenceadmissible}), it suffices to prove the theorem for any particular dissection.

We now show that for the (formal) admissible dissection $\Delta_0$ whose associated ribbon complex $\mathbb G (\Delta_0)$ is illustrated in Fig.~\ref{fig:ribbon}, the category $\mathbf \Gamma (\mathcal E_{\Delta_0})$ of global sections of $\mathcal E_{\Delta_0}$ is Morita equivalent to the A$_\infty$ category $\mathbf A_{\Delta_0}$.

Let us describe the cosheaf $\mathcal E_{\Delta_0}$. Under the above assumption, for each $x \in \Sing (S)$ the A$_\infty$ category $\mathcal E_{\Delta_0} (U_{\vtimes_x})$ is given by the following DG Kronecker quiver  
\begin{equation}
\label{eq:kronecker}
\begin{tikzpicture}[baseline=-2.6pt,description/.style={fill=white,inner sep=1pt,outer sep=0}]
\matrix (m) [matrix of math nodes, row sep=2.5em, text height=1.5ex, column sep=3em, text depth=0.25ex, ampersand replacement=\&, inner sep=3.5pt]
{
0 \& 1 \\
};
\path[->,line width=.4pt, font=\scriptsize, transform canvas={yshift=.5ex}] (m-1-1) edge node[above=.01ex] {$q_x$} (m-1-2);
\path[->,line width=.4pt, font=\scriptsize, transform canvas={yshift=-.5ex}] (m-1-1) edge node[below=.01ex] {$p_x$} (m-1-2);
\end{tikzpicture}
\end{equation}
with $|q_x| = |p_x| + 1$ and differential $\mutimes_1(\s p_x) = \s q_x$. The A$_\infty$ category $\mathcal E_{\Delta_0}(U_{\vtimes_x} \cap U_{\vbullet_0})$ is given by the $\mathrm A_2$ quiver 
\[
\begin{tikzpicture}[baseline=-2.6pt,description/.style={fill=white,inner sep=1pt,outer sep=0}]
\matrix (m) [matrix of math nodes, row sep=2.5em, text height=1.5ex, column sep=2.8em, text depth=0.25ex, ampersand replacement=\&, inner sep=3.5pt]
{
0 \& 1 \\
};
\path[->,line width=.4pt, font=\scriptsize] (m-1-1) edge node[above=-.4ex] {$q_x$} (m-1-2);
\end{tikzpicture}
\]

To the special vertex $\vbullet_0$ we associate the A$_\infty$ category $\mathcal E_{\Delta_0}(U_{\vbullet_0})$ given by the following graded quiver
\begin{equation}\label{equation:Anquvier}
\begin{tikzpicture}[x=1em,y=1em]
\begin{scope}
\node[shape=circle,inner sep=1.5pt] (0) at (0,0) {$\bullet$};
\node[shape=circle,inner sep=1.5pt] (1) at (4,0) {$\bullet$};
\node[shape=circle,inner sep=1.5pt] (d) at (8,0) {$...$};
\node[shape=circle,inner sep=1.5pt] (n) at (12,0) {$\bullet$};
\path[->, line width=.4pt] (0) edge node[font=\scriptsize, below=-.3ex] {$p_1$} (1) (1) edge node[font=\scriptsize, below=-.3ex] {$p_2$} (d) (d) edge node[font=\scriptsize, below=-.3ex] {$p_{n-1}$} (n);
\draw[dash pattern=on 0pt off 2pt, line width=.5pt, line cap=round] (1) ++(175:1.1em) arc[start angle=175, end angle=5, radius=1.1em] (7.8,0) ++(175:1.1em) arc[start angle=175, end angle=120, radius=1.1em] (8,0) ++(5:1.1em) arc[start angle=5, end angle=60, radius=1.1em];
\end{scope}
\end{tikzpicture}
\end{equation}
where $n = \val (\vbullet_0)$. For the other vertices $\vbullet_i$ with $i \neq 0$, we have $\mathcal E_{\Delta_0}(U_{\vbullet_i}) \simeq \Bbbk$.

Then the cosheaf $\mathcal E_{\Delta_0}$ is given as follows
\begin{equation}\label{equation:cosheafdelta0}
\prod_{u, v, w \in V} \!\! \mathcal E_{\Delta_0} (U_u \cap U_v \cap U_w) \triplerightarrow \! \prod_{v, w \in V} \mathcal E_{\Delta_0} (U_v \cap U_w) \doublerightarrow \prod_{v \in V} \mathcal E_{\Delta_0} (U_v)
\end{equation}
where the categories are either $\Bbbk$, the DG Kronecker quiver \eqref{eq:kronecker}, the A$_2$ quiver or the graded A$_n$ quiver with radical square zero \eqref{equation:Anquvier}.

We obtain the following result.

\begin{proposition}
\label{proposition:globalsection}
Let $\Delta_0$ be an admissible dissection such that the associated ribbon complex $\mathbb G (\Delta_0)$ is illustrated in  Fig.~\ref{fig:ribbon}. Then the category $\mathbf \Gamma (\mathcal E_{\Delta_0})$ of global sections of the cosheaf $\mathcal E_{\Delta_0}$ is Morita equivalent to $\mathbf A_{\Delta_0}$ as defined in Section \ref{section:ainfinity}.
\end{proposition}

\begin{proof}
Note that all the DG categories in \eqref{equation:cosheafdelta0} are cofibrant except $\mathcal E_{\Delta_0}(U_{\vbullet_0})$ in \eqref{equation:Anquvier} for the special vertex $\vbullet_0$. For this category we may take the cofibrant replacement $\widetilde{\mathcal E}_{\Delta_0}(U_{\vbullet_0})$ as in the proof of \cite[Proposition 3.7]{haidenkatzarkovkontsevich}, see also the proof of Proposition \ref{theorem:moritaequivalenceglobalsection} above.

Then the naive colimit $\mathbf \Gamma (\widetilde{\mathcal E}_{\Delta_0})$ of $\widetilde{\mathcal E}_{\Delta_0}$ is the homotopy colimit of $\widetilde{\mathcal E}_{\Delta_0}$ and thus of $\mathcal E_{\Delta_0}$.  Using the Bardzell resolution as in \cite{tamaroff}, the naive colimit $\mathbf \Gamma (\widetilde{\mathcal E}_{\Delta_0})$ is a cofibrant replacement of the DG category $\mathbf A_{\Delta_0}$, which are Morita equivalent. 
\end{proof}

\begin{remark}
\label{remark:deformationinterpretation}
Note that the underlying graded category of $\mathbf A_\Delta$ is a graded gentle category (algebra) $\Bbbk Q/I$, although it is defined on an orbifold surface. Moreover, if we denote by $\widehat{\mathbf A}_\Delta$ the underlying graded category together with the products $\mubar_2$ and $\mucirc_n$ (without $\mutimes_n$), then $\widehat{\mathbf A}_\Delta$ is an A$_\infty$ category as in \cite{haidenkatzarkovkontsevich}. Then $\widehat{\mathbf A}_\Delta$ describes the partially wrapped Fukaya category of the graded smooth surface $\widehat S$ which is obtained from the orbifold surface $S$ by replacing each orbifold points $x$ by a boundary component $\partial \mathbb D_x$ with one stop (see Remark \ref{remark:compactification}).

The A$_\infty$ category $\mathbf A_\Delta$ (with the additional higher multiplications $\mutimes_{\geq 1}$) can in fact be viewed as an A$_\infty$ deformation of $\widehat{\mathbf A}_\Delta$, since each minimal orbifold disk sequence $\theta$ of length $n$ corresponds to a $2$-cocycle
\begin{equation}
\label{eq:cocycle}
\mutimes_n^\theta \in \HH^2 (\widehat{\mathbf A}_\Delta, \widehat{\mathbf A}_\Delta) \simeq \HH^2 (\mathcal W (\widehat{\mathbf S}), \mathcal W (\widehat{\mathbf S})).
\end{equation}
Here the right-hand side is the (total) degree $2$ Hochschild cohomology of $\mathcal W (\widehat{\mathbf S})$ which, by definition, we take to be the Hochschild cohomology of its A$_\infty$ enhancement $\tw (\widehat{\mathbf A})$. From a geometric point of view, algebraic A$_\infty$ deformations of $\mathcal W (\widehat{\mathbf S})$ thus correspond to ``compactifications'' of certain boundary components with one stop and winding number $1$ into orbifolds points, see Fig.~\ref{fig:compactification}. The details of the full A$_\infty$ deformation theory of $\mathcal W (\widehat{\mathbf S})$ will be given in \cite{barmeierschrollwang}.
\end{remark}

\subsection{A$_\infty$ categories without admissibility}
\label{subsection:notstrong}

In the last subsection we briefly discuss how to obtain higher products for weakly admissible dissections which are not admissible.

Let $\mathbf S = (S, \Sigma, \eta)$ and let $\Delta$ be a weakly admissible dissection of $\mathbf S$ which is not admissible. Similar to the case of admissible dissections, we may define a graded category $\mathbf B_\Delta$ whose objects are given by the arcs in $\Delta$ and whose morphisms are given by the linear combinations of the boundary/orbifold paths (as well as the identity for each arc) as in \S\S\ref{subsection:objectsandmorphisms}--\ref{subsection:higher}. 

The notions of smooth/orbifold disk sequences (Definition \ref{definition:disksequence}) also make sense for weakly admissible dissections which are not admissible. Thus the smooth and orbifold disk sequences give rise to the higher products $\mutimes_n$ and $\mucirc_n$, respectively, similar to the admissible case. However, the following example shows that $\mathbf B_\Delta$ together with these higher products alone does {\it not} satisfy the A$_\infty$ relations if $\Delta$ is not admissible.

\begin{example}
\label{example:notstrong}
Let $\mathbf S = (S, \Sigma, \eta)$ be the orbifold disk with two orbifold points and two stops in the boundary and let $\eta$ be any line field on $S$. Consider the dissection $\Delta$ illustrated in Fig.~\ref{fig:nonadmissibledissection}. Then there are only two vertices in the boundary of the $2$-cell $d_x$ in $\mathbb G (\Delta)$ which are of type $\vtimes$, so that $\Delta$ is not admissible (cf.\ Definition \ref{definition:admissible}). 

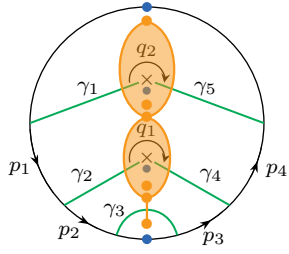
\begin{figure}
\begin{tikzpicture}[x=1em,y=1em,decoration={markings,mark=at position 0.55 with {\arrow[black]{Stealth[length=4.2pt]}}}]
\begin{scope}[xshift=0]
\draw[line width=.5pt] circle(4em);
\node[font=\scriptsize,shape=circle,scale=.6,fill=white] (X) at (90:1.5em) {};
\node[font=\scriptsize] at (90:1.5em) {$\times$};
\node[font=\scriptsize,shape=circle,scale=.6,fill=white] (Y) at (-90:1.2em) {};
\node[font=\scriptsize] at (-90:1.2em) {$\times$};
\draw[->, line width=.5pt] ($(-90:1.2em)+(195:.6em)$) arc[start angle=195, end angle=-15, radius=.6em];
\node[font=\scriptsize] at ($(-90:1.2em)+(90:.95em)$){$q_1$};
\node[font=\scriptsize,color=stopcolour] at (-90:1.6em) {$\bullet$};
\node[font=\scriptsize,color=stopcolour] at (90:1.1em) {$\bullet$};
\draw[->, line width=.5pt] ($(90:1.5em)+(190:.6em)$) arc[start angle=190, end angle=-10, radius=.6em];
\node[font=\scriptsize] at ($(90:1.5em)+(90:1em)$) {$q_2$};
\draw[line width=.75pt,color=arccolour] (-135:4em) to (Y);
\draw[line width=.75pt,color=arccolour] (-45:4em) to (Y);
\path[line width=.75pt,out=105,in=75,looseness=1.5,color=arccolour] (285:4em) edge (255:4em);
\draw[line width=.75pt,color=arccolour] (180:4em) to (X);
\draw[line width=.75pt,color=arccolour] (0:4em) to (X);
\draw[line width=0pt,postaction={decorate}] (200-17:4em) arc[start angle=200-17, end angle=200+17, radius=4em];
\node[font=\scriptsize] at (200:4.6em) {$p_1$};
\draw[line width=0pt,postaction={decorate}] (235:4em) arc[start angle=235, end angle=235+10, radius=4em];
\node[font=\scriptsize] at (235:4.6em) {$p_2$};
\draw[line width=0pt,postaction={decorate}] (305-10:4em) arc[start angle=305-10, end angle=305+5, radius=4em];
\node[font=\scriptsize] at (300:4.6em) {$p_3$};
\draw[line width=0pt,postaction={decorate}] (315+17:4em) arc[start angle=315+17, end angle=360-17+5, radius=4em];
\node[font=\scriptsize] at (340:4.75em) {$p_4$};
\node[font=\scriptsize] at (-2, 1.2em) {$\gamma_1$};
\node[font=\scriptsize] at (-2.2, -1.8em) {$\gamma_2$};
\node[font=\scriptsize] at (-1.1, -3em) {$\gamma_3$};
\node[font=\scriptsize] at (2.2, -1.8em) {$\gamma_4$};
\node[font=\scriptsize] at (2, 1.2em) {$\gamma_5$};
\draw[fill=stopcolour, color=stopcolour] (90:4em) circle(.15em);
\draw[fill=stopcolour, color=stopcolour] (-90:4em) circle(.15em);
\node[color=ribboncolour] at (90:.6em) {$\bullet$};
\node[color=ribboncolour] at (-90:2.2em) {$\bullet$};
\node[color=ribboncolour] at (0,.2em) {$\bullet$};
\node[color=ribboncolour] at (90:3.5em) {$\bullet$};
\node[color=ribboncolour] at (-90:2.6em) {$\bullet$};
\node[color=ribboncolour] at (-90:3.5em) {$\bullet$};
\draw[-, line width=.85, color=ribboncolour, line cap=round, bend left=75] (-90:2.6em) to (0,.2em);
\draw[-, line width=.85, color=ribboncolour, line cap=round, bend right=75] (-90:2.6em) to (0,.2em);
\draw[-, line width=.85, color=ribboncolour, line cap=round, bend left=75] (90:3.5em) to (0,.2em);
\draw[-, line width=.85, color=ribboncolour, line cap=round, bend right=75] (90:3.5em) to (0,.2em);
\draw[-, line width=.85, color=ribboncolour, line cap=round] (-90:2.6em) to (-90:3.5em);
\draw[line width=0, fill=ribboncolour, draw=ribboncolour, opacity=.5] (-90:2.6em) to[bend left=75] (0,.2em) to[bend left=75] (-90:2.6em);
\draw[line width=0, fill=ribboncolour, draw=ribboncolour, opacity=.5] (90:3.5em) to[bend left=75] (0,.2em) to[bend left=75] (90:3.5em);
\end{scope}
\end{tikzpicture}
\caption{A weakly admissible dissection which is not admissible.}
\label{fig:nonadmissibledissection}
\end{figure}

The only nontrivial compositions $\mubar_2$ on $\mathbf B_\Delta$ are given by
\begin{align*}
\mubar_2(\s p_2 \otimes \s p_1) = (-1)^{|p_1|} \s p_2p_1 \quad \text{and} \quad
\mubar_2(\s p_4 \otimes \s p_3) = (-1)^{|p_3|} \s p_4 p_3. 
\end{align*}
There are no smooth disk sequences, but there are two minimal orbifold disk sequences
\[
\gamma_2\overset{p_2}{\smile} \gamma_3 \overset{p_3}{\smile}\gamma_4 \quad \text{and} \quad \gamma_1\overset{p_1}{\smile} \gamma_2 \overset{q_1}{\smile} \gamma_4 \overset{p_4}{\smile} \gamma_5
\]
which give the following two higher products on $\mathbf B_\Delta$
\begin{align*}
\mutimes_2( \s p_3 \otimes \s p_2) = \s q_1 \quad \text{and} \quad
 \mutimes_3(\s p_4 \otimes \s q_1 \otimes \s p_1) = \s q_2.
\end{align*}
However, when equipped with these higher products alone, $\mathbf B_\Delta$ does not satisfy the A$_\infty$ relations. For instance, the terms appearing in the A$_\infty$ relation applied to $\s p_4 \otimes \s p_3 \otimes \s p_2 \otimes \s p_1$ are
\begin{align*}
 \mutimes_3(\s p_4 \otimes  \s p_3 \otimes \mubar_2(\s p_2 \otimes \s p_1)) & =(-1)^{|p_1|} \mutimes_3(\s p_4 \otimes  \s p_3 \otimes \s p_2 p_1) \\
 (-1)^{|\s p_1|}  \mutimes_3(\s p_4 \otimes \mutimes_2( \s p_3 \otimes \s p_2) \otimes \s p_1) &= (-1)^{|\s p_1|} \s q_2\\
(-1)^{|\s p_2|+|\s p_1|} \mutimes_3(\mubar_2(\s p_4 \otimes \s p_3) \otimes \s p_2 \otimes \s p_1) & =  (-1)^{|\s p_2|+|\s p_1|+|p_3|} \mutimes_3(\s p_4p_3 \otimes \s p_2 \otimes \s p_1).
\end{align*}
To obtain a well-defined A$_\infty$ structure on $\mathbf B_\Delta$ we need to define at least one additional higher product. For example, one may define
\begin{equation}
\label{eq:adhoc}
\begin{aligned}
\mutimes_3 (\s p_4 \otimes  \s p_3 \otimes \s p_2 p_1) &= \s q_2 \\
\mutimes_3 (\s p_4 p_3 \otimes \s p_2 \otimes \s p_1) &= 0
\end{aligned}
\end{equation}
which may be understood as the higher product associated to the gluing of the two minimal orbifold disk sequences. It is straightforward to check that with this additional higher product $\mathbf B_\Delta$ does satisfy the A$_\infty$ relations. Similar to the proof of Theorem \ref{theorem:moritaequivalenceadmissible} one may show that $\mathbf B_\Delta$ is Morita equivalent to $\mathbf A_{\Delta'}$ for any admissible dissection $\Delta'$ so that $\H^0 (\tw (\mathbf B_\Delta)^\natural) \simeq \mathcal W (\mathbf S)$.
\end{example}

Although in concrete examples such as Example \ref{example:notstrong}, it is possible to find explicit additional higher products such as \eqref{eq:adhoc} by hand, there is a natural systematic way to construct them via deformation theory. Recall from Remarks \ref{remark:compactification} and \ref{remark:deformationinterpretation} that we may view a graded orbifold surface with stops $\mathbf S$ as a ``compactification'' of a graded smooth surface $\widehat{\mathbf S}$ obtained from $\mathbf S$ by removing a small open disk $\mathbb D_x$ containing $x \in \Sing (S)$ for each orbifold point. 
Via an explicit homotopy transfer as in \cite{barmeierwang}, the cocycles $\mutimes_n^\theta$ \eqref{eq:cocycle} produce explicit A$_\infty$ products on $\mathbf B_\Delta$. This gives a way of defining an explicit A$_\infty$ category also for weakly admissible dissections which are not admissible. We refer to \cite{barmeierschrollwang} for more details.

\section{Formal generators and graded skew-gentle algebras}
\label{section:formal}

In this section, we study formal generators of the partially wrapped Fukaya category $\mathcal W (\mathbf S)$. In particular, we give further conditions on $\Delta$ which characterize when the A$_\infty$ category $\mathbf A_\Delta$ is a DG category and when it is formal (Definition \ref{definition:formal}). In the latter case $\mathbf A_\Delta$ is a formal generator of $\mathcal W (\mathbf S)$. We show that $\mathcal W (\mathbf S)$ always admits a formal generator whose cohomology is a graded skew-gentle algebra. This particular type of formal generator is closely related to surface dissections of orbifold surfaces studied for example in \cite{amiotbruestle,labardinifragososchrollvaldivieso,hezhouzhu,qiuzhangzhou}. However, our classification of formal generators produces many more algebras which are not skew-gentle themselves, but which are derived equivalent to skew-gentle algebras. In \S\ref{subsection:orbifolddiskthree} we use this to give new derived equivalences between associative algebras which are derived equivalent to skew-gentle algebras. Throughout this section, we consider weakly admissible dissections $\Delta$ containing no isotopic arcs as on the right of Fig.~\ref{fig:differential}, so that the resulting algebras are basic.

\subsection{DG and formal dissections}
\label{subsection:DGformal}

Let us introduce two further notions of weakly admissible dissections. A DG dissection $\Delta$ is such that the corresponding A$_\infty$ category $\mathbf A_\Delta$ is a DG category (Proposition \ref{proposition:dg}) and a formal dissection is a slightly stronger notion for which $\mathbf A_\Delta$ is a formal DG category (Theorem \ref{theorem:formaldg}).

\begin{definition}
\label{definition:DGformal}
Let $\mathbf S = (S, \Sigma, \eta)$ be a graded orbifold surface with stops. A {\it DG dissection} on $\mathbf S$ is a weakly admissible dissection $\Delta$ whose ribbon complex $\mathbb G (\Delta)$ contains no vertices of type $\vcirc$ and, moreover, $\val(\vtimes_x) \leq 3$ for each $x\in \Sing (S)$. The latter condition is equivalent to $\val(\vtimes_x)\in \{ 2, 3\}$ (cf.\ Remark \ref{remark:valency}). 

A {\it formal dissection} is a DG dissection $\Delta$ satisfying the following additional condition: 
If the orbifold path parallel to $p$ has length strictly greater than $1$, then $p$ cannot be concatenated with any other non-trivial path  as illustrated in Fig.~\ref{fig:dgformal}.  Dually, if $\val(\vtimes_x)=2$ and the number of vertices in the boundary of $d_x$ is strictly greater than $2$ then the missing relations at the two vertices (i.e.\ $u_1$ and $u_2$) adjacent to $\vtimes_x$ appear next to the edges connecting to $\vtimes_x$ as is also illustrated in Fig.~\ref{fig:dgformal}. See Example \ref{example:non-formaldg} for a DG dissection which is not formal. 
\end{definition}

\begin{figure}
\begin{tikzpicture}[x=1em,y=1em,decoration={markings,mark=at position 0.5 with {\arrow[black]{Stealth[length=4.8pt]}}}]
\begin{scope}
\draw[line width=.5pt,postaction={decorate}] ($(252:4.5)+(-1.5,0)$) -- ($(288:4.5)+(1.5,0)$);
\node[font=\scriptsize,color=stopcolour] at ($(252:4.5)+(-.75,0)$) {$\bullet$};
\node[font=\scriptsize,color=stopcolour] at ($(288:4.5)+(.75,0)$) {$\bullet$};
\node[font=\scriptsize] at (0,0) {$\times$};
\node[font=\scriptsize,shape=circle,scale=.6] (X) at (0,0) {};
\draw[line width=.75pt, color=arccolour, line cap=round] (X) -- (252:4.5);
\draw[line width=.75pt, color=arccolour, line cap=round] (X) -- (288:4.5);
\draw[line width=.75pt, color=arccolour, line cap=round] (X) -- (190:2em);
\draw[line width=.75pt, color=arccolour, line cap=round] (X) -- (128:2em);
\draw[line width=.75pt, color=arccolour, line cap=round] (X) -- (0:2em);
\node[font=\scriptsize] at (79:.7) {.};
\node[font=\scriptsize] at (59:.7) {.};
\node[font=\scriptsize] at (39:.7) {.};
\node[font=\scriptsize,left=-.3ex] at (252:3.5) {$\gamma_0$};
\node[font=\scriptsize,right=-.3ex] at (288:3.5) {$\gamma_1$};
\node[font=\scriptsize,color=stopcolour] at (270:1em) {$\bullet$};
\node[font=\scriptsize] at (270:5) {$p$};
\node[circle, fill=ribboncolour, minimum size=.3em, inner sep=0]  at (-90:3em) {};
\node[color=ribboncolour] at (-90:2.5em) {$\bullet$};
\node[font=\scriptsize] at ($(-35:2em)+(0:1em)$) {$u_2$};
\node[font=\scriptsize] at ($(220:2em)+(0:-1em)$) {$u_1$};
\node[color=ribboncolour] at ($(-35:2em) +(-90:.5em)$) {$\bullet$};
\node[color=ribboncolour] at ($(220:2em)+(-90:.5em)$) {$\bullet$};
\node[circle, fill=ribboncolour, minimum size=.3em, inner sep=0]  at (165:2em) {};
\node[circle, fill=ribboncolour, minimum size=.3em, inner sep=0]  at (220:2em) {};
\node[circle, fill=ribboncolour, minimum size=.3em, inner sep=0]  at (-35:2em) {};
\draw[-, line width=.85, color=ribboncolour, line cap=round] (220:2em) to ($(220:2em)+(230:1em)$);
\draw[-, line width=.85, color=ribboncolour, line cap=round] (220:2em) to ($(220:2em)+(145:1em)$);
\draw[-, line width=.85, color=ribboncolour, line cap=round] (-35:2em) to ($(-35:2em) +(45:1em)$);
\draw[-, line width=.85, color=ribboncolour, line cap=round] (-35:2em) to ($(-35:2em) +(-35:1em)$);
\draw[-, line width=.85, color=ribboncolour, line cap=round] (50:2em) to  (-35:2em) to (-90:3em)  to (220:2em) to (165:2em) to (90:2em);
\draw[line width=0, fill=ribboncolour, draw=ribboncolour, opacity=.5] (50:2em) to  (-35:2em) to (-90:3em)  to (220:2em) to (165:2em) to (90:2em) --cycle;
\end{scope}
\begin{scope}[xshift=12em]
\node[font=\scriptsize,color=stopcolour] at (-90:4.6) {$\bullet$};
\node[font=\scriptsize] at (0,0) {$\times$};
\node[font=\scriptsize,shape=circle,scale=.6] (X) at (0,0) {};
\node[font=\scriptsize,shape=circle,scale=.6] (Y) at (-90:4.2) {};
\node[font=\scriptsize] at (-90:4.2) {$\times$};
\draw[line width=.75pt, color=arccolour, line cap=round, bend left=45] (X) to (Y);
\draw[line width=.75pt, color=arccolour, line cap=round, bend right=45] (X) to (Y);
\draw[line width=.75pt, color=arccolour, line cap=round] (X) -- (190:2em);
\draw[line width=.75pt, color=arccolour, line cap=round] (X) -- (128:2em);
\draw[line width=.75pt, color=arccolour, line cap=round] (X) -- (0:2em);
\node[font=\scriptsize] at (79:.7) {.};
\node[font=\scriptsize] at (59:.7) {.};
\node[font=\scriptsize] at (39:.7) {.};
\node[font=\scriptsize,left=-.3ex] at (252:3.5) {$\gamma_0$};
\node[font=\scriptsize,right=-.3ex] at (288:3.5) {$\gamma_1$};
\node[font=\scriptsize,color=stopcolour] at (-90:.7em) {$\bullet$};
\node[font=\scriptsize] at (270:2.7) {$p$};
\draw[->, line width=.5pt] ($(-90:4.2)+(121:1)$) arc[start angle=121, end angle=55, radius=1];
\node[circle, fill=ribboncolour, minimum size=.3em, inner sep=0]  at (-90:2em) {};
\node[font=\scriptsize] at ($(-35:2em)+(0:1em)$) {$u_2$};
\node[font=\scriptsize] at ($(220:2em)+(0:-1em)$) {$u_1$};
\node[color=ribboncolour] at (-90:1.5em) {$\bullet$};
\node[color=ribboncolour] at ($(-35:2em) +(-90:.5em)$) {$\bullet$};
\node[color=ribboncolour] at ($(220:2em)+(-90:.5em)$) {$\bullet$};
\node[circle, fill=ribboncolour, minimum size=.3em, inner sep=0]  at (165:2em) {};
\node[circle, fill=ribboncolour, minimum size=.3em, inner sep=0]  at (220:2em) {};
\node[circle, fill=ribboncolour, minimum size=.3em, inner sep=0]  at (-35:2em) {};
\draw[-, line width=.85, color=ribboncolour, line cap=round] (220:2em) to ($(220:2em)+(230:1em)$);
\draw[-, line width=.85, color=ribboncolour, line cap=round] (220:2em) to ($(220:2em)+(150:1em)$);
\draw[-, line width=.85, color=ribboncolour, line cap=round] (-35:2em) to ($(-35:2em) +(45:1em)$);
\draw[-, line width=.85, color=ribboncolour, line cap=round] (-35:2em) to ($(-35:2em) +(-30:1em)$);
\draw[-, line width=.85, color=ribboncolour, line cap=round] (50:2em) to  (-35:2em) to (-90:2em)  to (220:2em) to (165:2em) to (90:2em);
\draw[line width=0, fill=ribboncolour, draw=ribboncolour, opacity=.5] (50:2em) to  (-35:2em) to (-90:2em)  to (220:2em) to (165:2em) to (90:2em) --cycle;
\end{scope}
\end{tikzpicture}
\caption{The local pictures in a formal dissection. The feature is that the path $p$ is maximal with respect to path concatenations.} 
\label{fig:dgformal}
\end{figure}
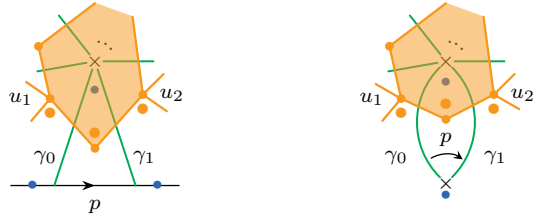

\begin{remark}
If $\Sing (S) = \varnothing$ so that $\mathbf S = (S, \Sigma, \eta)$ is a graded smooth surface with stops, then the condition on the vertices of type $\vtimes_x$ is vacuous and $\Delta$ is a formal dissection precisely when it dissects the surface into polygons which all contain exactly one boundary stop or one full boundary stop, i.e.\ when $\mathbb G (\Delta)$ only contains vertices of type $\vbullet$ and $\vodot$ and no vertices of type $\vcirc$. Definition \ref{definition:DGformal} thus generalizes the corresponding notion for smooth surfaces given in \cite[\S 3.4]{haidenkatzarkovkontsevich} (where this notion is called {\it full formal arc system}).
\end{remark}

\begin{remark}\label{remark:dgdissectiondg}
Let $\Delta$ be an admissible DG dissection. Since there are neither smooth disk sequences nor orbifold disk sequences of length $\geq 2$, see \S\ref{subsection:Aproduct}, it follows that there are no nonzero A$_\infty$ products $\mu_i$ with $i>2$ and thus $\mathbf A_\Delta$ is a DG category by Proposition \ref{proposition:ainfinity}. 
\end{remark}

\subsection{Explicit DG categories without admissibility}
\label{subsection:notstrongDG}

By Definition \ref{definition:DGformal} neither DG dissections nor formal dissections need be admissible. In fact, comparing to the definition of admissibility (Definition \ref{definition:admissible}) a DG dissection $\Delta$ is admissible if and only if the ribbon complex $\mathbb G (\Delta)$ contains no $2$-cells $d_x$ such that on the boundary of $d_x$ there are exactly $2$ vertices whose types are both $\vtimes$ and whose valencies are $2$ or $3$. The possible configurations of arcs which may be part of a formal dissection but are not part of an admissible dissection are illustrated in Fig.~\ref{fig:formalnonadmissible}.

\begin{figure}[ht]
\begin{tikzpicture}[x=1em,y=1em,decoration={markings,mark=at position 0.55 with {\arrow[black]{Stealth[length=4.8pt]}}}]
\begin{scope}
\draw[line width=.5pt] (-4em,0) ++(90:1.5em) arc[start angle=90, end angle=-15, radius=1.5em];
\draw[->, line width=.5pt] (-4em,0) ++(51:1.5em) arc[start angle=51, end angle=10, radius=1.5em];
\draw[->, line width=.5pt] (-4em,0) ++(70:1.5em) arc[start angle=70, end angle=50, radius=1.5em];
\draw[->, line width=.5pt] (3em,0) ++(83:1em) arc[start angle=83, end angle=-230, radius=1em];
\node[font=\scriptsize,shape=circle,scale=.6] (Y) at (3em,0) {};
\node[font=\scriptsize] at (3em, 0) {$\times$};
\node[font=\scriptsize,shape=circle,scale=.6] (X) at (0,4em) {};
\node[font=\scriptsize] at (0,4em) {$\times$};
\node[font=\scriptsize,color=stopcolour] at ($(0,4em)+(270:.65em)$) {$\bullet$};
\node[font=\scriptsize,color=stopcolour] at ($(3em,0)+(110:.8em)$) {$\bullet$};
\draw[->, line width=.5pt] ($(0,4em)+(225:.7em)$) arc[start angle=225, end angle=-45, radius=.7em];
\draw[line width=.75pt, color=arccolour, line cap=round] (-4em,0) ++(30:1.5em) -- (X);
\draw[line width=.75pt, color=arccolour, line cap=round] (Y) -- (X);
\draw[line width=.75pt, color=arccolour, line cap=round] (-4em,0) ++(0:1.5em) -- (Y);
\draw[line width=.75pt, color=arccolour, line cap=round] (-4em,0) ++(80:1.5em) to[out=85, in=85, looseness=2.8] (Y);
\draw[line width=.75pt, color=arccolour, line cap=round] (Y) -- ++(-75:1.5em);
\draw[line width=.75pt, color=arccolour, line cap=round] (Y) -- ++(10:1.5em);
\node[font=\scriptsize] at ($(0,4em)+(90:1.2em)$) {$q$};
\node[font=\scriptsize] at ($(-4em,0)+(75:.8em)$) {$r_1$};
\node[font=\scriptsize] at ($(-4em,0)+(5:.8em)$) {$r_2$};
\node[font=\scriptsize] at ($(3em,0)+(45:1.5em)$) {$p_1$};
\node[font=\scriptsize] at ($(3em,0)+(-35:1.5em)$) {$p_2$};
\node[font=\scriptsize] at ($(3em,0)+(160:1.5em)$) {$p_n$};
\node[font=\scriptsize] at ($(3em,0)+(-125:1.5em)$) {$.$};
\node[font=\scriptsize] at ($(3em,0)+(-131:1.5em)$) {$.$};
\node[font=\scriptsize] at ($(3em,0)+(-137:1.5em)$)  {$.$};

\node[font=\scriptsize] at (-.7em,2em) {$\vtimes$};
\node[font=\scriptsize] at (-.7em, 6em) {$\vtimes'$};
\node[color=ribboncolour] at (0,2.5em) {$\bullet$};
\node[color=ribboncolour] at (0,2em) {$\bullet$};
\node[color=ribboncolour] at (0,6em) {$\bullet$};
\node[color=ribboncolour] at (.5em,6em) {$\bullet$};
\node[color=ribboncolour] at (0,8em) {$\bullet$};
\node[color=ribboncolour] at (0,-1em) {$\bullet$};
\draw[-, line width=.85, color=ribboncolour, line cap=round, bend left=50] (0,2em) to (0, 6em);
\draw[-, line width=.85, color=ribboncolour, line cap=round, bend right=50] (0,2em) to (0, 6em);
\draw[-, line width=.85, color=ribboncolour, line cap=round] (0, 6em) to (0, 8em);
\draw[-, line width=.85, color=ribboncolour, line cap=round] (0, -1em) to (0, 2em);
\draw[line width=0, fill=ribboncolour, draw=ribboncolour, opacity=.5] (0,2em) to[bend left=50] (0,6em) to[bend left=50] (0,2em);
\end{scope}
\begin{scope}[xshift=22em]
\node[font=\scriptsize,shape=circle,scale=.6] (Y) at (0,0em) {};
\node[font=\scriptsize] at (0, 0) {$\times$};
\node[font=\scriptsize,shape=circle,scale=.6] (X) at (0,4em) {};
\node[font=\scriptsize] at (0,4em) {$\times$};
\node[font=\scriptsize,color=stopcolour] at ($(0,4em)+(270:.65em)$) {$\bullet$};
\node[font=\scriptsize,color=stopcolour] at ($(123:1.5em)$) {$\bullet$};
\draw[line width=.75pt, color=arccolour, line cap=round] (Y) to[out=60, in=-60, looseness=1] (X);
\draw[line width=.75pt, color=arccolour, line cap=round] (Y) to[out=120, in=-120, looseness=1] (X);
\draw[line width=.75pt, color=arccolour, line cap=round] (Y) to[out=35, in=145, looseness=75] (Y);
\draw[->, line width=.5pt] ($(0,4em)+(226:.6em)$) arc[start angle=226, end angle=-49, radius=.6em];
\node[font=\scriptsize] at ($(0,4em)+(90:1em)$) {$q$};
\draw[->, line width=.5pt] ($(0,0em)+(110:.8em)$) arc[start angle=110, end angle=-210, radius=.8em];
\node[font=\scriptsize] at ($(0,0em)+(90:1.3em)$) {$p_1$};
\node[font=\scriptsize] at ($(0,0em)+(50:1.7em)$) {$p_2$};
\node[font=\scriptsize] at ($(0,0em)+(-160:1.5em)$) {$p_n$};
\draw[line width=.75pt, color=arccolour, line cap=round] (Y) -- ++(-90:1.2em);
\node[font=\scriptsize] at ($(0,0)+(-20:1.2em)$) {$.$};
\node[font=\scriptsize] at ($(0,0)+(-29:1.2em)$) {$.$};
\node[font=\scriptsize] at ($(0,0)+(-38:1.2em)$)  {$.$};
\node[color=ribboncolour] at (-.5em,5.5em) {$\bullet$};
\node[color=ribboncolour] at (0,2.5em) {$\bullet$};
\node[color=ribboncolour] at (0,2em) {$\bullet$};
\node[color=ribboncolour] at (0,5.5em) {$\bullet$};
\node[color=ribboncolour] at (0,7em) {$\bullet$};
\draw[-, line width=.85, color=ribboncolour, line cap=round, bend left=50] (0,2em) to (0, 5.5em);
\draw[-, line width=.85, color=ribboncolour, line cap=round, bend right=50] (0,2em) to (0, 5.5em);
\draw[-, line width=.85, color=ribboncolour, line cap=round] (0, 5.5em) to (0, 7em);
\draw[line width=0, fill=ribboncolour, draw=ribboncolour, opacity=.5] (0,2em) to[bend left=50] (0,5.5em) to[bend left=50] (0,2em);
\end{scope}
\begin{scope}[xshift=12em]
\draw[line width=.5pt] (0,-2em) ++(160:2em) arc[start angle=160, end angle=20, radius=2em];
\draw[->, line width=.5pt] (0,-2em) ++(140:2em) arc[start angle=140, end angle=110, radius=2em];
\draw[->, line width=.5pt] (0,-2em) ++(90:2em) arc[start angle=90, end angle=60, radius=2em];\node[font=\scriptsize,shape=circle,scale=.6] (X) at (0,2.7em) {};
\node[font=\scriptsize] at (0,2.7em) {$\times$};
\node[font=\scriptsize,color=stopcolour] at (0, 3.1em) {$\bullet$};
\draw[line width=.75pt, color=arccolour, line cap=round] ($(0,2em)+(-90:2em)$) -- (X);
\draw[line width=.75pt, color=arccolour, line cap=round] ($(0,-2em) +(150:2em)$) to[out=85, in=95, looseness=5.8] ($(0,-2em)+(30:2em)$);
\node[color=ribboncolour] at (0,3.9em) {$\bullet$};
\node[font=\scriptsize] at (-.8em, -1em) {$q_1$};
\node[font=\scriptsize] at (.8em, -1em) {$q_2$};
\node[font=\scriptsize] at (.5em, 1em) {$\gamma$};
\node[color=ribboncolour] at (0,3.5em) {$\bullet$};
\draw[-, line width=.85, color=ribboncolour, line cap=round] (0,3em) circle(.9em);
\draw[line width=0, fill=ribboncolour, draw=ribboncolour, opacity=.5] (0,3em) circle(.9em);
\end{scope}

\end{tikzpicture}
\caption{Configurations of arcs in DG dissections which are not admissible. For the first and third configurations the mirrored configurations are also not admissible. The paths $r_1, r_2$ (left) can be either boundary or orbifold paths and the paths $q_1, q_2$ (middle) must be boundary paths.}
\label{fig:formalnonadmissible}
\end{figure}
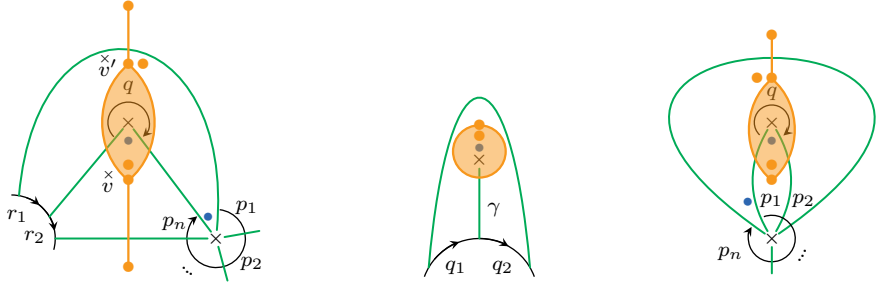

In Section \ref{section:ainfinity} we defined an explicit A$_\infty$ category $\mathbf A_\Delta$ for any admissible dissection $\Delta$. (See also \S\ref{subsection:notstrong} for a discussion of explicit A$_\infty$ categories without admissibility.) We now extend the definition of $\mathbf A_\Delta$ to also cover the case of DG dissections $\Delta$ which are not admissible.

If $\Delta$ is a DG dissection which is not admissible, we may also associate a DG category $\mathbf A_\Delta$ whose products are given by $\mubar_2$ and $\mutimes_i$ for $i = 1, 2$, similar to the admissible case, together with the additional differentials $\widehat \mu_1$ and products $\widehat \mu_2$ coming from the polygons in Fig.~\ref{fig:formalnonadmissible}.  Precisely, the polygons on the left of Fig.~\ref{fig:formalnonadmissible} contribute to $\mathbf A_\Delta$ the following products
\[
\mutimes_2(\s p_n \otimes \s r_2) =\s q, \quad \quad \mutimes_2(\s q \otimes \s r_1) = \s p_n \dotsb p_2p_1
\]
which correspond to the two orbifold disk sequences of length $2$,
and the additional product
\begin{align*}
 \widehat \mu_2(\s p_n \otimes \s r_2r_1) = \s p_n \dotsb p_2p_1
\end{align*}
which is needed to make the product of $\mathbf A_\Delta$ associative. Note that $\widehat \mu_2$ does not come from an orbifold or smooth disk sequence.

The second polygons contribute the products 
$$\widehat \mu_2(\s q_2 \otimes \s q_1) = \s \, \id_{\gamma}, \qquad \widehat \mu_2(\s q_2 \otimes \s q_1q_0) = (-1)^{|q_0|} \s q_0, \qquad \widehat \mu_2(\s q_3q_2 \otimes \s q_1) = \s q_3,$$
where the last two equations make sense if the boundary paths $q_0, q_3$ exist. 

The third polygons in Fig.~\ref{fig:formalnonadmissible} produce the following differential and products
\begin{align*}
\mutimes_1(\s p_1) = \s q, \qquad \mutimes_2 (\s p_2 \otimes \s  q) = \s p_n\dotsb p_2p_1,
\end{align*}
 from the two orbifold disk sequences and the additional differential 
 \[
\widehat \mu_1(\s p_2p_1) =(-1)^{|\s p_1|} \s p_n\dotsb p_2p_1.
 \]
which is needed to get a well-defined DG category. 

\begin{proposition}\label{proposition:dg}
Let $\Delta$ be a DG dissection (not necessarily admissible). Then the category $\mathbf A_\Delta$ with the products $\mubar_2$, $\mutimes_{\leq 2}$ and $\widehat \mu_{\leq 2}$ is a DG category.
\end{proposition}

\begin{proof}
If $\Delta$ is admissible then it is clear that $\mathbf A_\Delta$ is a DG category, see Remark \ref{remark:dgdissectiondg}. Let us assume that $\Delta$ is not admissible. If $\Delta$ contains polygons as on the left of Fig.~\ref{fig:formalnonadmissible} then we have 
\begin{align*}
\mutimes_2(\mutimes_2(\s p_n \otimes \s r_2) \otimes \s r_1) &= \s p_n\dotsb p_2 p_1\\
\widehat \mu_2( \s p_n \otimes \mubar_2( \s r_2 \otimes \s r_1)) & = (-1)^{|r_1|} \s p_n \dotsb p_2 p_1,
\end{align*}
where the sign $(-1)^{|r_1|}$ comes from $\mubar_2( \s r_2 \otimes \s r_1))= (-1)^{|r_1|}\s  r_2r_1.$ This shows that $\mutimes_2, \widehat \mu_2$ are associative when acting on $\s p_n \otimes \s r_2 \otimes \s r_1$. The associativity and Leibniz rule for other elements shall follow from Proposition \ref{proposition:ainfinity}. It is clear that the product $\widehat \mu_2$ is associative for the polygons in the middle. 

Similarly, if $\Delta$ contains polygons as on the right of Fig.~\ref{fig:formalnonadmissible} then we have 
\begin{align*}
\mutimes_2 ( \s p_2 \otimes \mutimes_1(\s p_1)) &=\s p_n \dotsb p_2 p_1 \\
\widehat \mu_1( \mubar_2( \s p_2 \otimes \s p_1)) & = (-1)^{|p_1| + | \s p_1|} \s p_n \dotsb p_2 p_1
\end{align*}
This yields the Leibniz rule since the differential acting on $p_2$ vanishes. 
\end{proof}

\begin{theorem}\label{theorem:dgtriangleequivalent}
Let $\mathbf S = (S, \Sigma, \eta)$ be a graded orbifold surface and let $\Delta$ be a DG dissection (not necessarily admissible). Then we have a triangulated equivalence
\[
\mathcal W (\mathbf S) \simeq \per (\mathbf A_\Delta)
\]
where $\per (\mathbf A_\Delta)$ is the perfect derived category of the DG category $\mathbf A_\Delta$.
\end{theorem}

\begin{proof}
If $\Delta$ is admissible, then $\mathbf A_\Delta$ is a DG category (cf.\ Remark \ref{remark:dgdissectiondg}) so that $\H^0 (\tw (\mathbf A_\Delta)^\natural) \simeq \per (\mathbf A_\Delta)$ and the result follows from Theorem \ref{theorem:moritapartiallywrapped}.

We thus need to consider the case where $\Delta$ is not admissible. If $\Delta$ contains polygons  on the left of Fig.~\ref{fig:formalnonadmissible}, then we introduce a new weakly admissible dissection $\Delta''$ obtained from $\Delta$ by adding an arc $\gamma$ as illustrated in Fig.~\ref{fig:formalnormalisation}. Similar to $\Delta$, we may associate to $\Delta''$  an A$_\infty$ category $\mathbf A_{\Delta''}$ whose objects are arcs in $\Delta''$ and morphisms are given by boundary and orbifold paths as before. The orbifold path $q$ in $\mathbf A_\Delta$ is divided into two paths by the arc $\gamma$, denoted by $q_1$ and $q_2$. Besides the products from the smooth and orbifold disk sequences, for instance, 
\begin{align*}
\mucirc_3( \s q_1 \otimes \s r_1 \otimes \s p_0) = \s e_{\gamma}, \quad \mutimes_1(\s q_2) = \s p_n \dotsb p_1p_0,\quad  \mutimes_3(\s p_n \otimes \s r_2r_1 \otimes \s p_0) = \s q_2,
\end{align*}
$\mathbf A_{\Delta''}$ carries the additional product
\[
\widehat \mu_2(\s p_n \otimes \s r_2r_1) = \s p_n \dotsb p_2p_1
\]
as in $\mathbf A_{\Delta}$. Here, $p_0$ is the new orbifold path starting at $\gamma$. Note that $\mathbf A_{\Delta''}$ with the above products is a well-defined A$_\infty$ category with $\mu_i =0$ for $i >3$.  The natural inclusions 
\[
\mathbf A_{\Delta}  \hookrightarrow \mathbf A_{\Delta'} \hookleftarrow \mathbf A_{\Delta''}
\]
are strict A$_\infty$ functors, which are Morita equivalences by Lemma \ref{lemma:twistedgenerators}.

\begin{figure}[ht]
\begin{tikzpicture}[x=1em,y=1em,decoration={markings,mark=at position 0.55 with {\arrow[black]{Stealth[length=4.8pt]}}}]
\begin{scope}
\draw[line width=.5pt] (-4em,0) ++(90:1.5em) arc[start angle=90, end angle=-15, radius=1.5em];
\draw[line width=.5pt] (-4em,0) ++(51:1.5em) arc[start angle=51, end angle=10, radius=1.5em];
\draw[line width=.5pt] (-4em,0) ++(70:1.5em) arc[start angle=70, end angle=50, radius=1.5em];
\node[font=\scriptsize,shape=circle,scale=.6] (Y) at (3em,0) {};
\node[font=\scriptsize] at (3em, 0) {$\times$};
\node[font=\scriptsize,shape=circle,scale=.6] (X) at (0,4em) {};
\node[font=\scriptsize] at (0,4em) {$\times$};
\node[font=\scriptsize,color=stopcolour] at ($(0,4em)+(270:.65em)$) {$\bullet$};
\node[font=\scriptsize,color=stopcolour] at ($(3em,0)+(110:.8em)$) {$\bullet$};
\draw[line width=.75pt, color=arccolour, line cap=round] (-4em,0) ++(30:1.5em) -- (X);
\draw[line width=.75pt, color=arccolour, line cap=round] (Y) -- (X);
\draw[line width=.75pt, color=arccolour, line cap=round] (-4em,0) ++(0:1.5em) -- (Y);
\draw[line width=.75pt, color=arccolour, line cap=round] (-4em,0) ++(80:1.5em) to[out=85, in=85, looseness=2.8] (Y);
\draw[line width=.75pt, color=arccolour, line cap=round] (Y) -- ++(-75:1.5em);
\draw[line width=.75pt, color=arccolour, line cap=round] (Y) -- ++(10:1.5em);
\node[font=\scriptsize] at ($(3em,0)+(-125:1.5em)$) {$.$};
\node[font=\scriptsize] at ($(3em,0)+(-131:1.5em)$) {$.$};
\node[font=\scriptsize] at ($(3em,0)+(-137:1.5em)$)  {$.$};
\node[color=ribboncolour] at (0,2.5em) {$\bullet$};
\node[color=ribboncolour] at (.4em,6em) {$\bullet$};
\draw[-, line width=.85, color=ribboncolour, line cap=round, bend left=50] (0,2em) to (0, 6em);
\draw[-, line width=.85, color=ribboncolour, line cap=round, bend right=50] (0,2em) to (0, 6em);
\draw[-, line width=.85, color=ribboncolour, line cap=round] (0, 6em) to (0, 8em);
\draw[-, line width=.85, color=ribboncolour, line cap=round] (0, -1em) to (0, 2em);
\draw[line width=0, fill=ribboncolour, draw=ribboncolour, opacity=.5] (0,2em) to[bend left=50] (0,6em) to[bend left=50] (0,2em);
\node[]  at (-3.5em, 6em) {$\Delta$};
\end{scope}
\begin{scope}[xshift=12em]
\node at (-6em,3em) {$\leftsquigarrow$};
\draw[line width=.5pt] (-4em,0) ++(90:1.5em) arc[start angle=90, end angle=-15, radius=1.5em];
\draw[line width=.5pt] (-4em,0) ++(51:1.5em) arc[start angle=51, end angle=10, radius=1.5em];
\draw[line width=.5pt] (-4em,0) ++(70:1.5em) arc[start angle=70, end angle=50, radius=1.5em];
\node[font=\scriptsize,shape=circle,scale=.6] (Y) at (3em,0) {};
\node[font=\scriptsize] at (3em, 0) {$\times$};
\node[font=\scriptsize,shape=circle,scale=.6] (X) at (0,4em) {};
\node[font=\scriptsize] at (0,4em) {$\times$};
\node[font=\scriptsize,color=stopcolour] at ($(0,4em)+(270:.65em)$) {$\bullet$};
\node[font=\scriptsize,color=stopcolour] at ($(3em,0)+(115:1.2em)$) {$\bullet$};
\draw[line width=.75pt, color=arccolour, line cap=round] (-4em,0) ++(30:1.5em) -- (X);
\draw[line width=.75pt, color=arccolour, line cap=round] (Y) -- (X);
\draw[line width=.75pt, color=arccolour, line cap=round] (-4em,0) ++(0:1.5em) -- (Y);
\draw[line width=.75pt, color=arccolour, line cap=round] (-4em,0) ++(80:1.5em) to[out=85, in=85, looseness=2.8] (Y);
\draw[line width=.75pt, color=arccolour, line cap=round] (3em,0) to[out=96, in=75, looseness=1] (X);
\node[font=\scriptsize] at ($(0,4)+(0:1.8em)$) {$\gamma$};
\draw[line width=.75pt, color=arccolour, line cap=round] (Y) -- ++(-75:1.5em);
\draw[line width=.75pt, color=arccolour, line cap=round] (Y) -- ++(10:1.5em);
\node[font=\scriptsize] at ($(3em,0)+(-125:1.5em)$) {$.$};
\node[font=\scriptsize] at ($(3em,0)+(-131:1.5em)$) {$.$};
\node[font=\scriptsize] at ($(3em,0)+(-137:1.5em)$)  {$.$};
\node[color=ribboncolour] at (0,2.5em) {$\bullet$};
\node[color=ribboncolour] at ($(0,4em)+(-40:2.5em)$) {$\bullet$};
\draw[-, line width=.85, color=ribboncolour, line cap=round] (0,2em) to ($(0,4em)+(-40:2em)$);
\draw[-, line width=.85, color=ribboncolour, line cap=round, bend left]  (0, 6em) to ($(0,4em)+(-40:2em)$);
\draw[-, line width=.85, color=ribboncolour, line cap=round, bend left=50] (0,2em) to (0, 6em);
\draw[-, line width=.85, color=ribboncolour, line cap=round] (0, 6em) to (0, 8em);
\draw[-, line width=.85, color=ribboncolour, line cap=round] (0, -1em) to (0, 2em);
\draw[line width=0, fill=ribboncolour, draw=ribboncolour, opacity=.5] (0,2em) to[bend left=50] (0,6em) to[bend left] ($(0,4em)+(-40:2em)$) to (0,2em);
\node[]  at (-3.5em, 6em) {$\Delta'$};
\end{scope}
\begin{scope}[xshift=24em]
\node at (-6em,3em) {$\rightsquigarrow$};
\draw[line width=.5pt] (-4em,0) ++(90:1.5em) arc[start angle=90, end angle=-15, radius=1.5em];
\draw[line width=.5pt] (-4em,0) ++(51:1.5em) arc[start angle=51, end angle=10, radius=1.5em];
\draw[line width=.5pt] (-4em,0) ++(70:1.5em) arc[start angle=70, end angle=50, radius=1.5em];
\node[font=\scriptsize,shape=circle,scale=.6] (Y) at (3em,0) {};
\node[font=\scriptsize] at (3em, 0) {$\times$};
\node[font=\scriptsize,shape=circle,scale=.6] (X) at (0,4em) {};
\node[font=\scriptsize] at (0,4em) {$\times$};
\node[font=\scriptsize,color=stopcolour] at ($(0,4em)+(270:.65em)$) {$\bullet$};
\node[font=\scriptsize,color=stopcolour] at ($(3em,0)+(115:1.2em)$) {$\bullet$};
\draw[line width=.75pt, color=arccolour, line cap=round] (-4em,0) ++(30:1.5em) -- (X);
\draw[line width=.75pt, color=arccolour, line cap=round] (Y) -- (X);
\draw[line width=.75pt, color=arccolour, line cap=round] (-4em,0) ++(0:1.5em) -- (Y);
\draw[line width=.75pt, color=arccolour, line cap=round] (3em,0) to[out=96, in=75, looseness=1] (X);
\draw[line width=.75pt, color=arccolour, line cap=round] (Y) -- ++(-75:1.5em);
\draw[line width=.75pt, color=arccolour, line cap=round] (Y) -- ++(10:1.5em);
\node[font=\scriptsize] at ($(3em,0)+(-125:1.5em)$) {$.$};
\node[font=\scriptsize] at ($(3em,0)+(-131:1.5em)$) {$.$};
\node[font=\scriptsize] at ($(3em,0)+(-137:1.5em)$)  {$.$};
\node[color=ribboncolour] at (0,2.5em) {$\bullet$};
\node[color=ribboncolour] at ($(0,4em)+(-40:2.5em)$) {$\bullet$};
\draw[-, line width=.85, color=ribboncolour, line cap=round] (0,2em) to ($(0,4em)+(-40:2em)$);
\draw[-, line width=.85, color=ribboncolour, line cap=round, bend left]  (0, 6em) to ($(0,4em)+(-40:2em)$);
\draw[-, line width=.85, color=ribboncolour, line cap=round, bend left=50] (0,2em) to (0, 6em);
\draw[-, line width=.85, color=ribboncolour, line cap=round] (0, 6em) to (1.5em, 7em);
\draw[-, line width=.85, color=ribboncolour, line cap=round] (0, 6em) to (-1.5em, 7em);
\node[font=\scriptsize] at ($(1.5em,6.7em)+(180:1.3em)$) {$.$};
\node[font=\scriptsize] at ($(1.5em,6.7em)+(180:1.5em)$) {$.$};
\node[font=\scriptsize] at ($(1.5em,6.7em)+(180:1.7em)$)  {$.$};
\draw[-, line width=.85, color=ribboncolour, line cap=round] (0, -1em) to (0, 2em);
\draw[line width=0, fill=ribboncolour, draw=ribboncolour, opacity=.5] (0,2em) to[bend left=50] (0,6em) to[bend left] ($(0,4em)+(-40:2em)$) to (0,2em);
\node[]  at (-3.5em, 6em) {$\Delta''$};
\end{scope}
\end{tikzpicture}
\caption{The dissection $\Delta''$ on the right is admissible which is related to $\Delta$ on the left by edge contractions.}
\label{fig:formalnormalisation}
\end{figure}
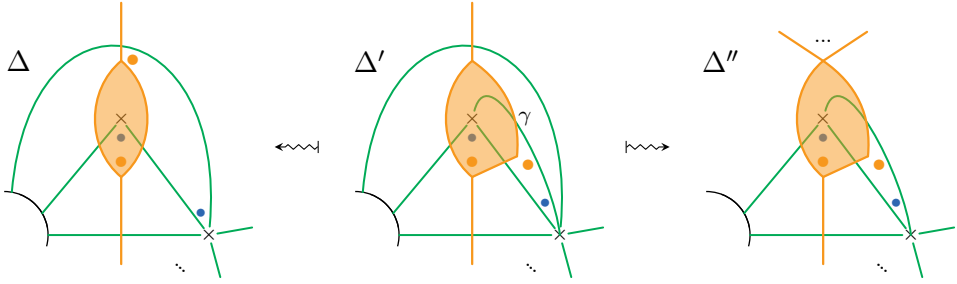

We may do the similar operations for the other cases  in Fig.~\ref{fig:formalnonadmissible}. Then the result follows from Theorems \ref{theorem:moritaequivalenceadmissible} and \ref{theorem:moritapartiallywrapped}.
\end{proof}

\subsection{Formal A$_\infty$ categories}
\label{subsection:formalexamples}

In this subsection, we show that the A$_\infty$ category $\mathbf A_\Delta$ is formal if and only if $\Delta$ is a formal dissection (Definition \ref{definition:DGformal}) --- except when $\mathbf S$ is the disk with two orbifold points and with one boundary stop in which case only the ``if'' part holds, see Remark \ref{remark:disktwoorbifoldonestop} below.

\begin{theorem}
\label{theorem:formaldg}
Let $\mathbf S = (S, \Sigma, \eta)$ be a graded orbifold surface with stops and let $\Delta$ be a weakly admissible dissection of $\mathbf S$.
\begin{enumerate}
\item If $\Delta$ is a formal dissection, then $\mathbf A_\Delta$ is a formal A$_\infty$ category, that is $\mathbf A_\Delta$ is A$_\infty$-quasi-isomorphic to its cohomology $\H^\bullet (\mathbf A_\Delta)$. \label{formaldg1}
\item Suppose that $\mathbf S$ is not the disk with two orbifold points and with one stop in the boundary. Then the converse of \textup{\ref{formaldg1}} also holds, that is $\mathbf A_\Delta$ is formal if and only if $\Delta$ is a formal dissection. \label{formaldg2}
\end{enumerate}
\end{theorem}

\begin{proof}
We first prove \ref{formaldg1}. Note that the morphism space of the cohomology $\H^\bullet(\mathbf A_\Delta)$ has a $\Bbbk$-linear basis given by the boundary and orbifold paths as in $\mathbf A_\Delta$, excluding the paths corresponding to  the orbifold disk sequences of length $2$ (i.e.\ the preimage and image of $\mutimes_1$ as illustrated on the left of Fig.~\ref{fig:differential}). 

We may construct an explicit A$_\infty$ quasi-equivalence 
\[
F \colon \H^\bullet  (\mathbf A_\Delta) \to \mathbf A_\Delta
\]
as follows. The functor $F$ is the identity on objects.  The map $F_1$ is the natural inclusion with respect to the basis.  Note that $F_1$ is not compatible with the multiplications, because for any orbifold disk sequence of length $2$ as illustrated below (where $n>1$ and $p_1$ can be either a boundary path or an orbifold path) and for any $1 \leq i < n$ we have 
\begin{align*}
F_1(\mu_2( \s q_n \dotsb q_{i+1} \otimes \s q_i \dotsb q_1)) &= 0\\
\mu_2 (F_1( \s q_n \dotsb q_{i+1}) \otimes F_1(\s q_i \dotsb q_1)) &=\s q_n \dotsb q_2 q_1 \neq 0.
\end{align*}

To resolve this issue, we need to introduce the higher morphism $F_2$, which is nontrivial only for 
\begin{align*}
F_2 (\s q_n \dotsb q_{i+1} \otimes \s q_i \dotsb q_1) &= -\s p_1 \quad  \text{for $1\leq i< n$}
\end{align*}
for any orbifold disk sequence of length $2$ with $n>1$.
 
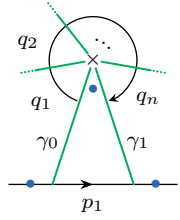
\[
\begin{tikzpicture}[x=1em,y=1em,decoration={markings,mark=at position 0.5 with {\arrow[black]{Stealth[length=4.8pt]}}}]
\draw[line width=.5pt,postaction={decorate}] ($(252:4.5)+(-1.5,0)$) -- ($(288:4.5)+(1.5,0)$);
\node[font=\scriptsize] at (0,0) {$\times$};
\draw[->, line width=.5pt] (249:1.5em) arc[start angle=249, end angle=-69, radius=1.5em];
\node[font=\scriptsize,shape=circle,scale=.6] (X) at (0,0) {};
\node[font=\scriptsize,color=stopcolour] at ($(252:4.5)+(-.75,0)$) {$\bullet$};
\node[font=\scriptsize,color=stopcolour] at ($(288:4.5)+(.75,0)$) {$\bullet$};
\draw[line width=.75pt, color=arccolour, line cap=round] (X) -- (252:4.5);
\draw[line width=.75pt, color=arccolour, line cap=round] (X) -- (288:4.5);
\draw[line width=.75pt, color=arccolour, line cap=round] (X) -- (190:2em);
\draw[dash pattern=on 0pt off 1.3pt, line width=.8pt, line cap=round, color=arccolour] (190:2.1em) -- (190:2.5em);
\draw[line width=.75pt, color=arccolour, line cap=round] (X) -- (128:2em);
\draw[dash pattern=on 0pt off 1.3pt, line width=.8pt, line cap=round, color=arccolour] (128:2.1em) -- (128:2.5em);
\draw[line width=.75pt, color=arccolour, line cap=round] (X) -- (350:2em);
\draw[dash pattern=on 0pt off 1.3pt, line width=.8pt, line cap=round, color=arccolour] (350:2.1em) -- (350:2.5em);
\node[font=\scriptsize] at (79:.7) {.};
\node[font=\scriptsize] at (59:.7) {.};
\node[font=\scriptsize] at (39:.7) {.};
\node[font=\scriptsize,left=-.3ex] at (252:3) {$\gamma_0$};
\node[font=\scriptsize,right=-.3ex] at (288:3) {$\gamma_1$};
\node[font=\scriptsize,color=stopcolour] at (270:1em) {$\bullet$};
\node[font=\scriptsize] at (270:5) {$p_1$};
\node[font=\scriptsize] at (220:2.3) {$q_1$};
\node[font=\scriptsize] at (165:2.3) {$q_2$};
\node[font=\scriptsize] at (-35:2.3) {$q_n$};
\end{tikzpicture}
\]
Note that we have 
\begin{align*}
&\mu_1(F_2( \s q_n \dotsb q_{i+1} \otimes \s q_i \dotsb q_1)) + \mu_2 ( F_1(\s q_n \dotsb q_{i+1}) \otimes F_1(\s q_i \dotsb q_1)) \\
={}& -\mu_1(\s p_1) + \s q_n \dotsb q_1 \\
={}& 0.
\end{align*}
This shows that for any $a_1, a_2 \in \H^\bullet  (\mathbf A_\Delta)$ we have
\[
F_1(\mu_2(\s a_2 \otimes \s a_1)) = \mu_1 (F_2(\s a_2 \otimes \s a_1)) + \mu_2 (F_1(\s a_2) \otimes F_1(\s a_1)).
\]

Since $\Delta$ is formal, it follows from the definition that for each orbifold disk sequence of length $2$ with $n>1$ as illustrated above, the path $p_1$ is maximal with respect to path concatenations. In other words, $\mubar_2(\s p_1 \otimes \s a) = 0 = \mubar_2 (\s a \otimes \s p_1)$ for any non-identity path $a$. We infer the following identities
 \begin{align*}
 0 &= \mu_2(F_2(\s a_3 \otimes \s a_2) \otimes F_1(\s a_1)) + \mu_2(F_1(\s a_3) \otimes F_2( \s a_2 \otimes \s a_1))\\
0& = \mu_2(F_2(\s a_4 \otimes \s a_3) \otimes F_2(s a_2\otimes \s a_1)) 
\end{align*}
 for any morphisms $a_1, \dotsc, a_4$ in $\H^
\bullet  (\mathbf A_\Delta)$.  As a result, $F=(F_1,F_2)$ (i.e.\ $F_i =0$ for $i>2$) is an A$_\infty$ quasi-equivalence, compare Definition \ref{definition:ainfinityfunctorquasi}. 

We now prove \ref{formaldg2}. By Proposition \ref{proposition:formalcriterion} it suffices to show that for any non-formal weakly admissible dissection $\Delta$ there exists a non-formal full subcategory $\mathbf B$ in the A$_\infty$ category $\mathbf A_\Delta$. 

First we may assume that $\Delta$ is a DG dissection. Otherwise, each smooth (resp.\ orbifold) disk sequence of length $n>2$ gives rise to a non-formal full A$_\infty$ subcategory whose objects are given by the arcs in the sequence.

By Definition \ref{definition:DGformal}, a DG dissection $\Delta$ is {\it non-formal} if and only if $\Delta$ contains an orbifold disk sequence of length $2$ as illustrated below with $n >1$. Here, the paths $p_i$ can be either orbifold paths or boundary paths. We claim that any such orbifold disk sequence of length $2$ gives rise to a non-formal full DG subcategory.
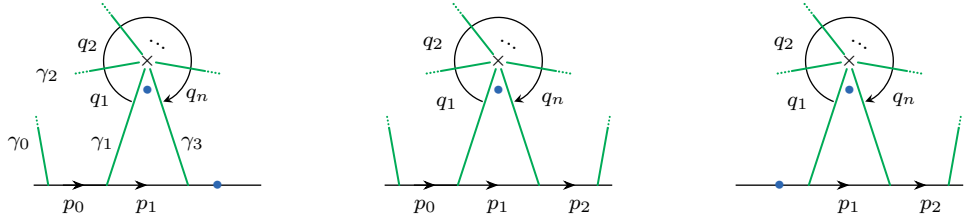
\[
\begin{tikzpicture}[x=1em,y=1em,decoration={markings,mark=at position 0.5 with {\arrow[black]{Stealth[length=4.8pt]}}}]
\begin{scope}[xshift=-12em]
\draw[line width=.5pt] ($(252:4.5)+(-2.5,0)$) --(252:4.5);
\draw[line width=.5pt,postaction={decorate}] ($(252:4.5)+(-1.5,0)$) --(252:4.5);
\draw[line width=.5pt,postaction={decorate}] (252:4.5)--(288:4.5);
\draw[line width=.5pt] (288:4.5)-- ($(288:4.5)+(2.5,0)$);
\node[font=\scriptsize,color=stopcolour] at ($(288:4.5)+(1,0)$) {$\bullet$};
\node[font=\scriptsize] at (0,0) {$\times$};
\draw[->, line width=.5pt] (249:1.5em) arc[start angle=249, end angle=-69, radius=1.5em];
\node[font=\scriptsize,shape=circle,scale=.6] (X) at (0,0) {};
\draw[line width=.75pt, color=arccolour, line cap=round] (X) -- (252:4.5);
\draw[line width=.75pt, color=arccolour, line cap=round] (X) -- (288:4.5);
\draw[line width=.75pt, color=arccolour, line cap=round] (X) -- (190:2em);
\draw[line width=.75pt, color=arccolour, line cap=round] ($(252:4.5)+(-2,0)$) -- ($(252:4.5)+(-2,0)+(100:2)$) ;
\draw[dash pattern=on 0pt off 1.3pt,line width=.75pt, color=arccolour, line cap=round] ($(252:4.5)+(-2,0)+(100:2)$) -- ($(252:4.5)+(-2,0)+(100:2.4)$) ;
\draw[dash pattern=on 0pt off 1.3pt, line width=.8pt, line cap=round, color=arccolour] (190:2.1em) -- (190:2.5em);
\draw[line width=.75pt, color=arccolour, line cap=round] (X) -- (128:2em);
\draw[dash pattern=on 0pt off 1.3pt, line width=.8pt, line cap=round, color=arccolour] (128:2.1em) -- (128:2.5em);
\draw[line width=.75pt, color=arccolour, line cap=round] (X) -- (350:2em);
\draw[dash pattern=on 0pt off 1.3pt, line width=.8pt, line cap=round, color=arccolour] (350:2.1em) -- (350:2.5em);
\node[font=\scriptsize] at (79:.7) {.};
\node[font=\scriptsize] at (59:.7) {.};
\node[font=\scriptsize] at (39:.7) {.};
\node[font=\scriptsize,left=-.2ex] at ($(252:3)+(-2.8,0)$) {$\gamma_0$};
\node[font=\scriptsize,left=-.3ex] at (252:3) {$\gamma_1$};
\node[font=\scriptsize,left=-.2ex] at (190:2.8) {$\gamma_2$};
\node[font=\scriptsize,right=-.3ex] at (288:3) {$\gamma_{3}$};
\node[font=\scriptsize,color=stopcolour] at (270:1em) {$\bullet$};
\node[font=\scriptsize] at (270:5) {$p_1$};
\node[font=\scriptsize] at ($(270:5)+(-2.5,0)$) {$p_0$};
\node[font=\scriptsize] at (220:2.1) {$q_1$};
\node[font=\scriptsize] at (165:2.1) {$q_2$};
\node[font=\scriptsize] at (-35:2.1) {$q_n$};
\end{scope}
\begin{scope}
\draw[line width=.5pt] ($(252:4.5)+(-2.5,0)$) --(252:4.5);
\draw[line width=.5pt,postaction={decorate}] ($(252:4.5)+(-1.5,0)$) --(252:4.5);
\draw[line width=.5pt,postaction={decorate}] (252:4.5)--(288:4.5);
\draw[line width=.5pt,postaction={decorate}] (288:4.5)-- ($(288:4.5)+(2.5,0)$);
\node[font=\scriptsize] at (0,0) {$\times$};
\draw[->, line width=.5pt] (249:1.5em) arc[start angle=249, end angle=-69, radius=1.5em];
\node[font=\scriptsize,shape=circle,scale=.6] (X) at (0,0) {};
\draw[line width=.75pt, color=arccolour, line cap=round] (X) -- (252:4.5);
\draw[line width=.75pt, color=arccolour, line cap=round] (X) -- (288:4.5);
\draw[line width=.75pt, color=arccolour, line cap=round] (X) -- (190:2em);
\draw[line width=.75pt, color=arccolour, line cap=round] ($(252:4.5)+(-2,0)$) -- ($(252:4.5)+(-2,0)+(100:2)$) ;
\draw[line width=.75pt, color=arccolour, line cap=round] ($(288:4.5)+(2,0)$) -- ($(288:4.5)+(2,0)+(80:2)$);
\draw[dash pattern=on 0pt off 1.3pt,line width=.75pt, color=arccolour, line cap=round] ($(252:4.5)+(-2,0)+(100:2)$) -- ($(252:4.5)+(-2,0)+(100:2.4)$) ;
\draw[dash pattern=on 0pt off 1.3pt, line width=.75pt, color=arccolour, line cap=round] ($(288:4.5)+(2,0)+(80:2)$) --($(288:4.5)+(2,0)+(80:2.4)$) ;

\draw[dash pattern=on 0pt off 1.3pt, line width=.8pt, line cap=round, color=arccolour] (190:2.1em) -- (190:2.5em);
\draw[line width=.75pt, color=arccolour, line cap=round] (X) -- (128:2em);
\draw[dash pattern=on 0pt off 1.3pt, line width=.8pt, line cap=round, color=arccolour] (128:2.1em) -- (128:2.5em);
\draw[line width=.75pt, color=arccolour, line cap=round] (X) -- (350:2em);
\draw[dash pattern=on 0pt off 1.3pt, line width=.8pt, line cap=round, color=arccolour] (350:2.1em) -- (350:2.5em);
\node[font=\scriptsize] at (79:.7) {.};
\node[font=\scriptsize] at (59:.7) {.};
\node[font=\scriptsize] at (39:.7) {.};
\node[font=\scriptsize,color=stopcolour] at (270:1em) {$\bullet$};
\node[font=\scriptsize] at (270:5) {$p_1$};
\node[font=\scriptsize] at ($(270:5)+(-2.5,0)$) {$p_0$};
\node[font=\scriptsize] at ($(270:5)+(2.8,0)$) {$p_2$};
\node[font=\scriptsize] at (220:2.3) {$q_1$};
\node[font=\scriptsize] at (165:2.3) {$q_2$};
\node[font=\scriptsize] at (-35:2.3) {$q_n$};
\end{scope}
\begin{scope}[xshift=12em]
\draw[line width=.5pt] ($(252:4.5)+(-2.5,0)$) --(252:4.5);
\node[font=\scriptsize,color=stopcolour] at ($(252:4.5)+(-1,0)$)  {$\bullet$};
\draw[line width=.5pt,postaction={decorate}] (252:4.5)--(288:4.5);
\draw[line width=.5pt,postaction={decorate}] (288:4.5)-- ($(288:4.5)+(2.5,0)$);
\node[font=\scriptsize] at (0,0) {$\times$};
\draw[->, line width=.5pt] (249:1.5em) arc[start angle=249, end angle=-69, radius=1.5em];
\node[font=\scriptsize,shape=circle,scale=.6] (X) at (0,0) {};
\draw[line width=.75pt, color=arccolour, line cap=round] (X) -- (252:4.5);
\draw[line width=.75pt, color=arccolour, line cap=round] (X) -- (288:4.5);
\draw[line width=.75pt, color=arccolour, line cap=round] (X) -- (190:2em);
\draw[line width=.75pt, color=arccolour, line cap=round] ($(288:4.5)+(2,0)$) -- ($(288:4.5)+(2,0)+(80:2)$);
\draw[dash pattern=on 0pt off 1.3pt, line width=.75pt, color=arccolour, line cap=round] ($(288:4.5)+(2,0)+(80:2)$) --($(288:4.5)+(2,0)+(80:2.4)$) ;

\draw[dash pattern=on 0pt off 1.3pt, line width=.8pt, line cap=round, color=arccolour] (190:2.1em) -- (190:2.5em);
\draw[line width=.75pt, color=arccolour, line cap=round] (X) -- (128:2em);
\draw[dash pattern=on 0pt off 1.3pt, line width=.8pt, line cap=round, color=arccolour] (128:2.1em) -- (128:2.5em);
\draw[line width=.75pt, color=arccolour, line cap=round] (X) -- (350:2em);
\draw[dash pattern=on 0pt off 1.3pt, line width=.8pt, line cap=round, color=arccolour] (350:2.1em) -- (350:2.5em);
\node[font=\scriptsize] at (79:.7) {.};
\node[font=\scriptsize] at (59:.7) {.};
\node[font=\scriptsize] at (39:.7) {.};
\node[font=\scriptsize,color=stopcolour] at (270:1em) {$\bullet$};
\node[font=\scriptsize] at (270:5) {$p_1$};
\node[font=\scriptsize] at ($(270:5)+(2.8,0)$) {$p_2$};
\node[font=\scriptsize] at (220:2.3) {$q_1$};
\node[font=\scriptsize] at (165:2.3) {$q_2$};
\node[font=\scriptsize] at (-35:2.3) {$q_n$};
\end{scope}
\end{tikzpicture}
\]
Indeed, we need to consider the following cases. In the first case where the arrows $p_0$ and $q_1$ are not in an orbifold disk sequence of length $2$, the full DG subcategory $\mathbf B$ of $\mathbf A_\Delta$ consisting of the arcs $\gamma_0, \gamma_1, \gamma_2, \gamma_n$ is given by 

\[
\begin{tikzpicture}[baseline=-2.6pt,description/.style={fill=white,inner sep=1pt,outer sep=0}]
\matrix (m) [matrix of math nodes, row sep=.5em, text height=1.5ex, column sep=3em, text depth=0.25ex, ampersand replacement=\&, inner sep=3.5pt]
{
0 \& 1 \& 2 \&  3\\
};
\path[->,line width=.4pt, font=\scriptsize] (m-1-1) edge node[above=.02ex] {$p_0$} (m-1-2);
\path[->,line width=.4pt, font=\scriptsize] (m-1-2) edge node[above=.01ex] {$q_1$} (m-1-3);
\path[->,line width=.4pt, font=\scriptsize] (m-1-3) edge node[above=.01ex] {$q_{2}$} (m-1-4);
\path[->,line width=.4pt, font=\scriptsize,bend right=30] (m-1-2) edge node[below=.01ex] {$p_{1}$} (m-1-4);
\end{tikzpicture}
\]
with the relation $q_1p_0 = 0$ and the differential $d(p_1) = q_2q_1$. Note that the cohomology $\H^\bullet(\mathbf B)$ is given by the following graded quiver
\[
\begin{tikzpicture}[baseline=-2.6pt,description/.style={fill=white,inner sep=1pt,outer sep=0}]
\matrix (m) [matrix of math nodes, row sep=.5em, text height=1.5ex, column sep=3em, text depth=0.25ex, ampersand replacement=\&, inner sep=3.5pt]
{
0 \& 1 \& 2 \& 3\\
};
\path[->,line width=.4pt, font=\scriptsize] (m-1-1) edge node[above=.02ex] {$p_0$} (m-1-2);
\path[->,line width=.4pt, font=\scriptsize] (m-1-2) edge node[above=.01ex] {$q_1$} (m-1-3);
\path[->,line width=.4pt, font=\scriptsize] (m-1-3) edge node[below=.01ex] {$q_{2}$} (m-1-4);
\path[->,line width=.4pt, font=\scriptsize,bend right=30] (m-1-1) edge node[below=.01ex] {$p_{01}$} (m-1-4);
\end{tikzpicture}
\]
with the relations $q_1p_0 = q_2q_1=0$. Here, the arrow $p_{01}$ represents the cocycle $p_1 p_0$ and the relation $q_2 q_1 = 0$ in $\H^\bullet (\mathbf B)$ follows from $q_2 q_1 = d (p_1)$ being a coboundary and thus zero in cohomology. Note that $\HH^2(\H^\bullet(\mathbf B), \H^\bullet(\mathbf B))$ of the graded algebra $\H^\bullet(\mathbf B)$ is $1$-dimensional. (This may be computed straightforwardly using the Chouhy--Solotar bimodule resolution see \cite{chouhysolotar}, see also \cite{barmeierwang}.) On the other hand, $\HH^2 (\mathbf B, \mathbf B) = 0$. As a result, $\H^\bullet(\mathbf B)$ is not A$_\infty$-quasi-isomorphic to $\mathbf B$. 

By constructing an explicit homotopy deformation retract between $\mathbf B$ and $\H^\bullet(\mathbf B)$, we obtain a nontrivial higher product on $\H^\bullet(\mathbf B)$ given as follows
\begin{align*}
\mu_3(\s q_2\otimes \s q_1 \otimes \s p_0) &= \s p_{01}.
\end{align*}

In the second case where the arrows $p_0$ and $q_1$ lie in an orbifold disk sequence of length $2$, the local configuration is illustrated as follows

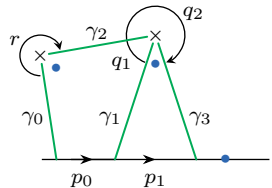
\[
\begin{tikzpicture}[x=1em,y=1em,decoration={markings,mark=at position 0.5 with {\arrow[black]{Stealth[length=4.8pt]}}}]
\begin{scope}[xshift=-12em]
\draw[line width=.5pt] ($(252:4.5)+(-2.5,0)$) --(252:4.5);
\draw[line width=.5pt,postaction={decorate}] ($(252:4.5)+(-1.5,0)$) --(252:4.5);
\draw[line width=.5pt,postaction={decorate}] (252:4.5)--(288:4.5);
\draw[line width=.5pt] (288:4.5)-- ($(288:4.5)+(2.5,0)$);
\node[font=\scriptsize,color=stopcolour] at ($(288:4.5)+(1,0)$) {$\bullet$};
\node[font=\scriptsize,color=stopcolour] at ($(190:4em)+(-40:.7em)$) {$\bullet$};
\node[font=\scriptsize] at (0,0) {$\times$};
\draw[->, line width=.5pt] (249:1em) arc[start angle=249, end angle=-69, radius=1em];
\draw[->, line width=.5pt] ($(190:4em)+(270:.7em)$) arc[start angle=270, end angle=15, radius=.7em];
\node[font=\scriptsize,shape=circle,scale=.6] (X) at (0,0) {};
\node[font=\scriptsize,shape=circle,scale=.6] (Y) at (190:4) {};
\node[font=\scriptsize] at (190:4) {$\times$};
\draw[line width=.75pt, color=arccolour, line cap=round] (X) -- (252:4.5);
\draw[line width=.75pt, color=arccolour, line cap=round] (X) -- (288:4.5);
\draw[line width=.75pt, color=arccolour, line cap=round] (X) -- (Y);
\draw[line width=.75pt, color=arccolour, line cap=round] ($(252:4.5)+(-2,0)$) -- (Y);
\node[font=\scriptsize,left=-.2ex] at ($(252:3)+(-2.6,0)$) {$\gamma_0$};
\node[font=\scriptsize,left=-.3ex] at (252:3) {$\gamma_1$};
\node[font=\scriptsize] at (178:2) {$\gamma_2$};
\node[font=\scriptsize] at ($(190:4em)+(150:1em)$) {$r$};
\node[font=\scriptsize,right=-.3ex] at (288:3) {$\gamma_{3}$};
\node[font=\scriptsize,color=stopcolour] at (270:1em) {$\bullet$};
\node[font=\scriptsize] at (270:5) {$p_1$};
\node[font=\scriptsize] at ($(270:5)+(-2.5,0)$) {$p_0$};
\node[font=\scriptsize] at (220:1.5) {$q_1$};
\node[font=\scriptsize] at (30:1.5) {$q_2$};
\end{scope}
\end{tikzpicture}
\]
Since the orbifold surface $S$ is not the disk with two orbifold points and with one stop on the boundary, there is at least one more arc $\gamma$ in $\Delta$ such that the morphism space between $\gamma$ and (the direct sum of) $\gamma_i$ is nonzero, and such that the possible local configurations are listed as follows. 
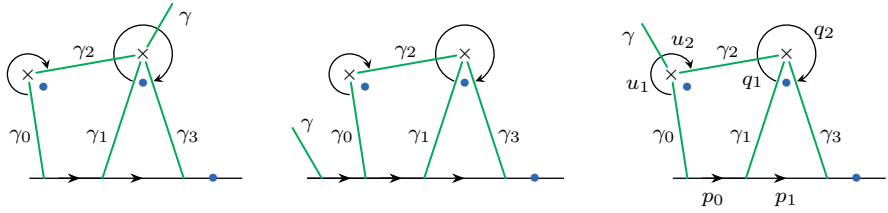
\[
\begin{tikzpicture}[x=1em,y=1em,decoration={markings,mark=at position 0.5 with {\arrow[black]{Stealth[length=4.8pt]}}}]
\begin{scope}
\draw[line width=.5pt] ($(252:4.5)+(-2.5,0)$) --(252:4.5);
\draw[line width=.5pt,postaction={decorate}] ($(252:4.5)+(-1.5,0)$) --(252:4.5);
\draw[line width=.5pt,postaction={decorate}] (252:4.5)--(288:4.5);
\draw[line width=.5pt] (288:4.5)-- ($(288:4.5)+(2,0)$);
\node[font=\scriptsize,color=stopcolour] at ($(288:4.5)+(1,0)$) {$\bullet$};
\node[font=\scriptsize,color=stopcolour] at ($(190:4em)+(-40:.7em)$) {$\bullet$};
\node[font=\scriptsize] at (0,0) {$\times$};
\draw[->, line width=.5pt] (249:1em) arc[start angle=249, end angle=-69, radius=1em];
\draw[->, line width=.5pt] ($(190:4em)+(270:.7em)$) arc[start angle=270, end angle=15, radius=.7em];
\node[font=\scriptsize,shape=circle,scale=.6] (X) at (0,0) {};
\node[font=\scriptsize,shape=circle,scale=.6] (Y) at (190:4) {};
\node[font=\scriptsize] at (190:4) {$\times$};
\draw[line width=.75pt, color=arccolour, line cap=round] (X) -- (252:4.5);
\draw[line width=.75pt, color=arccolour, line cap=round] (X) -- (288:4.5);
\draw[line width=.75pt, color=arccolour, line cap=round] (X) -- (60:2em);
\draw[line width=.75pt, color=arccolour, line cap=round] (X) -- (Y);
\draw[line width=.75pt, color=arccolour, line cap=round] ($(252:4.5)+(-2,0)$) -- (Y);
\node[font=\scriptsize,left=-.2ex] at ($(252:3)+(-2.6,0)$) {$\gamma_0$};
\node[font=\scriptsize,left=-.3ex] at (252:3) {$\gamma_1$};
\node[font=\scriptsize] at (178:2) {$\gamma_2$};
\node[font=\scriptsize] at (40:1.9) {$\gamma$};

\node[font=\scriptsize,right=-.3ex] at (288:3) {$\gamma_{3}$};
\node[font=\scriptsize,color=stopcolour] at (270:1em) {$\bullet$};
\end{scope}
\begin{scope}[xshift=11em]
\draw[line width=.5pt] ($(252:4.5)+(-4,0)$) --(252:4.5);
\draw[line width=.5pt,postaction={decorate}] ($(252:4.5)+(-1.5,0)$) --(252:4.5);
\draw[line width=.5pt,postaction={decorate}] ($(252:4.5)+(-4,0)$) --($(252:4.5)+(-1,0)$);
\draw[line width=.5pt,postaction={decorate}] (252:4.5)--(288:4.5);
\draw[line width=.5pt] (288:4.5)-- ($(288:4.5)+(2,0)$);
\node[font=\scriptsize,color=stopcolour] at ($(288:4.5)+(1,0)$) {$\bullet$};
\node[font=\scriptsize,color=stopcolour] at ($(190:4em)+(-40:.7em)$) {$\bullet$};
\node[font=\scriptsize] at (0,0) {$\times$};
\draw[->, line width=.5pt] (249:1em) arc[start angle=249, end angle=-69, radius=1em];
\draw[->, line width=.5pt] ($(190:4em)+(270:.7em)$) arc[start angle=270, end angle=15, radius=.7em];
\node[font=\scriptsize,shape=circle,scale=.6] (X) at (0,0) {};
\node[font=\scriptsize,shape=circle,scale=.6] (Y) at (190:4) {};
\node[font=\scriptsize] at (190:4) {$\times$};
\draw[line width=.75pt, color=arccolour, line cap=round] (X) -- (252:4.5);
\draw[line width=.75pt, color=arccolour, line cap=round] (X) -- (288:4.5);
\draw[line width=.75pt, color=arccolour, line cap=round] (X) -- (Y);
\draw[line width=.75pt, color=arccolour, line cap=round] ($(252:4.5)+(-3.5,0)$) -- ($(252:4.5)+(-3.5,0)+(120:2)$);
\draw[line width=.75pt, color=arccolour, line cap=round] ($(252:4.5)+(-2,0)$) -- (Y);
\node[font=\scriptsize,left=-.2ex] at ($(252:3)+(-2.6,0)$) {$\gamma_0$};
\node[font=\scriptsize,left=-.3ex] at (252:3) {$\gamma_1$};
\node[font=\scriptsize] at (178:2) {$\gamma_2$};
\node[font=\scriptsize] at  ($(252:4.5)+(-3.5,0)+(105:1.9)$) {$\gamma$};

\node[font=\scriptsize,right=-.3ex] at (288:3) {$\gamma_{3}$};
\node[font=\scriptsize,color=stopcolour] at (270:1em) {$\bullet$};
\end{scope}
\begin{scope}[xshift=22em]
\draw[line width=.5pt] ($(252:4.5)+(-2.5,0)$) --(252:4.5);
\draw[line width=.5pt,postaction={decorate}] ($(252:4.5)+(-1.5,0)$) --(252:4.5);
\draw[line width=.5pt,postaction={decorate}] (252:4.5)--(288:4.5);
\draw[line width=.5pt] (288:4.5)-- ($(288:4.5)+(2,0)$);
\node[font=\scriptsize,color=stopcolour] at ($(288:4.5)+(1,0)$) {$\bullet$};
\node[font=\scriptsize,color=stopcolour] at ($(190:4em)+(-40:.7em)$) {$\bullet$};
\node[font=\scriptsize] at (0,0) {$\times$};
\draw[->, line width=.5pt] (249:1em) arc[start angle=249, end angle=-69, radius=1em];
\draw[->, line width=.5pt] ($(190:4em)+(270:.7em)$) arc[start angle=270, end angle=15, radius=.7em];
\node[font=\scriptsize,shape=circle,scale=.6] (X) at (0,0) {};
\node[font=\scriptsize,shape=circle,scale=.6] (Y) at (190:4) {};
\node[font=\scriptsize] at (190:4) {$\times$};
\draw[line width=.75pt, color=arccolour, line cap=round] (X) -- (252:4.5);
\draw[line width=.75pt, color=arccolour, line cap=round] (X) -- (288:4.5);
\draw[line width=.75pt, color=arccolour, line cap=round] (Y) -- ($(190:4)+(120:2em)$);
\draw[line width=.75pt, color=arccolour, line cap=round] (X) -- (Y);
\draw[line width=.75pt, color=arccolour, line cap=round] ($(252:4.5)+(-2,0)$) -- (Y);
\node[font=\scriptsize,left=-.2ex] at ($(252:3)+(-2.6,0)$) {$\gamma_0$};
\node[font=\scriptsize,left=-.3ex] at (252:3) {$\gamma_1$};
\node[font=\scriptsize] at (178:2) {$\gamma_2$};
\node[font=\scriptsize] at  ($(190:4em)+(140:1.9em)$)  {$\gamma$};

\node[font=\scriptsize] at ($(190:4em)+(200:1.2em)$) {$u_1$};
\node[font=\scriptsize] at ($(190:4em)+(70:1.2em)$) {$u_2$};
\node[font=\scriptsize,right=-.3ex] at (288:3) {$\gamma_{3}$};
\node[font=\scriptsize,color=stopcolour] at (270:1em) {$\bullet$};
\node[font=\scriptsize] at (270:5) {$p_1$};
\node[font=\scriptsize] at ($(270:5)+(-2.5,0)$) {$p_0$};
\node[font=\scriptsize] at (220:1.5) {$q_1$};
\node[font=\scriptsize] at (30:1.5) {$q_2$};
\end{scope}
\end{tikzpicture}
\]

For the first  configuration we consider the full subcategory $\mathbf B$ consisting of the arcs $\gamma_0, \gamma_1, \gamma, \gamma_3$. It follows from the first case that $\mathbf B$ is non-formal. Similarly, for the second one the full subcategory consisting of the arcs $\gamma, \gamma_1,\gamma_2, \gamma_3$ is non-formal. For the third one the full subcategory $\mathbf B$ consisting of $\gamma, \gamma_0, \gamma_1, \gamma_2, \gamma_3$ is given by the following DG quiver 
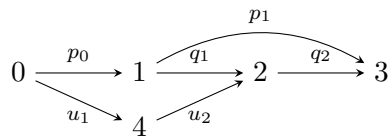
\[
\begin{tikzpicture}[baseline=-2.6pt,description/.style={fill=white,inner sep=1pt,outer sep=0}]
\matrix (m) [matrix of math nodes, row sep=.5em, text height=1.5ex, column sep=3em, text depth=0.25ex, ampersand replacement=\&, inner sep=3.5pt]
{
0 \& 1 \& 2 \&  3\\
\& 4\\
};
\path[->,line width=.4pt, font=\scriptsize] (m-1-1) edge node[above=.02ex] {$p_0$} (m-1-2);
\path[->,line width=.4pt, font=\scriptsize] (m-1-2) edge node[above=.01ex] {$q_1$} (m-1-3);
\path[->,line width=.4pt, font=\scriptsize] (m-1-3) edge node[above=.01ex] {$q_{2}$} (m-1-4);
\path[->,line width=.4pt, font=\scriptsize] (m-1-1) edge node[below=.01ex] {$u_1$} (m-2-2);
\path[->,line width=.4pt, font=\scriptsize] (m-2-2) edge node[below=.01ex] {$u_2$} (m-1-3);
\path[->,line width=.4pt, font=\scriptsize,bend left=30] (m-1-2) edge node[above=.01ex] {$p_{1}$} (m-1-4);
\end{tikzpicture}
\]
with the relation $q_1p_0 = u_2u_1, q_2u_2=0$ and the differential $d(p_1) = q_2q_1$. The cohomology $\H^\bullet(\mathbf B)$ is given by the following graded quiver
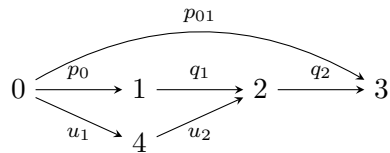
\[
\begin{tikzpicture}[baseline=-2.6pt,description/.style={fill=white,inner sep=1pt,outer sep=0}]
\matrix (m) [matrix of math nodes, row sep=.5em, text height=1.5ex, column sep=3em, text depth=0.25ex, ampersand replacement=\&, inner sep=3.5pt]
{
0 \& 1 \& 2 \&  3\\
\& 4\\
};
\path[->,line width=.4pt, font=\scriptsize] (m-1-1) edge node[above=.02ex] {$p_0$} (m-1-2);
\path[->,line width=.4pt, font=\scriptsize] (m-1-2) edge node[above=.01ex] {$q_1$} (m-1-3);
\path[->,line width=.4pt, font=\scriptsize] (m-1-3) edge node[above=.01ex] {$q_{2}$} (m-1-4);
\path[->,line width=.4pt, font=\scriptsize] (m-1-1) edge node[below=.01ex] {$u_1$} (m-2-2);
\path[->,line width=.4pt, font=\scriptsize] (m-2-2) edge node[below=.01ex] {$u_2$} (m-1-3);
\path[->,line width=.4pt, font=\scriptsize,bend left=30] (m-1-1) edge node[above=.01ex] {$p_{01}$} (m-1-4);
\end{tikzpicture}
\]
with the relations $q_1p_0 = u_2u_1, q_2u_2=q_2q_1=0$. By comparing the Hochschild cohomology of $\H^\bullet(\mathbf B)$ and $\mathbf B$ we obtain that $\mathbf B$ is non-formal and moreover there is an A$_\infty$ structure on $\H^\bullet(\mathbf B)$ 
\[
\mu_3(\s q_2 \otimes \s q_1 \otimes \s p_0) = \s p_{01} = \mu_3(\s u_2 \otimes \s u_1 \otimes \s p_0).
\]
This proves the claim. 
\end{proof}

\begin{remark}\label{remark:disktwoorbifoldonestop}
Let $\mathbf S $ be the disk with two orbifold points and one stop in the boundary. The dissections $\Delta$ in Fig.~\ref{fig:nonformaldissectionformaldg} are not formal but the DG categories $\mathbf A_\Delta$ are formal. 
\begin{figure}[ht]
\begin{tikzpicture}[x=1em,y=1em,decoration={markings,mark=at position 0.5 with {\arrow[black]{Stealth[length=4.8pt]}}}]
\begin{scope}[transform shape, scale=.9]
\begin{scope}
\draw[line width=.5pt] circle(4.5em);
\node[font=\scriptsize,color=stopcolour] at (-90:1em) {$\bullet$};
\node[font=\scriptsize] at (0,0) {$\times$};
\draw[->, line width=.5pt] (249:1em) arc[start angle=249, end angle=-69, radius=1em];
\draw[->, line width=.5pt] ($(180:2.5em)+(200:.7em)$) arc[start angle=200, end angle=15, radius=.7em];
\node[font=\scriptsize,shape=circle,scale=.6] (X) at (0,0) {};
\node[font=\scriptsize,shape=circle,scale=.6] (Y) at (180:2.5) {};
\node[font=\scriptsize] at (180:2.5) {$\times$};
\node[font=\scriptsize,color=stopcolour] at ($(180:2.5)+(-60:1em)$) {$\bullet$};
\draw[line width=.75pt, color=arccolour, line cap=round] (X) -- (252:4.5);
\draw[line width=.75pt, color=arccolour, line cap=round] (X) -- (288:4.5);
\draw[line width=.75pt, color=arccolour, line cap=round] (X) -- (Y);
\draw[line width=.75pt, color=arccolour, line cap=round] (200:4.5) -- (Y);
\node[font=\scriptsize] at ($(180:2.5em)+(150:1em)$) {$r$};
\node[font=\scriptsize,color=stopcolour] at (90:4.5em) {$\bullet$};
\draw[line width=0pt,postaction={decorate}] (-95:4.5em) arc[start angle=-95, end angle=-85, radius=4.5em];
\node[font=\scriptsize] at (-90:5em) {$p_1$};
\draw[line width=0pt,postaction={decorate}] (-140:4.5em) arc[start angle=-140, end angle=-125, radius=4.5em];
\node[font=\scriptsize] at (-135:5em) {$p_0$};
\node[font=\scriptsize] at (220:1.5) {$q_1$};
\node[font=\scriptsize] at (30:1.5) {$q_2$};
\end{scope}
\begin{scope}[xshift=15em]
\draw[line width=.5pt] circle(4.5em);
\node[font=\scriptsize,color=stopcolour] at (-90:1em) {$\bullet$};
\node[font=\scriptsize] at (0,0) {$\times$};
\draw[->, line width=.5pt] (249:1em) arc[start angle=249, end angle=-69, radius=1em];
\draw[->, line width=.5pt] ($(180:2.5em)+(-5:.7em)$) arc[start angle=-5, end angle=-305, radius=.7em];
\node[font=\scriptsize,shape=circle,scale=.6] (X) at (0,0) {};
\node[font=\scriptsize,shape=circle,scale=.6] (Y) at (180:2.5) {};
\node[font=\scriptsize] at (180:2.5) {$\times$};
\node[font=\scriptsize,color=stopcolour] at ($(180:2.5)+(18:1em)$) {$\bullet$};
\draw[line width=.75pt, color=arccolour, line cap=round] (X) -- (252:4.5);
\draw[line width=.75pt, color=arccolour, line cap=round] (X) -- (288:4.5);
\draw[line width=.75pt, color=arccolour, line cap=round] (X) -- (Y);
\draw[line width=.75pt, color=arccolour, line cap=round] (60:4.5) -- (Y);
\node[font=\scriptsize] at ($(180:2.5em)+(190:1em)$) {$r$};
\node[font=\scriptsize,color=stopcolour] at (90:4.5em) {$\bullet$};
\draw[line width=0pt,postaction={decorate}] (-95:4.5em) arc[start angle=-95, end angle=-85, radius=4.5em];
\node[font=\scriptsize] at (-90:5em) {$p_1$};
\draw[line width=0pt,postaction={decorate}] (-10:4.5em) arc[start angle=-10, end angle=0, radius=4.5em];
\node[font=\scriptsize] at (-5:5.2em) {$p_2$};
\node[font=\scriptsize] at (220:1.5) {$q_1$};
\node[font=\scriptsize] at (30:1.5) {$q_2$};
\end{scope}
\end{scope}
\end{tikzpicture}
\caption{The non-formal DG dissections $\Delta$ on the disk with two orbifold points and with one stop give formal DG categories $\mathbf A_\Delta$.}
\label{fig:nonformaldissectionformaldg}
\end{figure}
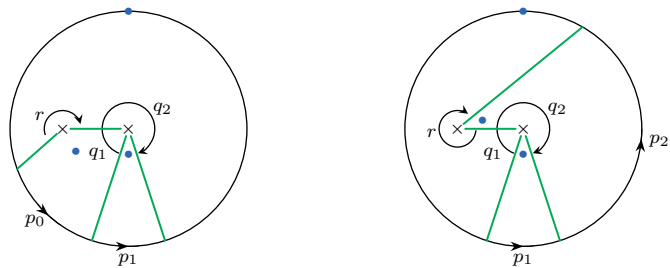

Let us only consider the left dissection. Note that the DG category $\mathbf A_\Delta$ is given by the DG quiver (left)  
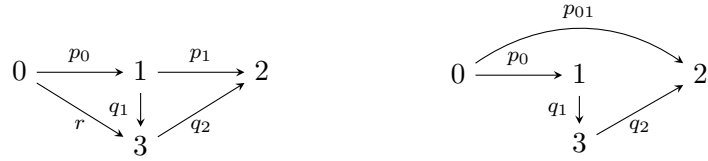
\[
\begin{tikzpicture}[baseline=-2.6pt,description/.style={fill=white,inner sep=1pt,outer sep=0}]
\begin{scope}
\matrix (m) [matrix of math nodes, row sep=1.2em, text height=1.5ex, column sep=3em, text depth=0.25ex, ampersand replacement=\&, inner sep=3.5pt]
{
0 \& 1 \& 2 \\
\& 3\\
};
\path[->,line width=.4pt, font=\scriptsize] (m-1-1) edge node[above=.02ex] {$p_0$} (m-1-2);
\path[->,line width=.4pt, font=\scriptsize] (m-1-2) edge node[above=.01ex] {$p_1$} (m-1-3);
\path[->,line width=.4pt, font=\scriptsize] (m-1-1) edge node[below=.01ex] {$r$} (m-2-2);
\path[->,line width=.4pt, font=\scriptsize] (m-2-2) edge node[below=.01ex] {$q_2$} (m-1-3);
\path[->,line width=.4pt, font=\scriptsize] (m-1-2) edge node[left=.01ex] {$q_1$} (m-2-2);
\end{scope}
\begin{scope}[xshift=15em]
\matrix (m) [matrix of math nodes, row sep=1em, text height=1.5ex, column sep=3em, text depth=0.25ex, ampersand replacement=\&, inner sep=3.5pt]
{
0 \& 1 \& 2 \\
\& 3\\
};
\path[->,line width=.4pt, font=\scriptsize] (m-1-1) edge node[above=.02ex] {$p_0$} (m-1-2);
\path[->,line width=.4pt, font=\scriptsize] (m-2-2) edge node[below=.01ex] {$q_2$} (m-1-3);
\path[->,line width=.4pt, font=\scriptsize] (m-1-2) edge node[left=.01ex] {$q_1$} (m-2-2);
\path[->,line width=.4pt, font=\scriptsize,bend left=35] (m-1-1) edge node[above=.01ex] {$p_{01}$} (m-1-3);
\end{scope}
\end{tikzpicture}
\]
with the relations $r = q_1p_0$ and $q_2r=0$ and the differential $d (p_1) = q_2q_1$. The cohomology $\H^\bullet(\mathbf A_\Delta)$ is given by the quiver (right) with the relation 
$ q_2q_1=0$. We may show that $\mathbf A_\Delta$ and $\H^\bullet(\mathbf A_\Delta)$ are A$_\infty$ quasi-isomorphic, since $\HH^2(\H^\bullet(\mathbf A_\Delta),\H^\bullet(\mathbf A_\Delta)) = 0$ and thus $\H^\bullet(\mathbf A_\Delta)$ is intrinsically formal by \cite[Theorem 4.7]{seidelthomas}. 
\end{remark}

We now give a simple example of an admissible DG dissection $\Delta$ which is not formal. (See also the dissection $\Delta_5$ in \S\ref{subsection:orbifolddiskthree} for a further example.)

\begin{example}\label{example:non-formaldg}
Consider the orbifold disk with one orbifold point and with three stops on the boundary. Take the admissible dissection $\Delta$ as illustrated in Fig.~\ref{fig:nonformaldg}. 
 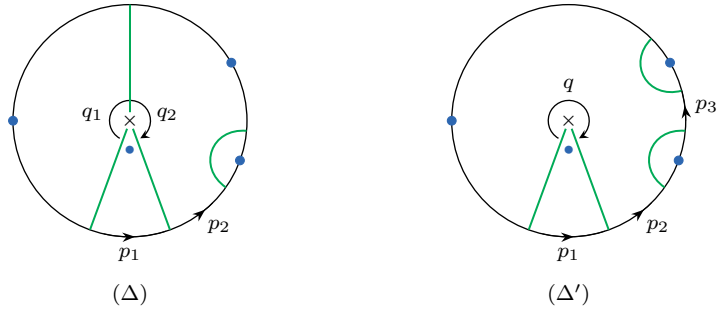
\begin{figure}[ht]
\begin{tikzpicture}[x=1em,y=1em,decoration={markings,mark=at position 0.5 with {\arrow[black]{Stealth[length=4.8pt]}}}]
\begin{scope}
\draw[line width=.5pt] circle(4em);
\node[font=\scriptsize,shape=circle,scale=.6,fill=white] (Y) at (0,0) {};
\node[font=\scriptsize] at (0,0) {$\times$};
\draw[->, line width=.5pt] (240:.7em) arc[start angle=240, end angle=-60, radius=.7em];
\node[font=\scriptsize] at (170:1.3em) {$q_1$};
\node[font=\scriptsize] at (10:1.3em) {$q_2$};
\node[font=\scriptsize,color=stopcolour] at (-90:1em) {$\bullet$};
\draw[line width=.75pt,color=arccolour] (-110:4em) to (Y);
\draw[line width=.75pt,color=arccolour] (90:4em) to (Y);
\draw[line width=.75pt,color=arccolour] (-70:4em) to (Y);
\path[line width=.75pt,out=175,in=145,looseness=1.5,color=arccolour] (-5:4em) edge (-35:4em);

\draw[line width=0pt,postaction={decorate}] (240+17:4em) arc[start angle=240+17, end angle=300-17+5, radius=4em];
\node[font=\scriptsize] at (270:4.6em) {$p_1$};
\draw[line width=0pt,postaction={decorate}] (-55:4em) arc[start angle=-55, end angle=-45, radius=4em];
\node[font=\scriptsize] at (-50:4.75em) {$p_2$};
\draw[fill=stopcolour, color=stopcolour] (-20:4em) circle(.15em);
\draw[fill=stopcolour, color=stopcolour] (180:4em) circle(.15em);
\draw[fill=stopcolour, color=stopcolour] (30:4em) circle(.15em);
\node[font=\scriptsize] at (-90:6em) {$(\Delta)$};
\end{scope}
\begin{scope}[xshift=15em]
\draw[line width=.5pt] circle(4em);
\node[font=\scriptsize,shape=circle,scale=.6,fill=white] (Y) at (0,0) {};
\node[font=\scriptsize] at (0,0) {$\times$};
\draw[->, line width=.5pt] (240:.7em) arc[start angle=240, end angle=-60, radius=.7em];
\node[font=\scriptsize] at (90:1.2em) {$q$};
\node[font=\scriptsize,color=stopcolour] at (-90:1em) {$\bullet$};
\draw[line width=.75pt,color=arccolour] (-110:4em) to (Y);
\draw[line width=.75pt,color=arccolour] (-70:4em) to (Y);
\path[line width=.75pt,out=175,in=145,looseness=1.5,color=arccolour] (-5:4em) edge (-35:4em);
\path[line width=.75pt,out=195,in=225,looseness=1.5,color=arccolour] (15:4em) edge (45:4em);
\draw[line width=0pt,postaction={decorate}] (240+17:4em) arc[start angle=240+17, end angle=300-17+5, radius=4em];
\node[font=\scriptsize] at (270:4.6em) {$p_1$};
\draw[line width=0pt,postaction={decorate}] (-55:4em) arc[start angle=-55, end angle=-45, radius=4em];
\node[font=\scriptsize] at (-50:4.75em) {$p_2$};
\draw[line width=0pt,postaction={decorate}] (0:4em) arc[start angle=0, end angle=15, radius=4em];
\node[font=\scriptsize] at (6:4.75em) {$p_3$};
\draw[fill=stopcolour, color=stopcolour] (-20:4em) circle(.15em);
\draw[fill=stopcolour, color=stopcolour] (180:4em) circle(.15em);
\draw[fill=stopcolour, color=stopcolour] (30:4em) circle(.15em);
\node[font=\scriptsize] at (-90:6em) {$(\Delta')$};
\end{scope}
\end{tikzpicture}
\caption{The DG dissection $\Delta$ (left) is  non-formal and $\Delta'$ (right) is formal.}
\label{fig:nonformaldg}
\end{figure}

The algebra (category) $\mathbf A_\Delta$ is a DG algebra given by the following quiver 
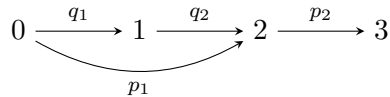
\[
\begin{tikzpicture}[baseline=-2.6pt,description/.style={fill=white,inner sep=1pt,outer sep=0}]
\matrix (m) [matrix of math nodes, row sep=2.5em, text height=1.5ex, column sep=3em, text depth=0.25ex, ampersand replacement=\&, inner sep=3.5pt]
{
0 \& 1 \& 2 \& 3 \\
};
\path[->,line width=.4pt, font=\scriptsize] (m-1-1) edge node[above=.01ex] {$q_1$} (m-1-2);
\path[->,line width=.4pt, font=\scriptsize] (m-1-2) edge node[above=.01ex] {$q_2$} (m-1-3);
\path[->,line width=.4pt, font=\scriptsize] (m-1-3) edge node[above=.01ex] {$p_2$} (m-1-4);
\path[->,line width=.4pt, font=\scriptsize,bend right=30] (m-1-1) edge node[below=.01ex] {$p_1$} (m-1-3);
\end{tikzpicture}
\]
with the relations $q_2p_2 = 0$ and the differential $d(p_1) = q_2q_1$. Here, we may choose a line field so that $|p_1|=-1$ and $|q_1|=|q_2|=|p_2|=0$. We claim that $\mathbf A_\Delta$ is {\it not} formal. Indeed, $\H^\bullet(\mathbf A_\Delta)$ is given by 
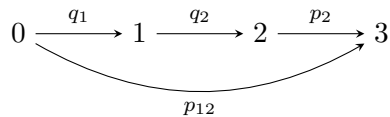
\[
\begin{tikzpicture}[baseline=-2.6pt,description/.style={fill=white,inner sep=1pt,outer sep=0}]
\matrix (m) [matrix of math nodes, row sep=2.5em, text height=1.5ex, column sep=3em, text depth=0.25ex, ampersand replacement=\&, inner sep=3.5pt]
{
0 \& 1 \& 2 \& 3 \\
};
\path[->,line width=.4pt, font=\scriptsize] (m-1-1) edge node[above=.01ex] {$q_1$} (m-1-2);
\path[->,line width=.4pt, font=\scriptsize] (m-1-2) edge node[above=.01ex] {$q_2$} (m-1-3);
\path[->,line width=.4pt, font=\scriptsize] (m-1-3) edge node[above=.01ex] {$p_2$} (m-1-4);
\path[->,line width=.4pt, font=\scriptsize,bend right=30] (m-1-1) edge node[below=.01ex] {$p_{12}$} (m-1-4);
\end{tikzpicture}
\]
with the relations $q_2q_1=0= p_2q_2$. Here, the arrow $p_{12}$ represents the nonzero cocycle $[p_2p_1]$ and thus $|p_{12}|=-1$. It follows from Theorem \ref{theorem:formaldg} that $\mathbf A_\Delta$ is non-formal and moreover the minimal A$_\infty$ structure on $\H^\bullet(\mathbf A_\Delta)$ is given by 
\[
\mu_3(\s p_2 \otimes \s q_2 \otimes \s q_1) = \s p_{12}.
\]

By Corollary \ref{corollary:typeD}, the minimal A$_\infty$ category $\H^\bullet(\mathbf A_\Delta)$ is Morita equivalent to the graded algebra $\H^\bullet(\mathbf A_{\Delta'})$ given by  the quiver of type D$_4$. 
\end{example}

\subsection{Graded skew-gentle algebras}

Let $\Bbbk Q/J$ be a graded gentle algebra. Pick a subset $\mathrm{Sp}$ of loops $\epsilon_i$ of degree $0$ in $Q$ such that $\epsilon_i^2 = 0$, i.e.\ 
\[
\mathrm{Sp}:= \{\epsilon_i \mid \epsilon_i ^ 2 \in J \}.
\]
The loop $\epsilon_i$ in $\mathrm{Sp}$ and the corresponding vertex $i$ are usually called {\it special}.
To $(Q, J, \mathrm{Sp})$ we may associate a graded algebra $A = \Bbbk Q/ I$ where the new ideal $I$ is obtained from $J$ by changing the relations $\epsilon_i^2 = 0$ at each loop $\epsilon_i \in \mathrm{Sp}$ into $\epsilon_i^2 = \epsilon_i$. The graded algebra $A$ is called a {\it graded skew-gentle} algebra, which was first introduced in \cite{geisspena}, see also \cite[\S 4]{bessenrodtholm}.

Since the relation $\epsilon_i^2 = \epsilon_i$ contains a linear term and is hence not admissible (in the sense of admissibility for relations in path algebras of quivers), we may split the idempotent $e_i$ into two primitive orthogonal idempotents $e_i - \epsilon_i$ and $\epsilon_i$. As a result, the algebra $A$ can be described by a quiver with admissible relations $(\bar Q, \bar I)$. That is, $A \simeq \Bbbk \bar Q/\bar I$. Roughly speaking, each special vertex $i$ splits into two vertices $i^+$ and $i^-$ in $\bar Q_0$, which corresponds to the idempotents $\epsilon_i$ and $e_i - \epsilon_i$. Then each arrow $p$ starting (resp.\ ending) at a special vertex $i$ in $Q$ also splits into two arrows starting (resp.\ ending) at $i^+$ and $i^-$, denoted by $p^+$ and $p^-$ (resp.\ ${^+}p$ and ${^-}p$), respectively. In particular, an arrow $p$ starting and ending at special vertices splits into four arrows ${^\pm} p ^\pm$. The degree of each new arrow equals the degree of the original one. 

The ideal $\bar I$ is generated by the following relations. Let $p_2p_1 = 0$ be a relation in $I$. If the ending vertex $i$ of $p_1$ is not special then we obtain relations   ${^\ast} p_2 p_1^{\ast}$ in $\bar I$ where $\ast= \emptyset, +, -$ depending on the existence of ${^\pm} p_2$ and $p_1^\pm$. Otherwise we obtain anti-commutative relations $$({^{\ast_2}} p_2^+) ({^+} p_1^{\ast_1}) + ({^{\ast_2}}p_2^-)({^-} p_1^{\ast_1} )= 0$$ where $\ast_1, \ast_2 \in \{\emptyset, +, -\}$.

\begin{theorem}\label{theorem:skewgentlealgebras}
Let $\mathbf S = (S, \Sigma, \eta)$ be a graded orbifold surface with stops. Then there exists a formal admissible dissection $\Delta$ of $\mathbf S$ such that there are exactly two arcs connecting to each orbifold point $x$ and $\val (\vtimes_x) = 2$.

Furthermore, $\H^\bullet (\mathbf A_\Delta)$ is isomorphic to a graded skew-gentle algebra and we have a triangulated equivalence
\[
\mathcal W (\mathbf S) \simeq \per (\H^\bullet (\mathbf A_\Delta)).
\]
\end{theorem}

\begin{proof}
Recall that the underlying graded algebra of $\mathbf A_\Delta$ is a graded gentle algebra $\Bbbk Q/I$, see Remark \ref{remark:deformationinterpretation}. By the assumption, each orbifold point $x$ provides a subquiver in $Q$  (it is allowed to have no arrows $p_0, p_1$)
\begin{equation*}
\begin{tikzpicture}[baseline=-2.6pt,description/.style={fill=white,inner sep=1pt,outer sep=0}]
\matrix (m) [matrix of math nodes, row sep=2.5em, text height=1.5ex, column sep=3em, text depth=0.25ex, ampersand replacement=\&, inner sep=3.5pt]
{
0\& 1 \& 2 \&3 \\
};
\path[->,line width=.4pt, font=\scriptsize] (m-1-1) edge node[above=.01ex] {$p_0$} (m-1-2);
\path[->,line width=.4pt, font=\scriptsize, transform canvas={yshift=.5ex}] (m-1-2) edge node[above=.01ex] {$q_x$} (m-1-3);
\path[->,line width=.4pt, font=\scriptsize, transform canvas={yshift=-.5ex}] (m-1-2) edge node[below=.01ex] {$p_x$} (m-1-3);
\path[->,line width=.4pt, font=\scriptsize] (m-1-3) edge node[above=.01ex] {$p_1$} (m-1-4);
\end{tikzpicture}
\end{equation*}
with $d(p_x) = q_x$, where the relations are given by $q_xp_0= 0 = p_1q_x$ if the arrows $p_0, p_1$ exist.  Taking the cohomology we obtain the following quiver
\begin{equation*}
\begin{tikzpicture}[baseline=-2.6pt,description/.style={fill=white,inner sep=1pt,outer sep=0}]
\matrix (m) [matrix of math nodes, row sep=2.5em, text height=1.5ex, column sep=3em, text depth=0.25ex, ampersand replacement=\&, inner sep=3.5pt]
{
0\& 1 \& 2 \&3 \\
};
\path[->,line width=.4pt, font=\scriptsize, bend left] (m-1-1) edge node[above=.01ex] {$p_0^-$} (m-1-3);
\path[->,line width=.4pt, font=\scriptsize] (m-1-1) edge node[below=.01ex] {$p_0^+$} (m-1-2);
\path[->,line width=.4pt, font=\scriptsize] (m-1-3) edge node[above=.01ex] {$p_1^+$} (m-1-4);
\path[->,line width=.4pt, font=\scriptsize, bend right] (m-1-2) edge node[below=.01ex] {$p_1^-$} (m-1-4);
\end{tikzpicture}
\end{equation*}
with the commutative square relation $p_1^+p_0^- = p_1^- p_0^-$. By the description of graded skew-gentle algebras before this proposition, we note that the above algebra is isomorphic to the following one
\begin{equation*}
\begin{tikzpicture}[baseline=-2.6pt,description/.style={fill=white,inner sep=1pt,outer sep=0}]
\matrix (m) [matrix of math nodes, row sep=2.5em, text height=1.5ex, column sep=3em, text depth=0.25ex, ampersand replacement=\&, inner sep=3.5pt]
{
0\& 12 \&3 \\ 
};
\path[->,line width=.4pt, font=\scriptsize] (m-1-1) edge node[above=.01ex] {$p_0$} (m-1-2);
\path[->, line width=.5pt, font=\scriptsize, looseness=12,in=150,out=30] (m-1-2);
\path[->,line width=.4pt, font=\scriptsize, out=120, in=60, looseness=8] (m-1-2) edge node[below=.01ex]{$\epsilon$} (m-1-2);
\path[->,line width=.4pt, font=\scriptsize] (m-1-2) edge node[above=.01ex] {$p_1$} (m-1-3);
\end{tikzpicture}
\end{equation*}
with relations $\epsilon^2 = \epsilon$ and $p_1p_0=0$. In this way, we see that $\H^\bullet(\mathbf A_\Delta)$ is isomorphic to a graded skew-gentle algebra. 
\end{proof}

\begin{corollary}
\label{corollary:formalgenerator}
Let $\mathbf S = (S, \Sigma, \eta)$ be a graded orbifold surface with stops and $\Gamma$ be any formal generator of $\mathcal W (\mathbf S)$. Then the graded endomorphism algebra of $\Gamma$ in $\mathcal W (\mathbf S)$ is (perfect) derived equivalent to a graded skew-gentle algebra.
\end{corollary}

\begin{proof}
If $\Gamma$ is any generator and $\mathbf A $ is its derived endomorphism A$_\infty$ algebra, then we have a triangulated equivalence
\[
\mathcal W (\mathbf S) \simeq \H^0 (\tw (\mathbf A)^\natural).
\]
Since $\Gamma$ is a formal generator, $\mathbf A$ is A$_\infty$-quasi-isomorphic to its cohomology algebra $\H^\bullet (\mathbf A)$ which is isomorphic to the graded endomorphism algebra of $\Gamma$ in $\mathcal W(\mathbf S)$. We thus have triangulated equivalences
\[
\per (\H^\bullet (\mathbf A)) \simeq \H^0 (\tw (\H^\bullet (\mathbf A))^\natural) \simeq \mathcal W (\mathbf S) \simeq \per (A)
\]
for some skew-gentle algebra $A$ associated to $\mathbf S$, where the last equivalence follows from Theorem \ref{theorem:skewgentlealgebras}.
\end{proof}

We formulate the following conjecture which can be viewed as a converse of Corollary \ref{corollary:formalgenerator}.

\begin{conjecture}
\label{conjecture:skewgentle}
Let $A$ be a graded skew-gentle algebra and let $\mathbf S = (S, \Sigma, \eta)$ be the graded orbifold surface associated to $A$.

For any graded associative algebra $B$ which is (perfect) derived equivalent to $A$ there exists a formal dissection $\Delta$ of $\mathbf S$ such that $B \simeq \H^\bullet (\mathbf A_\Delta)$ as graded associative algebras.

In particular, the formal dissections of $\mathbf S$ describe the complete class of graded associative algebras derived equivalent to $A$.
\end{conjecture}

It was shown in \cite[Theorem 3.8]{labardinifragososchrollvaldivieso} that, with the exception of the skew-gentle algebras given by one arrow between two vertices and either one or two special loops,
\[
\begin{tikzpicture}[baseline=-2.6pt,description/.style={fill=white,inner sep=1pt,outer sep=0}]
\matrix (m) [matrix of math nodes, row sep=2.5em, text height=1.5ex, column sep=3em, text depth=0.25ex, ampersand replacement=\&, inner sep=3.5pt]
{
0\& 1 \&  \& 0 \& 1 \&  \& 0 \& 1 \\ 
};
\path[->,line width=.4pt, font=\scriptsize] (m-1-1) edge node[above=.01ex] {} (m-1-2);
\path[->,line width=.4pt, font=\scriptsize, out=120, in=60, looseness=5] (m-1-2) edge node[below=.01ex]{} (m-1-2);
\path[->,line width=.4pt, font=\scriptsize] (m-1-4) edge node[above=.01ex] {} (m-1-5);
\path[->,line width=.4pt, font=\scriptsize, out=120, in=60, looseness=5] (m-1-4) edge node[below=.01ex]{} (m-1-4);
\path[->,line width=.4pt, font=\scriptsize] (m-1-7) edge node[above=.01ex] {} (m-1-8);
\path[->,line width=.4pt, font=\scriptsize, out=120, in=60, looseness=5] (m-1-7) edge node[below=.01ex]{} (m-1-7);
\path[->,line width=.4pt, font=\scriptsize, out=120, in=60, looseness=5] (m-1-8) edge node[below=.01ex]{} (m-1-8);
\end{tikzpicture}
\]
the orbifold surface associated to a skew-gentle algebra is unique up to diffeomorphism. These three cases are exactly those for which the skew-gentle algebra is isomorphic to a gentle algebra in which case the algebra also admits a smooth surface model, also see \cite[Lemma 3.6]{amiotbruestle}.

In Conjecture \ref{conjecture:skewgentle} the (orbifold) surface is thus unique up to diffeomorphism except for these three cases. In the case of one special loop, the orbifold surface model is a disk with one orbifold point and two stops (a Weinstein sector of type $\mathrm D_3$) and the smooth surface model (as a gentle algebra) is a disk with four stops (a Weinstein sector of type $\mathrm A_3$). The fact that there are two surface models can be viewed as an incarnation of the fact that the Dynkin diagrams of type $\mathrm A_3$ and $\mathrm D_3$ coincide. In the case of two special loops, the orbifold surface model is a disk with two orbifold points and one stop (see Remark \ref{remark:disktwoorbifoldonestop}) and the smooth surface model is an annulus with two stops on each boundary component.

In the case when $\mathbf S$ is a disk with one orbifold point and stops in the boundary, Conjecture \ref{conjecture:skewgentle} relates the dissections of $\mathbf S$ and the algebras derived equivalent to $\Bbbk Q$ where $Q$ is a quiver of type D. This case has recently been shown in \cite{amiotplamondon}.

Conjecture \ref{conjecture:skewgentle} generalizes the following folklore conjecture about graded gentle algebras.

\begin{conjecture}
\label{conjecture:gentle}
The class of graded gentle algebras is closed under derived equivalence. In particular, if $\mathbf S = (S, \Sigma, \eta)$ is a graded {\it smooth} surface with stops, then any formal generator of $\mathcal W (\mathbf S)$ is given by a formal dissection of $\mathbf S$.
\end{conjecture}

In the ungraded case, the first part of the statement of Conjecture \ref{conjecture:gentle} is a theorem of Schröer and Zimmermann \cite{schroeerzimmermann} who noted that ``[e]xamples of classes of algebras closed under derived equivalence are very rare and show that gentle algebras deserve much attention''. We also expect an ungraded version of Conjecture \ref{conjecture:skewgentle} to hold. Note that building on \cite{lekilipolishchuk2}, it is shown in  \cite{amiotplamondonschroll,jinschrollwang} that  the derived equivalence class of a (ungraded or graded) gentle algebra is determined by the homotopy class of the line field on the associated smooth surface and we expect an analogous result for orbifold surfaces.

\subsection{Examples of associative algebras derived equivalent to a skew-gentle algebra}\label{subsection:orbifolddiskthree}

In this subsection, we provide explicit examples of formal dissections on the orbifold disk with three orbifold points with one stop on the boundary. 

Let $\mathbf S = (S, \Sigma, \eta)$ be the graded orbifold disk with three orbifold points and with one stop on the boundary. Consider the five DG dissections $\Delta_1, \dotsc, \Delta_5$ on $\mathbf S$ illustrated in Figs.~\ref{fig:dgdissections} and \ref{fig:dgnonformaldissection}. Note that the first four are formal and the fifth is not formal (cf.\ Definition \ref{definition:DGformal}). 
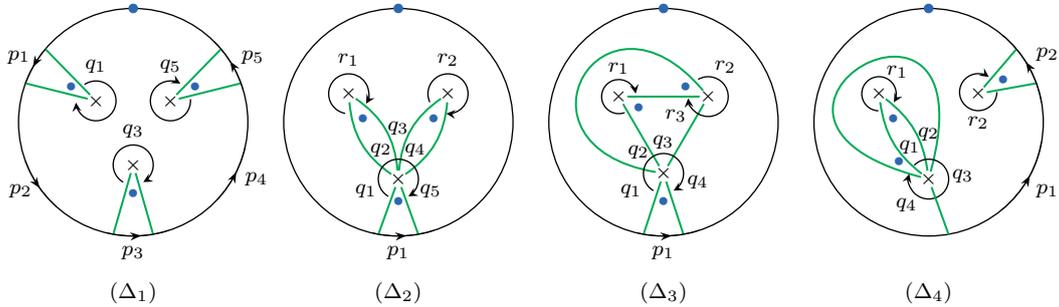
\begin{figure}[ht]
\begin{tikzpicture}[x=1em,y=1em,decoration={markings,mark=at position 0.55 with {\arrow[black]{Stealth[length=4.2pt]}}}]
\begin{scope}[scale=.98]
\begin{scope}[xshift=0]
\draw[line width=.5pt] circle(4em);
\node[font=\scriptsize,shape=circle,scale=.6,fill=white] (X) at (30:1.5em) {};
\node[font=\scriptsize] at (30:1.5em) {$\times$};
\node[font=\scriptsize,shape=circle,scale=.6,fill=white] (Y) at (-90:1.5em) {};
\node[font=\scriptsize] at (-90:1.5em) {$\times$};
\node[font=\scriptsize,shape=circle,scale=.6,fill=white] (Z) at (150:1.5em) {};
\node[font=\scriptsize] at (150:1.5em) {$\times$};
\draw[->, line width=.5pt] ($(30:1.5em)+(5:.7em)$) arc[start angle=5, end angle=-300, radius=.7em];
\node[font=\scriptsize] at ($(30:1.5em)+(90:1.2em)$) {$q_5$};
\node[font=\scriptsize,color=stopcolour] at (30:2.5em) {$\bullet$};
\draw[->, line width=.5pt] ($(-90:1.5em)+(240:.7em)$) arc[start angle=240, end angle=-60, radius=.7em];
\node[font=\scriptsize] at ($(-90:1.5em)+(90:1.2em)$){$q_3$};
\node[font=\scriptsize,color=stopcolour] at (-90:2.5em) {$\bullet$};
\draw[->, line width=.5pt] ($(150:1.5em)+(125:.7em)$) arc[start angle=125, end angle=-175, radius=.7em];
\node[font=\scriptsize] at ($(150:1.5em)+(90:1.2em)$) {$q_1$};
\node[font=\scriptsize,color=stopcolour] at (150:2.5em) {$\bullet$};
\draw[line width=.75pt,color=arccolour] (-100:4em) to (Y);
\draw[line width=.75pt,color=arccolour] (-80:4em) to (Y);
\draw[line width=.75pt,color=arccolour] (20:4em) to (X);
\draw[line width=.75pt,color=arccolour] (40:4em) to (X);
\draw[line width=.75pt,color=arccolour] (140:4em) to (Z);
\draw[line width=.75pt,color=arccolour] (160:4em) to (Z);
\draw[line width=0pt,postaction={decorate}] (140:4em) arc[start angle=140, end angle=160, radius=4em];
\node[font=\scriptsize] at (150:4.6em) {$p_1$};
\draw[line width=0pt,postaction={decorate}] (200:4em) arc[start angle=200, end angle=240-17+5, radius=4em];
\node[font=\scriptsize] at (211.5:4.6em) {$p_2$};
\draw[line width=0pt,postaction={decorate}] (240+17:4em) arc[start angle=240+17, end angle=300-17+5, radius=4em];
\node[font=\scriptsize] at (270:4.6em) {$p_3$};
\draw[line width=0pt,postaction={decorate}] (300+17:4em) arc[start angle=300+17, end angle=360-17+5, radius=4em];
\node[font=\scriptsize] at (335:4.75em) {$p_4$};
\draw[line width=0pt,postaction={decorate}] (17:4em) arc[start angle=17, end angle=40+5, radius=4em];
\node[font=\scriptsize] at (28.5:4.75em) {$p_5$};
\draw[fill=stopcolour, color=stopcolour] (90:4em) circle(.15em);
\node[font=\scriptsize] at (-90:6em) {$(\Delta_1)$};
\end{scope}
\begin{scope}[xshift=9.25em]
\draw[line width=.5pt] circle(4em);
\node[font=\scriptsize,shape=circle,scale=.6,fill=white] (X) at (30:2em) {};
\node[font=\scriptsize] at (30:2em) {$\times$};
\node[font=\scriptsize,shape=circle,scale=.6,fill=white] (Y) at (-90:2em) {};
\node[font=\scriptsize] at (-90:2em) {$\times$};
\node[font=\scriptsize,shape=circle,scale=.6,fill=white] (Z) at (150:2em) {};
\node[font=\scriptsize] at (150:2em) {$\times$};
\draw[-, line width=.75pt,color=arccolour, bend left=25] (Y) to (Z);
\draw[-,line width=.75pt,color=arccolour, bend right=25] (Y) to (Z);
\draw[-, line width=.75pt,color=arccolour, bend left=25] (X) to (Y);
\draw[-,line width=.75pt,color=arccolour, bend right=25] (X) to (Y);
\draw[->, line width=.5pt] ($(30:2em)+(205:.7em)$) arc[start angle=205, end angle=-95, radius=.7em];
\node[font=\scriptsize] at ($(30:2em)+(90:1.2em)$) {$r_2$};
\node[font=\scriptsize,color=stopcolour] at ($(30:2em)+(-120:1em)$) {$\bullet$};
\draw[->, line width=.5pt] ($(-90:2em)+(240:.7em)$) arc[start angle=240, end angle=-60, radius=.7em];
\node[font=\scriptsize] at ($(-90:2em)+(200:1.2em)$){$q_1$};
\node[font=\scriptsize] at ($(-90:2em)+(120:1.2em)$){$q_2$};
\node[font=\scriptsize] at ($(-90:2em)+(90:1.8em)$){$q_3$};
\node[font=\scriptsize] at ($(-90:2em)+(60:1.2em)$){$q_4$};
\node[font=\scriptsize] at ($(-90:2em)+(-20:1.2em)$){$q_5$};
\node[font=\scriptsize,color=stopcolour] at (-90:2.8em) {$\bullet$};
\draw[->, line width=.5pt] ($(150:2em)+(260:.7em)$) arc[start angle=260, end angle=-30, radius=.7em];
\node[font=\scriptsize] at ($(150:2em)+(90:1.2em)$) {$r_1$};
\node[font=\scriptsize,color=stopcolour] at ($(150:2em)+(-60:1em)$) {$\bullet$};
\draw[line width=.75pt,color=arccolour] (-100:4em) to (Y);
\draw[line width=.75pt,color=arccolour] (-80:4em) to (Y);
\draw[line width=0pt,postaction={decorate}] (240+17:4em) arc[start angle=240+17, end angle=300-17+5, radius=4em];
\node[font=\scriptsize] at (270:4.6em) {$p_1$};
\draw[fill=stopcolour, color=stopcolour] (90:4em) circle(.15em);
\node[font=\scriptsize] at (-90:6em) {$(\Delta_2)$};
\end{scope}
\begin{scope}[xshift=18.5em]
\draw[line width=.5pt] circle(4em);
\node[font=\scriptsize,shape=circle,scale=.6,fill=white] (X) at (30:1.8em) {};
\node[font=\scriptsize] at (30:1.8em) {$\times$};
\node[font=\scriptsize,shape=circle,scale=.6,fill=white] (Y) at (-90:1.8em) {};
\node[font=\scriptsize] at (-90:1.8em) {$\times$};
\node[font=\scriptsize,shape=circle,scale=.6,fill=white] (Z) at (150:1.8em) {};
\node[font=\scriptsize] at (150:1.8em) {$\times$};
\draw[line width=.75pt, color=arccolour, line cap=round] (X) to[out=130, in=170, looseness=4.2] (Y);
\draw[-, line width=.75pt,color=arccolour] (Y) to (Z);
\draw[-, line width=.75pt,color=arccolour] (X) to (Y);
\draw[-,line width=.75pt,color=arccolour] (Z) to (X);
\node[font=\scriptsize,color=stopcolour] at ($(150:2em)+(-30:1em)$) {$\bullet$};
\node[font=\scriptsize,color=stopcolour] at ($(30:2em)+(165:1em)$) {$\bullet$};
\draw[->, line width=.5pt] ($(30:1.8em)+(120:.7em)$) arc[start angle=120, end angle=-175, radius=.7em];
\node[font=\scriptsize] at ($(30:1.8em)+(60:1.2em)$) {$r_2$};
\node[font=\scriptsize] at ($(30:1.8em)+(-150:1.3em)$) {$r_3$};
\draw[->, line width=.5pt] ($(-90:1.8em)+(240:.7em)$) arc[start angle=240, end angle=-60, radius=.7em];
\node[font=\scriptsize] at ($(-90:2em)+(200:1.2em)$){$q_1$};
\node[font=\scriptsize] at ($(-90:2em)+(132:1.3em)$){$q_2$};
\node[font=\scriptsize] at ($(-90:2em)+(90:1.3em)$){$q_3$};
\node[font=\scriptsize] at ($(-90:2em)+(-5:1.2em)$){$q_4$};
\node[font=\scriptsize,color=stopcolour] at (-90:2.8em) {$\bullet$};
\draw[->, line width=.5pt] ($(150:1.8em)+(290:.6em)$) arc[start angle=290, end angle=10, radius=.6em];
\node[font=\scriptsize] at ($(150:1.8em)+(90:1em)$) {$r_1$};
\draw[line width=.75pt,color=arccolour] (-100:4em) to (Y);
\draw[line width=.75pt,color=arccolour] (-80:4em) to (Y);
\draw[line width=0pt,postaction={decorate}] (240+17:4em) arc[start angle=240+17, end angle=300-17+5, radius=4em];
\node[font=\scriptsize] at (270:4.6em) {$p_1$};
\draw[fill=stopcolour, color=stopcolour] (90:4em) circle(.15em);
\node[font=\scriptsize] at (-90:6em) {$(\Delta_3)$};
\end{scope}
\begin{scope}[xshift=27.75em]
\draw[line width=.5pt] circle(4em);
\node[font=\scriptsize,shape=circle,scale=.6,fill=white] (X) at (30:2em) {};
\node[font=\scriptsize] at (30:2em) {$\times$};
\node[font=\scriptsize,shape=circle,scale=.6,fill=white] (Y) at (-90:2em) {};
\node[font=\scriptsize] at (-90:2em) {$\times$};
\node[font=\scriptsize,shape=circle,scale=.6,fill=white] (Z) at (150:2em) {};
\node[font=\scriptsize] at (150:2em) {$\times$};
\draw[-, line width=.75pt,color=arccolour, bend left=22] (Y) to (Z);
\draw[-,line width=.75pt,color=arccolour, bend right=22] (Y) to (Z);
\draw[line width=.75pt, color=arccolour, line cap=round] (Y) to[out=165, in=75, looseness=50] (Y);
\draw[->, line width=.5pt] ($(30:2em)+(5:.65em)$) arc[start angle=0, end angle=-305, radius=.65em];
\node[font=\scriptsize] at ($(30:2em)+(-90:1em)$) {$r_2$};
\node[font=\scriptsize,color=stopcolour] at (30:3em) {$\bullet$};
\draw[->, line width=.5pt] ($(-90:2em)+(136:.7em)$) arc[start angle=136, end angle=-195, radius=.7em];
\node[font=\scriptsize] at ($(-90:2em)+(120:1.3em)$){$q_1$};
\node[font=\scriptsize] at ($(-90:2em)+(90:1.6em)$){$q_2$};
\node[font=\scriptsize] at ($(-90:2em)+(5:1.2em)$){$q_3$};
\node[font=\scriptsize] at ($(-90:2em)+(230:1.2em)$){$q_4$};
\draw[->, line width=.5pt] ($(150:2em)+(268:.6em)$) arc[start angle=268, end angle=-30, radius=.6em];
\node[font=\scriptsize] at ($(150:2em)+(50:1em)$) {$r_1$};
\node[font=\scriptsize,color=stopcolour] at ($(150:2em)+(-60:1em)$) {$\bullet$};
\node[font=\scriptsize,color=stopcolour] at ($(-90:2em)+(150:1.2em)$) {$\bullet$};
\draw[line width=.75pt,color=arccolour] (-80:4em) to (Y);
\draw[line width=.75pt,color=arccolour] (40:4em) to (X);
\draw[line width=.75pt,color=arccolour] (20:4em) to (X);
\draw[line width=0pt,postaction={decorate}] (30-20:4em) arc[start angle=30-20, end angle=30+20, radius=4em];
\node[font=\scriptsize] at (30:4.8em) {$p_2$};
\draw[line width=0pt,postaction={decorate}] (-30-20:4em) arc[start angle=-30-20, end angle=-30+20, radius=4em];
\node[font=\scriptsize] at (-30:4.8em) {$p_1$};
\draw[fill=stopcolour, color=stopcolour] (90:4em) circle(.15em);
\node[font=\scriptsize] at (-90:6em) {$(\Delta_4)$};
\end{scope}
\end{scope}
\end{tikzpicture}
\caption{Four formal dissections of a disk with three orbifold points and one boundary stop.}
\label{fig:dgdissections}
\end{figure}

Since the $\Delta_i$'s are DG dissections, it follows that the A$_\infty$ algebras (categories) $\mathbf A_{\Delta_i}$'s are DG algebras. We now describe the five DG algebras and their cohomology algebras explicitly. 

\noindent (1)\;
The DG algebra $\mathbf A_{\Delta_1}$ is given by the following DG quiver
\[
\begin{tikzpicture}[baseline=-2.6pt,description/.style={fill=white,inner sep=1pt,outer sep=0}]
\matrix (m) [matrix of math nodes, row sep=2.5em, text height=1.5ex, column sep=3em, text depth=0.25ex, ampersand replacement=\&, inner sep=3.5pt]
{
0 \& 1 \& 2 \& 3 \& 4 \& 5\\
};
\path[->,line width=.4pt, font=\scriptsize, transform canvas={yshift=.5ex}] (m-1-1) edge node[above=.01ex] {$q_1$} (m-1-2);
\path[->,line width=.4pt, font=\scriptsize, transform canvas={yshift=-.5ex}] (m-1-1) edge node[below=.01ex] {$p_1$} (m-1-2);
\path[->,line width=.4pt, font=\scriptsize] (m-1-2) edge node[below=.01ex] {$p_2$} (m-1-3);
\path[->,line width=.4pt, font=\scriptsize, transform canvas={yshift=.5ex}] (m-1-3) edge node[above=.01ex] {$q_3$} (m-1-4);
\path[->,line width=.4pt, font=\scriptsize, transform canvas={yshift=-.5ex}] (m-1-3) edge node[below=.01ex] {$p_3$} (m-1-4);
\path[->,line width=.4pt, font=\scriptsize] (m-1-4) edge node[below=.01ex] {$p_4$} (m-1-5);
\path[->,line width=.4pt, font=\scriptsize, transform canvas={yshift=.5ex}] (m-1-5) edge node[above=.01ex] {$q_5$} (m-1-6);
\path[->,line width=.4pt, font=\scriptsize, transform canvas={yshift=-.5ex}] (m-1-5) edge node[below=.01ex] {$p_5$} (m-1-6);
\end{tikzpicture}
\]
with the relations  $p_{2} q_1 = q_3p_2= p_4q_3 = q_5 p_4 = 0$ and the differential
\[
d(p_1) = q_1,\quad  d(p_3) = q_3, \quad d(p_5) = q_5.
\]
We choose a line field $\eta$ so that $|p_i|=0$ and $|q_j|=1$ for $1\leq i \leq 5$ and $j = 1, 3, 5$.

Note that the cohomology algebra $A_1 = \H^\bullet(\mathbf A_{\Delta_1})$ is isomorphic to the following ungraded skew-gentle algebra.
\[
\begin{tikzpicture}[baseline=-2.6pt,description/.style={fill=white,inner sep=1pt,outer sep=0}]
\matrix (m) [matrix of math nodes, row sep=2.5em, text height=1.5ex, column sep=3em, text depth=0.25ex, ampersand replacement=\&, inner sep=3.5pt]
{
0 \& 2 \& 4 \\
1 \& 3 \& 5\\
};
\path[->,line width=.4pt, font=\scriptsize] (m-1-1) edge node[above=.01ex] {$p_{12}$} (m-1-2);
\path[->,line width=.4pt, font=\scriptsize] (m-1-1) edge[] node[xshift=-4ex, yshift=1.5ex] {$p_{13}$} (m-2-2);
\path[->,line width=.4pt, font=\scriptsize] (m-2-1) edge[] node[xshift=-4ex, yshift=-1.5ex] {$p_2$} (m-1-2);
\path[->,line width=.4pt, font=\scriptsize] (m-2-1) edge node[below=.01ex] {$p_{23}$} (m-2-2);
\path[->,line width=.4pt, font=\scriptsize] (m-1-2) edge node[above=.01ex] {$p_{34}$} (m-1-3);
\path[->,line width=.4pt, font=\scriptsize] (m-1-2) edge[] node[xshift=-.5ex, yshift=-2.3ex] {$p_{4}$} (m-2-3);
\path[->,line width=.4pt, font=\scriptsize] (m-2-2) edge[] node[xshift=-.5ex, yshift=2.3ex] {$p_{35}$} (m-1-3);
\path[->,line width=.4pt, font=\scriptsize] (m-2-2) edge node[below=.01ex] {$p_{45}$} (m-2-3);
\end{tikzpicture}
\]
Namely, the relations are given by the four commutative squares. Here the arrow $p_{ij}$ represents the cocycle $[p_{j} \dotsb p_{i+1}p_i]$. 

\noindent{(2)}\; The DG algebra $\mathbf A_{\Delta_2}$ is given by the following DG quiver
\[
\begin{tikzpicture}[baseline=-2.6pt,description/.style={fill=white,inner sep=1pt,outer sep=0}]
\matrix (m) [matrix of math nodes, row sep=2.5em, text height=1.5ex, column sep=3em, text depth=0.25ex, ampersand replacement=\&, inner sep=3.5pt]
{
0 \& 1 \& 2 \& 3 \& 4 \& 5\\
};
\path[->,line width=.4pt, font=\scriptsize] (m-1-1) edge node[above=.01ex] {$q_1$} (m-1-2);
\path[->,line width=.4pt, font=\scriptsize, transform canvas={yshift=-.5ex}] (m-1-2) edge node[below=.01ex] {$q_2$} (m-1-3);
\path[->,line width=.4pt, font=\scriptsize, transform canvas={yshift=.5ex}] (m-1-2) edge node[above=.01ex] {$r_1$} (m-1-3);
\path[->,line width=.4pt, font=\scriptsize, transform canvas={yshift=.5ex}] (m-1-4) edge node[above=.01ex] {$r_2$} (m-1-5);
\path[->,line width=.4pt, font=\scriptsize, transform canvas={yshift=-.5ex}] (m-1-4) edge node[below=.01ex] {$q_4$} (m-1-5);
\path[->,line width=.4pt, font=\scriptsize] (m-1-3) edge node[below=.01ex] {$q_3$} (m-1-4);
\path[->,line width=.4pt, font=\scriptsize] (m-1-5) edge node[below=.01ex] {$q_5$} (m-1-6);
\path[->,line width=.4pt, font=\scriptsize,bend right=30] (m-1-1) edge node[below=.01ex] {$p_1$} (m-1-6);
\end{tikzpicture}
\]
with the relations  $r_1 q_1 = q_3 r_1= r_2q_3 = q_5 r_2 = 0$ and the differential
\[
d(p_1) = q_5\dotsb q_2 q_1,\quad  d(q_2) = r_1, \quad d(q_4) = r_2.
\]
Here, $|p_1|=-1$ and $|q_i|=0$ for $1\leq i \leq 4$ and $|r_1|=|r_2|=1$.

The cohomology algebra $A_2 = \H^\bullet(\mathbf A_{\Delta_2})$ is isomorphic to the following ungraded algebra
\[
\begin{tikzpicture}[baseline=-2.6pt,description/.style={fill=white,inner sep=1pt,outer sep=0}]
\matrix (m) [matrix of math nodes, row sep=1.2em, text height=1.5ex, column sep=3em, text depth=0.25ex, ampersand replacement=\&, inner sep=3.5pt]
{
\& 1 \& 3 \\
0 \& \& \&5 \\
 \& 2 \& 4\\
};
\path[->,line width=.4pt, font=\scriptsize] (m-2-1) edge node[above=.01ex] {$q_{1}$} (m-1-2);
\path[->,line width=.4pt, font=\scriptsize] (m-2-1) edge node[below=.01ex] {$q_{12}$ } (m-3-2);
\path[->,line width=.4pt, font=\scriptsize] (m-1-2) edge node[above=.01ex] {$q_{23}$} (m-1-3);
\path[->,line width=.4pt, font=\scriptsize] (m-1-2) edge[] node[xshift=-4ex, yshift=2.3ex] {$q_{24}$} (m-3-3);
\path[->,line width=.4pt, font=\scriptsize] (m-3-2) edge node[below=.01ex] {$q_{34}$} (m-3-3);
\path[->,line width=.4pt, font=\scriptsize] (m-3-2) edge[] node[xshift=-4ex, yshift=-2.3ex]  {$q_3$} (m-1-3);
\path[->,line width=.4pt, font=\scriptsize] (m-1-3) edge node[above=.01ex] {$q_{45}$} (m-2-4);
\path[->,line width=.4pt, font=\scriptsize] (m-3-3) edge node[below=.01ex] {$q_5$} (m-2-4);
\end{tikzpicture}
\] 
where the relations are given by that all the squares are commutative and nonzero, and that all the paths of length $3$ are zero (e.g.\ $q_{45} q_{23} q_1= 0$). Note that $A_2$ is not a skew-gentle algebra since there are cubic relations.

\noindent{(3)}\; The DG algebra $\mathbf A_{\Delta_3}$ is given by the following DG quiver
\[
\begin{tikzpicture}[baseline=-2.6pt,description/.style={fill=white,inner sep=1pt,outer sep=0}]
\matrix (m) [matrix of math nodes, row sep=1.5em, text height=1.5ex, column sep=3em, text depth=0.25ex, ampersand replacement=\&, inner sep=3.5pt]
{
\&\&\& 4 \&\\
0 \& 1 \& 2 \&3 \& 5\\
};
\path[->,line width=.4pt, font=\scriptsize] (m-2-1) edge node[above=.01ex] {$q_1$} (m-2-2);
\path[->,line width=.4pt, font=\scriptsize] (m-2-2) edge node[above=.01ex] {$q_2$} (m-2-3);
\path[->,line width=.4pt, font=\scriptsize] (m-2-3) edge node[above=.01ex] {$r_1$} (m-1-4);
\path[->,line width=.4pt, font=\scriptsize] (m-2-4) edge node[right=.01ex] {$r_3$} (m-1-4);
\path[->,line width=.4pt, font=\scriptsize] (m-2-3) edge node[above=.01ex] {$q_3$} (m-2-4);
\path[->,line width=.4pt, font=\scriptsize] (m-2-4) edge node[above=.01ex] {$q_4$} (m-2-5);
\path[->,line width=.4pt, font=\scriptsize,bend right=30] (m-2-2) edge node[below=.01ex] {$r_2$} (m-2-4);
\path[->,line width=.4pt, font=\scriptsize,bend right=30] (m-2-1) edge node[below=.01ex] {$p_1$} (m-2-5);
\end{tikzpicture}
\]
with the relations  $r_2 q_1 = q_4 r_2 = 0,  r_1q_2= r_3r_2, r_3q_3=r_1$ and the differential $d(p_1) = q_4q_3q_2 q_1.$
Here, $|p_1|=-1$ and $|q_i|=|r_j|=0$ for $1\leq i \leq 4$ and $1\leq j \leq 3$.

The cohomology algebra $A_3 = \H^\bullet(\mathbf A_{\Delta_3})$ is isomorphic to the following ungraded algebra 
\[
\begin{tikzpicture}[baseline=-2.6pt,description/.style={fill=white,inner sep=1pt,outer sep=0}]
\matrix (m) [matrix of math nodes, row sep=1.5em, text height=1.5ex, column sep=3em, text depth=0.25ex, ampersand replacement=\&, inner sep=3.5pt]
{
\&\&\& 4 \&\\
0 \& 1 \& 2 \&3 \& 5\\
};
\path[->,line width=.4pt, font=\scriptsize] (m-2-1) edge node[above=.01ex] {$q_1$} (m-2-2);
\path[->,line width=.4pt, font=\scriptsize] (m-2-2) edge node[above=.01ex] {$q_2$} (m-2-3);
\path[->,line width=.4pt, font=\scriptsize] (m-2-4) edge node[right=.01ex] {$r_3$} (m-1-4);
\path[->,line width=.4pt, font=\scriptsize] (m-2-3) edge node[above=.01ex] {$q_3$} (m-2-4);
\path[->,line width=.4pt, font=\scriptsize] (m-2-4) edge node[above=.01ex] {$q_4$} (m-2-5);
\path[->,line width=.4pt, font=\scriptsize,bend right=30] (m-2-2) edge node[below=.01ex] {$r_2$} (m-2-4);
\end{tikzpicture}
\]
with the relations $r_2 q_1 = q_4 r_2 = 0$, $r_3(q_3q_2-r_2)= 0$ and $q_4q_3q_2 q_1=0$. We note that here as in the previous case, the algebra $A_3$ is not skew-gentle. 

\noindent{(4)}\; The DG algebra $\mathbf A_{\Delta_4}$ is given by the following DG quiver
\[
\begin{tikzpicture}[baseline=-2.6pt,description/.style={fill=white,inner sep=1pt,outer sep=0}]
\matrix (m) [matrix of math nodes, row sep=2.5em, text height=1.5ex, column sep=3em, text depth=0.25ex, ampersand replacement=\&, inner sep=3.5pt]
{
0 \& 1 \& 2 \& 3 \& 4 \& 5\\
};
\path[->,line width=.4pt, font=\scriptsize, transform canvas={yshift=.5ex}] (m-1-1) edge node[above=.01ex] {$r_1$} (m-1-2);
\path[->,line width=.4pt, font=\scriptsize, transform canvas={yshift=-.5ex}] (m-1-1) edge node[below=.01ex] {$q_1$} (m-1-2);
\path[->,line width=.4pt, font=\scriptsize] (m-1-2) edge node[below=.01ex] {$q_2$} (m-1-3);
\path[->,line width=.4pt, font=\scriptsize,transform canvas={yshift=-.5ex}] (m-1-3) edge node[below=.01ex] {$q_3$} (m-1-4);
\path[->,line width=.4pt, font=\scriptsize, transform canvas={yshift=.5ex}] (m-1-4) edge node[above=.01ex] {$q_4$} (m-1-3);
\path[->,line width=.4pt, font=\scriptsize] (m-1-4) edge node[below=.01ex] {$p_1$} (m-1-5);
\path[->,line width=.4pt, font=\scriptsize, transform canvas={yshift=.5ex}] (m-1-5) edge node[above=.01ex] {$r_2$} (m-1-6);
\path[->,line width=.4pt, font=\scriptsize, transform canvas={yshift=-.5ex}] (m-1-5) edge node[below=.01ex] {$p_2$} (m-1-6);
\end{tikzpicture}
\] 
with the relations  $q_2 r_1 = q_4q_3q_2q_1, \  q_3 q_4 =p_1 q_3 = r_2 p_1 = 0$ and the differentials $$d(q_1) = r_1, \quad d(q_2q_1) = q_2r_1,\quad d(p_2) = r_2.$$ Here, $|r_1|=|r_2|=1$ and all the other arrows are of degree $0$.

The cohomology algebra $A_4 = \H^\bullet (\mathbf A_{\Delta_4})$ is isomorphic to the following ungraded algebra 
\[
\begin{tikzpicture}[baseline=-2.6pt,description/.style={fill=white,inner sep=1pt,outer sep=0}]
\matrix (m) [matrix of math nodes, row sep=2.5em, text height=1.5ex, column sep=3em, text depth=0.25ex, ampersand replacement=\&, inner sep=3.5pt]
{
0 \& 1 \& 2 \& 3 \& 4 \& 5\\
};
\path[->,line width=.4pt, font=\scriptsize, bend right=35] (m-1-1) edge node[below=.01ex] {$q_{13}$} (m-1-4);
\path[->,line width=.4pt, font=\scriptsize] (m-1-2) edge node[below=.01ex] {$q_2$} (m-1-3);
\path[->,line width=.4pt, font=\scriptsize,transform canvas={yshift=-.5ex}] (m-1-3) edge node[below=.01ex] {$q_3$} (m-1-4);
\path[->,line width=.4pt, font=\scriptsize, transform canvas={yshift=.5ex}] (m-1-4) edge node[above=.01ex] {$q_4$} (m-1-3);
\path[->,line width=.4pt, font=\scriptsize] (m-1-4) edge node[above=.01ex] {$p_1$} (m-1-5);
\path[->,line width=.4pt, font=\scriptsize, bend right=35] (m-1-4) edge node[below=.01ex] {$p_{12}$} (m-1-6);
\end{tikzpicture}
\] 
with the quadratic monomial relations 
\begin{align*}
q_4 q_{13} = p_{12} q_{13} = p_1 q_{13} = q_3q_4= p_1q_3=p_{12} q_3 =  0.
\end{align*}  
Here, the arrow $q_{13}$ (resp.\ $p_{12}$) represents the cocycle $q_3q_2q_1$ (resp.\ $p_2p_1$) and the relation $q_4q_{13}=0$ follows since $q_4q_3q_2q_1 =q_2r_1$ is a coboundary. 

Since the above four dissections are formal, it follows from Theorem \ref{theorem:formaldg} and Theorem \ref{theorem:dgtriangleequivalent} that
\[
\per (A_1) \simeq \per (A_2) \simeq \per (A_3) \simeq \per (A_4) \simeq \mathcal W (\mathbf S).
\]
In particular, the ungraded algebras $A_1, A_2, A_3$ and $A_4$ are (perfect) derived equivalent.

\noindent{(5)}\; Finally, we also give an example of a DG dissection of $\mathbf S$ which is not formal. Let $\Delta_5$ be the dissection illustrated in Fig.~\ref{fig:dgnonformaldissection}. 

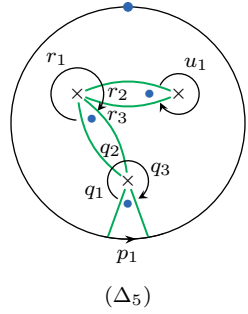
\begin{figure}[ht]
\begin{tikzpicture}[x=1em,y=1em,decoration={markings,mark=at position 0.55 with {\arrow[black]{Stealth[length=4.2pt]}}}]
\draw[line width=.5pt] circle(4em);
\node[font=\scriptsize,shape=circle,scale=.6,fill=white] (X) at (30:2em) {};
\node[font=\scriptsize] at (30:2em) {$\times$};
\node[font=\scriptsize,shape=circle,scale=.6,fill=white] (Y) at (-90:2em) {};
\node[font=\scriptsize] at (-90:2em) {$\times$};
\node[font=\scriptsize,shape=circle,scale=.6,fill=white] (Z) at (150:2em) {};
\node[font=\scriptsize] at (150:2em) {$\times$};
\draw[-, line width=.75pt,color=arccolour, bend left=22] (Y) to (Z);
\draw[-,line width=.75pt,color=arccolour, bend right=22] (Y) to (Z);
\draw[-, line width=.75pt,color=arccolour, bend left=22] (X) to (Z);
\draw[-,line width=.75pt,color=arccolour, bend right=22] (X) to (Z);
\draw[->, line width=.5pt] ($(30:2em)+(155:.7em)$) arc[start angle=155, end angle=-155, radius=.7em];
\node[font=\scriptsize] at ($(30:2em)+(60:1.2em)$) {$u_1$};
\node[font=\scriptsize,color=stopcolour] at ($(30:2em)+(180:1em)$) {$\bullet$};
\draw[->, line width=.5pt] ($(-90:2em)+(240:.7em)$) arc[start angle=240, end angle=-60, radius=.7em];
\node[font=\scriptsize] at ($(-90:2em)+(200:1.2em)$){$q_1$};
\node[font=\scriptsize] at ($(-90:2em)+(120:1.2em)$){$q_2$};
\node[font=\scriptsize] at ($(-90:2em)+(20:1.2em)$){$q_3$};
\node[font=\scriptsize,color=stopcolour] at (-90:2.8em) {$\bullet$};
\draw[->, line width=.5pt] ($(150:2em)+(268:.9em)$) arc[start angle=268, end angle=-40, radius=.9em];
\node[font=\scriptsize] at ($(150:2em)+(120:1.4em)$) {$r_1$};
\node[font=\scriptsize] at ($(150:2em)+(0:1.4em)$) {$r_2$};
\node[font=\scriptsize] at ($(150:2em)+(-30:1.6em)$) {$r_3$};
\node[font=\scriptsize,color=stopcolour] at ($(150:2em)+(-60:1em)$) {$\bullet$};
\draw[line width=.75pt,color=arccolour] (-100:4em) to (Y);
\draw[line width=.75pt,color=arccolour] (-80:4em) to (Y);
\draw[line width=0pt,postaction={decorate}] (240+17:4em) arc[start angle=240+17, end angle=300-17+5, radius=4em];
\node[font=\scriptsize] at (270:4.6em) {$p_1$};
\draw[fill=stopcolour, color=stopcolour] (90:4em) circle(.15em);
\node[font=\scriptsize] at (-90:6em) {$(\Delta_5)$};
\end{tikzpicture}
\caption{A DG and non-formal dissection.}
\label{fig:dgnonformaldissection}
\end{figure}

The DG algebra $\mathbf A_{\Delta_5}$ is given by the following DG quiver
\[
\begin{tikzpicture}[baseline=-2.6pt,description/.style={fill=white,inner sep=1pt,outer sep=0}]
\matrix (m) [matrix of math nodes, row sep=1.5em, text height=1.5ex, column sep=3em, text depth=0.25ex, ampersand replacement=\&, inner sep=3.5pt]
{
  \&   \& 2 \& 3 \&        \\
0 \& 1 \&   \&   \& 4 \& 5 \\
};
\path[->,line width=.4pt, font=\scriptsize] (m-2-1) edge node[above=.01ex] {$q_1$} (m-2-2);
\path[->,line width=.4pt, font=\scriptsize] (m-2-2) edge node[above=.01ex] {$q_2$} (m-2-5);
\path[->,line width=.4pt, font=\scriptsize] (m-2-2) edge node[above=.01ex] {$r_1$} (m-1-3);
\path[->,line width=.4pt, font=\scriptsize] (m-1-3) edge node[below=.01ex] {$r_2$} (m-1-4);
\path[->,line width=.4pt, font=\scriptsize] (m-1-4) edge node[above=.01ex] {$r_3$} (m-2-5);
\path[->,line width=.4pt, font=\scriptsize] (m-2-5) edge node[above=.01ex] {$q_3$} (m-2-6);
\path[->,line width=.4pt, font=\scriptsize,transform canvas={yshift=.5ex}] (m-1-3) edge node[above=.01ex] {$u_1$} (m-1-4);
\path[->,line width=.4pt, font=\scriptsize,bend right=20] (m-2-1) edge node[below=.01ex] {$p_1$} (m-2-6);
\end{tikzpicture}
\]
with the relations  $r_1 q_1 = u_1 r_1= r_3 u_1 = q_3 r_3 = 0$ and the differentials $$d(p_1) = q_3q_2 q_1, \quad d(q_3) = r_3r_2r_1,\quad d(r_2) = u_1.$$ Here, $|p_1|=-1, |r_2|=1, |u_1|=2$ and all the other arrows are of degree $0$.

The cohomology algebra $A_5 = \H^\bullet(\mathbf A_{\Delta_5})$ is isomorphic to the following graded algebra 
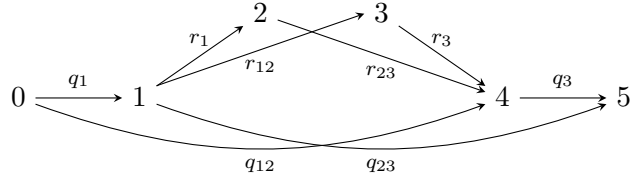
\[
\begin{tikzpicture}[baseline=-2.6pt,description/.style={fill=white,inner sep=1pt,outer sep=0}]
\matrix (m) [matrix of math nodes, row sep=1.5em, text height=1.5ex, column sep=3em, text depth=0.25ex, ampersand replacement=\&, inner sep=3.5pt]
{
  \&   \& 2 \& 3 \&        \\
0 \& 1 \&   \&   \& 4 \& 5 \\
};
\path[->,line width=.4pt, font=\scriptsize] (m-2-1) edge node[above=.01ex] {$q_1$} (m-2-2);
\path[->,line width=.4pt, font=\scriptsize] (m-2-2) edge node[above=.01ex] {$r_1$} (m-1-3);
\path[->,line width=.4pt, font=\scriptsize] (m-1-3) edge node[below=.01ex] {$r_{23}$} (m-2-5);
\path[->,line width=.4pt, font=\scriptsize] (m-1-4) edge node[above=.01ex] {$r_3$} (m-2-5);
\path[->,line width=.4pt, font=\scriptsize] (m-2-5) edge node[above=.01ex] {$q_3$} (m-2-6);
\path[->,line width=.4pt, font=\scriptsize,transform canvas={yshift=.5ex}] (m-2-2) edge node[below=.01ex] {$r_{12}$} (m-1-4);
\path[->,line width=.4pt, font=\scriptsize,bend right=20] (m-2-1) edge node[left=.03ex, below=.01ex] {$q_{12}$} (m-2-5);
\path[->,line width=.4pt, font=\scriptsize,bend right=20] (m-2-2) edge node[right=.03ex, below=.01ex] {$q_{23}$} (m-2-6);
\end{tikzpicture}
\]
where the relations are given by 
\begin{align*}
r_{3}r_{12}-r_{23}r_1=0, \quad r_1 q_1 = r_{12} q_1= q_3 r_{23} = q_3 r_3 = q_{23}q_1=q_3q_{12}=0
\end{align*}  
Note that $|r_{12}|=1=|r_{23}|$ and all the other arrows are of degree $0$. It follows from Theorem \ref{theorem:formaldg} that $\mathbf A_{\Delta_5}$ is not formal and we obtain an A$_\infty$ structure on $A_5$ given by 
\begin{align*}
(-1)^{|\s q_1|+|\s r_{12}|} \mu_3(\s r_3\otimes \s r_{12}\otimes \s q_1) &= \s q_{12}=(-1)^{|\s q_1|+|\s r_{1}|}\mu_3(\s r_{23}\otimes \s r_1\otimes \s q_1)\\
(-1)^{|\s r_{3}|}\mu_3(\s q_3\otimes\s r_3\otimes \s r_{12}) &= \s q_{23} =(-1)^{|\s r_{23}|} \mu_3(\s q_3\otimes \s r_{23}\otimes \s r_1). 
\end{align*}
With this higher structure $A_5$ is (perfect) derived equivalent to the ungraded algebras $A_1, A_2, A_3, A_4$.

\begin{remark}
The algebras $A_2, A_3$ and $A_4$ 
are not skew-gentle and (as far as we can tell) they do not appear to belong to a specific known class of (derived tame) associative algebras. However, they arise from different dissections of the same graded orbifold surface and we conjecture that algebras obtained from formal dissections of orbifold surfaces form a new class of associative algebras which is closed under derived equivalence (Conjecture \ref{conjecture:skewgentle}).
\end{remark}

\section{Partially wrapped Fukaya categories as A$_\infty$ orbit categories}
\label{section:doublecover}

In this final section, we show that the partially wrapped Fukaya categories of graded orbifold surfaces with stops are equivalent to {\it A$_\infty$ orbit categories} (as defined in \S\ref{subsection:orbitcategories}) of any double cover.

Let $\mathbf S = (S, \Sigma, \eta)$ be a graded orbifold surface with stops. Recall that a {\it double cover} $\widetilde{\mathbf S} = (\widetilde S, \widetilde \Sigma, \widetilde \eta)$ of $\mathbf S$ is a graded smooth surface  consisting of a smooth surface with boundary $\widetilde S$ with an orientation-preserving almost free $\mathbb Z_2$-action such that, a $\mathbb Z_2$-invariant set $\widetilde \Sigma$ of stops and a $\mathbb Z_2$-invariant line fieldwith stops $(\widetilde S, \widetilde \Sigma, \widetilde \eta)$ together with an orientation-preserving $\mathbb Z_2$-action such that $\widetilde{\mathbf S}/\mathbb Z_2 \simeq \mathbf S$, compare \cite[\S 4]{thurston}. (The special case of the disk with one orbifold point was treated in Section \ref{section:disk}.)

Consider the partially wrapped Fukaya category $\mathcal W(\widetilde{\mathbf S})$ of $\widetilde{\mathbf S}$ introduced in \cite{haidenkatzarkovkontsevich}. Let $\widetilde \Delta$ be an admissible dissection of $\widetilde{\mathbf S}$. Denote by $\mathbf A_{\widetilde \Delta}$ the associated graded gentle algebra (viewed as a $\Bbbk$-linear graded category) so that
$$
\mathcal W(\widetilde{\mathbf S}) \simeq \H^0(\tw(\mathbf A_{\widetilde \Delta})).
$$
If $\widetilde \Delta$ is $\mathbb Z_2$-invariant then it induces a $\mathbb Z_2$-action on $\mathbf A_{\widetilde \Delta}$ and thus on $\mathcal W(\widetilde{\mathbf S})$. 

We now show that the partially wrapped Fukaya category $\mathcal W(\mathbf S)$ can be described as the orbit category of the category $\mathcal W(\widetilde{\mathbf S})$ with respect to the $\mathbb Z_2$-action. The following theorem can be viewed as a global-to-local counterpart to the local-to-global construction of the partially wrapped Fukaya category given in Section \ref{section:cosheaves}.

\begin{theorem}
\label{theorem:doublecoverorbitcategory}
Let $\mathbf S = (S, \Sigma, \eta)$ be a graded orbifold surface with stops and let $\widetilde{\mathbf S}$ be a double cover of $\mathbf S$. There exists a weakly admissible dissection $\Delta$ of $\mathbf S$ and $\mathbb Z_2$-invariant dissection $\widetilde \Delta$ of $\widetilde{\mathbf S}$ such that
\[
(\tw (\mathbf A_\Delta))^\natural \simeq (\tw (\mathbf A_{\widetilde \Delta}) / \mathbb Z_2)^\natural \simeq (\tw (\mathbf A_{\widetilde \Delta} / \mathbb Z_2))^\natural
\]
are equivalences of pretriangulated A$_\infty$ categories, the middle term being an A$_\infty$ orbit category.

Moreover, this induces a triangulated equivalence
\[
\mathcal W (\mathbf S) \simeq (\mathcal W (\widetilde{\mathbf S}) / \mathbb Z_2)^\natural.
\]
\end{theorem}

\begin{proof}
By Lemma \ref{lemma:4gon} (see also \S\ref{subsection:special}) there is an admissible dissection $\Delta_0$ of $\mathbf S$ such that there are exactly two arcs connecting to each orbifold point $x$ and this dissection is also formal (cf.\ Definition \ref{definition:DGformal}). Since $\Delta_0$ is formal it follows from Proposition \ref{theorem:formaldg} that $\mathbf A_{\Delta_0}$ is A$_\infty$ quasi-isomorphic to $\H^\bullet(\mathbf A_{\Delta_0})$, which is a graded skew-gentle algebra by Theorem \ref{theorem:skewgentlealgebras}. As a result, we obtain a triangle equivalence 
\[
\mathcal W (\mathbf S) \simeq \H^0( \tw(\mathbf A_{\Delta_0}))^\natural \simeq \H^0 (\tw (\H^\bullet (\mathbf A_{\Delta_0})))^\natural.
\]

Consider the lifting $\widetilde \Delta_0$ of $\Delta_0$ to the double cover $\widetilde{\mathbf S}$. Note that the two arcs connecting to an orbifold point $x$ lift to two isotopic arcs through a fixed point, which represent the same object of $\mathcal W(\widetilde{\mathbf S})$. Therefore for each orbifold point we remove one of the two arcs from $\widetilde \Delta_0$ so that the resulting arc system $\widetilde\Delta$ is an  admissible dissection of $\widetilde{\mathbf S}$. Then by \cite{haidenkatzarkovkontsevich} we obtain a triangle equivalence
\begin{align*}
\mathcal W (\widetilde{\mathbf S}) \simeq \H^0 (\tw (\mathbf A_{\widetilde \Delta})).
\end{align*}

Note that $\mathbf A_{\widetilde \Delta}/\mathbb Z_2$ is also a graded skew-gentle algebra (cf.\ \cite[Proposition 4.3]{amiotbruestle}) and moreover, there is a natural functor (morphism) $F\colon \mathbf A_{\widetilde \Delta} / \mathbb Z_2 \to \H^\bullet(\mathbf A_{\Delta_0})$ given by
$$
F(\widetilde\gamma) = \begin{cases} \gamma  & \text{if $\widetilde \gamma$ is a lift of $\gamma$ which is not $\mathbb Z_2$-invariant}\\ \gamma_x^+ \oplus \gamma_x^- & \text{if $\widetilde \gamma$ is a $\mathbb Z_2$-invariant lift of $\gamma$} \end{cases}
$$
which is an equivalence of categories. Here, $\gamma_x^+$ and $\gamma_x^-$ are the two arcs in $\Delta_0$ connecting to the orbifold point $x$. In particular, we obtain triangle equivalences 
\[
\H^0( \tw(\mathbf A_{\widetilde \Delta}) / \mathbb Z_2)^\natural \simeq \H^0( \tw(\mathbf A_{\widetilde \Delta} / \mathbb Z_2))^\natural \simeq  \H^0( \tw(\H^\bullet(\mathbf A_{\Delta_0}))^\natural,
\]
where the first equivalence follows from \eqref{eq:idempotent} and Propositions \ref{proposition:groupactiontriangulated} and \ref{proposition:triangulatedgroupaction} and the second from the isomorphism $F$. Combining the above triangle equivalences we obtain the desired one. 
\end{proof}

\begin{remark}
Recall from \cite{haidenkatzarkovkontsevich} that the indecomposable objects in $\mathcal W(\widetilde{\mathbf S})$ are described by homotopy classes of graded curves with local system on $\widetilde{\mathbf S}$. Theorem \ref{theorem:doublecoverorbitcategory} allows us to describe the indecomposable objects of $\mathcal W(\mathbf S)$ in terms of certain graded curves with local system on $\mathbf S$, see also \cite{qiuzhangzhou} for the graded case and \cite{labardinifragososchrollvaldivieso,amiot23} for the ungraded case. 
\end{remark}

\subsection*{Acknowledgements}

When preparing this manuscript, we greatly benefitted from conversations with Xiaofa Chen, Xiao-Wu Chen, Sheel Ganatra, Fabian Haiden, Bernhard Keller, Chris Kuo, Wenyuan Li, Alexander Polishchuk, Alekos Robotis and Yilin Wu whom we all warmly thank for discussions, comments and explanations. We are also indebted to Denis Auroux for his illuminating Fall 2016 Eilenberg Lectures at Columbia University made available online \cite{aurouxlectures}.

This work originated during the Junior Trimester Program ``New Trends in Representation Theory'' at the Hausdorff Research Institute for Mathematics in Bonn, funded by the Deutsche Forschungsgemeinschaft (DFG, German Research Foundation) under Germany's Excellence Strategy -- EXC-2047/1 -- 390685813. This material is also based upon work supported by the National Science Foundation under Grant No.~DMS-1928930 and by the Alfred P.~Sloan Foundation under grant G-2021-16778 while S.\,B.\ was in residence at the Simons Laufer Mathematical Sciences Institute (formerly MSRI) in Berkeley, California, during the Spring 2024 semester. We thank both institutes for their welcoming atmosphere and the excellent working environments that greatly contributed to this work. In addition, this work was partially supported by the DFG through the project SFB/TRR 191 Symplectic Structures in Geometry, Algebra and Dynamics (Projektnummer 281071066-TRR191) and through the grant WA 5157/1-1, and by the National Natural Science Foundation of China (Grant Nos.\ 12371043 and 12071137).

\end{document}